%% file: maxclass-revised-AMS.tex
\documentclass[draft]{memo-l}
\usepackage{amsmath, amsthm,amssymb, mathrsfs, amsfonts}
\usepackage{float}
\usepackage{graphicx}
\usepackage{arydshln}
\usepackage[normalem]{ulem}
\usepackage{cite}
\usepackage{moreenum}
\newtheorem{innercustomthm}{Theorem}
\newenvironment{customthm}[1]
  {\renewcommand\theinnercustomthm{#1}\innercustomthm}
  {\endinnercustomthm}

\newtheorem{innercustomcor}{Corollary}
\newenvironment{customcor}[1]
  {\renewcommand\theinnercustomcor{#1}\innercustomcor}
  {\endinnercustomcor}

\theoremstyle{definition}
\newtheorem{definition}{Definition}[section]

\newcounter{claim}[definition]

\usepackage{color}
\newcommand{\blue}[1]{\,{\color{black} #1}\,}

\theoremstyle{remark}
\newtheorem{remark}[definition]{Remark}

\theoremstyle{plain}

\theoremstyle{plain}
\newtheorem{cor}[definition]{Corollary}

\theoremstyle{plain}
\newtheorem{lemma}[definition]{Lemma}

\theoremstyle{plain}
\newtheorem{theorem}[definition]{Theorem}
\newtheorem{prop}[definition]{Proposition}

\newtheorem{example}[definition]{Example}

\theoremstyle{remark}
\newtheorem{notation}[definition]{Notation}

\theoremstyle{definition}

\newtheorem{hypothesis}[definition]{Hypothesis}

\newcommand {\VV}{\mathbf V}

\newcommand{\B}{\mathrm{B}}
\newcommand{\C}{{C}}

\newcommand{\E}{\mathcal{E}}
\newcommand{\F}{\mathcal{F}}
\newcommand{\G}{\mathrm{G}}

\newcommand{\K}{\mathcal{K}}
\newcommand{\N}{{N}}

\newcommand{\J}{\mathrm{J}}

\newcommand{\Pp}{\mathcal{P}}

\newcommand{\Aut}{\mathrm{Aut}}
\newcommand{\Hom}{\mathrm{Hom}}
\newcommand{\Inj}{\mathrm{Inj}}
\newcommand{\Mor}{\mathrm{Mor}}
\newcommand{\Out}{\mathrm{Out}}
\newcommand{\Inn}{\mathrm{Inn}}
\newcommand{\Syl}{\mathrm{Syl}}
\newcommand{\syl}{\mathrm{Syl}}

\newcommand{\Ly}{\operatorname{Ly}}
\newcommand{\HN}{\operatorname{HN}}
\newcommand{\Mo}{\operatorname{M}}

\newcommand{\GF}{\mathrm{GF}}
\newcommand{\GL}{\mathrm{GL}}
\newcommand{\Sp}{\mathrm{Sp}}
\newcommand{\SL}{\mathrm{SL}}
\newcommand{\PGL}{\mathrm{PGL}}
\newcommand{\PSL}{\mathrm{PSL}}

\newcommand{\PSU}{\mathrm{PSU}}
\newcommand{\PSp}{\mathrm{PSp}}
\newcommand{\Sym}{\mathrm{Sym}}
\newcommand{\Alt}{\mathrm{Alt}}

\newcommand{\Frob}{\mathrm{Frob}}

\newcommand{\foc}{\mathrm{foc}}
\newcommand{\hyp}{\mathrm{hyp}}

\newcommand{\norm}{\mathrel{\unlhd}}

\newcommand{\wt}{\widetilde}

\newcommand{\eps} {\epsilon}
\newcommand{\Z} { {Z}}

\newcommand{\ov} {\overline}
\newcommand{\agemO} {\mho}
\newcommand{\SmallGroup}{\mathrm {SmallGroup}}

\theoremstyle{definition}

\begin{document}

\frontmatter

\title[Saturated Fusion Systems on $p$-groups of Maximal  Class ]{Saturated Fusion Systems on $p$-groups of Maximal  Class }

 \author{Valentina Grazian}
 \address{
Department of Mathematics and Applications\\
University of Milano -- Bicocca\\
Via Roberto Cozzi 55, 20125 Milano \\ Italy
} \email{valentina.grazian@unimib.it}

\author{Chris Parker}

\address{
School of Mathematics\\
University of Birmingham\\
Birmingham B15 2TT\\
United Kingdom
} \email{c.w.parker@bham.ac.uk}

\thanks{The authors would like to thank the Isaac Newton Institute for Mathematical Sciences for support and hospitality during the programme \emph{Groups, representations and applications: new perspectives} when work on this paper was undertaken. This work was supported by \textbf{EPSRC Grant Number EP/R01 146 04/1}.}
\date{\today}

  \begin{abstract} For a prime number $p$, a finite $p$-group of order $p^n$ has maximal class if and only if it has nilpotency class $n-1$.
 Here we examine saturated fusion systems $\mathcal F$ on maximal class $p$-groups $S$ of order at least $p^4$.
 The Alperin-Goldschmidt Theorem for saturated fusion systems yields that $\mathcal F$ is entirely determined by the $\mathcal F$-automorphisms of its $\mathcal F$-essential subgroups and of $S$ itself. If an $\mathcal F$-essential subgroup either has order $p^2$ or is non-abelian of order $p^3$, then it is called an $\mathcal F$-pearl.
 The facilitating and technical theorem in this work shows that an $\mathcal F$-essential subgroup is either an $\mathcal F$-pearl, or one of two explicitly determined maximal subgroups of $S$.  This result is easy to prove if $S$ is a $2$-group and can be read from the work of   D{\'{\i}}az,   Ruiz, and   Viruel together with that of Parker and Semeraro when $p=3$. The main contribution is for $p \ge 5$ as in this case there is no classification of the maximal class $p$-groups.
    The main Theorem describes all the reduced saturated fusion systems on a maximal class $p$-group of order at least $p^4$ and follows from two more extensive theorems. These two  theorems  describe   all saturated fusion systems, not restricting to the reduced ones for example, on exceptional and non-exceptional maximal class $p$-groups respectively. As a corollary we have the easy to remember result that states that, if $O_p(\mathcal F)=1$, then  either $\mathcal F$ has $\mathcal F$-pearls or $S$ is isomorphic to a Sylow $p$-subgroup of $\mathrm G_2(p)$ with $p\ge 5$ and the fusion systems are explicitly described.
\end{abstract}
  \maketitle
  \tableofcontents
\renewcommand\leftmark{SATURATED FUSION SYSTEMS ON $p$-GROUPS OF MAXIMAL CLASS}
\mainmatter
\include{Introduction}

\include{GeneralGroups}

\include{MaxClass}

\include{Reps}

\include{FusionSystems}

\include{FSMaxClass1}

\include{EssMC}
\include{NotExcep}

\include{g1sexcep}

\include{Loc1}

\include{Loc2}

\include{Exceptional}

\include{NotExceptional}

\include{Proofs}

\include{Examples}
\appendix
\renewcommand{\thesection}{\Alph{section}}
\include{AppendixA}

\include{AppendixB}

\include{AppendixC}

\backmatter

\bibliographystyle{plain}

\end{document}

%% file: Introduction.tex
\section{Introduction and main results}
Let $p$ be a prime and   $S$ be a finite $p$-group.
A \emph{fusion system} on $S$ is a category which has  objects  the subgroups of $S$ and   morphisms which are injective group homomorphisms between the subgroups of $S$. A fusion system is \emph{saturated} if it satisfies  certain technical axioms  which are described in some detail \blue{in Section \ref{sec:PF}, where standard terminology used for saturated fusion systems can also be found.} \blue{Of special importance is the Alperin-Goldschmidt  Theorem, Theorem~\ref{t:alp}, which says that  a saturated fusion system $\F$ on $S$ is entirely determined by  the $\F$-automorphisms of $S$, and the $\F$-automorphism groups of the so-called   $\F$-essential subgroups.} All  finite groups $G$ determine a saturated fusion system $\F_S(G)$ on  a fixed Sylow $p$-subgroup $S$ of $G$. In this case the morphisms between subgroups of $S$ are exactly the restrictions of conjugation maps induced by elements of $G$. When a saturated fusion system can be constructed from a group in this way, we say the system  is \emph{realizable}. A saturated fusion system  which is not realizable is called \emph{exotic}.

\blue{A $p$-group has maximal class if  it has order $p^n$ and  nilpotency class $n-1$.
The intention of this work is to study the structure of saturated fusion systems $\F$ on maximal class $p$-groups of order at least $p^4$. Our main theorems precisely describe the $\F$-essential subgroups of $\F$ and their $\F$-automorphism groups.
In contrast with other works on fusion systems, we will consider all saturated fusion systems on the selected class of $p$-groups, not limiting our investigation, for instance, to reduced saturated fusion systems. Aside from proving a broader result, our decision to do this  will become evident when we discuss the inductive approach to the proof of our theorems.}

\subsection{Historical context}
The categories we call saturated fusion systems were first introduced by Puig in the early 1990s and recorded in handwritten notes which developed much of the fundamental theory.   Puig did not formally publish his discoveries until 2006 \cite{Pg} and in this work, for a $p$-group $S$, a \emph{Frobenius $S$-category} is exactly what we call a saturated fusion system on $S$.
 Puig's motivation  was to gain a deeper understanding of the local to global conjectures which are of fundamental importance in modular representation theory. In particular, saturated fusion systems can be constructed on the defect group of a $p$-block in much the same way as they are constructed on Sylow $p$-subgroups. This connection is described by Kessar in \cite[Part IV]{AKO}. Recent developments in this direction include work of Kessar, Linckelmann, Lynd and Semeraro \cite{KLLS} in which certain numerical conjectures in representation theory are formulated for fusion systems; in particular, for exotic fusion systems. This is exciting as, to-date, no saturated fusion system for an odd prime $p$ has been shown to be exotic without using the classification of finite simple groups and the research in \cite{KLLS} offers potential for achieving this without the classification.  Finding an alternative way to show that exotic systems are exotic is listed as \cite[Question 7.7]{AO} in the survey by Aschbacher and Oliver.

 The  christening  of Frobenius $S$-categories   as saturated fusion systems is  traced back to \cite{BLO2} where in 2003 Broto, Levi and Oliver defined a \emph{centric linking system}  related to a saturated fusion system. Later in 2013, Chermak \cite{chermak} demonstrated that   each saturated fusion system determines a unique centric linking system. From the linking system of a saturated fusion system Broto, Levi and Oliver  were able to construct the $p$-completion of its geometric realization and this is the classifying space associated with the saturated fusion system. The homotopy properties of these spaces share properties of $p$-completed classifying spaces of finite groups. Readers are directed to \cite[Part III]{AKO} for details and motivation for this study.

 The theory of fusion systems provides an idealized environment in which to study the $p$-local structure of a finite group $G$ and, in particular, the structure of normalizers of non-trivial  $p$-subgroups of $G$. This formalism has been exploited most notably and energetically by Aschbacher in a deep series of papers aimed at simplifying some parts of the classification of finite simple groups. This approach is described in the surveys \cite{AStA,AKO,AO}. Other notable contributions to this goal are \cite{2rank4} in which Oliver determines reduced fusion systems on $2$-groups in which every subgroup is generated by at most $4$ elements, Andersen, Oliver and Ventura \cite{AOVsmall} which determines, using computational assistance, all reduced fusion systems on $2$-groups of order at most $2^9$ and Henke and Lynd \cite{henke2018fusion} where they consider saturated  fusion systems with components related to the Solomon fusion systems.\blue{ Using specially developed computational methods, Parker and Semeraro \cite[Main Result]{parkersemerarocomputing} have explicitly enumerated all saturated fusion systems with $O^p(\F)=\F$ and $O_p(\F)=1$ on $p$-groups of order $p^n$ with $p^n\in \{3^4,3^5,3^6,3^7,5^4,5^5, 5^6,7^4,7^5\}$. Their {\sc Magma} code for computing with fusion systems  is available publicly  \cite{PSGIT} and is used for some of the computations in this article (see Appendix~\ref{AppC}).}

The Solomon fusion systems \cite[III.6.7]{AKO}, which are supported on a Sylow $2$-subgroup of $\mathrm{Spin}_7(q)$, $q$ odd, are exotic. These are the only known  exotic fusion systems on $2$-groups and are still a fascinating subject of research \cite{henke2018fusion}. In contrast, there is a large variety  of exotic fusion systems on $p$-groups with $p$ odd and the reason for this is still not transparent.
Ruiz \cite{Ruiz} and more recently Oliver and Ruiz \cite{oliver2020simplicity} have considered non-abelian simple groups and determined instances where the fusion system $\F=\F_S(G)$ has $O^{p'}(\F)$ exotic. These fusion systems, from their very construction are closely related to finite simple groups. Other exotic systems  are constructed as fusion systems of free amalgamated products. \blue{Examples of these can be found} in \cite{BLO2,ClellandParker2010,ParkerStroth2015} and these are often far away from being realized by finite groups in that the $p$-groups are usually not closely related to the Sylow $p$-subgroup of a finite simple group. Many of these fusion systems are defined on  $p$-groups of \emph{maximal class}.

\blue{In the early 2000's  Ruiz and Viruel \cite{RV} famously handled the case when $\F$ is defined on a non-abelian $p$-group of order $p^3$. These  are maximal class $p$-groups. The shocking outcome was the discovery of three exotic fusion systems   when $p=7$.}  In their work on rank $2$ groups (groups which have no elementary abelian subgroups of order $p^3$ and are not cyclic or quaternion), D{\'{\i}}az, Ruiz and Viruel \cite{DRV} examine saturated fusion systems on the infinite families as classified by Blackburn \cite{black}. The outcome of their research was the discovery of several infinite  families of exotic fusion systems on certain of the maximal class $3$-groups. In fact, every maximal class $3$-group except the Sylow $3$-subgroup of the simple group $\Alt(9)$ has rank $2$.  In all of the examples they discovered, the $3$-group has an abelian subgroup of index $3$ \cite{black}. In 2019, Parker and Semeraro, using  their computational approach to saturated  fusion systems \cite{parkersemerarocomputing} uncovered a saturated fusion system on  a rank $2$ group of order $3^6$ which has maximal class and no abelian subgroups of index $3$.  This gave rise to the article  \cite{PSrank2}. \blue{A description of all the saturated fusion systems on maximal class $3$-groups is provided in Appendix \ref{AppB}.} From a different direction, Clelland and Parker \cite{ClellandParker2010} constructed families of saturated fusion systems on  a Sylow $p$-subgroup $T$ of groups of shape $q^a{:}\SL_2(q)$ where $2\le a \le p$ and $q=p^b$. When $a= 2$, $T$ is a Sylow $p$-subgroup of $\PSL_3(q)$ and, for $a=3$, $T$ is a Sylow $p$-subgroup of $\PSp_4(q)$.  For $a> 3$ and for $p \ge 5$, the fusion systems $\F$ with $O_p(\F)=1$ discovered in \cite{ClellandParker2010}  are typically exotic. If $q=p$, then $T$ has maximal class. This construction therefore yields infinite families of exotic saturated fusion systems on such maximal class $p$-groups.  In each case,  the underlying $p$-group has an abelian subgroup of index $p$.  In a remarkable series of articles, Oliver takes this property as his starting point and in \cite{p.index, p.index2, p.index3} he, and his co-authors, determine the reduced saturated fusion systems on $p$-groups with an abelian subgroup of index $p$.   A compilation of their results when applied to $p$-groups of maximal class is provided in  Appendix~\ref{AppA}.  It turns out that these saturated fusion systems are overwhelmingly exotic. This perhaps leaves the impression that for odd primes   exotic fusion systems are not exotic at all. This is possibly  an illusion created by  considering groups which are in some way small.  Evidence that exotic fusion systems may be exotic after all comes, for example, from \cite{Onofrei} where it is shown that, in certain good situations,   a saturated fusion system determines a locally finite classical Tits chamber system and so is not exotic. Work of van Beek generalizing the classification of groups with a weak $BN$-pair has also not revealed any surprises \cite{martin}.

Grazian has classified the saturated fusion systems on $p$-groups of rank $3$ for $p \ge 5$ \cite{rank3}. Unlike in the earlier work of D{\'{\i}}az, Ruiz and Viruel for rank $2$ groups, there is no list of groups to examine such as those given in \cite{black} and her methods invoke deep results from group theory developed for the classification of the finite simple groups. The resulting theorem reveals that saturated fusion systems on such groups are  realizable  with just one isolated exotic example on a maximal class $7$-group of order $7^5$.

\blue{Moragues Moncho \cite{Raul} classified all saturated fusion systems $\F$ on $p$-groups $S$ with an extraspecial subgroup of index $p$ and $O_p(\F)=1$.  The resulting theorem, which is required for the proof of our Theorem~\ref{Raul app}, is much more uniform than the result concerning $p$-groups with an abelian subgroup of index $p$. In particular, he shows that if the $p$-group considered has order at least $p^7$ then it has maximal class and it is  a uniquely determined group of order $p^{p-1}$. The fusion systems uncovered are closely related to those found by Parker and Stroth \cite{ParkerStroth2015} and are all exotic.}

\subsection{The main theorems}

The first step towards a classification of the saturated fusion systems $\F$ on a class of $p$-groups is the study of the so-called \emph{$\F$-essential subgroups} (see Definition \ref{def:ess}). The set of all $\F$-essential subgroups is denoted by $\E_\F$. In Grazian's research,  a certain special type of $\F$-essential subgroups played an important role.  She named  these subgroups \emph{$\F$-pearls}. They are defined as follows:

\begin{definition}\label{pearl:def} Let $p$ be a prime and   $\F$ be a saturated fusion system on a $p$-group $P$. An $\F$-essential subgroup $E$ of $P$ is called an \emph{$\F$-pearl} if it is isomorphic to either the elementary abelian group of order $p^2$,    or the extraspecial group of exponent $p$ and order $p^3$ if $p$ is odd or the quaternion of order $8$ if $p=2$. We denote the set  of abelian $\F$-pearls by $\mathcal P_a(\F)$, the set of extraspecial  $\F$-pearls by $\mathcal P_e(\F)$ and we write $\mathcal P(\F)= \mathcal P_a(\F) \cup \mathcal P_e(\F)$. \end{definition}

In \cite{pearls}, Grazian develops many of the fundamental properties of saturated fusion systems\blue{ $\F$ which have $\F$-pearls, one of the most basic being that the $p$-group on which $\F$ is defined must have maximal class.}  She also presents a fascinating lemma \cite[Lemma 3.7]{pearls} (included  here as Lemma~\ref{pearls4}), that constructs  proper saturated subfusion systems with $\F$-pearls  of any saturated fusion system that has $\F$-pearls.  These subsystems are frequently exotic and if the $\F$-pearls are\blue{abelian the systems $O^p(\F)$ are simple \cite[Definition I.6.1]{AKO} (see also Theorem~\ref{non exc}).} \blue{The exotic simple saturated fusion system on the group of order $7^5$ discovered in \cite{rank3}  can be constructed \blue{using Lemma~\ref{pearls4}} from the saturated fusion system of the monster sporadic simple group at the prime $7$.}

It is important to clarify that not all saturated fusion systems on a maximal class $p$-group contain $\F$-pearls. The determination of saturated fusion systems $\F$ on $S$ with $O_p(\F)=1$ and without $\F$-pearls is a consequence of the main results of this work and is recorded as Corollary~\ref{cor: no pearls}.

Before we can state our main theorems, we require some further notation related to maximal class groups. Assume that $S$ has maximal class and order at least $p^4$. Set $\gamma_2(S)= S'=[S,S]$ and define $\gamma_j(S)= [\gamma_{j-1}(S), S]$  for $j \ge 3$. We set $$\gamma_1(S)= C_{S}(\gamma_2(S)/\gamma_4(S)).$$  As $S$ has maximal class, $\gamma_2(S)/\gamma_4(S)$ has order $p^2$ and so $|S:\gamma_1(S)|=p$ which means that $\gamma_1(S)$ is a maximal subgroup of $S$. We also have $\gamma_{n-2}(S)$ is the second centre $Z_2(S)$ of $S$ and $|Z_2(S)|=p^2$. Hence $C_S(Z_2(S))$, just like $\gamma_1(S)$, is a maximal subgroup of $S$. These subgroups are examples of \emph{$2$-step centralizers}.  If $S$ has more than one  $2$-step centralizer,   $S$ is called \emph{exceptional}. It is a fact that $S$ is exceptional if and only if $\gamma_1(S) \ne C_S(Z_2(S))$.

 An example of an exceptional group of maximal class is a Sylow $p$-subgroup of the simple group of Lie type $\G_2(p)$ for $p\geq 5$.  The saturated fusion systems $\F$  on a Sylow $p$-subgroup of the simple group $\G_2(p)$ satisfying $O_p(\F)=1$ have been classified by Parker and Semeraro in \cite{G2p}. \blue{Their work uncovers 27 exotic fusion systems when $p=7$.}

The work of Oliver and his co-workers Craven, Ruiz and Semeraro \cite{p.index, p.index2, p.index3} described earlier can be applied to maximal class $p$-groups whenever $\gamma_1(S)$ is abelian. This straightforward application is presented for completeness as Theorem~\ref{TheoremAbelian} in  Appendix~\ref{AppA}. Hence throughout the main body of this work we may and do assume that $\gamma_1(S)$ is not abelian.

A consequence of the main results of our research  can be presented as follows.  The notation used for sporadic simple groups is consistent with \cite{GLS3}.

\begin{customthm}{A}\label{compound} Suppose that $\F$ is a reduced saturated fusion system on a $p$-group $S$ of maximal class \blue{of order at least $p^4$.}  Then  one of the following statements  holds.
\begin{enumerate}
\item $\gamma_1(S)$ is non-abelian, and  $S$ is not exceptional,  $\mathcal \E_\F=\mathcal P_a(\F)$, $\F$ is simple and exotic.
\item $\gamma_1(S)$ is non-abelian, $S$ is exceptional and either  \begin{enumerate}\item
$p \ge 5$ and $\F=\F_S(\G_2(p))$;
\item $p=5$, $S$ is isomorphic to a Sylow $5$-subgroup of $\G_2(5)$ and $\F=\F_S(G)$ where $G$ is one of the sporadic simple groups $\Ly, \HN$ or $\B$;
\item $p=7$,  $S$ is isomorphic to a Sylow $7$-subgroup of $\G_2(7)$ and either $\F$ is exotic (20 examples) or $\F= \F_S(\Mo)$ where $\Mo$ denotes the monster; or
\item $p \ge 11$, $S$ is uniquely determined of order $p^{p-1}$, $\mathcal P(\F)= \mathcal P_a(\F)\not=\emptyset$ and, if $\gamma_1(S)$ is $\F$-essential, then $\Out_\F(S) \cong \GF(p)^\times \times \GF(p)^\times$, $O^{p'}(\Out_\F(\gamma_2(S))) \cong \SL_2(p)$ and $\gamma_1(S)/Z(\gamma_1(S))$ is the  $(p-3)$-dimensional irreducible $\GF(p)\SL_2(p)$-module.
\end{enumerate}
\item $\gamma_1(S)$ is abelian and $\F$ is described by Theorem~\ref{TheoremAbelian}. \end{enumerate}
\end{customthm}

Theorem~\ref{compound} is proved by extracting special cases from Theorems~\ref{Raul app} and \ref{non exc} below.  We remark here that our proofs require the Classification Theorem of the non-abelian simple groups. This is used to provide the names of groups with a strongly $p$-embedded subgroup which contain an elementary abelian subgroup of order $p^2$ (Proposition~\ref{SE-p2}) and also to understand so-called quadratic pairs \cite{chermak2002}. Of course, it is also used when we assert that a given fusion system is exotic for otherwise we would have answered \cite[Question 7.7]{AO}.

The maximal class $2$-groups are either dihedral, quaternion or semidihedral \cite[Corollary 3.3.4(iii)]{Led-Green} and  the fusion systems on such groups are known; in particular it is easy to demonstrate that the $\F$-essential subgroups are all $\F$-pearls (see Lemma \ref{p=2}). As for $p=3$, the saturated fusion systems on maximal class $3$-groups are all known due to the work of  D{\'{\i}}az, Ruiz, Viruel and Parker and Semeraro (see Lemma \ref{p=3}).
Therefore for our main theorems we focus our attention on the primes $p\geq 5$.

 To state our next theorem we require some \blue{additional terminology}. Let   $S(p)$ be the unique split extension of an extraspecial group of exponent $p$ and order $p^{p-2}$ by a cyclic group of order $p$ which has maximal class \cite[Proposition 8.1]{Raul}. When $p=7$, we have $ S(7)$ is isomorphic to the Sylow $7$-subgroup of $\G_2(7)$. Also, the group denoted by $\SmallGroup(5^6,661)$ is group number  $661$ in the {\sc Magma}~\cite{Magma} small group library of groups of order $5^6$. \blue{In Lemma~\ref{mc-factsb} (v) we see that the maximal class $p$-groups which are exceptional have order at least $p^6$ and at most $p^{p+1}$. In particular,  there are no exceptional maximal class $3$-groups.} This explains our assumption that $p \ge 5$ and $|S|\ge p^6$ in the next theorem.
\begin{customthm}{B}\label{Raul app}
Suppose that $p \ge 5$,   $S$ is an exceptional maximal class $p$-group of order at least $p^6$ and $\F$ is a saturated fusion system on $S$. Assume that $\F \ne  N_\F(S)$.  Then one of the following holds.
\begin{enumerate}
\item $\gamma_1(S)$ is extraspecial, and, if  $\F\ne N_\F(\gamma_1(S))$, then  one of the following holds:
\begin{enumerate}
\item $S $ is isomorphic to a Sylow $p$-subgroup of $\G_2(p)$  and either
\begin{enumerate}[label=(\greek*)]
\item $\F= N_\F(C_S(Z_2(S)))$, $O^{p'}(\Out_\F(C_S(Z_2(S))))\cong \SL_2(p)$;
\item $p=5$, $1\ne O_p(\F) \le \gamma_2(S)$, $\F \cong \F_S(5^3{}^. \SL_3(5))$;
\item  $p \ge 5$ and $\F=\F_S(\G_2(p))$;
\item $p=5$ and $\F=\F_S(G)$ where $G= \Ly, \HN, \Aut(\HN)$ or $\B$; or
\item $p=7$ and either $\F$ is exotic (27 examples) or $\F= \F_S(\Mo)$.
 \end{enumerate}
 \item $p \ge 11$, $S\cong S(p)$, $\mathcal P(\F)= \mathcal P_a(\F)\not=\emptyset$ and, if $\gamma_1(S)$ is $\F$-essential, then $\Out_\F(S) \cong \GF(p)^\times \times \GF(p)^\times$, $O^{p'}(\Out_\F(\gamma_2(S))) \cong \SL_2(p)$ and $\gamma_1(S)/Z(\gamma_1(S))$ is the unique $(p-3)$-dimensional irreducible $\GF(p)\SL_2(p)$-module.

     \end{enumerate}
      \item $p=5$,  $S= \mathrm{SmallGroup}(5^6,661)$, $O_5(\F)=C_S(Z_2(S))$ is the unique $\F$-essential subgroup, $\Out_\F(S)$ is cyclic of order $4$, $\Out_\F(C_S(Z_2(S))) \cong \SL_2(5)$ and $\F$ is unique.
     \end{enumerate}
In particular,  if $\F \ne N_\F(\gamma_1(S))$, then $\F=O^p(\F)$ and, in addition,   $O_p(\F) =1 $ in all cases other than   (i)(a)($\alpha$), (i)(a)($\beta$) and   (ii).
     \end{customthm}

Theorem~\ref{Raul app} does not describe $\Aut_\F(\gamma_1(S))$  when $\F= N_\F(\gamma_1(S))$. This additional detail can be determined as follows. First observe that $|S|\le p^{p+1}$ as $S$ is exceptional.  Therefore $|\gamma_1(S)|= p^{1+2a}$ where $2\le a \le (p-1)/2$.  Now $\Aut_S(\gamma_1(S))$ acts on $\gamma_1(S)/Z(S)$ with a single Jordan block and so we can apply \cite{Craven2018} to determine the candidates for $\Out_\F(\gamma_1(S))$.
For non-exceptional $p$-groups,  we prove the following theorem.

\begin{customthm}{C}\label{non exc}
Suppose that $p$ is an odd prime, $S$ is a maximal   class $p$-group of order at least $p^4 $ and $\F$ is a saturated fusion system on $S$. Assume that $S$ is not exceptional,   $\gamma_1(S)$ is not abelian and $\F \ne N_\F(\gamma_1(S))$. Then one of the following holds:
\begin{enumerate}
\item $\E_\F=\mathcal P_a(\F)$,  $|S: \hyp(\F)|\le p$ with $|S: \hyp(\F)|= p$  if and only if $|S|=p^{j(p-1)+1}$  for some $j \ge 2$.   Furthermore,  \blue{either $O^p(\F)$ is simple and exotic or $p=3$ and   $O^3(\F)$ is realized by $\PSL_3(q)$ for suitable prime powers $q$.}
\item $p \ge 5$, $\E_\F=    \mathcal P_e(\F)$,     $O_p(\F)=Z(S)$, $|S: \hyp(\F)|\le p$ with $|S: \hyp(\F)|= p$  if and only if $|S|=p^{j(p-1)+2}$  for some $j \ge 2$.  Furthermore,    $O^p(\F/Z(S))$ is simple  and exotic.
\item $p \ge 5$, $\E_\F= \mathcal P_a(\F) \cup \{\gamma_1(S)\}$,  $O_p(\F)=1$, $\F \ne O^p(\F)$ and
\begin{enumerate}
\item $ \mathcal P_a(\F)$ is a single $\F$-class, $|S|=p^{j(p-1)+1}$ for some $j \ge 2$ and $S$ has sectional rank $p-1$;
\item $\Out_\F(\gamma_1(S)) \cong \Sym(p)$ or $\PGL_2(p)$;
\item$Z(\gamma_1(S))=\agemO^1(\gamma_1(S))$ has index $p^{p-1}$ in $\gamma_1(S)$, $\gamma_1(S)'< \Omega_1(\gamma_1(S))$ has order $p^{p-2}$ and  $\gamma_2(S)$ is abelian  but not elementary abelian;
\item  every composition factor of $\Aut_\F(\gamma_1(S))$ on $\gamma_1(S)$ has order $p$ or $p^{p-2}$ and the composition factors of order $p$ are centralized by the automorphism group $\Aut_\F(\gamma_1(S))$;
\item for $P \in \mathcal P_a(\F)$, $\hyp(\F)= P\gamma_2(S)$,   $O^p(\F)$ is a saturated fusion system on $P\gamma_2(S)$,  and $\Aut_{O^p(\F)}(\gamma_2(S)) \cong \Sym(p)$ or $\PGL_2(p)$.
 \end{enumerate}
\end{enumerate}
Furthermore, in all cases  $\Out_\F(S)$ is a Hall $p'$-subgroup of $\Out(S)$ and is cyclic of order $p-1$ and, if  $| S |  = p^n$, and $P\in \mathcal P(\F)$, then either $\mathcal P(\F)=P^S$ or  $\mathcal E_\F=\mathcal P(\F)$ and $n  \equiv \eps \pmod { p-1 }$
where $\eps = 0$ if $P\in \mathcal P_a(\F)$ and  $\eps = 1$ if $P\in \mathcal P_e(\F)$.\end{customthm}

Suppose that $p$ odd, $r $ a prime with $r^a-1 \equiv 0 \pmod p^k$. In Section~\ref{sec:examples}, we show that an automorphism group $G$ of $G_0=\PSL_p(r^{p})$ which projects diagonally into a Sylow $p$-subgroups of $\Out(G_0)$ between the image of $\PGL_p(r)$ and the image of a field automorphism of order $p$ provides realizable examples of Theorem~\ref{non exc} (iii) with, for $S \in \syl_p(G)$, $\Out_{\F_S(G)}(\gamma_1(S))\cong \Sym(p)$. By \cite[Theorem 6.2]{parkersemerarocomputing}, the subfusion systems generated by the $\F_S(G)$-pearls  gives an example of (i) in the case that $\gamma_2(S)$ is abelian. In Theorem~\ref{non exc} (iii) the fusion systems $O^p(\F)$ can be found in Table~\ref{Tab1} in  Appendix~\ref{AppA} and are listed in Lines (3) and (4) in the case that $\gamma_2(S)$ is elementary abelian and otherwise in Lines (29) and (33).
  If $\F$ is a saturated fusion system with $\Out_\F(\gamma_1(S))\cong \PGL_2(p)$ with $p \ge 7$ and $\gamma_1(S)$ is non-abelian, then, as $\gamma_1(\hyp(\F))= \gamma_2(S)$ is abelian, we may use \cite[Theorem 4.5]{p.index3} to see that $O^p(\F)$ is exotic.

Theorem~\ref{non exc} does not specify the structure of a saturated fusion system $\F= N_{\F}(\gamma_1(S))$ and, as we don't know the structure of $\gamma_1(S)$, it could be more difficult to determine the precise structure than  in the exceptional case. This leads to some complications in our inductive arguments.  Theorems~\ref{non exc} leaves open the question of what more we can say about saturated fusion systems which only have $\F$-pearls. For example, the isomorphism type of $S$ is not determined in Theorem~\ref{non exc} (i) and (ii). As an indicator that the structure of $S$ can be more complicated than the structures \blue{described} in Theorem~\ref{non exc} (iii) we have the following example which was obtained by computer using the procedures developed in \cite{parkersemerarocomputing} and implemented in {\sc Magma} \cite{Magma}. \blue{In Example~\ref{5^7thm}, $|S|=5^7$ and so $S$ is not exceptional.}

\begin{example}\label{5^7thm} Suppose that $S$ is a maximal class $5$-group of order $5^7$ and suppose that $\gamma_1(S)$ is not abelian. If $\F $ is a saturated fusion system on $S$ and $\mathcal P(\F)$ is non-empty, then $S$ is one of the \blue{seven groups
$$\SmallGroup(5^7,1297), \SmallGroup(5^7,1308),$$
$$\SmallGroup(5^7,1321),
\SmallGroup(5^7,1360),$$
$$\SmallGroup(5^7,1363),
\SmallGroup(5^7,1374),$$
$$\SmallGroup(5^7,1384).$$}Furthermore,  each possibility for $S$ supports a unique (up to isomorphism) saturated fusion system $\F$ with $\F$-pearls and $\gamma_1(S)$  is not $\F$-essential. For $S=\SmallGroup(5^7,1308)$, $\F$ has a unique $\F$-class of extraspecial $\F$-pearls and the saturated fusion systems on the remaining groups have a single $\F$-class of abelian $\F$-pearls.   \blue{For $S$ any of the groups $\SmallGroup(5^7,1360),$
$\SmallGroup(5^7,1363),$
$\SmallGroup(5^7,1374),$ or
$\SmallGroup(5^7,1384), $    $\gamma_1(S)$  has  nilpotency class $3$; otherwise  $\gamma_1(S)$ has nilpotency class $2$.}
\end{example}

 All the fusion systems in Example~\ref{5^7thm} with abelian pearls are simple and exotic by Theorem~\ref{pearls5}.
To provide some perspective to this calculation, we remark that  there are 99 maximal class groups of order $5^7$. Three of them have an abelian subgroup of index $5$. Suppose that $S$ is a maximal class $5$-group with $\gamma_1(S)$ non-abelian and let $\F$ be  saturated fusion system on $S$. Assume $\F$ has an $\F$-pearl $P$. Then Theorem~\ref{non exc}  says that $\Aut(S)$ has a Hall $5'$-subgroup of order $4$. Of the maximal class $5$-groups of order $5^7$, with non-abelian $2$-step centralizer just 12 of them have a Hall $5'$-subgroup of $\Aut(S)$ of order $4$. \blue{All of these have an elementary abelian subgroup of order $25$ not contained in the $2$-step centralizer and so have candidates for pearls.}  We also remark that the fusion system $\F$ on $S=\SmallGroup(5^7,1308)$ has a unique class of extraspecial $\F$-pearls and so this shows that there are examples in Theorem~\ref{non exc} (ii) when $p = 5$. \blue{The {\sc Magma} routines for this computation can be found in Subsection~\ref{C1.2}. The code exploits an idea we are just about to explain.}

 The study of maximal class $p$-groups with large automorphism group as in \cite{DietrichEick} requires more research  to make substantial headway on the problem of determining all fusion systems which have every essential subgroup a pearl. To make this statement clearer,
let us fix
$B= \left\{\left(\begin{smallmatrix} a&0&0\\b&a^{-1}&0\\c&d&1\end{smallmatrix}\right)\mid a,b,c,d\in \GF(p), a\ne 0\right\} $. Let $S$ be maximal class $p$-group of order at least $p^5$, and suppose that $x\in S$  has order $p$ and  $P= C_S(x)= \langle x, Z(S)\rangle$. Then $N_S(P) $ is extraspecial of order $p^3$.  If $S$ has an automorphism $\phi$  of order $p-1$ which normalizes $p$ and satisfies $N_S(P)\rtimes \langle \phi\rangle \cong B$, then  there exists a saturated fusion system $\F$ on $S$ with $P$ an abelian $\F$-pearl just as in \cite[Theorem 2.8]{p.index}. The existence of an automorphism of order $p-1$ is precisely the condition  that Dietrich and Eick use in their work \cite{DietrichEick}.

We now state the corollary as advertised above.

\begin{cor}\label{cor: no pearls} Let $p$ be a prime,  $S$ be a $p$-group of maximal   class and  let $\F$ be a  saturated fusion system on $S$ with $O_p(\F)=1$. If $\mathcal P(\F)$ is empty, then $S$   is isomorphic to a Sylow $p$-subgroup  of $\G_2(p)$ and either
 \begin{enumerate}
 \item $\F= \F_{S}(\G_2(p))$;
 \item  $p=5$ and  $\F=\F_S(G)$ where $G= \Ly, \HN, \Aut(\HN)$ or $\B$;
 \item  $p=7$,  $\F$ is exotic and the $\F$-essential subgroups are $C_S(Z_2(S))$ and $\gamma_1(S)$, with  $\Out_\F(C_S(Z_2(S)))\cong \GL_2(7)$, $\Out_\F(\gamma_1(S)) \cong 3\times 2^. \Sym(7)$,   and $\Out_\F(S) \cong \GF(7)^\times \times  \GF(7)^\times$.
 \end{enumerate} \end{cor}

 We remark that the exotic  saturated fusion system in part (iii) of the corollary is obtained from the saturated fusion system in the monster $\Mo$ by pruning the pearl \cite[Lemma 6.5]{parkersemerarocomputing}.

\subsection{An overview of the paper}

The article develops as follows.
We start in Section~\ref{sec:GrpTh} with some background group theoretical results. Especially, we explain what it means for a group to have a strongly $p$-embedded subgroup and present some elementary facts about such groups.

Section~\ref{sec:MC} commences with two lemmas which detail properties of maximal class $p$-groups, most of which are drawn from \cite{Led-Green}.  Of particular importance are the facts that $\gamma_1(S)$ is a regular $p$-group and $|\Omega_1(\gamma_1(S)) | \le p^p$, with equality if and only if $|S|= p^{p+1}$. In the second part of Section~\ref{sec:MC}, we start to study the automorphism group of $S$. As we are interested in the structure of $\Out_\F(S)$, we specifically study $p'$-automorphisms.  Two very important results for our work, which were known to Juh\'asz \cite{J}, are Lemma~\ref{action} which describes the action of a single $p'$-automorphism of $S$ on $\gamma_k(S)/\gamma_{k+1}(S)$ for $k \ge 1$, and Lemma~\ref{lem:autS p'elts} which says that if some $p'$-automorphism of $S$ centralizes $S/\gamma_1(S)$, then $\gamma_1(S)$ is either abelian or $S$ is exceptional and $\gamma_1(S)$ is extraspecial. In our work this lemma has the consequence that most of the time we can assume that $\Out_\F(S)$ is cyclic of order dividing $p-1$. But perhaps more important than both these results is \blue{Lemma~\ref{centralizer auto} which applies when $S$ is not exceptional} and, for example,  controls the size of $|C_{\gamma_1(S)}(\psi)|$ for $\psi$ a $p'$-automorphism of $S$. A particular consequence of Lemma~\ref{centralizer auto} is that if $\tau$ is an automorphism of $S$   and $\tau$ centralizes $\gamma_{k}(S)/\gamma_{k+2}(S)$ then $\tau$ is a $p$-element. A final important result is Theorem~\ref{J6.2} which is \cite[Theorem~6.2]{J}. This has the consequence that if $\gamma_w(S)$ is elementary abelian and $w \ge 3$, then $\gamma_{w-1}(S)$ has nilpotency class at most $2$. For completeness, we present a modestly simplified version of Juh\'asz's proof.

In Section~\ref{sec:Reps}, we gather a collection of results about representations of groups with cyclic Sylow $p$-subgroups of order $p$.  These results are applied later in the paper in the case when $\gamma_1(S)$ is known to be an $\F$-essential subgroup, in order to obtain the structure of $\Out_\F(\gamma_1(S))$. Of particular significance is Feit's Theorem (Theorem~\ref{feit}) which says that if a group with a cyclic Sylow $p$-subgroup is not closely related to $\PSL_2(p)$, then any faithful representation is relatively large. Section~\ref{sec:Reps} also contains a description of  the irreducible $\GF(p) \SL_2(p)$-modules, and results which help decompose tensor products of such modules (see Proposition~\ref{Clebsch Gordan}). These results are exploited throughout the proof of our main theorems.

 In Section~\ref{sec:PF} we recall basic definitions and known facts about fusion systems, referring mostly to \cite{AKO} (especially for the terminology). We also present a number of results about saturated fusion systems with $\F$-pearls: Lemmas ~\ref{pearls1}, \ref{pearls2}, \ref{pearls2.5} and \ref{pearls4}  as well as Theorems~\ref{pearls3} and \ref{pearls5} which are mostly taken from \cite{pearls}. The last result of Section~\ref{sec:PF} is Proposition~\ref{prop:sub sat} which allows us in Sections~\ref{sec: non excep g1} to construct a saturated subfusion system on a $p$-group with a maximal abelian subgroup.

 Because of the Alperin-Goldschmidt  Theorem, Theorem~\ref{t:alp}, our first significant objective is to determine all the candidates for $\F$-essential subgroups in a maximal class $p$-group. This is the foundation for our proof of Theorems~\ref{compound}, \ref{Raul app} and \ref{non exc}, for without knowing the $\F$-essential subgroups nothing more can be said.

\begin{customthm}{D}\label{MT1} Suppose that $p$ is a prime,  $S$ is a $p$-group of maximal  class and order at least $p^4$ and $\F$ is a saturated fusion system on $S$. If $E$ is an $\F$-essential subgroup, then either $E$ is an $\F$-pearl,   $E= \gamma_1(S)$ or $E=C_S(Z_2(S))$. Furthermore, if $S$ is exceptional, then $\mathcal P(\F)=\mathcal P_a(\F)$.  \end{customthm}

The proof of Theorem~\ref{MT1} spans Sections~\ref{sec:Generalities} to \ref{sec: proof MT1 2}. This means that  Section \ref{sec:Generalities} is where the real work begins. We start by analyzing the properties of $\F$-essential subgroups of maximal class $p$-groups. The most relevant results are Lemma~\ref{in.sub}, that says that an $\F$-essential subgroup $E$ of $S$  is not an $\F$-pearl if and only if it is contained in either $\gamma_1(S)$ or $C_S(Z_2(S))$, and Lemma~\ref{normal.ess}, in which we prove that every normal $\F$-essential subgroup of $S$ is a maximal subgroup of $S$.

In Section~\ref{sec:exceptional} we focus our attention on  the $\F$-essential subgroups of $S$ in the case that $S$ is an exceptional group. The goal of this section is to prove Proposition~\ref{Prop:ess.in.except}, which implies Theorem~\ref{MT1} in the case in which $\gamma_1(S)$ is extraspecial and  gives the first ingredients toward our  proof of Theorem \ref{Raul app}.

Section~\ref{sec: g1 ess} considers the case when $S$ is not exceptional and $\gamma_1(S)$ is $\F$-essential.
Of particular interest is Lemma \ref{gamma1-essC} in which we prove that $\Omega_1(\gamma_1(S))$ has nilpotency class at most $2$ and that, if the class is exactly $2$, then $O^{p'}(\Out_\F(\gamma_1(S))) \cong \PSL_2(p)$. This lemma is crucial for the proof of Theorem \ref{MT1}.

 In Section~\ref{sec: g1 ess2}, in the case that $S$ is exceptional, Proposition~\ref{gamma1-essD} states that if $\gamma_1(S)$ is $\F$-essential, then $\gamma_1(S)$ is extraspecial. Thus, in combination with  Proposition~\ref{Prop:ess.in.except}, we may assume that, if $S$ is exceptional, then $\gamma_1(S)$ is not $\F$-essential in our minimal counterexample. The proof of  Proposition~\ref{gamma1-essD} invokes Lemma \ref{gamma1-essC}, Proposition~\ref{Clebsch Gordan} and a detailed commutator calculation which in the end reaches contradiction by using a 1902 result due to Burnside (Theorem~\ref{Engel2}).

Sections \ref{sec: proof MT1 1} and \ref{sec: proof MT1 2} contain the series of results that leads to the proof of Theorem \ref{MT1}. Our  proof is achieved by contradiction, considering a minimal counterexample $\F$ to Theorem  \ref{MT1} (first with respect to the size of $S$ and then to the number of morphisms). In other words, we consider a minimal fusion system $\F$ containing an $\F$-essential subgroup $E$ that is not an $\F$-pearl and is not equal to $\gamma_1(S)$ or $C_S(Z_2(S))$ (we say that $E$ is a witness). So assume that $\F$ is such a fusion system. The most important result of Section~\ref{sec: proof MT1 1} is Proposition~\ref{propOp} which asserts $O_p(\F)=1$. The proof of Proposition~\ref{propOp} uses the fact that we know those non-abelian simple groups which have a $\GF(p)G$-module on which some non-trivial $p$-elements act with minimal polynomial of degree $2$. The fact that we require originally goes back to Thompson in unpublished work and we cite Chermak \cite{chermak} for the proof.   We also show that for $\F$ a minimal counterexample we must have $p\geq 7$,  $p^7 \leq |S| < p^{2p-4}$ and $\Omega_1(\gamma_1(S))$ non-abelian (Lemmas~\ref{p7} and \ref{Omega.non.ab}).
With the scene set, in Section \ref{sec: proof MT1 2} we choose a subgroup $T$ of $\gamma_1(S)$ whose automorphism group is not formed by restrictions of $\F$-automorphisms of $\gamma_1(S)$ and that is maximal  first with respect to the order of its normalizer in $S$  and second with respect to  its own order. Note that a witness $E$ is a candidate for $T$, but the key idea is that $T$ is not necessarily $\F$-essential. The study of the subgroup $T$ is divided in two major cases: the case in which $T$ is $\F$-centric and the case in which $T$ is not  $\F$-centric. We show that \blue{both cases are} impossible and so $T$ cannot exist, proving Theorem~\ref{MT1}.

With Theorem~\ref{MT1} proved, Section~\ref{sec: proof Raul app} contains the proof of Theorem \ref{Raul app}. By this stage, this is relatively straightforward because of  \cite{Raul}, where  Moragues Moncho   classified all saturated fusion systems $\F$ on $p$-groups with an extraspecial subgroup of index $p$ and $O_p(\F)=1$.

In Section~\ref{sec: non excep g1} we prepare for the proof of Theorem~\ref{non exc}.  We are interested in the case in which $S$ is not exceptional.  Because of Theorem~\ref{MT1}, once  we assume that $\F \ne N_\F(\gamma_1(S))$, we know that $\Pp(\F)$ is non-empty.  Thus we choose an $\F$-pearl $P$. If $\gamma_1(S)$ is not $\F$-essential, then we readily obtain (i) and (ii) of Theorem~\ref{non exc}.  So in Hypothesis~\ref{hyp last} we assume that $\gamma_1(S)$ is $\F$-essential and non-abelian. In this case we let $V=\Omega_1(Z(\gamma_1(S)))$ and $S_1= VP$. Then Proposition~\ref{prop:sub sat} can be applied to produce a saturated reduced subfusion system of $S_1$ in which $V$ is $\F$-essential. The application of \cite{p.index2} (see  Appendix~\ref{AppA} for the main result of \cite{p.index2, p.index3} applied to maximal class $p$-groups) gives us that $\Out_\F(\gamma_1(S))$ is either $\PGL_2(p)$ or $\Sym(p)$ and it also dictates the isomorphism type of $V$ as a $\GF(p)\Out_\F(\gamma_1(S))$-module. We then determine in Lemma~\ref{G action Q} detailed information about the action of $\Aut_\F(\gamma_1(S))$ on $\gamma_1(S)$. In particular, this shows that the chief factors alternate between having order $p$ and being central and having order $p^{p-2}$. After a few more observations, in Lemma~\ref{F not OpF} we show that $\F \ne O^p(\F)$; the proof of this uses detailed  knowledge about the submodule structure of exterior squares of the $(p-2)$-dimensional modules for $\PSL_2(p)$ and $\Alt(p)$. In Lemma~\ref{Orders}, we show that $P$ must be abelian and so there are no extraspecial $\F$-pearls and $O_p(\F)=1$.  We also determine the structure of $\gamma_1(S)$.  With this information available, the proof of Theorem~\ref{non exc} is quickly brought together in Section~\ref{sec:proofs}. In Section~\ref{sec:proofs} we also prove Theorem~\ref{compound} and Corollary~\ref{cor: no pearls}.

The final section of the paper presents a group which realizes examples of Theorem~\ref{non exc} and in Appendix~\ref{AppA} we give a description of the reduced fusion systems on maximal class $p$-groups with $p$ odd and having an abelian subgroup of index $p$ taken from \cite{p.index2, p.index3}. In Appendix~\ref{AppB}  we present the classification of saturated fusion systems on maximal class $3$-groups due to \cite{DRV,PSrank2}. Finally Appendix~\ref{AppC} lists the {\sc Magma} code used in various examples and results of the paper.

\subsection{Non-standard notation.} We follow one of \cite{Huppert, Gor, GLS2} for group theoretic notation and we follow Leedham-Green and McKay \cite{Led-Green} for notation and facts surrounding $p$-groups of maximal class. In particular, for a maximal class $p$-group $R$, we mention that  $R'=\Phi(R)$ is denoted by $\gamma_2(R)$. We apply almost all maps on the right. If $G$ is a group and $g \in G$ then $c_g$ is the conjugation map $c_g:G \rightarrow G$ defined by $x c_g= x^g= g^{-1}xg$ for $x \in G$.    If $X ,Y \le G$ are groups and $n \in \mathbb N$, then $[X,Y;n]$ is defined recursively by $[X,Y;1]=[X,Y]$ and $[X,Y;n]= [[X,Y;n-1],Y]$ for $n \ge 1$. \blue{It is also convenient to define $[X,Y;0]=X$.}  For $x,y, z\in G$, we write $[x,y,z]$ for $[[x,y],z]$.

Our nomenclature for specific  groups is for the most part standard or self-explanatory, for example, we use $\Alt(n)$ and  $\Sym(n)$, $n \ge 3$ to  denote the alternating and symmetric groups of degree $n$ respectively. \blue{Similarly, we use $\Frob(n)$ to denote a Frobenius group of order $n$, whenever such a group is uniquely defined. }  For $p$ odd, the extraspecial groups of order $p^3$ are denoted by $p^{1+2}_+$ and $p^{1+2}_-$ where the first group has exponent $p$ and the second exponent $p^2$. For a field $\mathbb K$, $\mathbb K^\times$ denotes its multiplicative group.

%% file: GeneralGroups.tex
\section{General group theoretical results}\label{sec:GrpTh}

We use the commutator formulae as in \cite[Theorem 2.2.1 and Lemma 2.2.2]{Gor} regularly and without reference. We also often refer to \cite[Sections 5.2 and 5.3]{Gor} for results concerning coprime action. Here we catalogue less familiar results.

Let $p$ be a prime.
A $p$-group $T$ is \emph{regular} if, and only if, for all $x, y \in T$, there exist $g_1, \dots, g_t \in \langle x, y \rangle' $ such that $$(xy)^{p}= x^{p}y^{p}g_1^{p}\dots g_t^p.$$
We will need a handful of properties of regular $p$-groups. These attributes say that regular $p$-groups have similar properties to abelian groups with respect to taking powers. Recall that, for a $p$-group $X$, $\Omega_1(X)$ is the subgroup of $X$ generated by elements of order $p$ and  $\agemO^1(X)= \langle x^p \mid x \in X\rangle$.

\begin{lemma}\label{lem:regular} Suppose that $P$ is a regular $p$-group and assume that $Q \le P$. Then
\begin{enumerate}
\item $Q$ is regular;
\item $\Omega_1(P)$ has exponent $p$;
\item $\Omega_1(Q)= Q \cap \Omega_1(P);$ and
\item $|P/\Omega_1(P)| = |\agemO^1(P)|$.
\end{enumerate}
\end{lemma}

\begin{proof} Part (i) is obvious from the definition of a regular $p$-group.  For (ii) see \cite[Haupsatz III.10.5]{Huppert}. Part (iii) follows from (ii). Part (iv) comes from  \cite[Satz III.10.7]{Huppert}.
\end{proof}

\begin{lemma}\label{agemo} Assume that $P$ is a regular $p$-group \blue{and $T\le P$. If $T \ge P' $ and $P= T\Omega_1(P)$, then $\agemO^1(P)=\agemO^1(T)$.}
\end{lemma}
\begin{proof} Let $x \in P$.  Then there is $t \in T$ and $w \in \Omega_1(P)$ such that $x=tw$.
By  Lemma~\ref{lem:regular} (ii) we have $w^p=1$.  Thus \cite[Lemma 1.2.10 (iii)]{Led-Green} together with  $T \ge P'\ge  \langle t,w\rangle '$, imply there exists $s \in \langle t,w\rangle '$ such that $$x^p= (tw)^p = t^pw^ps^p = t^p s^p \in  \agemO^1(T).$$ Hence $\agemO^1(P) \le \agemO^1(T)\le \agemO^1(P)$ and this gives the result.
\end{proof}

 For a group $P$, we define $$E_2(P)=\{x\in P\mid [x,y,y]=1 \text{ for all } y \in P\}.$$

\begin{lemma}\label{Engel1}
Suppose that $p$ is a prime and $P$ is a $p$-group of nilpotency class $3$. Then $E_2(P)$ is a characteristic subgroup of $P$ which contains $Z_2(P)$.
\end{lemma}

\begin{proof} Obviously $Z_2(P)\le E_2(P)$.
Assume that $a, b \in E_2(P)$ and let $y \in P$.  Then we calculate $$[ab^{-1},y]= [a,y]^{b^{-1}}[b^{-1},y]= [a,y][a,y,b^{-1}][b^{-1},y] \in C_P(y) Z(P)C_P(y) = C_P(y).$$
Hence $E_2(P)$ is a subgroup of $P$.  Since $E_2(P)$ is a characteristic subset of $P$, $E_2(P)$ is a characteristic subgroup of $P$.
\end{proof}

\begin{theorem}[Burnside]\label{Engel2} Suppose that $p \ne 3$ is a prime and $P$ is a finite $p$-group such that $P= E_2(P)$. Then $P$ has nilpotency class $2$.
\end{theorem}

\begin{proof} Notice that $[x,y,y]=1$ if and only  if $[(y^{-1})^x,y]=1$ if and only if $\langle y^x\rangle$ and $\langle y\rangle$ commute.  Therefore this theorem  dates back to 1902 \cite{Burnside1902}.
\end{proof}

\begin{prop}\label{Hall-Hig} Suppose that $p$ is a prime,  $L$ is a group and $P$ is a Sylow $p$-subgroup of $L$.  Let $V$ be a faithful, irreducible $\GF(p)L$-module. If $[O_{p'}(L),P] \ne 1$, then $\dim V \ge p-1$ and $[V,P;p-2]\ne 0$.
\end{prop}
 \begin{proof} Set $H= O_{p'}(L)\langle x\rangle$ where $x \in P$  does not centralize $O_{p'}(H)$. Let $U$ be a non-trivial composition factor for $H$ in $V$ which is not centralized by $[O_{p'}(H),x]$. Then $H$ is $p$-soluble, $O_p(H/C_H(U))=1$ and we may apply the Hall-Higman Theorem \cite[Theorem 11.1.1]{Gor} to $H/C_H(U)$ to obtain the minimal polynomial for $x$ \blue{acting} on $U$ is $(X-1)^r$ with $r \in \{p, p-1\}$. In particular, $\dim V \ge r \ge p-1$ and $[V,P;p-2]\ge [V,x;p-2]= V(x-1)^{(p-2)}\ne 0$.
\end{proof}

\begin{lemma}\label{lem:K in E} Suppose that $p$ is a prime, $S$ is a $p$-group, $ E, K \le S$ with $EK$ a subgroup of $S$.   If $N_K(E) \le E$, then $K \le E$.
\end{lemma}

\begin{proof} Let $t \in N_{EK}(E)$.  Then $t = ek$ for some $e \in E$ and $k \in K$.  Thus
$$E= E^t= E^{ek}= E^k$$ and so $k \in N_K(E) \le E$.  Hence $N_{EK}(E) \le E$ and this means that $E=EK\ge K$.
\end{proof}

We require the following cohomological type result which is a consequence of a theorem of Gasch\"utz.

\begin{lemma}\label{lem:mss} Suppose that $p$ is a prime, $G$ a group and $V$ a $\GF(p)G$-module. Let $W= [V,O^p(G)]$ and
 $T \in \syl_p(G)$. Then $W+C_V(T)= W+C_V(G)$.  In particular, if $C_V(G)=0$, then $C_V(T) \le W$.
\end{lemma}

\begin{proof} See \cite[Lemma C.17]{mss}. \end{proof}

\subsection{Groups with a strongly $p$-embedded subgroup}
In this subsection we collect  together  results about groups with a strongly $p$-embedded subgroup.

\begin{definition}
 Suppose that $p$ is a prime, $H$ is a group and $M$ is a proper subgroup of $H$ of order
divisible by $p$.  Then $M$ is strongly $p$-embedded in $H$ if and only if  $M \cap M^h$  has order coprime to $p$ for all $h \in H\setminus M$.
\end{definition}

It is easy to establish, see \cite[Definition 17.11, Proposition 17.11]{GLS2}, that   $M$ is strongly $p$-embedded in $H$ if and only if  $M$ contains a Sylow $p$-subgroup $T$ of $H$ and $N_H(R)\leq M$ for all $1 \neq R \leq T$. In particular, if $H$ has a strongly $p$-embedded subgroup, then $O_p(H)=1$.

\begin{lemma}\label{cor-p'} Suppose that $p$ is a prime, $H$ is a group and  $M$ is strongly $p$-embedded in $H$. If $K \le M$ is a subnormal subgroup of $H$, then $K $ is  a $p'$-subgroup.
\end{lemma}

\begin{proof}
We may suppose that $N_H(K) \not \le M$. \blue{ Let $R\in \Syl_p(K)$.  Then, by the Frattini Argument, $N_H(K)= N_{N_H(K)}(R)K$ and so $N_H(R) \not \le M$. As $M$ is strongly $p$-embedded in $M$, $R=1$.} Hence $K  $ is a $p'$-subgroup.\end{proof}

\begin{lemma}\label{strongly p structure}
Suppose that $p$ is a prime and $H$ is a group with a strongly $p$-embedded subgroup $M$. If $M$ contains an elementary abelian subgroup  of order $p^2$, then $O_{p'}(H) \le M$, $M/O_{p'}(H)$ is strongly $p$-embedded in  $H/O_{p'}(H)$ and $H/O_{p'}(H)$ is an almost simple group.
\end{lemma}

\begin{proof} See \cite[Lemma 4.3]{parkersemerarocomputing}. \end{proof}

  \begin{lemma}\label{strongly p-embedded}
Suppose that $p$ is a prime, $H$ is a group with a strongly $p$-embedded subgroup and that $K$ is a normal subgroup of $H$  which commutes with an element of order $p$.  Then $H/K$  has a strongly $p$-embedded subgroup.
 \end{lemma}

\begin{proof} Assume that $M$ is strongly $p$-embedded in $H$. Then $M$ contains a Sylow $p$-subgroup $T$ of $H$ and $N_H(R) \le M$ for all $1\ne R\le T$.   Since $K$ centralizes a $p$-element in $H$ and $K$ is normal in $H$ by assumption, $K \le M$. Hence $K \le O_{p'}(H)$ by Lemma~\ref{cor-p'}. Set $\ov H=H/K$. Then $\ov H$ and $\ov M$ have order divisible by $p$. Let $\ov R$ be a non-trivial $p$-subgroup of $\ov T$ with $R \le T$.  Then, by coprime action, $N_{\ov H}(\ov R) = \ov{N_H(R)} \le \ov M$ and so we conclude $\ov H$ has a strongly $p$-embedded subgroup as claimed.
\end{proof}

\blue{For a group $X$, $F^*(X)$ denotes the \emph{generalized Fitting subgroup} of $X$. This is the subgroup of $X$ generated by the subnormal nilpotent subgroups and subnormal  quasisimple subgroups of $X$. See \cite[Definition 3.4]{GLS2}.}

\begin{prop} \label{SE-p2} Suppose that $p$ is a prime, $X$ is a  group,
$K=F^*(X)$   and $T \in \Syl_p(X)$. Assume that $O_{p'}(X)=1$ and that $M$ is a strongly $p$-embedded subgroup of $X$ containing  $T$. Then $O_p(X)=1$, $K$ is a non-abelian simple group and $M \cap K$ is strongly $p$-embedded in $K$,
and  $p$ and $K$ are as follows:
\begin{enumerate}
\item $p$ is any prime, $a \ge 1$ and $K \cong \PSL_2(p^{a+1})$, $\PSU_3(p^a)$ ($p^a \ne 2$),  ${}^2\mathrm B_2(2^{2a+1})$ $(p=2)$ or ${}^2{\rm G}_2(3^{2a+1})$ $(p=3)$ and $X/K$ is a
$p'$-group.
 \item $p > 3$, $K \cong \Alt(2p)$,
$|X/K| \le 2$ and $T$ is elementary abelian of order $p^2$.
\item $p=3$, $K \cong \PSL_2(8)$, $X \cong \Aut(\PSL_2(8))\cong {}^2\mathrm G_2(3)\cong \PSL_2(8){:}3$, $T\cong 3^{1+2}_-$ and $T \cap K$ is cyclic of order $9$.
\item $p=3$, $K \cong \PSL_3(4)$, $X/K$ is a $2$-group and $T$ is elementary abelian of order $3^2$.
\item $p=3$, $X=K\cong \mathrm{M}_{11}$ and $T$ is elementary abelian of order $3^2$.
\item $p=5$,  $K \cong {}^2\mathrm B_2(32)$, $X \cong \Aut({}^2\mathrm B_2(32)) \cong {}^2\mathrm B_2(32){:}5$, $T\cong 5^{1+2}_-$ and $T \cap K$ is cyclic  of order $25$.
\item $p=5$, $K \cong
{}^2\mathrm F_4(2)^\prime$,  $|X/K|\le 2$ and $T$ is elementary abelian of order $5^2$.
 \item $p=5$, $K \cong \mathrm {McL}$,  $|X/K|\le 2$ and $T\cong 5^{1+2}_+$.
\item $p=5$, $K  \cong \mathrm{Fi}_{22}$, $|X/K|\le 2$ and  $T$ is elementary abelian of order $5^2$. \item
$p=11$,  $X=K \cong \mathrm J_4$ and   $T\cong 11^{1+2}_+$.
\item $p$ is odd and $T=T \cap K$ is cyclic.
\end{enumerate}\end{prop}

\begin{proof} This is mainly  \cite [Chapter 4,  Lemma 10.3 ]{GLS4}. The formulation presented here comes from \cite[Proposition 4.5]{parkersemerarocomputing}.
\end{proof}

%% file: MaxClass.tex
\section{Maximal class $p$-groups}\label{sec:MC}

Throughout this section, $S$ represents  a $p$-group of maximal class of order $p^n$ with $n \ge 3$. This means that $S$ has nilpotency class $n-1$.
The maximal class $2$-groups are the dihedral groups, generalized quaternion groups and semidihedral groups and so everything about these groups is easy to calculate.
  The maximal class $p$-groups of order $p^3$ are extraspecial and we assume that the reader is familiar with their structure.
   Thus for this section we concentrate on the case when $p$ is odd, even though many of the results are obviously true when $p=2$, and we also assume that $n \ge 4$.  We set $Z_1(S)= Z(S)$ and, for $2\le j \le n-1$, $Z_j(S)$ is the complete pre-image of $Z(S/Z_{j-1}(S))$. So the $Z_j(S)$  are the terms of the \emph{upper central series} of $S$. Similarly, we put $\gamma_2(S)= S'=[S,S]$ and define $\gamma_j(S)= [\gamma_{j-1}(S), S]$ for $j \ge 3$. These are the members of the \emph{lower central series} of $S$. Notice that $\gamma_j(S) = 1$ for $j \ge n$. As $S$ has maximal  class, $Z_j(S)= \gamma_{n-j}(S)$ for $1\le j \le n-2$ and this subgroup has order $p^{j}$ and index $p^{n-j}$ in $S$. The \emph{$2$-step centralizers } in $S$ are the subgroups $$C_S(\gamma_j(S)/\gamma_{j+2}(S))$$ where $2\le j \le n-2$.
As \blue{in the introduction,} we   define $$\gamma_1(S)= C_{S}(\gamma_2(S)/\gamma_4(S)).$$ Notice that, \blue{as} $|S| \ge p^4$, then $|S:\gamma_1(S)|= |\gamma_1(S):\gamma_2(S)|= p$ and all the $2$-step centralizers in $S$ are maximal subgroups of $S$.\begin{definition} A maximal class group with more than one $2$-step centralizer is  called an \emph{exceptional group}.\end{definition}

By \cite[Lemma 1.1.25 (i)]{Led-Green}, we always have $[\gamma_i(S),\gamma_j(S)] \le \gamma_{i+j}(S)$. Hence  $\gamma_{t}(S)$ is abelian for all $t \ge n/2$.
The  \emph{degree of commutativity} of $S$, is the greatest integer  $c$  such that $[\gamma_i(S),\gamma_j(S)] \le \gamma_{i+j+c}(S)$ for all $1 \le i \le j \le n$. We say that $S$ has a \emph{positive degree of commutativity} if and only if $c$ is positive.

\begin{lemma}\label{mc-facts}  Suppose that $S$ is a maximal class $p$-group of order $p^n$ with $n \ge 4$.
 \begin{enumerate}
 \item If $N$ is a proper normal subgroup of $S$, then either $N=\gamma_j(S)$ for some $j \ge 2$ or $N$ is a  maximal subgroup of $S$. Furthermore, $S/\gamma_j(S)$ has maximal class for all $j \ge 2$.
 \item $\gamma_1(S)$ is a regular $p$-group.

\item If $n> {p+1}$, then $\Omega_1(\gamma_1(S))=\gamma_{n-(p-1)}(S)$ has exponent $p$ and order $ p^{p-1}$  and $\agemO^1(\gamma_i(S))=\gamma_{i+(p-1)}(S)$ for all $1 \leq i \leq n-(p-1)$.
\item $|\Omega_1(\gamma_1(S))| \le p^p$ and if equality holds then $\Omega_1(\gamma_1(S))=\gamma_1(S)$ and $|S|= p^{p+1}$.
 \item If $n > p+1$,  $ A \le \gamma_1(S)$  and $B$ is normal in $A$ with $A/B$ elementary abelian, then $|A/B| \le p^{p-1}$.

\item If $n \le {p+1}$, then $\gamma_2(S)$ and $S/Z(S)$ have exponent $p$.

\item There exists a unique maximal abelian normal subgroup $\gamma_w(S)$ of $S$ and either $\Omega_1(\gamma_1(S)) \le \gamma_w(S)$ or $\gamma_w(S)$ is elementary abelian.
\end{enumerate}
\end{lemma}

\begin{proof} The first part of (i) is \cite[Proposition 3.1.2]{Led-Green} while the second is obvious by definition.

 For part (ii), if $n \le p+1$, then $|\gamma_1(S)|\le p^p$ and \cite[Lemma 1.2.11]{Led-Green} yields $\gamma_1(S)$ is regular. If $n > p+1$, \cite[Corollary 3.3.4 (i)]{Led-Green}  gives the same result.

For statement (iii) we first note that $\gamma_1(S)$ is regular by (ii) and so $\Omega_1(\gamma_1(S)$ has exponent $p$ by Lemma~\ref{lem:regular} (ii). The result now follows from \cite[Corollary 3.3.6(i)]{Led-Green}.

Part (iv) is a consequence of (iii).

Part (v) is  included in \cite[Exercise 3.3 (3)]{Led-Green} and part (vi)   is \cite[ Proposition 3.3.2]{Led-Green}.

Part (vii) follows from (i) as $n \ge 4$.

\end{proof}

\begin{lemma}\label{mc-factsb}  Suppose that $S$ is a maximal class $p$-group of order $p^n$ with $n \ge 4$ and $M$ is a maximal subgroup of $S$.
 \begin{enumerate}
 \item   If $M \not \in  \{\gamma_1(S),C_S(Z_2(S))\}$, then $M$  has maximal  class  and $\gamma_i(M)= \gamma_{i+1}(S)$ for \blue{$i =2, \dots ,n-1$ and $\gamma_1(M)=\gamma_2(S)$ whenever $n \ge 5$.}
 \item    $M$ is a $2$-step centralizer in $S$ if and only if  $M\in \{\gamma_1(S),C_S(Z_2(S))\}$.

 \item   $S$ is exceptional if and only if  $\gamma_1(S)\ne C_S(Z_2(S))$.
 \item The degree of commutativity of $S$ is positive  if and only if $S$ is not exceptional.
\item If $n$ is odd, or $n=4$ or $n>p+1$, then $S$ is not exceptional.
\item Assume that $|S| \ge p^5$. Then $S/Z(S)$ is not exceptional.
 \item  $\gamma_1(S)= C_S(\gamma_j(S)/\gamma_{j+2}(S))$ for all $2 \le j \le n-3$.

\end{enumerate}
\end{lemma}

\begin{proof}
For statement (i) see \cite[Lemma (1.2)]{VL-L}.

For (ii),  we know that $\gamma_1(S)$  and $C_S(Z_2(S))$ are $2$-step centralizers by definition. That they are the only $2$-step centralizers follows from (i).

Part (iii) follows from (i) and the definition of an exceptional group.

Part (iv) is \cite[Corollary 3.2.7]{Led-Green}.

 Part (v) follows from \cite[Theorems 3.2.11 and 3.3.5]{Led-Green} for $|S|\geq p^5$. If $|S|= p^4$ then $\gamma_1(S)=\C_S(\Z_2(S))$ is abelian and is the unique $2$-step centralizer.

For part (vi). If $n$ is even, then $|S/Z(S)|=p^{n-1}\ge p^5$ is not exceptional by (v).  Whereas, if $n$ is odd, then $S$ is not exceptional by (v) and $\gamma_1(S)$ is the unique $2$-step centralizer in $S$ by (ii).  Since the preimage of a $2$-step  centralizer in $S/Z(S)$, is a $2$-step centralizer in $S$, we deduce that $S/Z(S)$ is not exceptional.

Finally (vii) follows from (vi).

\end{proof}

\begin{lemma}\label{p2centralizer} The following hold:
\begin{enumerate}\item Suppose that $t\in S$ and $t \not \in \gamma_1(S)\cup C_S(Z_2(S))$.  Then $C_S(t)=\langle t, Z(S)\rangle$ has order $p^2$ and $C_S(t) \cap \gamma_1(S)= C_S(t) \cap C_S(Z_2(S))=Z(S)$.
\item If $T$ is a $p$-group and there exists  $t \in T$ with  $|C_T(t)|=p^2$, then $T$ has maximal class.
\end{enumerate}
\end{lemma}

\begin{proof} Suppose that $t\not\in \gamma_1(S)\cup C_S(Z_2(S))$. Then $t$ is not contained in any $2$-step centralizer by Lemma~\ref{mc-factsb} (ii).

Assume for a contradiction that $|C_S(t)| \ge p^3$. Then there exists $y \in (C_S(t) \cap \gamma_1(S))\setminus Z(S)$ so that $y \in \gamma_j(S) \setminus \gamma_{j+1}(S)$ with $\gamma_{j+1}(S) \ge Z(S)$. Now $t$ centralizes $$\gamma_j(S)/\gamma_{j+2}(S)= \langle \gamma_{j+1}(S)/\gamma_{j+2}(S) , y \gamma_{j+2}(S)\rangle,$$ contrary to $t$ not being in a $2$-step centralizer. Since $\gamma_1(S)$ and $C_S(Z_2(S))$ have index $p$ in $S$ and $t \not \in \gamma_1(S)\cup C_S(Z_2(S))$, $C_S(t) \cap \gamma_1(S)= C_S(t) \cap C_S(Z_2(S))=Z(S)$. This proves (i).

Part (ii) is \cite[Satz III.14.23]{Huppert}.
 \end{proof}

The next lemma provides  a weak upper bound for the number of commutators by  $\gamma_1(S)$ required to annihilate $\Omega_1(\gamma_1(S))$.
\begin{lemma}\label{nilp.class.gamma1}
Suppose that $S$ is a maximal class $p$-group of order at least $p^4$. Then $[\Omega_1(\gamma_1(S)),\gamma_1(S);\frac{p+1}{2}]=1$. Furthermore, if $[\Omega_1(\gamma_1(S)),\gamma_1(S);\frac{p-1}{2}]\ne 1$, then $S$ is exceptional.
\end{lemma}

\begin{proof} If $|S|=p^4$, then $\gamma_1(S)$ is abelian and so the statement holds in this case.  Suppose $|S| \ge p^5$. Lemmas~\ref{mc-facts} (iv) and ~\ref{mc-factsb} (vi) give that $|\Omega_1(\gamma_1(S))|\le p^p$ and $S/Z(S)$ is not exceptional.
Hence, the fact that $\gamma_1(S)/Z(S)$ is the unique $2$-step centralizer of $S/\Z(S)$ implies $[\Omega_1(\gamma_1(S)), \gamma_1(S); \frac{p-1}{2}]\leq Z(S)$. This proves the first bound.

Suppose now that $[\Omega_1(\gamma_1(S)),\gamma_1(S);\frac{p-1}{2}]\ne 1$.
 We first show that $$[\Omega_1(\gamma_1(S)), \gamma_1(S); \frac{p-1}{2}- 1]\le  Z_2(S).$$ This is clear if $|\Omega_1(\gamma_1(S))|\le p^{p-1}$, since $\gamma_1(S)$ is the unique $2$-step centralizer of $S/Z(S)$.  Assume $|\Omega_1(\gamma_1(S))|=p^p$. Then $\Omega_1(\gamma_1(S))=\gamma_1(S)$ by Lemma~\ref{mc-facts}(iv) and we know that $\gamma_1(S)/\gamma_4(S)$ is abelian. Hence,in this case we  also have $$[\Omega(\gamma_1(S)), \gamma_1(S);\frac{p-1}{2} - 1]\le  Z_2(S).$$ Therefore we get $1 \neq [\Omega_1(\gamma_1(S)),\gamma_1(S);\frac{p-1}{2}] \leq [Z_2(S), \gamma_1(S)]$. This means that $Z_2(S) \not \le \Z(\gamma_1(S))$ and so  $S$ is exceptional by Lemma \ref{mc-factsb}(iii).
\end{proof}

\begin{lemma}\label{sbgrp-maxclass}
Suppose that $S$ is a maximal class $p$-group of order at least $p^4$. If $T \le S$  and $T\not \subseteq \gamma_1(S)\cup C_S(Z_2(S))$, then $T$ has maximal class.
\end{lemma}

\begin{proof} Suppose that $S$ is a counterexample of minimal order and let $T \leq S$ with $T\not \subseteq \gamma_1(S)\cup C_S(Z_2(S))$ be a subgroup that does not have maximal class. In particular, $T$ is not abelian of order $p$ or $p^2$, and $T \ne S$.    So we have $|T| \ge p^3$.    Let $M$ be a maximal subgroup of $S$ which contains $T$.  Then $M \not \subseteq \gamma_1(S)\cup C_S(\Z_2(S))$ and so $M$ has maximal class by Lemma~\ref{mc-factsb} (i). Therefore $T < M$ and $|M| \ge p^4$.   By Lemma~\ref{mc-factsb} (i), $\gamma_1(M) = \gamma_2(S)$ and $Z_2(M) = Z_2(S)$. In particular, $\gamma_1(M)$ centralizes $Z_2(S)$ and so $\gamma_1(M)= C_{M}(Z_2(M))$ and $T \not \subseteq  \gamma_1(M)\cup C_{M}(Z_2(M))$. Using the minimality of $S$, we conclude that $T$ has maximal class, a contradiction.
\end{proof}

\begin{lemma}\label{mc-normalizer} Suppose that $S$ is a maximal class $p$-group of order at least $p^4$.  If $T \subseteq \gamma_1(S)\cup C_S(Z_2(S))$ and $N_S(T)  \not \subseteq \gamma_1(S)\cup C_S(Z_2(S))$, then $T$ is normal in $S$.
  Furthermore, if $Z(S) \le W \subseteq \gamma_1(S)\cup C_S(Z_2(S))$ and $N_S(W) \not\le \gamma_1(S)$, then $W$ is normal in $S$.
\end{lemma}
\blue{
\begin{proof}
 By Lemma~\ref{sbgrp-maxclass},  $N_S(T) $ has maximal class. Assume that $T$ is not normal in $S$.  Then $T \ne C_S(Z_2(S))$ and $T \ne \gamma_1(S)$. Hence, as $N_S(T) \not \le \gamma_1(S)$ or $C_S(Z_2(S))$, we have $|N_S(T):T|\ge p^2$. Since $N_S(T)$ has maximal class and normalizes $T$, $T= \gamma_i(N_S(T))$ for some $i \ge 2$.  In particular, $T$ is characteristic in $N_S(T)$ and this means $T$ is normal in $S$, a contradiction.
 
 Now suppose that $Z(S) \le W$. If $S$ is not exceptional, then we have the $N_S(W)\not \le\gamma_1(S)= \gamma_1(S) \cup C_S(Z_2(S))$ and so $W$ is normal in $S$. If $S$ is exceptional, then $|S| \ge p^6$ and $S/Z(S)$ is not exceptional. Now the result follows by applying our primary statement to $W/Z(S)$ in $S/Z(S)$.
\end{proof}}

We close this subsection by determining some of the finite simple groups which have a maximal class Sylow $p$-subgroup.

\begin{lemma}\label{maxlclass simple} Suppose that $p \ge 5$ is a prime,  $G$ is a finite simple group and $S \in \syl_p(G)$.  Assume that $S$ has maximal class and has no abelian subgroup of index $p$. Then $S$ is isomorphic to a Sylow $p$-subgroup of $\G_2(p)$, $\gamma_1(S)$ is extraspecial and either $G\cong \G_2(p)$ or $p=5$ and $G\cong \Ly, \HN$ or $\B$ or $p=7$ and $G \cong\Mo$.
\end{lemma}

\begin{proof} \blue{We may assume that $|S| \ge p^4$.
 We consider each of the types of non-abelian simple groups}.  If $G$ is an alternating group of degree $d$, then, as $S$ has  no abelian subgroup of index $p$, $d\ge  p^3$. But then  $S$ contains $p^p$ commuting $p$-cycles \blue{and has order greater than $p^{p+2}$.} This contradicts Lemma~\ref{mc-facts} (iv).

Suppose $G$ is a Lie type group in characteristic $p$. We use the fact that $|S/\gamma_2(S)|=p^2$. Using \cite[Theorem 3.3.1]{GLS3} we see that $G$ has untwisted \blue{Lie-rank} at most $2$ and that \blue{$G$} is defined over $\GF(p)$ if it has rank 2. Since $S$ has no abelian subgroup of index $p$, $G$ is not $\PSL_2(p^2)$,  $\PSU_3(p)$, $\PSL_3(p)$, $\PSp_4(p)$. This leaves $\G_2(p)$.  Hence the result holds for Lie type groups in characteristic $p$.

Suppose that $G$ is a Lie type group in characteristic $r$ with $r \ne p$. We use \cite[Theorem 4.10.2]{GLS3}. Thus we can write $S= P_T P_W$ where $P_T$ is an abelian normal subgroup of $S$ and $P_W$ is a complement to $P_T$.  Since $$\gamma_2(S)=[S,S]= [P_TP_W,P_TP_W] = [P_T,P_W]P_W',$$ $|S/\gamma_2(S)|=p^2$ implies that $P_W/P_W'$ has order $p$.  But then, $P_W$ has order $p$ and $S$ has an abelian subgroup of index $p$, a contradiction.

Finally, suppose that $G$ is a sporadic simple group.  We deploy the tables in \cite[Tables 5.3]{GLS3}.  The sporadic groups listed in the conclusion of our lemma have Sylow $p$-subgroups isomorphic to those of $\G_2(5)$  in the first three cases and to those of $\G_2(7)$ in the last case.
Since $p \ge 5$, we only need to think about the Sylow $5$-subgroup of the monster $\Mo$. This group has order $5^9$. Using \cite[Table 5.3z]{GLS3}, the monster has a $5$-local subgroup $5^{1+6}.((4\circ 2^{.}\J_2).2)$. Thus $S/Z(S)$ is of maximal class and has a normal subgroup of $5$-rank $6$, as $|S/Z(S)|= 5^8> 5^7$, this contradicts Lemma~\ref{mc-facts}(v).
\end{proof}

We give a list of the realizable fusion systems on maximal class $p$-groups which have an abelian subgroup of index $p$ in  Appendix~\ref{AppA}, Table~\ref{Tab3}. This table is extracted from \cite[Table 2.2]{p.index2}.

\subsection{Automorphisms of $p'$-order}\label{SS3.1}

We now present some conclusions about the automorphism group   of $p$-groups $S$ of maximal class of order $p^n$ with $n \ge 4$. \blue{Our first result is well-known.
\begin{theorem}\label{full auto S}
Suppose that $S$ is a $p$-group of maximal class and order at least $p^4$. Then $\Aut(S)/O_p(\Aut(S))$ is isomorphic to a subgroup of the diagonal matrices in $\GL_2(p)$. In particular, if $H$ is a subgroup of $\Aut(S)$ and $|H|$ is coprime to $p$, then $H$ is abelian and is isomorphic to a subgroup of the diagonal matrices in $\GL_2(p)$.
\end{theorem}

\begin{proof} We know that $|S/\gamma_2(S)|=p^2$ and $S/\gamma_2(S)$ is elementary abelian. Hence $\Aut(S)/C_{\Aut(S)}(S/\gamma_2(S))$ is isomorphic to a subgroup of $\GL_2(p)$. Since $\gamma_1(S)$ is a characteristic subgroup of $S$ and $|\gamma_1(S)/\gamma_2(S)|=p$, $\Aut(S)/C_{\Aut(S)}(S/\gamma_2(S))$ is isomorphic to a subgroup of the lower triangular matrices of $\GL_2(p)$. By Burnside's Theorem \cite[5.1.4]{Gor}, $C_{\Aut(S)}(S/\gamma_2(S))$ is a $p$-group. The result now follows.
 \end{proof}}
One consequence of Theorem~\ref{full auto S} is that if $\varphi \in \Aut(S)$ has order coprime to $p$, then by Maschke's Theorem there exists a maximal subgroup $M$ of $S$ with $M \ne \gamma_1(S)$ such that $M \varphi=M$.
Choose $x \in M \setminus \gamma_1(S)$ and $s_1 \in \gamma_1(S)\setminus \gamma_2(S)$ and define
\[ s_i = [x, s_{i-1}] ~ \text{ for every } ~ 2\leq i \leq n-2 ~ \text{ and } \]
\[ s_{n-1} = \begin{cases} [x, s_{n-2}] \quad \text{ if }M =\C_S(\Z_2(S)) \\
						[s_1, s_{n-2}] \quad \text{otherwise.}\end{cases}\]

The choice of $x$ and $s_1$ is plainly not unique, and it is even possible that $[s_1,s_2] \in \gamma_j(S)$ and for a different choice of $x$ and $s_1$ we have $[s_1,s_2] \in \gamma_k(S)$ with $k < j$. In particular, there are examples of maximal class groups $S$ and a choice of $x$ and $s_1$ such that $[s_1,s_2]=1$  and $\gamma_1(S)$ is non-abelian.

The next lemma is extracted from part of the  proof of \cite[Theorem 2.19]{pearls}.

\begin{lemma}\label{action}
\blue{Suppose that $S$ is a  maximal class $p$-group of order  $p^n$ with $n \ge 4$.} Let $\varphi \in \Aut(S)$ be an automorphism of order coprime to $p$ and let $x \in S \setminus \gamma_1(S)$ be such that $\varphi$ leaves   $\langle x \rangle\gamma_2(S)$ invariant.
If $a,b \in \GF(p)$ are such that $x\varphi \equiv x^a \mod \gamma_2(S)$ and $s_1\varphi  \equiv s_1^b \mod \gamma_2(S)$ then
\begin{equation*}
\begin{split}
 s_i \varphi & \equiv s_i^{a^{i-1}b} \mod \gamma_{i+1}(S) ~ \text{ for every } ~ 1 \leq i\leq n-2 ~ \text{ and } \\
s_{n-1}\varphi &= \begin{cases} s_{n-1}^{a^{n-2}b} &\text{ if  } S \text{ is not exceptional} \\s_{n-1}^{a^{n-3}b^2} &\text{ if  } S \text{ is  exceptional.}
\end{cases}
\end{split}
\end{equation*}
\end{lemma}

\blue{
\begin{proof} We demonstrate the result by induction on $i$.
If $i=1$, then this is just the definition of $b$.
Assume $1<i< n-2$. Then by the inductive hypothesis there exists $ u\in \gamma_2(S)$, $v\in \gamma_i(S)$ such that
\[ s_i \varphi = [x, s_{i-1}]\varphi = [x^a u, s_{i-1}^{a^{i-2}b}v] \]
Thus
\[    s_i \varphi \equiv [x^a, s_{i-1}^{a^{i-2}b}] \mod \gamma_{i+1}(S) \equiv s_i^{a^{i-1}b} \mod \gamma_{i+1}(S). \]

The same argument works for $i=n-1$ when $S$ is not exceptional and this yields the result in this case.

If $S$ is exceptional and $i=n-1$, then we have
\[ s_{n-1}\varphi = [s_1, s_{n-2}]\varphi = [s_1^b u, s_{n-2}^{a^{n-3}b}v] = s_{n-1}^{a^{n-3}b^2},\]
for some $u\in \gamma_2(S)$ and $v\in Z(S)$ and this yields the result.
\end{proof}}

\begin{lemma}\label{lem:autS p'elts}
Suppose that $S$ is a $p$-group of maximal class and order at least $p^4$. Assume that $\gamma_1(S)$ is not abelian or extraspecial. Then every non-trivial automorphism of $S$ of order coprime to $p$ acts faithfully on $S/\gamma_1(S)$. \blue{In particular, $C_{\Aut(S)}(S/\gamma_1(S))$ is a $p$-group and $\Aut(S)/C_{\Aut(S)}(S/\gamma_1(S))$ has order dividing $p-1$.}
\end{lemma}

\begin{proof}
Suppose $\alpha\in\Aut(S)$ has order coprime to $p$ and centralizes $S/\gamma_1(S)$. We show that $\alpha$ is the trivial automorphism. Aiming for a contradiction, suppose that $\alpha$ is non-trivial.
Let $b\in \GF(p)$ be such that $s_1\alpha \equiv s_1^b \mod \gamma_2(S)$. Because $\alpha$ is non-trivial, $b \ne 1$ as $\alpha$ acts faithfully on $S/\gamma_2(S)$ by \cite[Theorem 5.1.4]{Gor}. Employing Lemma \ref{action} with $a=1$, we get
 \begin{equation*}
\begin{split}
 s_i \alpha & \equiv s_i^{b} \mod \gamma_{i+1}(S) ~ \text{ for every } ~ 1 \leq i\leq n-2 ~ \text{ and } \\
s_{n-1}\alpha &= \begin{cases} s_{n-1}^{b} &\text{ if  } S \text{ is not exceptional} \\
						s_{n-1}^{b^2}& \text{ if  } S \text{ is  exceptional}
\end{cases}
\end{split}
\end{equation*}
Since $\gamma_1(S)$ is not abelian, we have $\Z(\gamma_1(S)) = \gamma_{k+1}(S)$ for some $2\leq k \leq n-2$. Then $\gamma_k(S)$ is abelian and is non-central.  Choose $j$ maximal such that $[s_k,s_j] \ne 1$.
As $[s_j, s_k] \neq 1$, $[s_j,s_k] \in \gamma_i(S)\backslash \gamma_{i+1}(S)$ for some $k+1 \leq i \leq n-1$.
Note that $[s_j, s_k] \in \Z(\gamma_1(S))$ commutes with both $s_j$ and $s_k$. Also,
 $s_j\alpha = s_j^bz$ for some $z \in \gamma_{j+1}(S)$, that commutes with $s_k$ by the maximal choice of $j$, and $s_k\alpha = s_k^bu$, for some $u \in \gamma_{k+1}(S) = \Z(\gamma_1(S))$. Therefore using the commutator formulae   we get
\[   [s_j,s_k]\alpha = [s_j\alpha, s_k\alpha] = [s_j^bz, s_k^bu] = [s_j^b, s_k^b]  = [s_j,s_k]^{b^2}.\]
Suppose $S$ is not exceptional.
Then
\[  [s_j,s_k]^{b^2} \equiv [s_j,s_k]^b \mod \gamma_{i+1}(S)\]
and so $b=1$, which is a contradiction.

Suppose $S$ is exceptional. Then $\ov{S}=S/\Z(S)$ is not exceptional by Lemma~\ref{mc-factsb} (vi) and $\alpha$ is a non-trivial automorphism of $\ov{S}$ of order coprime to $p$ centralizing $\ov{S}/\gamma_1(\ov{S})$. Since $S/Z(S)$ is not exceptional,  using what we proved above we obtain $\gamma_1(\ov{S}) = \gamma_1(S)/\Z(S)$ is abelian. Hence $\gamma_1(S)$ is extraspecial, contradicting the hypothesis.
\end{proof}

\begin{cor}\label{c12}
Suppose that $S$ is a $p$-group of maximal class and order at least $p^4$. Assume that $\gamma_1(S)$ is not abelian or extraspecial. If $H$ is a subgroup of $\Aut(S)$ of order coprime to $p$, then $H$ is cyclic of order $m$ dividing $(p-1)$ and $H$ acts faithfully on $S/\gamma_1(S)$. Furthermore, $|\Aut(S)|= p^am$ for some natural number $a$.
\end{cor}

\begin{proof} We know that $\gamma_1(S)$ is a characteristic subgroup of $S$. Hence there is a homomorphism $\theta:\Aut(S) \rightarrow \Aut(S/\gamma_1(S))$ and the latter group is cyclic of order $p-1$.  By Lemma~\ref{lem:autS p'elts}, $\ker \theta$ is a $p$-group. This proves the claim.
\end{proof}

\begin{lemma}\label{action.exceptional1}
Suppose that $S$ is a $p$-group of maximal class and order at least $p^4$. If $S$ is exceptional and $\alpha \in \Aut(S)$ is an involution which inverts $S/\gamma_1(S)$, then $\alpha$  inverts $Z(S)$.
\end{lemma}

\begin{proof} \blue{Because $S$ is exceptional, $n$ is even by Lemma~\ref{mc-factsb} (v).} Since $\alpha$ is an involution which inverts $S/\gamma_1(S)$, $(x\gamma_2(S)) \alpha = x^{-1}\gamma_2(S)$ and  $(s_1\gamma_2(S))\alpha = s_1^b\gamma_2(S)$ where $b= \pm 1$.
By Lemma \ref{action},  $\alpha$ acts on $Z(S)$ \blue{by raising elements to the power}
\[(-1)^{n-3} b^2 = (-1)^{n-3} = -1.\]
Hence $\alpha$ inverts $Z(S)$.
\end{proof}

\begin{lemma}\label{centralizer auto} Suppose that $S$ is a $p$-group of maximal class and order at least $p^4$ that is not exceptional. Let $\alpha$ be an automorphism of $S$  of order $m\neq 1$  in its action on $S/\gamma_1(S)$.  Assume that there exists $c \in \GF(p)$ such that $ s_j \alpha \equiv s_j^c \mod{\gamma_{j+1}(S)}$ and  $s_k\alpha \equiv s_k^c \mod{\gamma_{k+1}(S)}$  for some $j,k\geq 1$.
 Then $j \equiv k \pmod m$.
\end{lemma}

\begin{proof}
As $m$ and $p$ are coprime,  there is $x \in S \setminus \gamma_1(S)$ such that $\alpha$ leaves invariant  $\langle x \rangle\gamma_2(S)$.
Let $a,b \in \GF(p)$ be such that $x\alpha \equiv x^a \mod \gamma_2(S)$ and $s_1\alpha \equiv s_1^b \mod \gamma_2(S)$.
By assumption there is $c\in \GF(p)$ and $j,k\geq 1$ such that
$$s_j \alpha \equiv s_j^c  \mod \gamma_{j+1}(S)  \text{ and } s_k \alpha \equiv s_k^c  \mod \gamma_{k+1}(S) .$$
Since $S$ is not exceptional, by Lemma \ref{action} we have
\[  s_j \alpha  \equiv s_j^{a^{j-1}b} \mod \gamma_{j+1}(S) \text{ and } s_k \alpha  \equiv s_k^{a^{k-1}b} \mod \gamma_{k+1}(S).\]
Therefore $$a^{j-1}b \equiv c  \equiv a^{k-1}b \pmod p.$$
Thus  $a^{j-k} \equiv 1 \pmod p$, that is $j-k \equiv 0 \pmod m$ and so
\[j \equiv k \pmod m.\]
\end{proof}

\subsection{A theorem of Juh\'{a}sz}

For completeness, we now present a proof of  \cite[Theorem 6.2]{J}  slightly modified for our application.

\begin{theorem}[Juh\'{a}sz]\label{J6.2}
Assume that $S$ has maximal class and order at least $p^4$.  If $\gamma_3(S)$ is abelian, then either $\gamma_2(S)$ has nilpotency class at most $2$ or $|S| \ge p^{2p+4}$. Furthermore, if $\gamma_j(S)>1$ is   elementary abelian for some $j \ge 3$, then $\gamma_{j-1}(S)$ has nilpotency class at most $2$.
\end{theorem}
\begin{proof}  Let $k,\ell \geq 1$ be such that $$\gamma_k(S)=[\gamma_3(S),\gamma_2(S)]=\gamma_2(S)' \text{ and }\gamma_\ell(S)=[\gamma_3(S),\gamma_2(S),\gamma_2(S)].$$Assume that $\gamma_2(S)$ does not have nilpotency class $2$.  Then $\gamma_{\ell}(S) \ne 1$. We shall show that $|S| \ge p^{2p+4}$.

Observe that for $j\ge 3$, $$[\gamma_{j}(S),\gamma_2(S)] = \gamma_{j+k-3}(S).$$
To see this, we note the case $j=3$ is just the definition of $k$. Assume that the statement holds for $j \ge 3$. Then $[S,\gamma_2(S),\gamma_j(S)]=[\gamma_3(S),\gamma_j(S)]=1$ because $\gamma_j(S) \le \gamma_3(S)$ and $\gamma_3(S)$ is abelian.  Hence the Three Subgroup Lemma yields the second equality of   \begin{eqnarray*}[\gamma_{j+1}(S),\gamma_2(S)] &=&[\gamma_j(S),S,\gamma_2(S)]= [\gamma_2(S),\gamma_j(S),S]\\&=& [\gamma_{j+k-3}(S),S]= \gamma_{(j+1)+k-3}(S)\end{eqnarray*} and this verifies the observation.
In particular, $\gamma_\ell(S)=[\gamma_{k}(S),\gamma_2(S)] =\gamma_{2k-3}(S)$ and so \begin{center} $\ell =2k-3$ and  $|S|=p^n \ge p^{2k-2}$.\end{center}

Recall the definition of $x$ and $s_i$, $1\le i\le n-1$  from Subsection~\ref{SS3.1} and for $i \ge n$ we set $s_i=1$.

Define $u_3=[s_2,s_1] \in  \gamma_{3}(S)$ and $u_i=[x,u_{i-1}]\in \gamma_i(S)$ for $i \ge 4$.
Since  $\gamma_2(S)/\gamma_k(S)$ is abelian and $[x,s_1]\in \gamma_2(S)$, using the  variant of the Hall-Witt identity
 \[ [w,y,z^w][y,z,w^y][z,w,y^z] =1 \]  and using $[s_2^{-1},s_j] \in Z(S/\gamma_{j+k-2}(S))$ we calculate

\begin{equation*}\label{eq2}
\begin{aligned}\gamma_{j+k-2}(S)&=
[x,s_j,s_1^x][s_j,s_1,s^{s_j}][s_1,x,s_j^{s_1}]\gamma_{j+k-2}(S)\\&=
[s_{j+1},s_1^x][s_j,s_1,x^{s_j}][s_2^{-1},s_j^{s_1}]\gamma_{j+k-2}(S)
\\&=
[s_{j+1},s_1[s_1,x]][[s_j,s_1]^{s_j^{-1}},x]^{s_j} [s_2^{-1},s_j]\gamma_{j+k-2}(S)\\&
=[s_{j+1},s_1s_j^{-1}][s_j,s_1,x][s_j,s_2]\gamma_{j+k-2}(S)\\
&=[s_{j+1},s_1 ][s_j,s_1,x][s_j,s_2]\gamma_{j+k-2}(S).
\end{aligned}\end{equation*}
Hence
\begin{equation}\label{eq2}
\begin{aligned}\;
[s_{j+1},s_1][s_j,s_2]\gamma_{j+k-2}(S)&=[x,[s_j,s_1]]\gamma_{j+k-2}(S);\text{ and }\\
 [s_3,s_1 ]\gamma_k(S)&=u_4\gamma_k(S).
 \end{aligned}
\end{equation}
Using $\gamma_2(S)/\gamma_\ell(S)$ has class $2$ and $\gamma_3(S)$ is abelian together with  the Hall-Witt identity and we have, for $t \ge 3$, \begin{eqnarray*}\gamma_\ell(S)
&=&[x,s_t,s_2]^{s_t^{-1}}[s_t^{-1},s_2^{-1},x]^{s_2}
[s_2,x^{-1},s_t^{-1}]^x\gamma_\ell(S)\\&=&
[x,s_t,s_2][s_t^{-1},s_2^{-1},x]\gamma_\ell(S)\\&=& [x,s_t,s_2][s_t,s_2,x]\gamma_\ell(S). \end{eqnarray*}
So we know for $t \ge 3$
\begin{equation}\label{eqnew} [s_{t+1},s_2]\gamma_\ell(S)=[s_t,s_2,x]^{-1}\gamma_\ell(S)=[x,[s_t,s_2]]\gamma_\ell(S).\end{equation}

We claim that, for $ j \ge 3$,  \begin{equation}\label{eq3}[s_j,s_1][ s_{j-1},s_2]^{j-3}\gamma_{j+k-3}(S)\gamma_\ell(S)= u_{j+1}\gamma_{j+k-3}(S)\gamma_\ell(S).\end{equation}
 This is valid for $j=3$ by (\ref{eq2}).
  We prove the claim by induction. We have $$[x,[s_j,s_1][ s_{j-1},s_2]^{j-3}]\in [x,u_{j+1}]\gamma_{j+k-2}(S)\gamma_\ell(S)= u_{j+2}\gamma_{j+k-2}(S)\gamma_\ell(S).$$
  \begin{eqnarray*}
 [x,[s_j,s_1][ s_{j-1},s_2]^{j-3}] &\in&
[x,[s_{j-1},s_2]^{j-3}][x,[s_j,s_1]]^{[s_j,s_1]^{j-3}}\gamma_{j+k-2}(S)\gamma_\ell(S)\\
&=&[x,[ s_{j-1},s_2]]^{j-3}[x,[s_j,s_1]]\gamma_{j+k-2}(S)\gamma_\ell(S)\\
&=
&[ s_{j},s_2]^{j-2} [s_{j+1},s_1]\gamma_{j+k-2}(S)\gamma_\ell(S) \end{eqnarray*}
where we have used $\gamma_3(S)$ is abelian for the second equality and (\ref{eq2}) and (\ref{eqnew}) for the   third.
Thus
(\ref{eq3}) follows by induction.

Write $[s_3,s_2]= \prod_{j=k}^{n-1} s_j^{a_j}$ where $0 \le a_j \le p-1$ and $a_k \ne 0$. As $\gamma_3(S)$ is abelian, we obtain
$$[s_3,s_2,s_1]=[ \prod_{j=k}^{n-1} s_j^{a_j},s_1]= \prod_{j=k}^{n-1}[ s_j,s_1]^{a_j}.$$
Hence, using equation (\ref{eq3}), taking suitable $g_{j+k-3} \in \gamma_{j+k-3}(S) $,   observing that $[s_2,s_j ] \in \gamma_{j+k-3}(S) \le \gamma_\ell(S)$ for $j \ge k$ and remembering $\gamma_3(S)$ is abelian, we calculate
\begin{equation}\label{eq5}\begin{aligned}&[s_3,s_2,s_1]\gamma_{\ell}(S)= \prod_{j=k}^{n-1}[ s_j,s_1]^{a_j}\gamma_{\ell}(S)=\prod_{j=k}^{n-1}( u_{j+1}[s_2,s_{j-1}]^{j-3}g_{j+k-3}) ^{a_j}\gamma_{\ell}(S)\\
&= [s_2, s_{k-1}]^{a_k(k-3)}\prod_{j=k}^{n-1} u_{j+1}^{a_j}g_{j+k-3}^{a_j} \gamma_{\ell}(S)= [s_2 , s_{k-1}]^{a_k(k-3)}\prod_{j=k}^{n-1} u_{j+1}^{a_j}  \gamma_{\ell}(S)\end{aligned}\end{equation}
where $g_{j+k-3} \in \gamma_{j+k-3}(S) \le \gamma_{\ell}(S)$ as, for $j \ge k$,  $j+k-3 \ge 2k-3\ge\ell$.
We next determine
$\prod_{j=k}^{n-1} u_{j+1}^{a_j}  \gamma_{\ell}(S).$
  So decompose $$u_3= [s_2,s_1]= \prod_{t= 3}^n s_t^{b_t}.$$ Then, since $\gamma_3(S)$ is abelian, $$u_4=[x,u_3]= [x,\prod_{t=3}^n s_t^{b^t}]=\prod_{t=3}^n[x,s_t]^{b_t}=\prod_{t=3}^n s_{t+1}^{b_t}.$$
and by induction, for $r \ge 3$, $u_{r}= \prod_{t=3}^n s_{t+r-3}^{b_t}$.

  Using induction and (\ref{eqnew}), for $t \ge 3$, $[x,[s_t,s_2]]\gamma_\ell(S)=\prod_{j=k}^ns_{j+t-2}^{a_j}\gamma_\ell(S)$ and so
\begin{eqnarray*}[u_4,s_2]\gamma_\ell(S)
& =&  \prod_{t= 3}^n [s_{t+1},s_2]^{b_t}\gamma_\ell(S)
=\prod_{t= 3}^n ([x,[s_t,s_2])^{b_t}\gamma_\ell(S)
\\&=&\prod_{t=3}^n \left(\prod_{j=k}^ns_{j+t-2}^{a_j}\right)^{b_t}\gamma_\ell(S)\\
&=&\prod_{j=k}^n\left(\prod_{t=3}^ns_{j+t-2}^{b_t}\right)^{a_j}\gamma_\ell(S)=
\prod_{j=k}^nu_{j+1}^{a_j}\gamma_\ell(S).\end{eqnarray*}
In combination with equation (\ref{eq5}) this provides
\begin{equation}\label{eq52}[s_3,s_2,s_1]\gamma_{\ell}(S)= [s_2,s_{k-1}]^{a_k(k-3)}[u_4,s_2]  \gamma_{\ell}(S).\end{equation}
On the other hand, commutating Equation (\ref{eq2}) on the right with $s_2$ yields
\begin{equation}\label{eq2.5}[s_3,s_1,s_2 ]\gamma_{\ell}(S)=[u_4g_k,s_2]\gamma_{\ell}(S)= [u_4,s_2]\gamma_{\ell}(S).\end{equation}
Since $[s_3,s_2,s_1]\gamma_\ell(S)=[s_3,s_1,s_2 ] \gamma_\ell(S)$, we deduce that $[s_2,s_{k-1}]^{a_k(k-3)}\in \gamma_\ell(S)$.
However, $$\langle [s_{k-1},s_2] \rangle \gamma_\ell(S)= [\gamma_{k-1}(S),\gamma_2(S)]= \gamma_{2k-4}(S) > \gamma_\ell(S)=\gamma_{2k-3}(S),$$  and therefore $[s_{k-1},s_2] \not\in \gamma_\ell(S)$ and so $a_k(k-3) \equiv 0 \pmod p$.  Since $0<a_k \le p-1$, $k-3 \equiv 0 \pmod p$ and, as $k> 3$, $k-3= mp$ for some $m\ge 1$.
Hence $n \ge 2k-2 \ge 2p+4$.
We conclude, $|S| \ge p^{2p+4}$. This proves the main statement.

Assume that $\gamma_j(S)$ is non-trivial and elementary abelian for some $j \ge 3$. Let $x \in S \setminus (\gamma_1(S)\cup C_S(Z_2(S)))$. Then $x^p \in C_{S}(x) \cap \gamma_1(S)= Z(S) \le \gamma_2(S)$ by Lemma~\ref{p2centralizer}. Hence $T=\langle x \rangle \gamma_{j-2}(S)$ has order $p^{n-j+2} \ge p^4$. By Lemma~\ref{sbgrp-maxclass}, $T$ has maximal class and $\gamma_j(S)= \gamma_3(T)$. Hence $\gamma_3(T)$ is elementary abelian. Lemma~\ref{mc-facts} (iv) implies $|\gamma_3(T)| \le p^{p-1}$ and thus $|T| \le p^{p+2}$.  Since $p+2 < 2p+4$, the main claim yields $\gamma_2(T)= \gamma_{j-1}(S)$ has nilpotency class at most $2$. This completes the proof.
\end{proof}

%% file: Reps.tex
\section{Representations of groups with a cyclic Sylow $p$-subgroup}\label{sec:Reps}

In this section, for $p$ an odd prime, we gather together various facts about  representations of groups with cyclic Sylow $p$-subgroups.

\begin{definition}\label{def:l2p} A group $X$ is said to be of \emph{$\mathrm L_2(p)$-type} provided each composition factor of $X$ is either a $p$-group, a $p'$-group or is isomorphic to $\PSL_2(p)$.
\end{definition}
We will require the following result due to Feit.
\begin{theorem}[Feit]\label{feit}  Suppose that $p$ is a prime, $\mathbb K$ is a field of characteristic $p$, $L$ is a finite group with a cyclic Sylow $p$-subgroup $P$ and $V$ is a faithful indecomposable $\mathbb KL$-module with $d=\dim V \le p$. Assume that $L$ is not of  $\mathrm L_2(p)$-type. Then $p$ is odd, $|P|=p$, $V|_P$ is indecomposable, $C_L(P)= P\times Z(L)$ and $d \ge \frac 2 3 (p-1)$.
\end{theorem}

\begin{proof} This is \cite[Theorem 1]{Feit}. \end{proof}

Theorem~\ref{feit} illuminates the importance of representations of  $\SL_2(p)$. Some of the results in this section hold for an arbitrary field $\mathbb K$ but our applications will only be for $\mathbb K =\GF(p)$. We follow the standard construction of certain  irreducible modules for $\GL_2(\mathbb K)$. Let $ \mathbb K[x,y]$ be the polynomial ring in two commuting variables $x$ and $y$ and coefficients in $\mathbb K$. Then, for $ \left(\begin{smallmatrix}\alpha &\beta\\\gamma&\delta \end{smallmatrix}\right)\in \GL_2(\mathbb K)$ and $a,b \ge 0$ natural numbers, the extension linearly of   $$x^ay^b \cdot  \left(\begin{smallmatrix}\alpha &\beta\\\gamma&\delta \end{smallmatrix}\right)=  (\alpha x + \beta y)^a(\gamma x + \delta y)^b$$
 makes $\mathbb K[x,y]$ into a  $\mathbb K\GL_2(\mathbb K)$-module. The subspaces of  $ \mathbb K[x,y]$ consisting of homogenous polynomials of fixed  degree at most $p-1$ provide us with an explicit construction of  the \emph{basic} irreducible $\mathbb K\SL_2(p)$-modules by restriction.

  \begin{notation} For $0  \le e \le p-1$,  $\VV_e$   represents the $(e+1)$-dimensional $\GL_2(\mathbb K)$-submodule of $\mathbb K[x,y]$ consisting of degree $e$ homogeneous polynomials. We use the same notation for $\VV_e$ when we consider $\VV_e$ as a module for certain subgroups of $\GL_2(\mathbb K)$, for example, when considered as a $\mathbb K\SL_2(p)$-module. We often call $\VV_1$ the natural $\mathbb K\SL_2(p)$-module.
  \end{notation}

 For  $\SL_2(p)$, every irreducible $\mathbb K\SL_2(p)$-module is basic and can be realized over $\GF(p)$ (see \cite[page 15]{Alperin} for example). In particular, in this case there are $p$ irreducible modules and they have dimensions $1, 2, \dots ,p$. The faithful modules are the ones of even-dimension and the odd-dimensional modules are   representations of $\PSL_2(p)$.

The next eight results provide  the facts that we shall require about these  representations. The first result is used silently in the text.

\begin{lemma}\label{uniserial} Suppose that $L \cong \SL_2(p)$,  $T \in \Syl_p(L)$ and $V$ is an irreducible $\GF(p)\SL_2(p)$-module.  Then $V$ is indecomposable as a $\GF(p)T$-module. In particular, for all $0\le k \le \dim V-1$, $\dim [V,T;k]/[V,T;k+1]=1$   and,  if $\dim [V,T]=1$, then $V\cong \VV_1$   has dimension $2$.
\end{lemma}

\begin{proof} This is calculated using the description of the modules above.\end{proof}

Obviously, if $p$ is odd and $V$ is a faithful $\SL_2(p)$-module, then the centre of   $\SL_2(p)$ negates $V$ and so a complement to a Sylow $p$-subgroup of $\SL_2(p)$ acts fixed-point-freely on any faithful module.  The same is not true for the irreducible $\PSL_2(p)$-modules.

\begin{lemma}\label{L2p}
Suppose that   $L \cong \SL_2(p)$, $T \in \Syl_p(L)$ and $H $ is a complement to $T$ in $N_L(T)$. Assume that $V$ is an   irreducible $d$-dimensional  $\GF(p)L$-module. If $C_V(H)\ne 0$, then $d$ is odd  and either
\begin{enumerate}
\item  $d\le p-2$, $\dim C_V(H)=1$ and $[V,T;(d-1)/2]/[V,T;(d+1)/2]$ is centralized by $H$; or
\item $d= p$, $\dim C_V(H)=3$ and $V/[V,T]$, $[V,T;(p-1)/2]/[V,T;(p+1)/2]$ and $C_V(T)$ are centralized by $H$.
\end{enumerate}
In particular, if $C_V(T)$ is centralized by $H$, then either $V$ is the trivial module or $\dim V= p$ and $\dim C_V(H)=3$.
\end{lemma}

\begin{proof}  Let $e=d-1$ and remember that $V= \VV_{e}$. Take $\tau =\left(\begin{smallmatrix}1&1\\0&1\end{smallmatrix}\right)$ to be a generator of $T$ and $\delta = \left(\begin{smallmatrix}\lambda&0\\0&\lambda^{-1}\end{smallmatrix}\right)$ where $\lambda \in \GF(p)$ has order $p-1$ to be a generator of $H$.
 Then we calculate $$[V,T;k] = \langle x^{j}y^{e-j}\mid 0\le j\le e-k\rangle.$$ We also calculate that $\delta$ acts as the scalar $\lambda^{e-k} \lambda^{-k}= \lambda ^{e-2k}$ on the quotient $$[V,T;k]/[V,T;k+1]= \langle x^{e-k}y^k+[V,T;k+1]\rangle.$$ Hence $\tau$ centralizes $[V,T;k]/[V,T;k+1]$ if and only if either $k= e/2$ or $e=p-1$ and $k= 0$ or $p-1$. In particular, $e=d-1$ is even and this gives the result.
\end{proof}

Assume that $\mathbb K$ has characteristic $p \ge 0$ \blue{(allowing $p=0$ for a moment)} and that $\VV_d$ is the $(d+1)$-dimensional $\mathbb K\GL_2(\mathbb K )$-module of homogeneous polynomials of degree $d$.
 Define
 \begin{eqnarray*} \Omega : \mathbb K[x,y] \otimes  \mathbb K[x,y] &\rightarrow & \mathbb K[x,y]\otimes \mathbb K[x,y]\\
 (f\otimes g) &\mapsto & \frac{ \partial f}{\partial x}\otimes \frac{\partial g}{\partial y}-\frac{ \partial f}{\partial y}\otimes \frac{\partial g}{\partial x}.
\end{eqnarray*}
Then $\Omega$ is $\mathbb K$-linear. Let  $A=\left(\begin{smallmatrix}1 &1\\0&1\end{smallmatrix}\right) $ and  $B_\theta=\left(\begin{smallmatrix}0 &-\theta\\1&0\end{smallmatrix}\right)$ with $\theta \in \mathbb K^\times$ be elements of $\GL_2(\mathbb K)$. Then $\GL_2(\mathbb K)=\langle A, B_\theta \mid \theta \in \mathbb K^\times \rangle$. We  calculate

\begin{eqnarray*}
(x^ay^b\otimes x^cy^d)A\Omega &=&((x+ y)^a y^b\otimes  (x+ y)^c y^d)\Omega\\
&=&a( x+ y)^{a-1}y^{b}\otimes\left(c(x+y)^{c-1}y^{d}+d(x+y)^c y^{d-1}\right)\\
&&-\left(a(x+y)^{a-1}y^{b}+ b(x+y)^ay^{b-1}\right)\otimes c(x+y)^{c-1}y^{d}\\
&=&ad( x+ y)^{a-1}y^{b}\otimes (x+y)^c y^{d-1}\\&&- bc (x+y)^ay^{b-1}\otimes (x+y)^{c-1}y^{d}\\
&=&\left(adx^{a-1}y^{b}\otimes x^cy^{d-1}- bc x^{a}y^{b-1}\otimes x^{c-1}y^{d}\right)A\\&=&
(x^ay^b\otimes x^cy^d)\Omega A.
\end{eqnarray*}
Hence $\Omega A=  A \Omega.$
Similarly, we calculate  that
$B_\theta \Omega  = \theta\Omega B_\theta$ and so $C\Omega = (\det C) \Omega C$ for all $C \in \GL_2(\mathbb K)$.    In particular, $\Omega$ is a $\mathbb K\SL_2(\mathbb K)$-module homomorphism. The multiplication map $\mu: \mathbb K[x,y]\otimes \mathbb K[x,y]\rightarrow\mathbb K[x,y]$ defined by  $(f\otimes g)\mu \mapsto fg$ is also a $\mathbb K\SL_2(\mathbb K)$-module homomorphism. It follows that, for $r \ge 0$, the \emph{$r$-transvectant} $$\Theta_r=\Omega^r \mu: \mathbb K[x,y]\otimes \mathbb K[x,y]\rightarrow \mathbb K[x,y]$$ is also a  $\mathbb K\SL_2(\mathbb K)$-module homomorphism.

Suppose that $d $ and $e$ are natural numbers and observe that by restriction $$\Theta_r: \VV_d\otimes \VV_e \rightarrow \VV_{d+e-2r}.$$
Assume now that $\mathbb K$ has characteristic $p>0$. Then, for $\ell \le p-1$, $\VV_{\ell}$ is an irreducible $\mathbb K\SL_2(\mathbb K)$-module. Therefore, if $d+e-2r \le p-1$, the restriction of $\Theta_r$ is either the zero map or is a surjection. In particular, if $r\le e \le d$ and $d+e \le p-1$, then, as $(x^d\otimes y^e)\Theta_r \ne 0$, $\Theta_r$ restricts to a surjection. By counting  dimensions, we conclude  $$\VV_d \otimes \VV_e \cong \VV_{d+e}\oplus \VV_{d+e-2}\oplus \dots \oplus \VV_{d-e}$$ which in characteristic $0$ is known as the \emph{Clebsch-Gordan decomposition}. Let $\iota$ be the $\mathbb K \SL_2(\mathbb K)$-module endomorphism of $\mathbb K[x,y]$ which maps $f \otimes g$ to $g \otimes f$. Then $\Theta_r \iota=(-1)^r\Theta_r $ and, for a fixed natural number $d$, $\Lambda^2(\VV_d)$ is the submodule of $\VV_d\otimes \VV_d$ negated by $\iota$ and $S^2(\VV_d)$ is the submodule centralized by $\iota$.
We have recreated the following
\begin{prop}\label{Clebsch Gordan} Suppose that $d, e\in \mathbb N$ with $d\ge e $ and $d+e \le p-1$. Then, as $\mathbb K\SL_2(\mathbb K)$-modules,$$\VV_d \otimes \VV_e \cong \VV_{d+e}\oplus \VV_{d+e-2}\oplus \dots \oplus \VV_{d-e}.$$
Furthermore, if $2d \le p-1$,
 $$S^2(\VV_d)\cong \VV_{2d}  \oplus \VV_{2d-4}\oplus \dots\oplus \VV_{a} $$ and $$\Lambda^2(\VV_d)= \VV_{2d-2} \oplus \VV_{2d-6}\oplus \dots \oplus \VV_{b}$$
where $a = 2d \pmod 4$ and $b = 2d-2 \pmod 4$ with $a, b \le 2$.\qed
\end{prop}

Notice that, if we take $d+1= (p-1)/2$, then $\VV_d$ is involved in $\Lambda^2(\VV_d)$ if and only if $d \equiv 2 \pmod 4$.

\begin{lemma}\label{tensor1} We have $$\VV_{p-3} \otimes \VV_{p-3} = \VV_0\oplus \VV_2\oplus \VV_{p-1}\oplus  P(p-3)\oplus \dots  \oplus P(4)$$
where, for $1 \le j \le p-1$,  $P(j)$ is the projective cover of $\VV_j$. In particular, $(\VV_{p-3} \otimes \VV_{p-3})/\mathrm{Rad}(\VV_{p-3} \otimes \VV_{p-3})$ is the direct sum of all the irreducible $\GF(p)\PSL_2(p)$-modules.
\end{lemma}

\begin{proof} It suffices to work over an algebraically closed field.
From \cite[Lemma 3.1 (ii)]{Craven2013}, we have $$\VV_{p-3} \otimes \VV_{p-3} = \VV_0\oplus \VV_2\oplus \VV_{p-1}\oplus  T(p+1)\oplus \dots  \oplus T(2p-6)$$ where, for an integer $j$, $T(j)$ is the \emph{tilting module} associated to $j$.  We know $\VV_{p-1}$ has  dimension $p$ and  is a projective module (see \cite{Alperin}). Using \cite[Lemma 3.1 (iii)]{Craven2013} $$\VV_{p-3}\otimes \VV_{p-1}=\VV_{p-1}\oplus T(p+1)\oplus \dots \oplus T(2p - 4).$$ Since $\VV_{p-1}$ is projective, so is   $\VV_{p-3}\otimes \VV_{p-1}$ and every direct summand of  $\VV_{p-3}\otimes \VV_{p-1}$  by \cite[Lemma 1.5.2]{Benson1}. Hence the tilting modules $T(p+1), \dots, T(2p-6)$ are projective $\GF(p)\SL_2(p)$-modules. Now, the discussion before \cite[Lemma 3.1]{Craven2013} reveals that for $p+1\le j \le 2p-2$,  $T(j)$ has a quotient $\VV_{2(p-1)-j}$ and dimension $2p$.  We conclude that $T(j)$ is the projective cover $P(2(p-1)-j)$ of $\VV_{2(p-1)-j}$.  With this notation, we have
$$\VV_{p-3} \otimes \VV_{p-3} = \VV_0\oplus \VV_2\oplus \VV_{p-1}\oplus  P(p-3)\oplus \dots  \oplus P(4).$$
In particular, we note that every irreducible $\GF(p)\PSL_2(p)$-module  appears exactly once as a quotient of $(\VV_{p-3} \otimes \VV_{p-3})/\mathrm{Rad}(\VV_{p-3} \otimes \VV_{p-3})$. This proves the claim.
\end{proof}

We recall from \cite[pages 15 and 48]{Alperin} that $P(0) $ has dimension $p$ and is a uniserial $\GF(p)\PSL_2(p)$-module with composition factors $\VV_0$, $\VV_{p-3}$ and $\VV_0$.  In particular, there is a unique indecomposable $\GF(p)\PSL_2(p)$-module with socle of dimension $1$ and quotient $\VV_{p-3}$.

\begin{lemma}\label{tensor2}  Suppose that $L \cong \PSL_2(p)$ and  $W$ is the  unique indecomposable $\GF(p)L$-module with socle of dimension $1$ and quotient $\VV_{p-3}$. Then $\Lambda^2 (W)$ has a submodule $U\le \mathrm{Rad}(\Lambda^2(W))$ with $U \cong \VV_{p-3}$ and  $\Lambda^2(W)/U \cong \Lambda^2(\VV_{p-3})$.
\end{lemma}

\begin{proof} Let $R$ be the socle of $W$.  Then  $W \otimes W$ has a submodule $$U^*=\langle r\otimes w, w \otimes r\mid r \in R, w \in W\rangle.$$ Plainly $W/U^* \cong \VV_{p-3}\otimes \VV_{p-3}$. Set $U= \langle w\otimes r - r\otimes w \mid r \in R, w \in W\rangle$.  Then $U \cong W/R= \VV_{p-3}$ and $\Lambda^2(W)/U \cong \Lambda^2(\VV_{p-3})$.  Hence we only need to show that $U \le \mathrm{Rad}(\Lambda^2(W))$. Since $\Lambda^2(W)$ is a direct summand of $W\otimes W$, we have $\mathrm{Rad}(\Lambda^2(W))= \mathrm{Rad}(W \otimes W) \cap \Lambda^2(W)$. In particular, if $U \not \le \mathrm{Rad}(\Lambda^2(W))$, then $U$ is a direct summand of $W \otimes W$. So suppose that this is the case.  Then $W \otimes W \cong \VV_{p-3} \oplus (\VV_{p-3} \otimes \VV_{p-3})$.  Let $T \in \syl_p(L)$.  Then using Lemma~\ref{tensor1}, as a $\GF(p)T$-module $\VV_{p-3} \oplus (\VV_{p-3} \otimes \VV_{p-3})$ is a sum of indecomposable modules $\VV_0$ of dimension 1, $\VV_2$ of dimension $3$, $\VV_{p-3}$ of dimension $p-2$ and a number of free modules of dimension $p$. On the other hand, as $W$ is indecomposable with socle $\VV_0$  and quotient $\VV_{p-3}$, using Lemma~\ref{uniserial} and \cite[Corollary 3.6.10]{Benson1} we have that  $W$ restricted to $T$ is indecomposable of dimension $p-1$. It follows that,  as a $\GF(p)T$-module, $W \otimes W$ is a direct sum of a trivial module and a free module (for example  use \cite[Lemma 3.1 (ii)]{Craven2013} to write down $\VV_{p-2}\otimes \VV_{p-2}$ and then restrict to $T$). Since the two structures are incompatible, we conclude that $U \le \mathrm{Rad}(\Lambda^2(W))$.
\end{proof}

We can now establish the technical point that we require.

\begin{lemma}\label{non-zero}  Suppose that $L \cong \PSL_2(p)$, $T \in \Syl_p(L)$, $H\le N_L(T)$ is a complement to $T$  and  $W$ is the  unique indecomposable $\GF(p)L$-module with socle of dimension $1$ and quotient $\VV_{p-3}$.  Assume that $\theta: W \times W \rightarrow \VV_{p-3}$ is a surjective alternating $L$-invariant bilinear map.  Let $u \in W\setminus [W,T]$ and $w \in [W,T;p-3]\setminus C_W(T)$ be such that $\langle w\rangle$ and $\langle u\rangle$ are $H$-invariant. Then $(u,w)\theta \ne 0$.
\end{lemma}

\begin{proof} Since $\theta: W \times W \rightarrow \VV_{p-3}$ is a surjective $L$-invariant bilinear map, there is a unique surjective $\GF(p)L$-module homomorphism $\wt \theta$ from $\Lambda^2(W)$ to $\VV_{p-3}$. Using Lemma~\ref{tensor2}, we have that $\wt \theta$ determines a surjective homomorphisms $\theta^*$ from  $\Lambda^2(\VV_{p-3})$ to $\VV_{p-3}$.  By Lemma~\ref{tensor1}, $\Lambda^2(\VV_{p-3})$ either has no quotient isomorphic to $\VV_{p-3}$, which is against our assumption that $\theta$ is surjective, or $\ker \theta^*$ is the unique maximal submodule of  $\Lambda^2(\VV_{p-3})$ which has quotient $\VV_{p-3}$. Now the $(p-3)/2$-transvectant
$\Theta_{(p-3)/2}$ restricted to $\Lambda^2(\VV_{p-3})$ also has image in $\VV_{p-3}$. Hence $\ker\theta^*= \ker \Theta_{(p-3)/2}$.

We may take $u= x^{p-3}$ and $w= y^{p-3}$ in $\VV_{p-3}$. We calculate  $(u\otimes w) \Theta_{(p-3)/2}= \left( {(p-3)}\atop {(p-3)/2}\right)^2(xy)^{(p-3)/2}\ne 0$.
Therefore  $(u,w) \not \in \ker \Theta_{(p-3)/2}=\ker \theta^*$.  This shows that $(u,w)\theta \ne 0$.
\end{proof}

In the next lemma we are interested in modules for $L=\Sym(p)$ defined in characteristic $p$.  The notation $S^\lambda $ denotes the \emph{Specht module} for $L$ corresponding to the partition $\lambda$ of $p$. The module $D^\lambda$ is the unique irreducible quotient of $S^\lambda$ (see \cite{James}). Thus $S^{p-1,1}$ is a characteristic $p$ representation of $L$ of dimension $p-1$ and can be identified with the submodule of the natural $\GF(p)L$-permutation module $\langle v_i\mid 1 \le i \le p\rangle$ which is the kernel of the augmentation map $\sum_{i=1}^p\lambda_iv_i \mapsto \sum_{i=1}^p\lambda_i$. In this case, $D^{p-1,1}= S^{p-1,1}/\langle\sum_{i=1}^pv_i\rangle$ has dimension $p-2$. The result we shall need is as follows.

\begin{lemma}\label{Symmetric(p)Module} Suppose that \blue{$p\ge 5$} is a prime, $L \cong \Sym(p)$ and $V=S^{p-1,1}$ considered as a $\GF(p)L$-module. Then $\Lambda^2(V)= S^{p-2,1^2}$ has irreducible composition factors $D^{p-1,1}$ and \blue{$D^{p-2,1^2}$} both with multiplicity $1$ and, furthermore, $\Lambda^2(V)$ has no quotient of dimension $1$ or isomorphic to $D^{p-1,1}$ .
\end{lemma}

\begin{proof} For prime $p\ge 11$, the isomorphism type and composition factors of $\Lambda^2 V$ are explicitly given in \cite[Section 2]{MagaardMalle}.  The fact that $D^{p-1,1}$ is not a quotient of $\Lambda^2(V)$ follows from \cite[Corollary 12.2]{James}. For $p=5$ and $p=7$, we have checked the assertion by computer \blue{(see Subsection~\ref{CSym(p)}).}
\end{proof}

\begin{lemma}\label{H1dim}
Assume that \blue{$p\ge 5$ is a prime,} $X = \Sym(p)$, $Y= X'=\Alt(p)$ and $V= D^{p-1,1}$ is the $p-2$-dimensional module.  Let $W=V|_{Y}$. Then $\dim\mathrm H^1(Y,W)=1$.
\end{lemma}

\begin{proof} Let $\tau= (1,2,3)$ and $\sigma= (3, \dots, p)$.  Then $H= \langle \tau,\sigma\rangle$ acts transitively on $\Omega =\{ 1, \dots, p\}$ and, as $p$ is a prime, it is primitive.  Since $H$ contains a $3$-cycle, $Y=H$  by Jordan's Theorem \cite[II.4.5]{Huppert}.  If $U$ is a $\GF(p)Y$-module with $ W$ of codimension $2$ and $[U,Y]= W$, then $\dim C_U(\tau)= 2+\dim C_W(\tau)= 2+p-4= p-2$ and $\dim C_U(\sigma)= 2+ \dim C_W(\sigma) =2+ 1= 3$. Hence $\dim C_U(H)\ge 1$ and this proves $\dim \mathrm H^1(Y,W)\le 1$. Since the natural $p$-point permutation module has a quotient $T$ with $[T,Y]= W$ and $C_T(Y)=0$, we have  $\dim \mathrm H^1(Y,W)= 1$.
\end{proof}

%% file: FusionSystems.tex
\section{A primer on fusion systems}\label{sec:PF}

We   assume some basic familiarity with fusion systems and recommend the references \cite{AKO, Craven} as introductory texts. We follow the notation from these sources. We start by flying over the standard definitions and at the same time introduce some of the standard terminology from  \cite{AKO, Craven}.

\begin{definition}
For a finite group $G$ and     subgroups $H, K \leq G$,   define
\[ \Hom_G (H,K) = \{ \varphi \in \Hom(H, K) \mid \varphi = c_g \text{ for some } g \in G \text{ such that } H^g \leq K \} \]
and   set $\Aut_G(H) = \Hom_G(H,H)\cong N_G(H)/C_G(H)$.
\end{definition}

More generally,  if $H,K \leq G$,    we define  $\Aut_K(H) = \{c_k \mid k \in K \cap N_G(H)\}$  to be the group of automorphisms of $H$ induced by conjugation by elements of $K$ which normalize $H$. Visibly  $\Aut_K(H) \le \Aut_G(H)\le \Aut(H)$. Similarly, if $K \le L$ and $A \le \Aut(L)$, we write $\Aut_A(K)$ to represent the group of automorphisms of $K$ generated by the restriction of automorphisms in $N_A(K)= \{\gamma\in A\mid  K\gamma=K\}$. In this case $\Out_A(K)= \Aut_A(K)\Inn(K)/\Inn(K)$. \blue{For groups $P$ and $Q$, $\Inj(P,Q)$ is the set of injective group homomorphisms from $P$ to $Q$.}

\begin{definition}\label{def:fussy}
A \emph{fusion system} on a $p$-group $S$ is a category $\F$, with objects   the set of all subgroups of $S$, and \blue{morphisms $\Mor_\F(P,Q)$ between objects $P$ and $Q$   which  satisfy the following two properties:}
\begin{enumerate}
\item $\Hom_S(P,Q)\subseteq \Mor_\F(P,Q) \subseteq \mathrm{Inj} (P,Q)$; and
\item each $\varphi \in \Mor_\F(P,Q)$ is the composite of an $\F$-isomorphism followed by an inclusion.
\end{enumerate}
\end{definition}

If $\F$ is a fusion system and $P,Q \le S$, then we write $\Hom_\F(P,Q)= \Mor_\F(P,Q)$ and $\Aut_\F(P)= \Mor_\F(P,P)$.

\begin{definition}\label{def:fussydefs}
Suppose that $\F$ is a fusion system on a finite $p$-group $S$ and $P  \le S$. Then
\begin{enumerate}
\item the $\F$-\emph{conjugacy class} of $P$, is $P^\F=\{P\alpha\mid \alpha \in \Hom_\F(P,S)\}$;
    \item $P$ is \emph{strongly $\F$-closed} if and only if $Q^\F \subseteq P$ for all $Q\le P$;
 \item for $R \in P^\F$, $\alpha \in \Hom_\F(R,P)$,  $\alpha^*$ is the isomorphism between $\Aut_\F(R)$ and $\Aut_\F(P)$  defined by $\gamma\mapsto \alpha^{-1}\gamma \alpha$;
\item  $P$ is \emph{fully $\F$-normalized} if and only if  $|N_S(P)| \ge |N_S(R)|$ for all $R \in P^\F$;
\item $P$ is \emph{fully $\F$-centralized} if and only if  $|C_S(P)| \ge |C_S(R)|$  for all $R \in P^\F$;
\item  $P$ is  \emph{$S$-centric} if and only if  $C_S(P)=Z(P)$, and $P$ is \emph{$\F$-centric} if and only if $R$ is  $S$-centric for all $R \in P^\F$;
\item  if $R \in P^\F$ and $\alpha \in \Hom_\F(R,P)$, $$N_\alpha=\{g \in N_S(R) \mid \alpha^{-1}c_g \alpha \in \Aut_S(P)\}$$ is the $\alpha$-\emph{extension control subgroup} of $S$;
\item $P$ is $\F$-\emph{receptive} if and only if  for all $R \in P^\F$ and  $\alpha \in \Hom_\F(R,P)$, there exists $\widetilde{\alpha} \in \Hom_\F(N_\alpha,S)$ such that $\widetilde{\alpha}|_R =\alpha;$
    \item   $P$ is \emph{fully $\F$-automized} if and only if  $\Aut_S(P) \in \Syl_p(\Aut_\F(P))$;
\item  $P$ is $\F$-\emph{saturated} provided there exists $R \in P^\F$ such that $R$ is simultaneously
\begin{enumerate}
\item  fully $\F$-automized; and
\item  $\F$-receptive.
\end{enumerate}
\item $\F$ is \emph{saturated} if every subgroup of $S$ is $\F$-saturated.
\end{enumerate}
\end{definition}

\blue{If $\F$ is a fusion system and $X$ is a set of morphisms in $\F$, $\langle X\rangle$ is the intersection of all the fusion systems  on $S$ which contain $X$. We say that $\langle X\rangle$ is the \emph{fusion system generated} by $X$. Obviously $\langle X\rangle$ is contained in $\F$. }

In our arguments, an important role is played by \emph{normalizer fusion systems.} Suppose that $\F$ is a saturated fusion system on $S$, \blue{ $T \le S$ and $K \le \Aut(T)$.} Then $N_\F^K(T)$ is the fusion system on $N_S(T)$ with, for $P, Q \le N_S(T)$, $\Hom_{N_\F(T)}(P,Q)$ consisting of morphisms $\alpha \in \Hom_\F(P,Q)$ such that there is $\wt \alpha \in \Hom_\F(PT,QT)$ with  \blue{$T\wt \alpha = T$, $\wt \alpha|_T\in K$ and } $\alpha = \wt \alpha|_P$. Importantly, if $T$ is fully $\F$-normalized, then $N_\F^K(T)$ is saturated \cite[Theorem I.5.5]{AKO}. \blue{The two extreme cases $K= \Aut(T)$, and $K=1$ are of main interest. In the former case we have the \emph{$\F$-normalizer} of $T$ and we write $N_\F(T)=N_\F^{\Aut(T)}(T)$ whereas in the latter we have the \emph{$\F$-centralizer} of $T$ and  we define  $C_\F(T)=N_\F^1(T)$.}

\blue{A subgroup $T\le S$ is \emph{normal} in $\F$ if and only if $\F= N_\F(T)$.}
The subgroup $O_p(\F)$ is the product of all subgroups $T  \le S$ such that \blue{$T$ is normal in $\F$.} It follows that $\F= N_\F(O_p(\F))$.
A saturated fusion system $\F$ is \emph{constrained} provided $O_p(\F)$ is $\F$-centric.

\begin{theorem}[The Model Theorem]\label{model} Let $\F$ be a constrained fusion system on $S$. Then there exists a finite group $G$ with $O_{p'}(G)=1$, $S \in\Syl_p(G)$, $\F= \F_S(G)$ and $C_G(O_p(G)) \le O_p(G)$.
\end{theorem}
\begin{proof} See \cite[Theorem III.5.10]{AKO}.
\end{proof}

We now gather some elementary consequences of the definitions above.
We have the following well-known fact and, as we shall use it several times, we provide the proof.

\begin{lemma}  Suppose that $\F$ is a   fusion system on $S$ and $P \le S $.
\begin{enumerate}
\item If $\alpha \in N_{\Aut_\F(P)}(\Aut_S(P))$, then $N_\alpha = \N_S(P)$.
\item If $P $ is $\F$-receptive, then $$ N_{\Aut_\F(P)}(\Aut_S(P))= \{\alpha|_{P}\mid \alpha \in N_{\Aut_\F(N_S(P))}(P)\}.$$
\end{enumerate}
\end{lemma}

\begin{proof} Part (i) is just the definition of $N_\alpha$.

Suppose that  $\alpha \in N_{\Aut_\F(P)}(\Aut_S(P))$. Then $N_\alpha=N_S(P)$ by (i). Since $P$ is $\F$-receptive, $\alpha=\wt \alpha|_P$  where $\wt \alpha \in \Hom_\F(N_S(P),S)$. As $P\wt \alpha = P\alpha =P$,  we know $\N_S(P)\wt \alpha = N_S(P\wt \alpha)= N_S(P)$. Hence $\wt \alpha \in N_{\Aut_\F(N_S(P))}(P)$ and so  $$N_{\Aut_\F(P)}(\Aut_S(P))\subseteq \{\alpha|_{P}\mid \alpha \in N_{\Aut_\F(N_S(P))}(P)\}.$$
Conversely, assume $g \in N_S(P)$. Then  $c_g \in \Aut_S(P)$  and, for $\beta \in N_{\Aut_\F(N_S(P))}(P)$, we have  $(\beta|_P)^{-1}c_g\beta|_P = c_{g\beta} \in \Aut_S(P)$. Hence $\beta|_P \in N_{\Aut_\F(P)}(\Aut_S(P))$.  This proves (ii).
\end{proof}

\begin{lemma}\label{NFTNFR} Suppose that $\F$ is a   fusion system on $S$, $T \le  S$ and $\mathcal K=N_{\mathcal K}(T)$ is a subfusion system of $\F$ on $N_S(T)$.  Assume that $R \le T$ is $\Aut_{\mathcal K}(T)$-invariant. Then $\mathcal K \subseteq N_\F(R)$. In particular, if $R\le T$ is $\Aut_\F(T)$-invariant, then $N_\F(T) \subseteq N_\F(R)$.
\end{lemma}

\begin{proof} Suppose that $X, Y \le N_S(T)$ and $\theta \in \Hom_{\mathcal K}(X,Y)$. Then, as $R$ is $\Aut_{\mathcal K}(T)$-invariant, it is also $\Aut_S(T) $-invariant and thus  $R$ is normalized by $N_S(T)$.  In particular, $X, Y \le N_S(T)\le  N_S(R)$ and so $X$ and $Y$ are objects in $\N_\F(R)$.  Since $\theta \in \Hom_{\mathcal K}(X,Y)$ and $\mathcal K= N_{\mathcal K}(T)$, the morphism $\theta$ extends  to $\hat \theta\in \Hom_{\mathcal K}(XT,YT)$  so that $T\hat \theta= T$. Thus  $\hat \theta|_T \in \Aut_{\mathcal K}(T)$ and $R\hat \theta= R$ because $R$ is $\Aut_{\mathcal K}(T)$-invariant.  As $\hat \theta$ extends $\theta$, $\hat \theta|_{XR} \in \Hom_\F(XR,YR)$ also extends $\theta$ and this means that $\theta \in \Hom_{N_\F(R)}(X,Y)$. Hence $\Hom_{\mathcal K}(X,Y) \subseteq \Hom_{N_\F(R)}(X,Y)$ and this proves the main statement of the lemma.

Taking $\mathcal K= N_\F(T)$ and noting $\Aut_{N_\F(T)}(T)= \Aut_\F(T)$, yields the remaining statement.
\end{proof}

If $Q$ is normal in $\F$, then the \emph{factor system} $\F/Q$ has objects $\{T/Q\mid Q \le S\}$ and, for $Q \le T,R \le S$,  morphisms  $\Hom_{\F/Q}(T/Q,R/Q)= \{\bar \phi\mid \phi \in \Hom_\F(T,R)\}$ where $(tQ)\bar \phi=t\phi Q$ for  $\phi \in \Hom_\F(T,R)$.

\begin{lemma}\label{F/Q} If $\F$ is saturated and $Q$ is normal in $\F$,  then $\F/Q$ is saturated
\end{lemma}

\begin{proof} This is \cite[Lemma II.5.5]{AKO}.
\end{proof}

\begin{definition}\label{def:ess} Suppose that $\F$ is a fusion system. A subgroup $P$  of $S$ is \emph{$\F$-essential} if $P\ne S$, $P$ is $\F$-centric,  fully $\F$-normalized and $\Out_\F(P)$ contains a strongly $p$-embedded subgroup. We write $\E_\F$  to denote the set of $\F$-essential subgroups of $\F$.
\end{definition}

\blue{Note that if $E \in \E_\F$ then $O_p(\Out_\F(E)) = 1$, that is, $O_p(\Aut_\F(E)) = \Inn(E)$. This fact will be used several times.}

The main tool for classifying \blue{saturated} fusion systems is provided by the following lemma.

\begin{theorem}[Alperin-Goldschmidt]\label{t:alp}
If $\F$ is a saturated fusion system on the $p$-group $S$, then $$\F= \langle \Aut_\F(S), \Aut_\F(E) \mid E \in \E_\F \rangle.$$
\end{theorem}

\begin{proof}
See \cite[Theorem I.3.5]{AKO}.
\end{proof}

\begin{lemma}\label{OpinE} Suppose that $\F$ is a saturated fusion system on $S$ and $E$ is an $\F$-essential subgroup. Then $O_p(\F)\le E$  and $O_p(\F)$ is  $\Aut_\F(E)$-invariant.
\end{lemma}

\begin{proof} Since $N_{O_p(\F)}(E)$ is $\Aut_\F(E)$-invariant and $O_p(\Aut_\F(E))=\Inn(E)$, we have $N_{O_p(\F)}(E) \le E$ by Lemma~\ref{lem:K in E}. Hence $O_p(\F) \le E$ and so is normal in $E$ and $\Aut_\F(E)$-invariant.
\end{proof}

We will also meet the fusion subsystems $O^p(\F)$ and $O^{p'}(\F)$. We  define the \emph{focal} and  \emph{hyperfocal} subgroups as follows
 $$\foc(\F)=\langle [g,\alpha]\mid g \in Q \le S \text { and } \alpha\in \Aut_\F(Q) \rangle$$

 $$\hyp (\F) = \langle [g,\alpha] \mid g \in Q \le S \text { and } \alpha \in O^{p}( \Aut_\F(Q))\rangle$$ where $[g,\alpha] = g^{-1} (g)\alpha$. The subfusion system $O^{p}(\F)$ is a saturated fusion system on $\hyp(\F)$ defined as follows:
$$O^p(\F)= \langle \Inn(\hyp(\F)), O^{p}(\Aut_\F(Q))\mid Q \le \hyp(\F)\rangle.$$

\begin{lemma}\label{l:focf}
Let $\F$ be a saturated fusion system on $S$. The following hold:
\begin{enumerate}
\item $\foc(\F)=\langle [g,\alpha]\mid g \in Q \le S, Q \text{ is } \F\text{-essential or } Q=S, \alpha \in \Aut_\F(Q) \rangle$.
\item $O^p(\F)=\F$ if and only if $\foc(\F)= S$ if and only if $\hyp(\F)=S$.
\end{enumerate}
\end{lemma}

\begin{proof}
 This follows from Theorem \ref{t:alp}   and \cite[Corollary I.7.5]{AKO}.
\end{proof}

The subfusion system $O^{p'}(\F) $ is more complicated to define, so we just settle for saying that it is the unique saturated subfusion system on $S$ minimal subject to containing $O^{p'}(\Aut_\F(Q))$ for all $Q \le S$ (see \cite[Definition I.7.3]{AKO}).

\blue{\begin{definition}\label{def:reduced}
A saturated fusion system $\F$ is called \emph{reduced} if and only if $O_p(\F)=1$ and $\F=O^p(\F)=O^{p'}(\F)$.\end{definition}}
Normal subfusion systems are defined in \cite[Definition I.6.1]{AKO} and a \emph{simple} saturated fusion system is a fusion system which has no proper normal subfusion system.  Simple fusion systems are reduced as $O^p(\F)$, $O^{p'}(\F)$ and $\F_{O_p(\F)}(O_p(\F))$ are normal subsystems of $\F$.

\begin{lemma}\label{oli 1.4} Suppose that $\F$ is a saturated fusion system on a $p$-group $S$. Assume that each $P \in\E_\F$ is minimal among all $\F$-centric subgroups.
  For each $P \in \E_\F$ define
$$\Aut^{(P)}_\F(S) =\langle\alpha \in \Aut_\F(S)
\mid P\alpha = P, \alpha|_P\in O^{p'}(\Aut_\F(P))\rangle.$$
Then $O^{p'}(\F)= \F$ if and only if $\Aut_\F(S)= \langle \Inn(S), \Aut_\F^{(P)}(S) \mid P \in \E_\F\rangle$.
\end{lemma}

\begin{proof} This is \cite[Lemma 1.4]{p.index}.
\end{proof}

In the next lemma, the containment $\Aut_\F(E) \subseteq \mathcal G$ means that $E$ is an object in $\mathcal G$ and that $\Aut_\F(E)=\Aut_{\mathcal G}(E)$. We will use this notation from here on.

\begin{lemma}\label{F-essential G-essential} Suppose that $\F$ and $\mathcal G$ are saturated fusion systems with ${\mathcal G} \subseteq \F$. Assume that $E$ is $\F$-essential and $\Aut_\F(E) \subseteq \mathcal G$. Then $E$ is ${\mathcal G}$-essential.
\end{lemma}

\begin{proof} We have $\Aut_\F(E)= \Aut_{\mathcal G}(E)$ and $E^{\mathcal G} \subseteq E^\F$. Therefore, as $E$ is fully $\F$-normalized and $\F$-centric it is also fully ${\mathcal G}$-normalized and ${\mathcal G}$-centric. Since $\Out_\F(E)=\Out_{\mathcal G}(E)$, we know $\Out_{\mathcal G}(E)$ has a strongly $p$-embedded subgroup. Thus $E$ is $\mathcal G$-essential.
\end{proof}

Recall that, following \cite[Proposition I.3.3]{AKO}, if $P < S$ is fully $\F$-normalized,  $H_\F(P)$ is defined to be the subgroup of $\Aut_\F(P)$ which is generated by those morphisms of $P$ which extend to $\F$-isomorphisms between strictly larger subgroups of $S$.  Then the statement in \cite[Proposition I.3.3]{AKO}, includes the fact that, if $P$ is $\F$-essential, then $H_\F(P)/\Inn(P)$ is  strongly $p$-embedded in $\Aut_\F(P)$.

The next result  is widely used  to show that certain subgroups are not $\F$-essential.

\begin{lemma}\label{lem:series}
Suppose that $\F$ is a saturated fusion system on $S$ and $Q \le S$.  Assume that $Q_s < Q_{s-1}< \dots <Q_0=Q$ are    $\Aut_\F(Q)$-invariant with $Q_s \le \Phi(Q)$.  If $  A \le \Aut_S(Q)$ and $[Q_i,A] \le Q_{i+1}$ for $0 \le i \le s-1$, then $A \le O_p(\Aut_\F(Q))$. In particular, if $A \not \le \Inn(Q)$, then $Q$ is not $\F$-essential.
\end{lemma}

\begin{proof} Set $B= \langle A^{\Aut_\F(Q)}\rangle$. Assume that $\beta \in B$ has order coprime to $p$.  Then $[Q,\beta;s] \le Q_s$ and so \cite[Theorem 5.3.6]{Gor} implies that $[Q,\beta]\le Q_s\le \Phi(Q)$.  Thus   \cite[Theorem 5.1.4]{Gor} implies $\beta=1$ and so $B$ is a $p$-group.  Therefore $A \le B \le  O_p(\Aut_\F(Q))$ and this proves the first claim. Since $O_p(\Aut_\F(E))=\Inn(E)$ for $E \in \E_\F$, if $A \not \le \Inn(Q)$, then  $Q$ is not $\F$-essential.
\end{proof}

If $Q$ has a series of subgroups as in Lemma~\ref{lem:series} and \blue{$P \leq S$}  centralizes \blue{all quotients} $Q_i/Q_{i+1}$, then we say that \blue{$P$} \emph{stabilizes the series} $Q_s < Q_{s-1}< \dots <Q_0=Q$ \blue{and we conclude that, if $P\not\le Q$, then $Q$ is not $\F$-essential.}

\blue{The next lemma relies on Proposition~\ref{SE-p2} which provides a list of all non-abelian simple groups with a strongly $p$-embedded subgroup containing an elementary abelian subgroup of order $p^2$.}

\begin{lemma}\label{|E| bound} Suppose that $\F$ is a saturated fusion system on $S$ and $E$ is $\F$-essential. Then $|E/\Phi(E)| \ge |N_S(E)/E|^2$.
\end{lemma}

\begin{proof} This is \cite[Proposition  4.6 (4)]{parkersemerarocomputing}. \end{proof}

 We now employ  work of Henke \cite{Ellen}.

\begin{prop}\label{Ellen} Suppose that $\F$ is a saturated fusion system on the $p$-group $S$ and $E$ is an $\F$-essential subgroup.  Assume that $U$ and $W$ are $\Aut_\F(E)$-invariant subgroups of $E$ with $V=U/W$  centralized by $\Inn(E)$ and   non-trivial as a $\GF(p)O^{p'}(\Aut_\F(E))$-module.  If $|[V,\Out_S(E)]|\le p$ or $|V:C_V(\Out_S(E))|\le p$, then $$O^{p'}(\Aut_\F(E))/C_{O^{p'}(\Aut_\F(E))}(V) \cong \SL_2(p),$$ and $V/C_V(O^{p'}(\Aut_\F(E)))$ is the natural $\GF(p)\SL_2(p)$-module.
\end{prop}

\begin{proof} The hypothesis that $V$ is a non-trivial $\GF(p)O^{p'}(\Aut_\F(E))$-module implies that $\Aut_S(E)$ does not centralize $V$. Hence $$|[V,\Out_S(E)]|= p \le |\Out_S(E)/C_{\Out_S(E)}(V)|$$ or  $$|V:C_V(\Out_S(E))|= p  \le |\Out_S(E)/C_{\Out_S(E)}(V)|.$$  This, by definition, means that either the dual of $V$ or $V$ is a  failure of factorisation  module for the group $O^{p'}(\Aut_\F(E))/C_{O^{p'}(\Aut_\F(E))}(V)$. Using  \cite[Theorem 5.6]{Ellen}, we obtain  $O^{p'}(\Aut_\F(E))/C_{O^{p'}(\Aut_\F(E))}(V)\cong \SL_2(p^a)$ and $V/C_V(O^{p'}(\Aut_\F(E)))$ is the natural $\GF(p)\SL_2(p^a)$-module. Since we know one of $|[V,\Out_S(E)]|= p$ or $|V:C_V(\Out_S(E))|= p$, we deduce that $p=p^a$ and this completes the proof.
\end{proof}

We continue this section by recalling important properties of $\F$-pearls (see Definition~\ref{pearl:def}).

\begin{lemma}\label{pearls1}
Suppose that $p$ is an odd prime, $\F$ is a saturated fusion system on a $p$-group $S$ and $P \in \mathcal P(\F)$. Then
\begin{enumerate}
\item $S$ has maximal class;
\item $\Out_\F(P)$ is isomorphic to a subgroup of $\GL_2(p)$ and  $O^{p'}(\Out_\F(P)) \cong \SL_2(p)$;
\item $P/\Phi(P)$ is a natural $\GF(p)\SL_2(p)$-module for $O^{p'}(\Out_\F(P))$;
\item  $[\N_S(P) \colon P] =p$; and
\item every subgroup of $S$ that is $\F$-conjugate to $P$ is an $\F$-pearl.
\end{enumerate}
\end{lemma}

\begin{proof} This is  a combination of \cite[Lemma 1.5, Corollary 1.11 and Lemma 1.13]{pearls}.
\end{proof}

\begin{lemma}\label{pearls2}  Suppose that $p$ is an odd prime and $\F$ is a saturated fusion system on a $p$-group $S $ of maximal class with $|S| \ge p^4$. If
$ E\le S$ is an $\F$-essential subgroup of $S$, then the following are equivalent:
\begin{enumerate}\item $E$ is an $\F$-pearl;
\item $E$ is not contained in  $\gamma_1(S)$ or in $C_S(Z_2(S))$; and
\item there exists an element $x \in S\setminus C_S(Z_2(S)) $ of order $p$ such that
either $ E = \langle x
\rangle Z(S) $ or $E = \langle x\rangle Z_2 (S)$.

\end{enumerate}
\end{lemma}

\begin{proof}  This is  a restatement of \cite[Lemma 2.4]{pearls}.  \end{proof}

  The next lemma helps when we     apply Lemma~\ref{oli 1.4}.

\begin{lemma}\label{pearls2.5}  Suppose that $p$ is an odd prime and $\F$ is a saturated fusion system on a $p$-group $S $ with $|S| \ge p^4$. If $P \in \mathcal  P(\F)$, then
 no proper subgroup of $P$ is $\F$-centric.
  \end{lemma}

  \begin{proof} There is nothing to do if $P\in \mathcal P_a(\F)$ is abelian. Assume $P\in \mathcal P_e(\F)$ is extraspecial. Since no subgroup of order $p$ is $\F$-centric we can assume $Q\le P$ has order $p^2$. Then parts (ii) and (iii) of Lemma~\ref{pearls1} imply that $Q$ is $\Aut_\F(P)$ conjugate to $Z_2(S)\le P$. Since $|S| \ge p^4$, and $|S:C_S(Z_2(S))|=p$, we conclude $Z_2(S)$ is not $\F$-centric. Hence  $Q$ is not $\F$-centric.
  \end{proof}

\blue{For a subgroup $A$ of a group $B$, we define $N^0(A)= A$ and then, for $i > 0$, $N^i(A)= N_B(N^{i-1}(A))$. The ordered collection of these subgroups is called the \emph{normalizer tower} of $A$  in $B$ and its length is the minimal $\ell$ such that $N^{\ell+1}(A)= N^\ell(A)$.}

\begin{theorem}\label{pearls3} Suppose that $p$ is an odd prime, $\F$ is a saturated fusion system on a $p$-group $S$ and $P \in \mathcal P(\F)$ is an $\F$-pearl with $|S:P|=p^m$.
Then
\begin{enumerate}
\item the members of the normalizer tower of $P$ \[\N^0(P) < \N^1(P) <\N^2(P)< \dots <\N^{m-1}(P) < \N^m(P)=S\]
are the only subgroups of $S$  which contain $P$; also $|\N^i(P) \colon \N^{i-1}(P)| = p$ for every $1 \leq i \leq m$;
\item $P$ is not properly contained in any $\F$-essential subgroup of $S$; and
\item every morphism in $\N_{\Aut_\F(P)}(\Aut_S(P))$ is the restriction of an automorphism of $S$ that normalizes each $\N^i(P)$.
\end{enumerate}
\end{theorem}
\begin{proof} This is an application of \cite[Theorem 3.6]{pearls}.
 \end{proof}

\begin{lemma}\label{pearls4} Suppose that $p$ is an odd prime and $\F$ is a saturated fusion system on a $p$-group $S$.
Let $P \in \mathcal{P}(\F)$ and let $m$ be such that $|S :P| =p^m$. For every $1 \leq i \leq m-1$ let $\F_i$ be the smallest fusion subsystem of $\F$ defined on $\N^i(P)$ such that  $\Aut_{\F_i}(P)=\Aut_\F(P)$ and $\Aut_{\F_i}(\N^i(P)) = \Inn(\N^i(P))\N_{\Aut_\F(\N^i(P))}(P)$. Then $\F_i$ is a saturated fusion subsystem of $\F$ and $P\in \mathcal{P}(\F_i)$.
\end{lemma}

\begin{proof} This is \cite[Lemma 3.7]{pearls}. \end{proof}

\begin{lemma}\label{pearlprune} Assume that $\F$ is a saturated fusion system on a $p$-group $S$, $\mathcal C$ is a set of $\F$-class representatives of $\F$-essential subgroups, and $P \in\mathcal C$. If $P$ is an $\F$-pearl, then $\mathcal G=\langle \Aut_\F(S), \Aut_\F(E)\mid E \in \mathcal C\setminus\{P\}\rangle$ is saturated.
\end{lemma}

\begin{proof} This is \cite[Lemma 6.5]{parkersemerarocomputing}.
\end{proof}

\blue{ The next lemma is part of  \cite[Theorem 3.15]{pearls}.
\begin{lemma}\label{type.pearl}  Suppose that $\F$ is a saturated fusion system on a $p$-group $S$ with $|S| \ge p^4$ and $P \in \mathcal P(\F)$.
Assume that the group $\gamma_1(S)$ is not abelian or extraspecial. Then
\begin{enumerate}
 \item  $\Out_\F(S)$ is a Hall $p'$-subgroup of $\Out(S)$, is cyclic of order  $p-1$ and acts faithfully on $S/\gamma_1(S)$.
 \item $\Aut_\F(S)= N_{\Aut_\F(S)}(P) \Inn(S)$.
 \item $\Out_\F(P) \cong \SL_2(p)$.
\end{enumerate}
\end{lemma}
\begin{proof}
Because $P \in \mathcal P(\F)$, $S$ has maximal class. Since $\gamma_1(S)$ is not abelian or extraspecial, Corollary~\ref{c12} yields the Hall $p'$-subgroups of $\Out(S)$ are cyclic of order at most $p-1$ and act faithfully on $S/\gamma_1(S)$. By Theorem~\ref{pearls3}, every automorphism in $N_{\Aut_\F(P)}(\Aut_S(P))$ is the restriction of an automorphism of $\Aut_\F(S)$. Hence $|N_{\Aut_\F(P)}(\Aut_S(P))/\Aut_S(P)| \le p-1$. Since $P$ is an $\F$-pearl, Lemma~\ref{pearls1} (ii) implies that $\Out_\F(P) \cong \SL_2(p)$ and we deduce that $\Out_\F(S)$ is cyclic of order $p-1$.
\end{proof}}

\blue{

\begin{theorem}\label{pearls5}  Suppose that \blue{$p$ is an odd prime} and $\F$ is a saturated fusion system on a $p$-group $S$ of order $p^n$ with $n \ge 4$. Assume that $P\in \mathcal P_a(\F)$.
\begin{enumerate}
\item If $T$ is strongly $\F$-closed in $S$, then either $T=S$ or $T=\gamma_2(S)P$ has index $p$ in $S$ and $\mathcal P_a(\F)= P^S$.
     \item If $\E_\F= \mathcal P_a(\F)$ and $\gamma_1(S)$ is not abelian or extraspecial, then either \begin{enumerate}\item $\F$ is simple; or \item  $O^p(\F)$ is simple, $\foc (\F)= \gamma_2(S)P$, $S$ is not exceptional and $ n = j(p-1) + 1$ for some $j\geq 2$.\end{enumerate}
\item If $p\ge 5$, $\gamma_1(S)$ is not abelian or extraspecial and $\F$ is simple, then $\F$ is exotic.
\end{enumerate}
\end{theorem}

\begin{proof}
Let $T\leq S$ be strongly $\F$-closed in $S$. Since $T$ is normal in $S$ and $|Z(S)|=p$, we have $Z(S) \le T$.  Recall that $P \in \mathcal P_a(\F)$.  Hence $Z(S) \le P$, $P \not \le C_S(Z_2(S))$ and $P \not \le \gamma_1(S)$ by Lemma~\ref{pearls2}. As $P$ is an abelian $\F$-pearl, there exists $\beta \in \Aut_\F(P)$ such that $P=Z(S)Z(S)\beta$. Since $T$ is strongly $\F$-closed, this implies $P \le T$. Using $T$ is normal in $S$ yields, $T\in \{S, \gamma_2(S)P\}$. Assume that $P_2 \in \mathcal P(\F) \setminus P^S$.
 Since $S$ has maximal class, $S$ acts transitively by conjugation on the subgroups of order $p^2$ containing $Z(S)$ which are not contained in $\gamma_1(S)$ but which are contained in $\gamma_2(S)P$. Hence $P_2 \not \le \gamma_2(S) P$ whereas we know $P_2 \le T$.  Hence $T= S$ in this case. This proves (i).

   Suppose that $\E_\F=\mathcal P_a(\F)$ and that $\gamma_1(S)$ is not abelian or extraspecial. This implies $|S| \ge p^5$.
Assume that $\mathcal G$ is a weakly normal subsystem of  $\F$ defined on a non-trivial subgroup $T$ of $S$. Then $T$ is strongly $\F$-closed in $S$ and so $T \in \{\gamma_2(S)P, S\}$ by (i). Since $\gamma_1(S)$ is not abelian or extraspecial,
   Lemma~\ref{type.pearl} implies that $\Aut_\F(P)\cong \SL_2(p)$ and $\Out_\F(S)$ is cyclic of order $p-1$. By \cite[Lemma 5.33]{Craven}, $\Aut_\mathcal G(P)$ is a normal subgroup of $\Aut_\F(P)\cong \SL_2(p)$. Since $|\Aut_S(P)|=p$, we deduce that $\Aut_{\mathcal G}(P) = \Aut_\F(P)$ and this is true for all pearls $P \in \mathcal P(\F)$.
   As $\E_\F=\mathcal P_a(\F)$,  for $R \le S$ with $|R| >p^2$, we have $\Aut_\F(R)$ is generated by restrictions of automorphisms of $S$ by Theorem~\ref{t:alp}. Hence no such subgroup can be $\mathcal G$-essential.
    In particular, $\E_\F= \E_\mathcal G$.
   By Theorem~\ref{pearls3}(iii), $\Aut_{\mathcal G}(T)$ has order divisible by $p-1$. On the other hand, $T$ is not an $\F$-essential subgroup by  Theorem~\ref{pearls3}(ii) and thus every element of $\Aut_{\mathcal G}(T)$ is the restriction of some element of $\Aut_\F(T)$.  We deduce that $\Out_{\mathcal G}(T)$ is cyclic of order $p-1$.

   If $T=S$, we have $\Aut_\F(S)=\Aut_\mathcal G(S)$ and Theorem~\ref{t:alp} implies that $\mathcal F= \mathcal G$. Therefore $\mathcal F$ is simple in this case and this is listed as (ii)(a).

   Suppose that $T=\gamma_2(S)P$.
    Then, by (i), $\mathcal P_a(\F)= P^S$. In particular, $\mathcal P_a(\mathcal G)= \mathcal P_a(\F)\ne P^{T}$.
    Since $\Aut_\mathcal G(T)$ is normal in $\Aut_\F(T)$ and $\Out_\mathcal G(T)$  has order $p-1$, we conclude that $\Aut_\mathcal G(T)\ge O^p(\Aut_\F(T))$.  Using (i), we now know $\foc(\F)=\hyp(\F)=\gamma_2(S)P$ and so $\mathcal G= O^p(\F)$. Since   $ \mathcal P_a(\mathcal G)\ne P^{T}$, $T$ is the only strongly $\mathcal G$-closed subgroup of $\mathcal G$ by (i). Since $\Out_{\mathcal G}(T)$ has order $p-1$, we now have $\mathcal G$ is simple as before.

   Continue to assume that $T=\gamma_2(S)P$.  It remains to prove that $S$ is not exceptional and that $n \equiv 1 \pmod{p-1}$.    Let $\phi \in \Aut_{\mathcal G}(T)$ have order $p-1$ be such that $P\phi= P$ and $\phi|_P$ acts as the diagonal  matrix $\begin{pmatrix}\lambda&0\\0&\lambda^{-1} \end{pmatrix}$. Then, choosing $x \in P \setminus \gamma_1(S)$ and using the notation introduced in Lemma~\ref{action}, we have
   $x\phi \equiv x^a \pmod {\gamma_2(S)}$, and $s_{n-1}\phi = s_{n-1}{a^{-1}}$ where $a\in \GF(p)$ has order $p-1$.
   Since $\phi \in O^p(\Aut_\F(T))=\Aut_\mathcal G(T)$, $\phi$ centralizes $S/T$ and so
    $s_1\phi \equiv s_1 \pmod {\gamma_2(S)}$ ($b=1$ in Lemma~\ref{action}). If $S$ is exceptional, then $\phi$ leaves $C_S(Z_2(S))$, $T$ and $\gamma_1(S)$ invariant, contrary to $a\ne 1$. Hence $S$ is not exceptional. Now Lemma ~\ref{action} applies to give  $$s_{n-1}^{a^{-1}}=s_{n-1}\phi=  s_{n-1}^{a^{n-2}}.$$ Thus $a^{n-1}=1$ so that $n  \equiv 1 \pmod {p-1}$, that is, $n = j(p-1) + 1$ for some $j\geq 1$. To complete the proof of (ii), note that if $j=1$ then $n=p$ and $S$ has a maximal subgroup that is abelian by\cite[Theorem A]{pearls}, a contradiction. Hence $j\geq 2$.

Finally, assume that $\F$ is simple and $p\ge 5$.  If $\F$ is not exotic, then $\F$ is realised by a finite simple group by \cite[Theorem 5.71]{Craven}. But then Lemma~\ref{maxlclass simple} yields $\gamma_1(S)$ is extraspecial, a contradiction. This demonstrates (iii).
\end{proof}}

Semeraro's theorem of can be used to add pearls to saturated fusion systems without destroying saturation.

\begin{theorem}\label{SemThmC}
Let $\F_0$ be a saturated fusion system on a finite $p$-group $S$. Let $V  \le S$ be a fully $\F_0$-normalized subgroup, set $H= \Out_{\F_0} (V)$ and let
$\widetilde{\Delta} \le \Out(V)$ be such that $H$ is a strongly $p$-embedded subgroup of $\widetilde{\Delta}$. For $\Delta$ the full preimage of $\tilde{\Delta}$ in $\Aut(V)$, write
$$\F =\langle \Mor(\F_0), \Delta\rangle.$$
Assume further that
\begin{enumerate}
\item   $V$ is $\F_0$-centric  and minimal under inclusion amongst all $\F$-centric subgroups; and
\item no proper subgroup of $V$ is $\F_0$-essential.
\end{enumerate}
Then $\F$ is saturated.
\end{theorem}

\begin{proof}
This is a statement of \cite[Theorem C]{Semeraro} for the special case $m=1$.
\end{proof}

In Section~\ref{sec: non excep g1} we exploit results about fusion systems on $p$-groups with an abelian subgroup of index $p$ \cite{p.index,p.index2, p.index3}. The tool for doing this is the following proposition.

\begin{prop}\label{prop:sub sat} Suppose that $p$ is an odd prime, $\F$ is a saturated fusion system on $S$, $P \in \mathcal P(\F)$ and $\Aut_\F(S)=\Inn(S)\N_{\Aut_\F(S)}(P)$. Assume that $Z(S)<V\le Z(\gamma_1(S))$ is $\Aut_\F(S)$-invariant, put $S_1= VP$, $$H= \Inn(S_1)\langle \phi|_{S_1} \mid \phi \in N_{\Aut_\F(S)}(P)\rangle \le \Aut_\F(S_1)$$ and
$$B= \{\phi|_V\mid \phi \in \Aut_\F(\gamma_1(S))\} \le \Aut_\F(V).$$
Then the fusion system $$\mathcal G=\langle\Aut_\F(P),B, H \rangle$$ on $S_1$ is saturated and $\Aut_{\mathcal G}(S_1)= H$.
\end{prop}

\begin{proof}  We have $P \le S_1$ and so, if $|S_1:P|= p^k$, $S_1= N^{k}(P)$ by Theorem~\ref{pearls3} (i).
Let $\F_k= \langle \Aut_\F(P),H\rangle$. Then $\F_k$ is the smallest fusion subsystem of $\F$ defined on $S_1$ such that $\Aut_{\F_k}(P) =\Aut_\F(P)$ and $\Aut_{\F_k}(S_1)= \Inn(S_1)N_{\Aut_{\F}(S_1)}(P)$. By Lemma~\ref{pearls4}, $\F_k$ is saturated.

Since $V$ is normal in $S_1$,  $V$ is fully $\F_k$-normalized and, as $V > Z(S)$, $V$ is abelian and $[S_1 \colon V] =p$, $V$ is an $\F_k$-centric subgroup of $S_1$. We have $\Aut_{\F_k}(V) = \Aut_{H}(V)$ is a subgroup of $B \le \Aut_\F(V)$. If $\Aut_H(V)= B$, then $\mathcal G= \F_k$ and $\mathcal G$ is saturated.  So we may assume that  $\Aut_H(V)< B$. In particular,  as $\Aut_\F(S)=\Inn(S)\N_{\Aut_\F(S)}(P)$, $\Aut_\F(\gamma_1(S)) > \Aut_{\Aut_\F(S)}(\gamma_1(S))$ and so $\gamma_1(S)$ is $\F$-essential and $\Out_{\Aut_\F(S)}(\gamma_1(S))$  is strongly $p$-embedded in $\Out_\F(\gamma_1(S))$.  Therefore $\Out_H(V)$ is strongly $p$-embedded in $\Out_B(V)$. Furthermore, as $V$ is abelian, $V$ is minimal by inclusion amongst all $\F_k$-centric subgroups and no proper subgroup of $V$ is $\F_k$-essential.  Now application of Theorem~\ref{SemThmC} delivers $\mathcal G$ is saturated. That $\Aut_{\mathcal G}(S_1)= H$ is apparent from its construction.
\end{proof}

%% file: FSMaxClass1.tex
\section{Fusion systems on groups of maximal class: generalities}\label{sec:Generalities}

This section begins the study of saturated fusion systems on maximal class $p$-groups.

\begin{lemma}\label{p=2}
 Suppose that $S$ is a non-abelian $2$-group of maximal  class and   $\F$ is a saturated fusion system on $S$. \blue{Then $S$ is dihedral, semidihedral or generalized quaternion of order at least $8$ and }$\F$ is known and realizable.  In particular, if $E$ is an $\F$-essential subgroup of $S$, then $E$ is  an $\F$-pearl and $\gamma_1(S)$ is \blue{cyclic}.
 \end{lemma}

\begin{proof} If $|S|=8$, then $S$ is either dihedral or quaternion and it is an easy exercise to write down all the saturated fusion systems on $S$.  \blue{That the maximal class $2$-groups of order at least $2^4$ are dihedral, semidihedral or generalized quaternion  is well-known and can be found in \cite[Corollary 3.3.4 (iii)]{Led-Green} for example.   In this case, the totality of the saturated fusion systems on such groups can be found in \cite[Example I.3.8]{AKO}. } \end{proof}

\begin{lemma}\label{p=3}
 Suppose that $\F$ is a saturated fusion system on a non-abelian maximal class $3$-group. Then $\F$ is known, $\E_\F\subseteq \mathcal P(\F)\cup\{ \gamma_1(S)\}$ and, if  $\gamma_1(S)\in \E_\F$, then $\gamma_1(S)$ is abelian.
 \end{lemma}

\begin{proof} This is  Theorem~\ref{outcome}.
\end{proof}

Because of Lemmas~\ref{p=2} and \ref{p=3}, for the remainder of this \blue{work} we focus on the case $p \ge 5$.

\begin{hypothesis}\label{hyp1} \blue{ The prime $p$ is at least $5$, $S$ is a maximal class $p$-group of order at least $p^4$ and  $\F$ is a saturated fusion system on $S$.}
\end{hypothesis}

\blue{Throughout this section and Sections~\ref{sec:exceptional}, \ref{sec: g1 ess}, \ref{sec: g1 ess2}, \ref{sec: proof MT1 1} and \ref{sec: proof MT1 2}, we assume that Hypothesis~\ref{hyp1} holds sway and adopt its notation.}
We start with some lemmas which loosely locate the potential $\F$-essential subgroups within $S$. First we recall

\begin{lemma}\label{in.sub}
If  $E$  is an $\F$-essential subgroup which is not an $\F$-pearl, then either   $E \leq \gamma_1(S)$ or $E \leq \C_S(\Z_2(S))$.
Furthermore, there are no $\F$-pearls in $\gamma_1(S)$ or in $\C_S(\Z_2(S))$.
\end{lemma}
\begin{proof}\blue{This is just a restatement of Lemma~\ref{pearls2}.}
\end{proof}

The next result  is included for completeness.
\begin{cor}\label{gamma.ab}
If   $\gamma_1(S)$ is abelian, then the candidates for $\F$-essential subgroups are $\gamma_1(S)$ and $\F$-pearls. In particular, if $O_p(\F) =1$, then $\F$ has an $\F$-pearl.
\end{cor}

\begin{proof} (See also  \cite[Lemma 2.3 (a)]{p.index}).  Since $\gamma_1(S)$ is abelian, $S$ is not exceptional. Thus, if $E$ is $\F$-essential and not an $\F$-pearl, then $E \le \gamma_1(S)$ by Lemma~\ref{in.sub}. As $E$ is $\F$-centric, this yields $E=\gamma_1(S)$ as claimed.  In particular, if $\F$ has no $\F$-pearls, then $\gamma_1(S)$ is the only $\F$-essential subgroup in $\F$. Since $\gamma_1(S)$ is a characteristic subgroup of $S$, the Alperin-Goldschmidt Theorem yields $\gamma_1(S)= O_p(\F)$. Hence, if \blue{$O_p(\F) =1$,} then $\F$ has an $\F$-pearl.
 \end{proof}

\blue{
\begin{lemma}\label{morphism.fixing.Z2}
If there exists a morphism $\varphi \in \Aut_\F(\Z_2(S))$ such that $\Z(S) \varphi \neq \Z(S)$, then $\C_S(\Z_2(S)) \in \E_\F$.
\end{lemma}

\begin{proof}
 Since $\Z_2(S)$ is fully $\F$-normalized and so $\F$-receptive, there exists an automorphism $\ov{\varphi}\in \Aut_\F(\C_S(\Z_2(S)))$ such that $\ov{\varphi}|_{\Z_2(S)} = \varphi$. In particular $\ov{\varphi}$ does not normalize $\Z(S)$ and so it cannot be the restriction of an automorphism of $S$. Since $\C_S(\Z_2(S))$ is characteristic in $S$, by the Alperin-Goldschmidt fusion theorem we deduce that $\C_S(\Z_2(S))$ is $\F$-essential.
\end{proof}}

\begin{prop}\label{normal.ess} Suppose  $E$ is  $\F$-essential and $E$ is normal in $S$. Then $E$ is a maximal subgroup of $S$. Moreover,  either  $E \in \{\gamma_1(S), \C_{S}(\Z_2(S))\}$ or $|S|=p^4$ and $E$ is a non-abelian $\F$-pearl.
\end{prop}
\begin{proof} Suppose for a contradiction that $E$ is not a maximal subgroup of $S$. Since $E\norm S$ and $S$ has maximal nilpotency class,  $E$ is a member of the lower central series of $S$. Hence $E= \gamma_k(S)$ for some $k\ge 2$. In particular, $\Out_S(E)\cong S/E$ is a $p$-group of maximal class.   Since $p$ is odd, $S/E$ contains an elementary abelian subgroup of order $p^2$ by \cite[5.4.10 (ii)]{Gor}.  Lemma~\ref{strongly p structure} yields  that $F^*(\Out_\F(E)/O_{p'}(\Out_\F(E)))$ is a non-abelian simple group. Set $W= F^*(\Out_\F(E)/O_{p'}(\Out_\F(E)))$ and $X= W \Out_S(E)O_{p'}(\Out_\F(E))/O_{p'}(\Out_\F(E))$. Using Proposition~\ref{SE-p2} we know the candidates for $X$. Remember also that $p \ge 5$.  Suppose first that $\Out_S(E)$ is abelian. Then $\Out_S(E)\cong S/S'$ is elementary abelian of order $p^2$. The possibilities for $X$ are $\PSL_2(p^2)$, $\Alt(2p)$ for arbitrary odd $p$,  and ${}^2\mathrm F_4(2)'$ or $\mathrm{Fi}_{22}$ with $p=5$.  By a Frattini Argument and using \cite[Theorem 7.8.1]{GLS3},   $\N_{\Aut_\F(E)}(\Aut_S(E))$ acts irreducibly on $\Out_S(E)$. As every morphism in $\N_{\Aut_\F(E)}(\Aut_S(E))$ is the restriction of an $\F$-automorphism of $S$ (because $E$ is $\F$-receptive), there exists an automorphism $\tau\in \Aut_\F(S)$ such that $\gamma_1(S)\tau  \ne \gamma_1(S) $, which is absurd as this subgroup is characteristic in $S$. If $\Out_S(E)$ is non-abelian, then  $\Out_S(E)$ is extraspecial and the possibilities for $X$ are $\PSU_3(p)$ for all $p\ge 5$, $\mathrm{McL}$ or ${}^2\B_2(32){:}5 $ with $p=5$, and  $\J_4$ with $p=11$. This time \cite[Theorem 7.6.2]{GLS3} shows that either $\N_{\Aut_\F(E)}(\Aut_S(E))$ acts irreducibly on $\Out_S(E)/Z(\Out_S(E))$ or $X \cong {}^2\B_2(32){:}5$. In the former case, we obtain a contradiction exactly as in the abelian case.  Suppose that $X \cong {}^2\B_2(32){:}5$. Then $\Out_S(E)\cong 5^{1+2}_-\cong S/\gamma_3(S)$ has exponent 25. Since $|S| \ge 5^4$, applying Lemma~\ref{mc-facts} (vi) to  $S/\gamma_4(S)$  yields $S/\gamma_3(S)$ has exponent $5$, which is a contradiction. Therefore, if $E$ is normal in $S$, then $E$ is a maximal subgroup of $S$.

If $E \not \in \{\gamma_1(S), \C_S(\Z_2(S))\}$, then Lemma~\ref{pearls2}  implies $E$ is an $\F$-pearl and so $|S|=p^4$. This proves the lemma.   \end{proof}

\begin{lemma}\label{lem:not p^3}
Suppose that $|S| \ge p^5$ and $E$ is an $\F$-essential subgroup which is not an $\F$-pearl. Then $|E| \ge p^4$.
\end{lemma}

\begin{proof} Suppose that $|E| \le p^3$.  Then, as $E$ is $\F$-essential and $E$ is not an $\F$-pearl, $|E|=p^3$ and $E$ is abelian. Furthermore, Lemma~\ref{in.sub} indicates that either $E \le \gamma_1(S)$ or $E \le C_S(Z_2(S))$. Since abelian groups of order $p^3$ and exponent at least $p^2$ cannot be $\F$-essential, we have $E$ is elementary abelian. Of course $Z(S) \le E$. By  \cite[Corollary 1.23]{rank3}, every automorphism in $N_{\Aut_\F(E)}(\Aut_S(E))$ is the restriction of an $\F$-automorphism of $S$ and $|\Aut_S(E)|= p$.  Since $E$ is not normal in $S$ by Proposition~\ref{normal.ess}, $E \ne Z_3(S)$. If $E \le \gamma_1(S)$, then $Z_3(S) \le N_S(E)$  and, if $E \le C_S(Z_2(S))$, then $Z_2(S) \le E$ and $Z_3(S) \le N_S(E)$. In particular, as $E\ne Z_3(S)$ and $|N_S(E)|=p^4$,  $N_S(E)= Z_3(S)E$ in both cases. Since $|S| \ge p^5$, $Z_3(S)$ is abelian and so $\blue{Z(N_S(E)) = }E \cap Z_3(S)$ is a maximal subgroup of $E$.  Proposition~\ref{Ellen} \blue{(applied with $U=E$ and $W=1$)} implies that $O^{p'}(\Aut_\F(E)) \cong \SL_2(p)$ and, if we take $\tau \in Z(O^{p'}(\Aut_\F(E)))$ to be an involution, then $[E,\tau]$ is a natural $\GF(p)O^{p'}(\Aut_\F(E))$-module. Let $\hat \tau \in \Aut_\F(S)$ be such that $\hat \tau|_E=\tau$ and $\hat \tau$ has $p'$-order. Then $$\hat \tau\text{ centralizes }Z_3(S)E/E \cong Z_3(S)/(Z_3(S) \cap E).$$

Suppose that $Z_2(S)= E \cap Z_3(S)$.  Then $Z_2(S)= C_E(Z_3(S))$. If $[E,Z_3(S)]=Z(S)$, then $\hat \tau$ inverts $Z(S)$ and centralizes $Z_2(S)/Z(S)$. Hence $\hat \tau $ centralizes the group $Z_3(S)/Z(S)$ and this contradicts Lemma~\ref{centralizer auto}.  Hence $[E,Z_3(S)] \le Z_2(S)$ but is not contained in $Z(S)$.  Therefore $\hat \tau$ centralizes $Z(S)$ and inverts $Z_2(S)/Z(S)$.  Since $[Z_3(S), \gamma_1(S)] \le Z(S)$, we have $E \not \le \gamma_1(S)$ and, \blue{in particular, $S$ is exceptional and $n$ is even by Lemma~\ref{mc-factsb} (v).}  We apply Lemma~\ref{action} and use the notation from there with $\varphi= \hat \tau$. Thus, as $\hat \tau$ inverts $E\gamma_1(S)/\gamma_1(S)\cong E/(E\cap \gamma_1(S))= E/Z_2(S)$, $a=-1$, we have $\hat \tau$ acts as  $$(-1)^{n-3}b=-b=-1$$ on $Z_2(S)/Z(S)$ and as $$(-1)^{n-3}b^2= -b^2=1$$ on $Z(S)$. This is impossible and thus $Z_2(S) \not \le E$. Now $N_S(E)= Z_2(S)E$ and $[E,Z_2(S)]= Z(S)$.  It follows that $\hat \tau$ centralizes $Z_2(S)/Z(S)$ and $(Z_3(S)\cap E)/Z(S)$.  Hence $\hat \tau$ centralizes $Z_3(S)/Z(S)$, which is impossible by Lemma~\ref{centralizer auto}. We have proved, that, if $|E|= p^3$ and $E$ is not an $\F$-pearl, then $E$ is not $\F$-essential.
\end{proof}

\begin{lemma}\label{no inv} If  $O_{p'}(Z(\Out_\F(\gamma_1(S))))$ is non-trivial, then either
\begin{enumerate}
\item $\gamma_1(S)$ is abelian; or
\item $S$ is  exceptional and $\gamma_1(S)$ is extraspecial.
\end{enumerate}
\end{lemma}
 \begin{proof} Assume that $\tau \in \Aut_\F(\gamma_1(S))$ projects to a non-trivial element of the group $O_{p'}(Z(\Out_\F(\gamma_1(S))))$. Then $\tau$ is the restriction of $\hat \tau\in \Aut_\F(S)$ of $p'$-order and $[S,\hat \tau]\le \gamma_1(S)$.
Lemma~\ref{lem:autS p'elts} now gives the result.\end{proof}

%% file: EssMC.tex
\section{Essential subgroups in exceptional maximal class groups}\label{sec:exceptional}

In this section, we start the investigation of $\F$-essential subgroups when $S$ is exceptional of order $p^n$. We assume that

\blue{\begin{hypothesis} \label{Hyp7.1}  Hypothesis~\ref{hyp1} holds with $S$   exceptional.
\end{hypothesis}}

\blue{Because $S$ is exceptional, we know from Lemma~\ref{mc-factsb} (v) that $n$ is even and $$p^6\le |S| =p^n\le p^{p+1}.$$}
Our aim is to prove the following proposition.

\begin{prop}\label{Prop:ess.in.except}
Suppose Hypothesis~\ref{Hyp7.1} holds.  Then
\[ \E_\F \subseteq \Pp_a(\F) \cup \{\C_S(Z_2(S))\} \cup \{ E \mid \Z(S)<E\leq \gamma_1(S)\}.\]
Furthermore,
\begin{enumerate}
\item if $\Pp_a(\F)\neq \emptyset$, then $|S|=p^{p-1}$ and $\gamma_1(S)$ is extraspecial;
\item if $\C_S(\Z_2(S)) \in \E_\F$, then $|S|=p^6$, $O^{p'}(\Out_\F(C_S(Z_2(S))))\cong \SL_2(p)$ and either \begin{enumerate}\item $\gamma_1(S)$ is extraspecial; or \item  $p=5$,  $S= \mathrm{SmallGroup(5^6,661)}$, $\E_\F= \{C_S(Z_2(S))\}$, $\Out_\F(S)$ is cyclic of order $4$, $\Out_\F(C_S(Z_2(S))) \cong \SL_2(5)$ and $\F$ is unique.
    \end{enumerate}
     and
\item if  $\gamma_1(S)$ is extraspecial, then \[ \E_\F \subseteq \Pp_a(\F) \cup \{\C_S(\Z_2(S)), \gamma_1(S)\}. \]
\end{enumerate}
\end{prop}

We prove Proposition~\ref{Prop:ess.in.except} via a series of lemmas.

\begin{lemma}\label{lem:peals-abelian}
Assume Hypothesis~\ref{Hyp7.1} holds. Then $\Pp(\F) = \Pp_a(\F)$ (that is, every $\F$-pearl is abelian).
\end{lemma}

\begin{proof}
Aiming for a contradiction, suppose $E$ is an extraspecial $\F$-pearl. Then we have $\Z(S)=\Z(E)$ and the involution $\tau \in \Aut_\F(E)$ which maps into $ \Z(O^{p'}(\Out_\F(E)))$ inverts $E/\Z(S)$ and centralizes $\Z(S)$. In addition, by  Theorem~\ref{pearls3} (iii), there is an automorphism $\hat \tau \in \Aut_\F(S)$ \blue{of $p'$-order} such that $\hat \tau|_E = \tau$. Let $M$ be the maximal subgroup of $S$ containing $E$. Then $M \notin \{ \gamma_1(S), C_S(Z_2(S)) \}$ by Lemma \ref{pearls2}. Thus $MC_S(Z_2(S))=\gamma_1(S)C_S(Z_2(S)) = S$ and so, as $\gamma_1(S)$ and  $C_S(Z_2(S))$ are characteristic in $S$, $$\gamma_1(S)/\gamma_2(S) \cong S/C_S(Z_2(S)) \cong  M/\gamma_2(S)=E\gamma_2(S)/\gamma_2(S) \cong E/\Z_2(S)$$ as $\hat \tau$-groups. Hence $\gamma_1(S)/\gamma_2(S)$ and $M/\gamma_2(S)$ are inverted by $\hat \tau$. Therefore $\hat \tau$ inverts $S/\gamma_2(S)$ and so is an involution. Since $\hat \tau$ centralizes $Z(S)$,  this contradicts  Lemma \ref{action.exceptional1}. Thus every $\F$-pearl is abelian.
\end{proof}

\begin{lemma}\label{Lem:no ess in extrasp}
Assume Hypothesis~\ref{Hyp7.1} holds with  $\gamma_1(S)/\Z(S)$   abelian.  Then $\gamma_1(S)$ is extraspecial and no proper subgroup of $\gamma_1(S)$ is $\F$-essential.
\end{lemma}

\begin{proof} Since $S$ is exceptional, $Z(S)=Z (\gamma_1(S))$.
By assumption we have $\Z(S) = \Z(\gamma_1(S)) = [\gamma_1(S), \gamma_1(S)]$. Hence $\gamma_1(S)$ is extraspecial.

Aiming for a contradiction, suppose there exists a subgroup $E < \gamma_1(S)$ that is $\F$-essential.
Note that
\[[E,E] \le [E,\gamma_1(S)] \le [\gamma_1(S), \gamma_1(S)]  = \Z(S) \leq E.\]
In particular, $E$ is normal in $\gamma_1(S)$. If $E$ is not elementary abelian,   then $E$ has normal series $E> \Phi(E)=Z(S)$ which is stabilized by $\gamma_1(S)$.  This contradicts Lemma~\ref{lem:series}. Hence $E$ is elementary abelian. Since $\gamma_1(S)$ is extraspecial of order $p^{n-1}$, this implies   $|E|\leq p^{n/2}$. In particular the quotient $\gamma_1(S)/E$ is elementary abelian of order $[\gamma_1(S) \colon E] \geq p^{(n-2)/2}$. On the other hand,
Lemma~\ref{|E| bound} now yields that $|E| \ge p^{n-2}$. Thus $n= 4$ and   Lemma \ref{mc-factsb}(v) contradicts the assumption that $S$ is exceptional. Therefore  no proper subgroup of $\gamma_1(S)$ is $\F$-essential.
\end{proof}

\begin{lemma}\label{prop:CSZ2 essential} Assume Hypothesis~\ref{Hyp7.1} holds with $C_S(Z_2(S)) \in \E_\F$.  Then $|S|= p^6$, $O^{p'}(\Out_\F(C_S(Z_2(S))))\cong \SL_2(p)$ and \[ \E_\F \subseteq \Pp_a(\F) \cup \{\C_S(\Z_2(S)), \gamma_1(S)\}. \]  In addition, either  $\gamma_1(S)$ is extraspecial or $p=5$,  $S= \mathrm{SmallGroup(5^6,661)}$, $\E_\F= \{C_S(Z_2(S))\}$, $\Out_\F(S)$ is cyclic of order $4$, $\Aut_\F(C_S(Z_2(S))) \cong \SL_2(5)$ and $\F$ is unique.
\end{lemma}

\begin{proof} Set $R= C_S(Z_2(S))$.  Then, as $R \ne \gamma_1(S)$ and $S/Z(S)$ is not exceptional by Lemma~\ref{mc-factsb} (vi), we glean  $[\gamma_2(S), R]=\gamma_3(S)$. It follows that $R' = \Phi(R)=\gamma_3(S)$ and $R/\Phi(R)$ has order $p^2$. Therefore $\Out_\F(R) $ embeds into $\GL_2(p)$ and $$O^{p'}(\Out_\F(R)) = \langle \Out_S(R)^{\Out_\F(R)}\rangle \cong \SL_2(p)$$ acts irreducibly on $R/\gamma_3(S)$. In particular, $\Inn(R) \le O^p(\Aut_\F(R))$.  Suppose that $|S| \ge p^7$.
 Then $$\gamma_4(S)=[\gamma_3(S), S]= [\gamma_3(S),\gamma_1(S)R] \ge [\gamma_3(S),\gamma_1(S)][\gamma_3(S),R]$$
Since $[\gamma_3(S),\gamma_1(S)] \le \gamma_5(S)$ and $[\gamma_3(S),R]$ is normal in $S$, we deduce   $[\gamma_3(S),R]=\gamma_4(S)$ and similarly $[\gamma_4(S),R]= \gamma_5(S)$. Since $[R,R]= \gamma_3(S)$, we have $\gamma_4(S)= [R,R,R]$ and $\gamma_5(S)=[R,R,R,R]$.  Now we must have $O^{p}(\Aut_\F(R))$ centralizes $\gamma_3(S)/\gamma_5(S)$. Since $\Inn(R) \le O^{p}(\Aut_\F(R))$, this means that $[\gamma_3(S),R]\le \gamma_5(S)$, a contradiction.
Hence $|S|\leq p^6$. Since $S$ is exceptional,   Lemma~\ref{mc-factsb} (v) implies that $|S|=p^6$.

We next show that there are no $\F$-essential subgroups properly contained in $R$.
Aiming for a contradiction, suppose that $E<R$  is an $\F$-essential subgroup $E$. Then, as $E$ is not an $\F$-pearl,  Lemma~\ref{lem:not p^3} implies $|E|= p^4$ and so $E$ is a maximal subgroup of $R$.
Recall that $\Aut_\F(R)$ acts irreducibly on $R/\gamma_3(S)$. In particular, it acts transitively on the maximal subgroups of $R$ containing $\gamma_3(S)=\Phi(R)$. Thus $E$ is $\F$-conjugate to $\gamma_2(S)$. Since $E$ is fully $\F$-normalized, we deduce that $E=\gamma_2(S)$ is normal in $S$, contradicting Proposition \ref{normal.ess}. This proves the claim.

It remains to show that either $\gamma_1(S)$ is extraspecial or $p=5$ and, in the latter case, determine the structure of $\F$.
Let $\varphi \in N_{O^{p'}(\Out_\F(R))}(\Out_S(R))$  be the automorphism of order $p-1$ corresponding to the matrix
$\begin{pmatrix} \lambda^{-1} & 0 \\ 0 & \lambda \end{pmatrix}$ for a fixed $\lambda \in \GF(p)$ of order $p-1$. Then by saturation there is an automorphism $\hat{\varphi}\in \Aut_\F(S)$ such that $\hat{\varphi}|_R = \varphi$. In particular $\hat{\varphi}$ acts on $\gamma_i(S)$ for every $i\geq 1$. Thus
for every $x,y\in S$ such that $R = \langle x \rangle \gamma_2(S)$ and $\gamma_2(S) = \langle y \rangle \gamma_3(S)$ we have
$$(x\gamma_2(S))\varphi = x^{\lambda^{-1}}\gamma_2(S) \quad \text{ and } \quad (y\gamma_3(S))\varphi = y^\lambda\gamma_3(S).$$
Since $[R, \gamma_1(S)] = \gamma_2(S)$ we also deduce that $\hat\varphi$ raises the elements of $\gamma_1(S)/\gamma_2(S)$ to the power $\lambda^2$.
By Lemma~\ref{action} we deduce that $\hat{\varphi}$ raises the elements of $\Z_2(S)/\Z(S)$ to the power $\lambda^{-1}$.
Let $s_1,s_2 \in S$ be such that $\gamma_1(S) = \langle s_1 \rangle \gamma_2(S)$ and $\gamma_2(S) = \langle s_2 \rangle \gamma_3(S)$. Then $s_1\hat\varphi = s_1^{\lambda^2}u$ and $s_2\hat\varphi = s_2^{\lambda}v$ for some $u\in \gamma_3(S)$ and $v\in \Z_2(S)$. Thus
\[ [s_1,s_2]\hat\varphi = [s_1^{\lambda^2}u, s_2^{\lambda}v] = [s_1,s_2]^{\lambda^3} \mod \Z(S).\]
Suppose $[s_1,s_2]\notin \Z(S)$. Then, since $[s_1,s_2] \in [\gamma_1(S),\gamma_2(S)] \leq  \gamma_4(S)=\Z_2(S)$, we get
\[ [s_1,s_2]\hat\varphi = [s_1, s_2]^{\lambda^{-1}} \mod \Z(S)\] from Lemma~\ref{action}.
Therefore  $\lambda^3 \equiv \lambda^{-1} \pmod {p}$  and so  $\lambda^4 \equiv 1 \pmod{p}$.  Hence either $[s_1,s_2] \in \Z(S)$ or $p=5$.
In the former case, as $[\gamma_1(S),\gamma_3(S)] \leq \Z(S)$, we deduce that
\[[\gamma_1(S), \gamma_1(S)] =[\gamma_1(S), \gamma_2(S)] = [\langle s_1\rangle \gamma_2(S), \langle s_2 \rangle \gamma_3(S)] \leq \Z(S).\]
Thus $\gamma_1(S)/\Z(S)$ is abelian \blue{and  Lemma \ref{Lem:no ess in extrasp} implies $\gamma_1(S)$ is extraspecial.} This proves that either $\gamma_1(S)$ is extraspecial or $p=5$.

Suppose that $p=5$ and assume that $\gamma_1(S)$ is not extraspecial. Then $|S|=5^6$. All the groups of order $5^6$ are known. We use {\sc Magma} and the package from \cite{parkersemerarocomputing} to check that there are 39 maximal class $5$-groups, $16$ of which are exceptional.  Four of these groups have $\Aut(R)$ non-soluble and two of them have $\gamma_1(S)$ not extraspecial. This leaves two groups to consider. For one of the cases $\Aut(S)$ is a $5$-group and so this cannot support a fusion system with $\Aut_\F(R)$ non-soluble.  In $\mathrm{SmallGroup}(5^6,661)$, we have $\Out_\F(S)$ has cyclic Sylow $2$-subgroups of order $4$. In particular, $\Aut_\F(S)$ is uniquely determined up to isomorphism and by restriction we have a subgroup $Y$ of $\Aut_\F(R)$ of order $2^2.5^4$.   Calculating in  $\Aut_\F(R)$ we find two conjugacy classes of subgroups $X$ containing $\Inn(R)$ and with $X/\Inn(R) \cong\SL_2(5)$. Exactly one of these classes contains an $\Aut(R)$-conjugate of  $Y$.  We check that the corresponding fusion system is saturated using \cite{parkersemerarocomputing} (though it obviously is). This is the fusion system described in (ii). It remains to prove that there are no other candidates for $\F$-essential subgroups on $S=\mathrm{SmallGroup}(5^6,661)$ when $R$ is $\F$-essential. \blue{Computer code to do this using \cite{parkersemerarocomputing} is described in Subsection~\ref{subsec:CSZ2}.}
However, we can also present an argument which does not require a computer. Suppose that $E \le S$ is an $\F$-pearl. Then there exist $\phi \in N_{\Aut_\F(E)}(\Out_S(E))$ of order $4$ which is the restriction  $\hat \phi\in \Aut_\F(S)$. Recall that the Sylow $2$-subgroup of $\Aut(S)$ has order $4$.  It follows that $\hat \phi|_{R}\in \Aut_\F({R})$.  But then $\hat \phi$ normalizes ${R}$, $\gamma_1(S)$ and $E \gamma_2(S)$ and these are distinct maximal subgroups of $S$. In particular, $\phi$ acts on $S/\gamma_2(S)$ as a scalar. However, we know that   $\hat \phi|_{R}\in N_{\Aut_\F({R})}(\Aut_S({R}))$ and this element does not act as a scalar on $S/\gamma_2(S)$, a contradiction.  Hence $\F$ has no $\F$-pearls. Now suppose that $E$ is $\F$-essential and $E \le \gamma_1(S)$ with $E \not \le R$. Notice that $Z(\gamma_1(S))= Z(S)$ and, as $\gamma_1(S)$ is not extraspecial, $\gamma_1(S)'= Z_2(S)$ and $\gamma_3(S)/Z(S)= Z(\gamma_1(S)/Z_2(S))$. Moreover,  $\gamma_2(S)= C_{\gamma_1(S)}(Z_2(S))$. Thus $S $ stabilizes the characteristic series $\gamma_1(S)>\gamma_2(S)>\gamma_3(S) > \gamma_4(S)$ of $\gamma_1(S)$. Lemma~\ref{lem:series} implies that $\gamma_1(S)$ is not $\F$-essential.  Hence $E < \gamma_1(S)$. By Lemma \ref{lem:not p^3} we deduce that $|E|=5^4$. Thus $E$ is a maximal subgroup of $\gamma_1(S)$. In particular $Z_2(S)=\gamma(S)'=[\gamma_1(S),E] \leq E$. Since $E \nleq R$ we deduce that $Z_2(S) \nleq \Z(E)$ and so $E$ is not abelian. Thus  $1\neq [E,E] < Z_2(S)$, that implies $|[E,E]|=p$. In particular $Z(S)[E,E] \leq \Z(E)$. If $Z(S) \neq [E,E]$ then $Z(S)[E,E] = Z_2(S)$ and so $Z_2(S)\leq \Z(E)$, a contradiction. Thus $Z(S)=[E,E]$. The group $\gamma_3(S)$ stabilizes the sequence $1 < Z < E$ and by Lemma~\ref{lem:series} we deduce that $\gamma_3(S) \leq E$. Also $E$ is not extraspecial (it has order $5^4$), hence $|Z(E)|=5^2$. The group $\Out_S(E)$ acts non-trivially on $E/Z(E)$ which is elementary abelian of order $25$. Hence by Proposition~\ref{Ellen} we deduce that $O^{5'}(\Out_\F(E))\cong \SL_2(5)$. Let $\tau \in Z(O^{5'}(\Out_\F(E)))$ be an involution. We have already proved that $\gamma_1(S)$ is not $\F$-essential, hence $E$ is maximal $\F$-essential in $S$ and so there is $\hat{\tau}\in \Aut_\F(S)$ such that $\hat{\tau}|_E = \tau$. Note that $\hat{\tau}$ centralizes $\gamma_2(S)/\gamma_3(S) \cong \gamma_1(S)/E$ and $Z(E)/Z(S)$. Since $E$ is non abelian and $\gamma_3(S)$ is abelian, we deduce that $Z(E) \leq \gamma_3(S)$ and  so $\gamma_3(S) = Z_2(S)Z(E)$. Hence $\gamma_3(S)/Z_2(S)$ is congruent to $Z(E)/Z(S)$ as a $\hat{\tau}$-group and it is therefore centralized by $\hat{\tau}$. We showed that $\hat{\tau}$ centralizes $\gamma_2(S)/ Z_2(S)$. Now, $\hat{\tau}$ acts non-trivially on $S/\gamma_1(S)$ and $\hat{\tau}$ is an automorphism of the group $S/Z(S)$, that is not exceptional. Therefore we get a contradiction from  Lemma \ref{centralizer auto}.

 This proves that $\E_\F=\{R\}$ and completes the description of $\F$ in the  case when $S =
 \SmallGroup(5^6,661)$ and $R$ is $\F$-essential.
\end{proof}

\begin{lemma}\label{lem:p^6CSZ2essential} \blue{Assume Hypothesis~\ref{Hyp7.1} holds.  Suppose that $E \in \E_\F$ with $E \not \le \gamma_1(S)$.}  Then either $E $ is an $\F$-pearl or $|S| =p^6$ and $E= C_S(Z_2(S))$.
\end{lemma}

\begin{proof} \blue{Suppose that the lemma is false. Set $R= C_S(\Z_2(S))$, and let $E$ be an $\F$-essential subgroup of $S$ chosen of maximal order with $E$ not contained in $\gamma_1(S)$.  Then  $E$ is not an $\F$-pearl. } Lemma~\ref{in.sub} and Lemma~\ref{lem:not p^3} together imply that $E \le R$ and $|E| \ge p^4$.  By Lemma~\ref{prop:CSZ2 essential}, we have $R \not \in \E_\F$ and so $E<R$. Because $E < R$ and $E$ is $\F$-essential, we have $\Z_2(S) < E$.  By Lemma~\ref{mc-factsb} (vi),  $S/Z(S)$ is not exceptional.  Since $E/\Z(S)$ is not contained in $\gamma_1(S/\Z(S))=\gamma_1(S)/\Z(S)$, Lemma~\ref{sbgrp-maxclass} implies that $E/Z(S)$ has maximal class. Since $|E|\ge p^4$, $E/Z(S)$ is not abelian.
In particular, $Z(E/Z(S))= Z_2(S)/Z(S)$ and this  implies  $Z(E)= Z_2(S)$ is $\Aut_\F(E)$-invariant.
Since $R \not \in \E_\F$, Lemma~\ref{morphism.fixing.Z2} implies that $\Aut_\F(E)$ leaves $Z(S)$ invariant. Now $C_{\Aut_\F(E)}(Z(S))$ has $p'$-index in $\Aut_\F(E)$ and,  since $E/Z(S)$ has maximal class, Theorem~\ref{full auto S} implies that $|E/Z(S)|=p^3$. Therefore $O^{p'}(\Aut_\F(E))$ acts on $E/Z(E)$ as $\SL_2(p)$.  Let $\tau \in \Aut_\F(E)$ project to $Z(O^{p'}(\Out_\F(E)))$ be an involution. As $O^{p'}(\Aut_\F(E))$ centralizes $Z_2(S)$, $  \tau$ centralizes  $Z_2(S)$. Then, the maximal choice of $E$ implies that there exists $\hat \tau \in \Aut_\F(S)$ so that $\hat \tau |_E=\tau$.  In particular, $\hat \tau$ has even order and inverts $E/Z(E)$. Since $Z(E)=Z_2(S)$ and $E\gamma_1(S)= S$, $\hat \tau$ inverts $S/\gamma_1(S)$. Assume that $(y\gamma_2(S))\hat \tau = y^b \gamma_2(S)$ for some $b \in \GF(p)^\times.$ Then, in Lemma~\ref{action}, we have $n$ is even and $a=-1$ and, as $\tau$ centralizes $Z_2(S)$, we obtain the unfathomable equations \begin{eqnarray*}
a^{n-3}b &=& -b = 1\\a^{n-3}b^2&=&-b^2=1.\end{eqnarray*} This contradiction completes the proof of the lemma.
\end{proof}

\begin{proof}[Proof of Proposition~\ref{Prop:ess.in.except}] Assume Hypothesis~\ref{Hyp7.1} holds.
Let $E$ be an $\F$-essential subgroup. By Lemma~\ref{in.sub} either $E$ is an $\F$-pearl, $E \leq \gamma_1(S)$ or $E \leq \C_S(\Z_2(S))$. If $E$ is an $\F$-pearl, then $E$ is abelian by Lemma \ref{lem:peals-abelian}. If $E$ is not an $\F$-pearl and it is not contained in $\gamma_1(S)$ then $E=\C_S(Z_2(S))$ by Lemma~\ref{lem:p^6CSZ2essential}. This proves that
\[ \E_\F \subseteq \Pp_a(\F) \cup \{\C_S(\Z_2(S))\} \cup \{ E \mid \Z(S)<E\leq \gamma_1(S)\}\] which is the displayed statement of the proposition.

If $\Pp_a(\F) \ne \emptyset$, then \cite[Theorem 3.14]{pearls} implies $|S|=p^{p-1}$ and $\gamma_1(S)$ is extraspecial.  Hence (i) holds.

If $\C_S(Z_2(S))$ is $\F$-essential, then Lemma~\ref{prop:CSZ2 essential} gives (ii).

 If $\gamma_1(S)$ is extraspecial, then $\gamma_1(S)/\Z(S)$ is abelian and Lemma~\ref{Lem:no ess in extrasp} gives (iii).
\end{proof} 

%% file: NotExcep.tex
\section[$\gamma_1(S)$ is $\F$-essential and $S$ is not exceptional]{The structure of  $\gamma_1(S)$ when $\gamma_1(S)$ is $\F$-essential and $S$ is not exceptional}\label{sec: g1 ess}

In this section we continue to assume Hypothesis~\ref{hyp1}. In addition, we assume that $S$ is not exceptional and $\gamma_1(S)$ is $\F$-essential. So we work with
\begin{hypothesis}\label{Hyp8.1}  Hypothesis~\ref{hyp1} holds with $S$   not exceptional and  $\gamma_1(S)\in\E_\F$.
\end{hypothesis}

Our objective in this section is to explore the structure of $\Omega_1(\gamma_1(S))$ when Hypothesis~\ref{Hyp8.1} holds.

\begin{lemma}\label{gamma1-essB} Assume that Hypothesis~\ref{Hyp8.1} holds.    If $O_{p'}(\Out_\F(\gamma_1(S)))$ is not centralized by $\Out_S(\gamma_1(S))$, then $\Omega_1(\gamma_1(S))$ is elementary abelian of order either $p^{p-1}$ or $p^p$.\end{lemma}

\begin{proof} Assume that  $O_{p'}(\Out_\F(\gamma_1(S)))$ is not centralized by $\Out_S(\gamma_1(S))$. Set  $R_0=O_{p,p'}(\Aut_\F(\gamma_1(S)))$ and $$R= [R_0,\Aut_S(\gamma_1(S))]\Inn(\gamma_1(S)).$$ Then $R_0/\Inn(\gamma_1(S)) = O_{p'}(\Out_\F(\gamma_1(S)))$, $R> \Inn(\gamma_1(S))$ and, by \cite[Theorem 5.3.10]{Gor},  there exists an $\Out_\F(\gamma_1(S))$-chief factor $V$ in $\Omega_1(\gamma_1(S))$ which is not centralized by $R$. The definition of $R$ and coprime action implies that  $R/C_R(V)$ is not centralized by $\Aut_S(\gamma_1(S))C_R(V)/C_R(V)$.  Applying Proposition~\ref{Hall-Hig} delivers $|V| \ge p^{p-1}$. Since $V$ is elementary abelian, either  $n=p+1$ \blue{and $\gamma_1(S)=\Omega_1(\gamma_1(S))$}  or $V=\Omega_1(\gamma_1(S)) $ by  Lemma~\ref{mc-facts} (iv). \blue{In the latter case, $\Omega_1(\gamma_1(S))$ is elementary abelian and we are done.}
\blue{Assume $|S|=p^{p+1}$ and $\gamma_1(S)=\Omega_1(\gamma_1(S))$ has order $p^p$. }If $V= \gamma_2(S)$, then, as $V$ is irreducible, $V \le Z(\gamma_1(S))$ and $\Omega_1(\gamma_1(S))$ is abelian.  If  $V= \gamma_1(S)/Z(S)$ and $\gamma_1(S)$ is extraspecial, then $C_S(Z_2(S)) \ne \gamma_1(S)$ and $S$ is exceptional, a contradiction.
We conclude that $\Omega_1(\gamma_1(S))$ is elementary abelian of order $p^{p-1}$ or $p^p$. 
\end{proof}

\begin{remark}
When $p=5$ and $|S|= 5^6$, the baby Monster sporadic simple group provides an example which demonstrates that when $S$ is exceptional  $\Out_\F(\gamma_1(S))$ may have a non-central normal subgroup of $5'$-order which is not centralized by $\Out_S(\gamma_1(S))$. See \cite[Table 5.1]{G2p} for example.
\end{remark}

Recall that groups of $\mathrm L_2(p)$-type are defined in Definition~\ref{def:l2p}.

\begin{lemma}\label{gamma1-essA}  Assume that Hypothesis~\ref{Hyp8.1} holds.     If $\Out_\F(\gamma_1(S))$ is not of $\mathrm{L}_2(p)$-type, then either
\begin{enumerate}
 \item $\Omega_1(\gamma_1(S))\le Z(\gamma_1(S))$; or\item  $|\Omega_1(\gamma_1(S)): Z(\gamma_1(S))|=p$.\end{enumerate} In particular, $\Omega_1(\gamma_1(S))$ is elementary abelian.
\end{lemma}

\begin{proof} Assume that $\Out_\F(\gamma_1(S))$ is not of $\mathrm{L}_2(p)$-type. We have $|\Omega_1(\gamma_1(S))|\le p^{p}$ by Lemma~\ref{mc-facts} (iv). Set $Z=Z( \gamma_1(S))$. If $Z \not < \Omega_1(\gamma_1(S))$, then, as $S$ has maximal class, $\Omega_1(\gamma_1(S)) \le Z$ and (i) holds.  So suppose that $Z < \Omega_1(\gamma_1(S))$. Then $Z$ is elementary abelian because $\Omega_1(\gamma_1(S))$ has exponent $p$. 

Since $S$ is not exceptional, $Z\ge Z_2(S)$ and so $Z$ is not centralized by $S$. In particular, $Z$ admits a non-trivial action of $\Out_\F(\gamma_1(S))$ and   $C_{\Out_\F(\gamma_1(S))}(Z)$ is a $p'$-group. 
As  $\Out_\F(\gamma_1(S))$ is not of $\mathrm L_2(p)$-type, $\Out_\F(\gamma_1(S))/C_{\Out_\F(\gamma_1(S))}(Z)$
 is not of $\mathrm L_2(p)$-type.
Therefore Theorem~\ref{feit}  implies $|Z| \ge p^{\frac 2 3 (p-1)}$.

Define $X=\Omega_1(\gamma_1(S))/Z$. Then, assuming that (ii) does not hold, $|X| \ge p^2$.   Since $\gamma_1(S)$ is the $2$-step centralizer, $V=X/[X,\gamma_1(S)]Z$ has order at least $p^2$. In particular, $V$ is a $\GF(p)\Out_\F(S)$-module and it is not centralized by $S$.
Furthermore, $C_{\Out_\F(S)}(V)$ is a $p'$-group and  $\Out_\F(S)/C_{\Out_\F(S)}(V)$ is not of $\mathrm L_2(p)$-type.
Applying Theorem~\ref{feit} yields $|V| \ge p^{\frac 2 3 (p-1)}$.
Therefore, as $p \ge 5$, we obtain the contradiction
$$p^p \ge |\Omega_1(\gamma_1(S))| \ge |Z||V|\ge  p^{2\frac 2 3 (p-1)} = p^{\frac 4 3 (p-1)} >p^p.$$ Hence $|X| \le p$ and (ii) holds.  This completes the proof.
  \end{proof}

\begin{lemma}\label{gamma1-essC}  Assume that Hypothesis~\ref{Hyp8.1} holds.    If $\Omega_1(\gamma_1(S))$ is non-abelian, then $|\Omega_1(\gamma_1(S))|=p^{p-1}$, $$Z(\Omega_1(\gamma_1(S)))= \Phi(\Omega_1(\gamma_1(S)))=[\Omega_1(\gamma_1(S)),\gamma_1(S)]=Z(\gamma_1(S)).$$  In particular, $\Omega_1(\gamma_1(S))$ has nilpotency class $2$. Furthermore, $$L=\langle \Out_S(\gamma_1(S))^{\Out_\F(\gamma_1(S))}\rangle \cong \PSL_2(p)$$ and, if $H\le \Aut_\F(S)$ has $p'$-order and projects to a complement to $\Out_S(\gamma_1(S))$ in $N_{L}(\Out_S(E))$, then $|C_{\Omega_1(\gamma_1(S))}(H)| = p^2$.
\end{lemma}

\begin{proof} Set $V= \Omega_1(\gamma_1(S))/[\Omega_1(\gamma_1(S)),\gamma_1(S)]$ and $$V_1=[\Omega_1(\gamma_1(S)),\gamma_1(S)]/[\Omega_1(\gamma_1(S)),\gamma_1(S),\gamma_1(S)].$$
Then, as $\gamma_1(S)$ is the $2$-step centralizer, $|V|\ge p^2$ and so, $V$ and $V_1$    are not centralized by $S$ unless $V_1=[\Omega_1(\gamma_1(S)),\gamma_1(S)]=Z(S)$ has order $p$.

By Lemmas~\ref{gamma1-essB}, and \ref{gamma1-essA}, $\Out_\F(\gamma_1(S))$   is of $\mathrm L_2(p)$-type  and $O_{p'}(\Out_\F(\gamma_1(S)))$ is  centralized by $\Out_S(\gamma_1(S))$. Since $O_p(\Out_\F(\gamma_1(S))) =1$,  $\Out_\F(\gamma_1(S))$ is not $p$-soluble and so, setting
$L=\langle \Out_S(\gamma_1(S))^{\Out_\F(\gamma_1(S))}\rangle$, we have $L$
 has a quotient isomorphic to $\PSL_2(p)$ and $L$ is   quasisimple. Hence $L \cong \SL_2(p)$ or $\PSL_2(p)$. Using Lemma~\ref{no inv}, we obtain $Z(L)=1$ and so   $L \cong \PSL_2(p)$. Let $\ov H$ be a complement to $\Out_S(\gamma_1(S))$ in $N_L(\Out_S(\gamma_1(S)))$ and $H$ be a preimage of $\ov H$ of $p'$-order. Then $H$ is cyclic of order $(p-1)/2$.

In $\Omega_1(\gamma_1(S))$, assume that $W>U>T\ge 1$ are $ \Aut_\F(\gamma_1(S))$-invariant subgroups with $\ov W=W/U$ and $\ov U=U/T$ both $ \Aut_\F(\gamma_1(S))$-chief factors. Then $\ov W$ and $\ov U$ can be regarded as $\GF(p)L$-modules.
 If $|\ov W|=|\ov U|=p$, then, as $\gamma_1(S)$ is the $2$-step centralizer, $\Aut_S(\gamma_1(S)) \le O^{p'}(\Aut_\F(\gamma_1(S))) \Inn(\gamma_1(S))$ centralizes $U/T$, contrary to $S$ having maximal class.  Hence at least one of the chief factors has order greater then $p$.
Write $|\ov W|=p^w$ and $|\ov U|=p^u$ with $p^3\le p^{u+w} \le |\Omega_1(\gamma_1(S))| \le p^p$. Using  Lemma~\ref{L2p} and   $L \cong \PSL_2(p)$, gives $u$ and $v$ are odd and, as $u+w \le p$, $H$ centralizes $$[\ov W,S;(w-1)/2]/[\ov W,S;(w+1)/2]\text{ and }[\ov U,S;(u-1)/2]/[\ov  U,S;(u+1)/2].$$
Because $S$ has maximal class,  there exists $\ell \ge 1$ such $$[W,S;(w-1)/2]= \gamma_\ell(S)/U$$ and $$[U,S;(u-1)/2]= \gamma_{\ell+(w+1)/2+(u-1)/2 }(S)/T.$$  By Lemma~\ref{centralizer auto}, $$\ell - (\ell+(w+1)/2+(u-1)/2 ) \equiv 0 \pmod {(p-1)/2}.$$
Hence $$w+u = k(p-1)$$ for some integer $k$. Therefore $|\Omega_1(\gamma_1(S))| \ge p^{p-1}$. Assume   $|\Omega_1(\gamma_1(S))| = p^{p-1}$. Then  $W=\Omega_1(\gamma_1(S))$, and $$U= Z(\Omega_1(\gamma_1(S)))= \Phi(\Omega_1(\gamma_1(S)))=[\Omega_1(\gamma_1(S)),\gamma_1(S)]=Z(\gamma_1(S))$$ and the statements in the lemma hold.
Assume that  $|\Omega_1(\gamma_1(S))| > p^{p-1}$. Then Lemma~\ref{mc-facts} (iv)  implies that $|S|=p^{p+1}$. This gives $\Omega_1(\gamma_1(S))=\gamma_1(S)$ and there exists $\gamma_1(S)\ge W>U>T>A \ge 1$ with $W/U$, $U/T$ and $T/A$ each $\Aut_\F(\gamma_1(S))$-chief factors. We assume that they have order $p^u$, $p^w$ and $p^t$ respectively.  Then $u+w= w+ t=p-1$ and $u+w+t\le p$. This means that $w=t=1$ and $u=p-2$.  Notice that $Z(\gamma_1(S))$ is normalized by $S$ and  is $L$-invariant. So $Z(\gamma_1(S))\in \{W,U,T\}$.   Since $\gamma_1(S)$ is non-abelian by assumption, we must have $Z(\gamma_1(S))= T = Z(S)$. Hence $S$ is exceptional, a contradiction.
\end{proof}

We remark that we have  constructed a saturated fusion system which satisfies  the conclusion of Lemma~\ref{gamma1-essC} using {\sc Magma} \blue{(see Subsection~\ref{CEx}). The example is realized by a group $G$ of shape $7^{3+3}{:}\PGL_2(7)$. Taking $S \in \syl_7(G)$, $S$ has exponent $7$, $\gamma_1(S)$ is special with centre of order $7^3$. Setting $\F=\F_S(G)$, we get
 $\Aut_\F(\gamma_1(S)) \cong 7^3{:}\PGL_2(7)$.} In this case, $\gamma_1(S)$ is the unique $\F$-essential subgroup \blue{and $\F$ cannot be further  decorated with pearls to create a larger saturated fusion system $\mathcal G$ with $\mathcal G$-pearls.}

%% file: g1sexcep.tex
\section[$\gamma_1(S)$ is $\F$-essential and $S$ is  exceptional]{The structure of  $\gamma_1(S)$ when $\gamma_1(S)$ is $\F$-essential and $S$ is  exceptional}\label{sec: g1 ess2}

 \blue{We now consider  Hypothesis~\ref{Hyp7.1} once again, applying results from Section~\ref{sec: g1 ess} to fusion systems on the non-exceptional group  $S/Z(S)$.
  The objective of this section is to prove}

\begin{prop}\label{gamma1-essD}   \blue{Assume that Hypothesis~\ref{Hyp7.1} holds with $\gamma_1(S)\in \E_\F$. Then $\gamma_1(S)$ is extraspecial.}
\end{prop}

\begin{proof}
Assume that $S$ is  exceptional and $\gamma_1(S)$ is $\F$-essential but is not extraspecial.

Set  $Q= \gamma_1(S)$, $V= Z_2(Q)$, $Z= Z(Q)=Z(S)$ and $\mathcal K= N_\F(Z)$. Then $Z$ is fully $\F$-normalized and so $\mathcal K$ is saturated, $Q$ is a $ \mathcal K$-essential subgroup and $Q/Z$ is a $\mathcal K/Z$-essential subgroup. By Lemma \ref{mc-factsb} (v) and (vi), $S/Z(S)$ is not exceptional and $|S|\le p^{p+1}$ and
  Lemma~\ref{mc-facts} (vi) yields $S/Z(S)$  has exponent $p$.
Because $Q$ is not extraspecial and $Q/Z=\Omega_1(Q/Z)$, we know $\Omega_1(Q/Z)$ is not elementary abelian. Hence Lemma~\ref{gamma1-essC} yields $(Q/Z)'= Z(Q/Z)= V/Z$ and $Q$ has order $p^p$. Furthermore, $O^{p'}(\Out_\F(Q)) \cong \PSL_2(p)$.

Let $A$ be the preimage of $O^{p'}(\Out_\F(Q)) $ in $\Aut_\F(Q)$ and  $a$ be a natural number such that $|V/Z|= p^a$. Then, as $V/Z$ is an irreducible $\GF(p)A$-module by Lemma~\ref{gamma1-essC}, $a$ is odd and so  $V$ is not extraspecial and, as $A$ acts irreducibly on $V/Z$, we deduce that   $V$ is elementary abelian. Because $A$ acts irreducibly on $Q/V$ and $V$ is abelian, we also have $C_{Q}(V)= V$. Therefore, using $ Z$ has order $p$ and $Q/Z(Q)$ and $V/Z$ are irreducible $\GF(p)A$-modules, we have  $Q/V \cong \Hom_{\GF(p)}(V,Z) $ (the dual of $V$) which also has order $p^a$. Since $|Q|= p^{p}$, we infer that $a=(p-1)/2$. As $a$ is odd, we also have $p \equiv 3 \pmod 4$. Since $Q/Z$ is special, the commutator map $\kappa: Q/V \times Q/V \rightarrow V/Z$ induces a non-trivial $\GF(p)A$-module homomorphism $\kappa^*: \Lambda^2(Q/V) \rightarrow V/Z$. In particular, $V/Z$ is isomorphic to a quotient of $\Lambda^2(Q/V)$ as $\GF(p)A$-modules. Writing $d= a-1= (p-1)/2-1$, Proposition~\ref{Clebsch Gordan} yields $d \equiv 2 \pmod 4$. Hence $p \equiv 7 \pmod 8$.
 Finally, we note that by Lemma~\ref{mc-facts}(vi), $\gamma_2(S)$ has exponent $p$ and thus as $Q$ is regular by Lemma~\ref{mc-facts}(ii), Lemma~\ref{lem:regular}(ii) and the irreducible action of $A$ on $Q/V$  implies $Q$ has exponent $p$.
We summarise what has been established.

\medskip

\begin{claim}\label{clm:qstruc}The following hold:\begin{enumerate}
\item $p \equiv 7 \pmod 8$ and $d \equiv 2 \pmod 4$.
 \item  $Q$ has exponent $p$ and  nilpotency class $3$, $V=Q'=C_Q(V)$ has order $p^{(p+1)/2}$ and $Z=Z(Q)$ has order $p$.
 \item  $A/\Inn(Q)\cong \PSL_2(p)$ acts irreducibly  on $Q/V\cong V/Z \cong \VV_{\frac{p-3}{2}}$ and $A$ centralizes $Z$.
\end{enumerate}
\end{claim}

\medskip
We intend to make an explicit calculation and so we begin by establishing some notation. Let $\epsilon$ be a generator of $\GF(p)^\times=(\mathbb Z/p\mathbb Z)^\times$. Thus $\epsilon$ is an integer with $1 \le \eps \le p-1$. We regard $Q/V$ and $V/Z$ as unfaithful representations of $\SL_2(p)$ and we identify $\Out_S(Q)$ with  $\langle \left(\begin{smallmatrix}1&1\\0&1 \end{smallmatrix}\right)\rangle$ and let $\tau \in N_{A}(\Aut_S(Q))$ have order $(p-1)/2$  be  $\left(\begin{smallmatrix}\eps&0\\0&\eps^{-1} \end{smallmatrix}\right)$.
The element $\iota$ has order $2$ and corresponds to $\langle \left(\begin{smallmatrix}0&1\\-1&0 \end{smallmatrix}\right)\rangle$ in $A$. Hence $\iota$ inverts $\tau$ and does not normalize $\Out_S(Q)$.

We have $Q/V$ is semisimple as a $\GF(p)\langle \tau \rangle$-module and $\tau$ normalizes $\gamma_j(S)$ for all $j \ge 1$.  The semisimple action of $\tau$ implies that we may select $t_i\in \gamma_i(S)\setminus \gamma_{i+1}(S)$, $1\le i \le (p-1)/2$ such that $t_i V$ is an eigenvector for $\tau$ on $Q/V$.

By \ref{clm:qstruc}(iii) the $\GF(p)\PSL_2(p)$-modules $Q/V$ and $V/Z$ can identified with the $\GF(p)\SL_2(p)$-module $\VV_{\frac{p-3}{2}}=\VV_d$.    When we make this identification, we may suppose that $t_1$ corresponds to $x^d$, $t_{(p-1)/2}$ is $y^d$ and generally $t_j$ corresponds to $x^{d-j+1}y^{j-1}$ for $1 \le j \le (p-1)/2= d+1$.  We calculate that
$$(t_jV)\tau=  t_j^{\eps^{d-2j+2}}V.$$
Define $m = (d+2)/2$. Then, as $p \equiv 7 \pmod 8$ by \ref{clm:qstruc}(i), $m$ is even and we also have  $(t_mV)\tau =t_mV$.
Furthermore,
 we can calculate

 \begin{claim}\label{clm:iota1}$(t_jV)\iota = (t_{d+2-j}V)^{(-1)^{(j-1)}}.$\end{claim}

\medskip

 For $1 \le \rho_i\in V$, $1\le i \le d+1$ be such that $\rho_i Z$ are  eigenvectors for $\tau$ with $\rho _i \in \gamma_{(p-1)/2+i}(S) \setminus\gamma_{(p-1)/2+i+1}(S)$. Lemma~\ref{lem:autS p'elts} implies that for each $j$ we have $$(\rho_jZ)\tau=  \rho_j^{\eps^{d-2j+2}}Z$$ which is consistent with making the standard identification with $\VV_{d}$.
 Just as above we have

 \begin{claim}\label{clm:iota2} $(\rho_jZ)\iota = (\rho_{d+2-j}Z)^{(-1)^{(j-1)}}.$\end{claim}

\medskip

Observe that, as $m$ is even, \ref{clm:iota1} and \ref{clm:iota2} demonstrate that $\iota$ inverts $t_mV$ and $\rho_m Z$.

We use the commutator relations as follows $$[\rho_j,t_k]\tau= [\rho_j^{\eps^{d-2j+2}}z,t_k^{\eps^{d-2k+2}}v] = [ \rho_j^{\eps^{d-2j+2}},t_k^{\eps^{d-2k+2}}]= [\rho_j,t_k]^{\eps^{2d-2(j+k)+4}}$$
where  $v \in V$ and $z \in Z$. Since $\tau$ centralizes $Z$ by \ref{clm:qstruc}(iii), using the above commutator calculation and exploiting \ref{clm:qstruc}(ii) to obtain equality  demonstrates that the following statement holds.

\begin{claim}\label{clm:clm2qess}
We have $C_{Q}(\rho_j) =V\langle t_k\mid k \ne d+2-j \rangle$.

\end{claim}

\bigskip
Set
 $r_j=[t_j,t_m]$ for $1\le j \le d+1$,  then $r_m=1$ and the action of $\tau$ shows that $r_j \in \langle \rho_j\rangle Z$.

 Since $Q$ has nilpotency class $3$ and $p \ge 5$, Lemma~\ref{Engel1} and Theorem~\ref{Engel2} imply that $V\le E_2(Q) < Q$ and the irreducible action of $A$ on $Q/V$ then yields $V=E_2(Q)$.  We obtain a contradiction by demonstrating that $t_m \in E_2(Q)$. Thus we   calculate  $[t_m,y,y] $ where $y= \prod_{i=1}^{d+1} t_i^{a_i}$ with $1 \le a_i \le p$ is a coset representative of $V$ in $Q$. Notice that in this next calculation, we have $[t_m,w_1w_2] \in [t_m,w_1][t_m,w_2]Z$ and so $[t_m,w_1w_2,w]= [[t_m,w_1][t_m,w_2],w] $ for $w, w_1, w_2 \in Q$.
\begin{eqnarray*}
[t_m,y,y]&=&[t_m,\prod_{i=1}^{d+1} t_i^{a_i},y]= [\prod_{i=1}^{d+1} [t_m,t_i^{a_i}], y]=[\prod_{i=1}^{d+1} r_i^{a_i}, y]=\prod_{i=1}^{d+1} [r_i, y]^{a_i}\\
&=& \prod_{i=1}^{d+1}[ r_i,\prod_{k=1}^{d+1} t_k^{a_k} ]^{a_i}=   \prod_{i=1}^{d+1}[ r_i, t_{d+2-i} ]^{a_i+a_{d+2-i}}
\end{eqnarray*}
where the last equality follows from \ref{clm:clm2qess}.
 Notice that, as $m$ is even by \ref{clm:qstruc}(i), we apply $\iota$ as follows $$r_i \iota =[t_i\iota ,t_m\iota]= [t_{d+2-i}^{(-1)^{(i-1)}}, t_m^{-1}]= r_{d+2-i}^{(-1)^{i} } .$$
Hence, as $d$ is even by \ref{clm:qstruc}(i), \ref{clm:iota1} and \ref{clm:iota2} yield
$$[ r_i, t_{d+2-i} ]\iota =
[ r_{d+2-i}^{(-1)^{(i)}}, t_{i}^{(-1)^{(d+2-i-1)}} ]= [ r_{d+2-i}, t_{i}] ^{(-1)^{i+(d+2-i-1)}}= [ r_{d+2-i}, t_{i}] ^{-1} $$
 and so, as $r_m=1$, we can pair elements in the product $\prod_{i=1}^{d+1}[ r_i, t_{d+2-i} ]^{a_i+a_{d+2-i}}$ to obtain $[t_m,y,y]=1$ for all $y \in Q$.
Hence $t_m \in E_2(Q)$ and this is our contradiction which establishes that if $Q$ is $\F$-essential, then $Q$ is extraspecial.
 \end{proof}

%% file: Loc1.tex
\section{Locating $\F$-essential subgroups in groups of maximal class I}\label{sec: proof MT1 1}

This section and the following section are devoted to the proof of Theorem~\ref{MT1}. To set the scene we repeat its statement.

\begin{customthm}{D}\label{MT1} Suppose that $p$ is a prime,  $S$ is a $p$-group of maximal  class and order at least $p^4$ and $\F$ is a saturated fusion system on $S$. If $E$ is an $\F$-essential subgroup, then either $E$ is an $\F$-pearl, or $E= \gamma_1(S)$ or $E=C_S(Z_2(S))$. Furthermore, if $S$ is exceptional, then $\mathcal P(\F)=\mathcal P_a(\F)$.  \end{customthm}

 Lemmas~\ref{p=2} and \ref{p=3} show that Theorem~\ref{MT1} holds if $p\le 3$. Hence we may and do assume that $p \ge 5$ and so Hypothesis~\ref{hyp1} holds and we take our notation from there.
\begin{lemma}\label{not abelian} Assume Hypothesis~\ref{hyp1} holds. If  $\gamma_1(S)$ is abelian and $E$ is $\F$-essential, then either $E$ is  an $\F$-pearl or $E=\gamma_1(S)$.
\end{lemma}

\begin{proof}  This is a consequence of Corollary~\ref{gamma.ab}.  \end{proof}

We now show that Theorem~\ref{MT1} holds when $|S|$ is small.

\begin{lemma}\label{p^5-groups ok} Assume Hypothesis~\ref{hyp1} holds. Suppose that $p^4\le |S| \le p^5$ and $\F$ is a saturated fusion system on $S$.  If $E$ is $\F$-essential, then $E$ is either an $\F$-pearl or $E \in \{\gamma_1(S),C_S(Z_2(S))\}$.
\end{lemma}
\begin{proof}  Suppose $E$ is not an $\F$-pearl. If $|S| = p^4$, then $E$ is normal in $S$ and the result follows from Proposition~\ref{normal.ess}. Hence $|S| =p^5$ and $|E|= p^4$ by Lemma~\ref{lem:not p^3}. In particular, $E$ is a normal subgroup of $S$. Again Proposition~\ref{normal.ess} yields the result.
 \end{proof}

Because of Corollary~\ref{c12},  Proposition~\ref{Prop:ess.in.except} and Lemmas~\ref{not abelian} and \ref{p^5-groups ok} we work under the following hypothesis until the proof of Theorem~\ref{MT1} is complete.

\begin{hypothesis}\label{hp.conj} Hypothesis~\ref{hyp1} holds with $|S| \ge p^6$ and the saturated fusion system  $\F$ is a minimal counterexample to Theorem \ref{MT1}, first with respect to  $|S|$ and second with respect to the number of morphisms in $\F$. Furthermore, \begin{enumerate}
\item $\gamma_1(S)$ is not abelian or extraspecial.
\item $\Out_\F(S)$ is cyclic of order dividing $p-1$ and acts faithfully on $S/\gamma_1(S)$.
\end{enumerate}
\end{hypothesis}

 \blue{We now assume that Hypothesis~\ref{hp.conj} is satisfied and say } that an $\F$-essential subgroup which is not contained in $\Pp(\F) \cup  \{ \gamma_1(S), \C_S(\Z_2(S)) \}$ is a \emph{witness} (to the fact that Theorem~\ref{MT1} is false).

 \begin{lemma}\label{not normal} If $E$ is a witness, then $E$ is not normal in $S$ and $E < \gamma_1(S)$.
 \end{lemma}
 \begin{proof}
   This follows from Lemma~\ref{in.sub} and  Propositions~\ref{normal.ess} and \ref{Prop:ess.in.except}.
 \end{proof}

\begin{lemma} \label{no pearls} The saturated fusion system $\F$ has no   $\F$-pearls. Furthermore, every witness $E$ is properly contained in $\gamma_1(S)$, $|E| \ge p^4$ and $\Omega_1(E)=E \cap \Omega_1(\gamma_1(S))$.
\end{lemma}

\begin{proof}
The fact that $\F$ has no $\F$-pearls follows from Lemma~\ref{pearlprune} and the minimality of $\F$. \blue{The first part of the second statement  is in Lemma~\ref{not normal}}. Since $\F$ has no $\F$-pearls, Lemma~\ref{lem:not p^3} implies $|E|\ge p^4$. Finally, we note that $\gamma_1(S)$ is regular by Lemma~\ref{mc-facts} (ii) and so Lemma~\ref{lem:regular} (iii) gives the final statement.
\end{proof}

Our first major consequence of Hypothesis~\ref{hp.conj} is as follows.

\begin{prop}\label{propOp} We have $O_p(\F)=1$.
\end{prop}

\begin{proof}   Let $Q= O_p(\F)$ and assume that $Q \ne 1$. Let  $E$ be a witness.
As $Q$ is normal in $S$ and contained in $E$ by Lemma \ref{OpinE}, Lemma~\ref{not normal} implies $Q< E$ with $|S:E|\ge p^2$ and so there exists $j\geq 3$ such that $Q= \gamma_j(S)$.

We first demonstrate that $Q$ is $\F$-centric. Suppose this is false. Then $C_S(Q) \not \le Q$ and so $C_S(Q) = \gamma_k(S)$ for some $k < j$ where we assume that $S=\gamma_0(S)$. In particular, $Q$ is abelian and $\Aut_\F(Q)\cong\Out_\F(Q)$.
If $C_S(Q) \le E$, then $C_S(Q)= C_E(Q)$ is $\Aut_\F(E)$-invariant. Hence  $N_\F(E) \subseteq N_\F(C_S(Q))\subset \F$ as $C_S(Q) >Q$. Since  $E$ is $\F$-essential and $\Aut_\F(E) \subset N_\F(C_S(Q))$, Lemma~\ref{F-essential G-essential} implies $E$ is  $N_\F(C_S(Q))$-essential and this contradicts the minimality of $\F$.
 Therefore $C_S(Q) =\gamma_k (S)\not \le E$ and we choose $k<\ell \le j $ so that $\gamma_{\ell}(S) \le E$ and $\gamma_{\ell-1}(S) \not \le E$. Then $\gamma_{\ell -1}(S)\le N_S(E)$ and  $\gamma_{\ell -1}(S)$ centralizes $ Q$.

 Set $L= \langle \Aut_{\gamma_{\ell-1}(S)}(E)^{\Aut_\F(E)}\rangle$. Then $L$ centralizes $Q$ and $C_L(E/Q)$ is a $p$-group by coprime action \cite[Theorem 5.3.6]{Gor}.
 As $L$ is normal in $\Aut_\F(E)$, we have $O_p(L) \le O_p(\Aut_\F(E))= \Inn(E) $. Hence $$\Inn(E)\ge O_p(L) \ge C_L(E/Q).$$

For $K \le \Aut_\F(E)$, set $$\ov K = K\Inn(E)/\Inn(E) \le \Out_\F(E)=\ov{\Aut_\F(E)}.$$

By Lemma~\ref{F/Q}, $\F/Q$ is saturated on $S/Q$ and we know $S/Q$ has maximal class. Since $Q \ne 1$, Lemma~\ref{mc-factsb}(vi) implies that, if $|S/Q| \ge p^4$, $S/Q$ is not exceptional.
  We claim $E/Q$ is $\F/Q$-essential.
 Certainly $E/Q$ is fully $\F/Q$-normalized.  Let $J= C_{\Aut_\F(E)}(E/Q)$.  We have $[J,L] \le J \cap L \le C_L(E/Q) \le \Inn(E)$. Hence $\ov J$ and $\ov L$ commute.  Therefore $\ov J$ normalizes $\ov{\Aut_{\gamma_{\ell-1}(S)}(E)}$.

 Let $T= J \cap \Aut_S(E) \in \Syl_p(J)$ and assume that $\ov T \ne 1$.  Then we obtain $\ov J= N_{\ov{J}} (\ov{\Aut_{\gamma_{\ell-1}(S)}(E)})$ and, by the Frattini Argument, $$\Out_\F(E) = \ov J N_{\Out_\F(E) }(\ov T).$$ Thus  $$\Out_\F(E)= N_{\Out_\F(E) }(\ov T)N_{\ov{J} } (\ov{\Aut_{\gamma_{\ell-1}(S)}(E)}),$$ which contradicts $\Out_\F(E) $ having a strongly $p$-embedded subgroup.  This proves that $T \le \Inn(E)$ and that $\ov J$ is a $p'$-group. Assume that $E/Q$ is not   $\F/Q$-centric. Then $C_{S/Q}(E/Q) \not \le E/Q$. Hence there exists $x \in N_S(E)\setminus E$ such that $c_x \in J$ and $\ov {c_x} \ne 1$, contrary to $\ov J$ having $p'$-order. Hence $E/Q$ is $\F/Q$-centric. Finally, we note that $\Aut_{\F/Q}(E/Q)\cong \Aut_\F(E)/J$ and so $\Out_{\F/Q}(E/Q) \cong \Out_\F(E)/\ov J$. As $J$ has $p'$-order and is centralized by $\ov L$, Lemma~\ref{strongly p-embedded} implies   $\Out_{\F/Q}(E/Q)$ has a strongly $p$-embedded subgroup. Thus $E/Q$ is $\F/Q$-essential as claimed.
 Therefore, if $Q$ is not $\F$-centric, then $E/Q$ is $\F/Q$-essential. In particular, $|E/Q|\ge p^2$ and $|S:E|\ge p^2$ as $E$ is not normal in $S$. Hence $|S/Q|\ge p^4$ and $E/Q < \gamma_1(S)/Q$.  This contradicts the minimal choice of $\F$.
 Hence $Q$ is $\F$-centric.

By Theorem~\ref{model}, there exists a group $G$ which is a model for $\F$. For this model, we have $Q=O_p(G)$ and $C_G(Q)=Z(Q)$. Furthermore, $C_G(Q/\Phi(Q))= Q$ and $E/\Phi(Q)$ is $\F_{S/\Phi(Q)}(G/\Phi(Q))$-essential.  If $\Phi(Q)\ne 1$, we apply induction to obtain a contradiction as $\F_{S/\Phi(Q)}(G/\Phi(Q))$ has
essential subgroups which are properly contained $\gamma_1(S/\Phi(Q))$. Hence $Q$ is elementary abelian. Since $Q$ is $\F$-centric, it is the largest normal abelian subgroup of $S$, say $Q=\gamma_w(S)$. Then, as $|S:E| \ge p^2$ and $E>Q$,  $w \ge 3$ and $|Q| \le |\Omega_1(\gamma_1(S))| \le p^p$ by Lemma~\ref{mc-facts} (iv).

Since $Q= \gamma_w(S)$ is elementary abelian, Theorem~\ref{J6.2} implies  that $\gamma_{w-1}(S)$ has nilpotency class   $2$. Hence   $\gamma_{w-1}(S)$ acts quadratically on $Q$. Consider $G/Q$.  We have $O_p(G/Q)=1$. Since $\gamma_{w-1}(S)$ acts quadratically on $Q$, and $p>3$, $\gamma_{w-1}(S)/Q$ centralizes $O_{p,p'}(G)/Q$ by \cite[Lemma 1.2]{chermak2002}. In particular, $G/Q$ has a component of order divisible by $p$. Suppose that $K$ is the preimage of such a component. Then $S \cap K > Q$. If $K$ is not normalised by $S$, then $|K^S| \ge p$ and $\langle (K^S\rangle \cap S)/Q$ contains an elementary abelian $p$-group of rank $p$. Since $|S| \ge |Q||\langle K^S\rangle \cap S||S:N_S(K)|\ge p^2.p^p.p=p^{p+3}$, this contradicts Lemma~\ref{mc-facts}(v). Hence $K$ is normalized by $S$, it follows that $K \cap S $ is normal in $S$.  If $L/Q \ne K/Q$ is a component of $G/Q$ with $S \cap L> Q$, then $L \cap S$ is normal in $S$ and we see that $Z(S/Q) \ge Z((L\cap S)(K\cap S)/Q)$ has order at least $p^2$. This is impossible as $S/Q$ has maximal class and $|S/Q| \ge p^3$.  Hence $K/Q$ is the unique such component and we have $F^*(G/Q) \le  O_{p'}(G/Q)  K/Q$ with   $\gamma_{w-1}(S)  $ acting quadratically on $Q$.  Since $\gamma_{w-1}(S)  /Q$ has order $p$ and centralizes $O_{p'}(G/Q)$,   $\gamma_{w-1}(S)  /Q$ acts faithfully on $K/Q$. Because $S/Q$ has maximal class and order at least $p^3$, $\gamma_{w-1}(S)  /Q$ is contained in every non-trivial normal subgroup of $S/Q$. Hence
 $S/Q$ acts faithfully on $K/Q$.  As $p \ge 5$, \cite[Theorem A]{chermak2002} implies that $K/Q$ is a group of Lie type defined in characteristic $p$. We refer to \cite[Theorem 2.5.12]{GLS3} for properties of automorphism groups of groups of Lie type.  Since $p\ge 5$, $K/Q$ has no graph automorphisms of order $p$ and so, if $S \not \le K$, $SK/K$ is a cyclic group of field automorphisms of $K/Q$. If $SK/K > 1$, it  follows from \cite[Theorem 3.3.1]{GLS3} that $|\Omega_1(Z((S\cap K)/Q))| \ge p^p$ and again we contradict Lemma~\ref{mc-facts}(iv). Hence $S \le K$.
Since $S/Q$ is $2$-generated and non-abelian, we now see that $K$ is a rank at most $2$ Lie type group defined over $\GF(p)$. Hence \cite[Theorem 3.3.1(b)]{GLS3} yields $K/Q$ mod its centre is one of $\PSU_3(p)$, $\PSL_3(p)$, $\PSp_4(p)$, $\G_2(p)$,    and in each case $S/Q$ has maximal class. However in these groups we have $|N_{K/Q}(S/Q)| \ge (p-1)^2/3 > p-1$, contrary to Hypothesis~\ref{hp.conj} (ii).  This contradiction illustrates that $Q=1$ and completes the proof of the proposition.
\end{proof}

\begin{lemma}\label{gamma.Op} Assume that $E$ is a witness and suppose there is $n>j\geq 1$ such that $\gamma_j(S) \le  E$. Then  $\gamma_j(S)$ is not  $\Aut_\F(E)$-invariant.
\end{lemma}

\begin{proof}
Note that $\gamma_j(S)$ is fully $\F$-normalized  since it is normal in $S$.
Consider the saturated fusion system $\N_\F(\gamma_j(S))$ on $S$. Aiming for a contradiction, suppose  $\Aut_\F(E)\subseteq \N_\F(\gamma_j(S))$, Lemma~\ref{F-essential G-essential} implies that  $E$ is $\N_\F(\gamma_j(S))$-essential.  Thus $\N_\F(\gamma_j(S))$ is a counterexample to Theorem \ref{MT1} and by the minimality of $\F$ we deduce that $\F = \N_\F(\gamma_j(S))$. But then $O_p(\F) \ne 1$, contrary to Proposition~\ref{propOp}.
\end{proof}

In the next three lemmas we exploit Proposition~\ref{propOp} to provide both lower and upper bounds for the order of $S$.

\begin{prop}\label{prop:not p^6 and CSZ2 essential}  If $S$ is exceptional then $|S|\geq p^8$.
In particular, if $S$ is exceptional, then $\C_S(Z_2(S))$ is not $\F$-essential.
\end{prop}

\begin{proof} \blue{Because $S$ is exceptional, Hypothesis~\ref{Hyp7.1} holds. Aiming for a contradiction,  suppose that $S$ has order $p^6$} and let $E $ be a witness. By Hypothesis~\ref{hp.conj} and Lemma~\ref{not normal} we know $E < \gamma_1(S)$  and $\gamma_1(S)$ is not extraspecial.
 Note that $Z(S)$ is not $\Aut_\F(E)$-invariant by Lemma \ref{gamma.Op}.  In particular, $Z(S) < Z(E)$ and so $|Z(E)|\geq p^2$.  By Lemmas~\ref{lem:not p^3} and \ref{not normal}, $|E|=p^4$ and $E\neq \gamma_2(S)$ as $E$ is not normal in $S$. So $\gamma_1(S) = EE^s$ for some $s \in S$. Also, as $\gamma_2(S) = \C_S(Z_2(S)) \cap \gamma_1(S)$, $E \ne \gamma_2(S)$ implies $Z_2(S) \cap Z(E) = Z(S)$.
If $E$ is abelian then $Z(\gamma_1(S)) = E \cap E^s$ has order $p^3$, contradicting the fact that $\Z(\gamma_1(S))=Z(S)$ because $S$ is exceptional. Hence $E$ is non-abelian and $|Z(E)|=p^2$. This implies $$[E,E] \leq [\gamma_1(S),\gamma_1(S)] \cap Z(E) \leq Z_2(S) \cap Z(E) = Z(S).$$ Since $E$ is not abelian, we get that $[E,E] = Z(S)$ and so $\Z(S)$ is $\Aut_\F(E)$-invariant, a contradiction.  This proves $|S| \ge p^8$ and the last statement follows from Proposition~\ref{Prop:ess.in.except}.
\end{proof}

\begin{lemma}\label{p7} We have $|S|\geq p^7$.
\end{lemma}

\begin{proof} By Hypothesis~\ref{hp.conj} we have $|S|\geq p^6$ and $\gamma_1(S)$ is not abelian.
Aiming for a contradiction, suppose $|S|=p^6$.  By Lemma \ref{prop:not p^6 and CSZ2 essential} the group $S$ is not exceptional.  Let $E < \gamma_1(S)$ be a witness. By Lemmas~\ref{lem:not p^3} and \ref{not normal}, $|E|=p^4$, $\gamma_1(S)= N_S(E)$ and $Z_2(S) \leq Z(\gamma_1(S)) < E$. In addition $E \ne \gamma_2(S)$ as $E$ is not normal in $S$. The group $Z_2(S)$ is not $\Aut_\F(E)$-invariant by Lemma~\ref{gamma.Op}, hence  $Z_2(S) < Z(E)$. So $|E \colon Z(E)| \leq p$ and we conclude that $E$ is abelian. Since $S^p \leq Z(S)$ by Lemma \ref{mc-facts}(vi) and $Z(S)$  is not $\Aut_\F(E)$-invariant by Lemma~\ref{gamma.Op}, we also get that $E$ has exponent $p$, that is, $E$ is elementary abelian.   Let  $s \in S\setminus N_S(E)$. Then  $\gamma_1(S) = EE^s$ and so $Z(\gamma_1(S)) = E \cap E^s$ has order $p^3$ and $\gamma_1(S)= \Omega_1(\gamma_1(S))$ is non-abelian. Also, $\gamma_3(S) = Z(\gamma_1(S))=\C_E(\gamma_1(S)) $ and $[\gamma_1(S) \colon E] = p = [E \colon \C_E(\gamma_1(S))]$. Proposition~\ref{Ellen} applied with $V=E$ implies $O^{p'}(\Aut_\F(E))\cong \SL_2(p)$ and $E/\C_E(O^{p'}(\Aut_\F(E)))$ is  a natural module for $O^{p'}(\Aut_\F(E))$. In particular the group $K = \C_E(O^{p'}(\Out_\F(E)))$ has order $p^2$ and it is contained in $\gamma_3(S) = \Z(\gamma_1(S))$.

If $\gamma_1(S)$ is $\F$-essential, then \blue{Hypothesis~\ref{Hyp8.1} holds and Lemma~\ref{gamma1-essC}} implies that $p^5= p^{p-1}$, a contradiction. Hence $\gamma_1(S)$ is not $\F$-essential and $E$ is not properly contained in any $\F$-essential subgroup of $S$.
Let $\tau \in \Z(O^{p'}(\Aut_\F(E)))$ be an involution. Then there is $\hat{\tau}\in \Aut_\F(S)$ such that $\hat{\tau}|_E = \tau$ and $\hat \tau$ has $p'$-order. So $\C_{\gamma_3(S)}(\hat{\tau}) =C_{E}(\hat{\tau})= K$ has order $p^2$ and $[\gamma_1(S), \hat \tau]\le E $. Hence $|C_{\gamma_1(S)}(\hat \tau)|=p^3$. The only way this is compatible with Lemma~\ref{action} is if $\hat \tau$ centralizes $\gamma_1(S)/\gamma_2(S)$, $\gamma_3(S)/\gamma_4(S)$ and $\gamma_5(S)= Z(S)$ and inverts $\gamma_2(S)/\gamma_3(S)$ and $\gamma_4(S)/Z(S)$. As $\gamma_1(S)/E\cong \gamma_2(S)/\gamma_3(S)$ which is inverted by $\hat \tau$, we have a contradiction. Hence $|S| \ge p^7$.
\end{proof}

We now turn to determining an upper bound for $|S|$.

\begin{lemma}\label{Omega.non.ab} Suppose that $E$ is a witness.  Then $Z(\Omega_1(\gamma_1(S))) \le E$ and $\Omega_1(\gamma_1(S)) \not \le E$. In particular, $\Omega_1(\gamma_1(S))$ is not abelian, $|S| < p^{2p-4}$ and $p\geq 7$.
\end{lemma}

\begin{proof} Using Lemma~\ref{lem:regular}(iii), we have  $$[\N_{Z(\Omega_1(\gamma_1(S)))}(E),E]\le \Omega_1(\gamma_1(S)) \cap E \leq \Omega_1(E) \leq \Omega_1(\gamma_1(S))$$  and so the group $\N_{Z(\Omega_1(\gamma_1(S)))}(E)$ stabilizes the characteristic series $1 < \Omega_1(E) < E$. Since $E$ is $\F$-essential, Lemmas~\ref{lem:K in E} and \ref{lem:series} imply  that $Z(\Omega_1(\gamma_1(S))) \leq E$.

If $\Omega_1(\gamma_1(S)) \le E$, then $\Omega_1(\gamma_1(S))=\Omega_1(E)$ is $\Aut_\F(E)$-invariant, contrary to  Lemma \ref{gamma.Op}. Hence $\Omega_1(\gamma_1(S)) \not\le E$. This proves the first two statements. In particular, as $Z(\Omega_1(\gamma_1(S))) \le E$, $\Omega_1(\gamma_1(S))$ is not abelian.

Assume that $|S| \geq p^{2p-4}$. Then $|S|> p^{p+1}$ since $2p-4 > p+1$ unless $p=5$ and in the latter case we know $|S|>5^6$ by Lemma \ref{p7}. Thus $S$ has positive degree of commutativity by Lemma~\ref{mc-factsb}(iv) and  (v), and  $\Omega_1(\gamma_1(S))=\gamma_{n-p+1}(S)$ by Lemma~\ref{mc-facts}(iii).
Since $\Omega_1(\gamma_1(S))$ is not abelian and  $S$ has positive degree of commutativity we get
\begin{eqnarray*}Z(S)&=& \gamma_{n-1}(S) \le [\gamma_{n-p+1}(S),\gamma_{n-p+1}(S)]\\& = & [\gamma_{n-p+1}(S),\gamma_{n-p+2}(S)]\le \gamma_{2n-2p+3+1}(S).\end{eqnarray*} This implies that $n \leq 2p-4 -1$, contradicting our assumptions. Therefore $|S| < p^{2p-4}$. Now, if $p=5$ we obtain $|S|<5^6$, contradicting Lemma \ref{p7}. Hence $p\geq 7$.
\end{proof}

We conclude this section with a characterization of the proper subsystems of $\F$.

\begin{lemma}\label{lem:subsystem1} Suppose that $\mathcal K$ is a proper saturated fusion subsystem of $\F$ on $S$. Then    $\mathcal K \subseteq N_\F(\gamma_1(S))$.
\end{lemma}

\begin{proof}
Suppose $\mathcal K \not \subseteq N_\F(S)$. Then there is a $\mathcal K$-essential subgroup $P$.
Since $\K$ is properly contained in $\F$, it is not a counterexample to Theorem \ref{MT1}. Hence either $P$ is a $\K$-pearl or $P \in \{C_S(Z_2(S)), \gamma_1(S)\}$.
If $P$ is a $\mathcal K$-pearl then $P$ is contained in some $\F$-essential subgroup $E^*$. By Lemma~\ref{no pearls}  we have  $P \le E^* \le \gamma_1(S)$ or $P \le E^* \le C_S(Z_2(S))$ both of which are impossible by Lemma~\ref{in.sub} applied to $P$. Thus $P$ is not a $\mathcal K$-pearl.  If $P= C_S(Z_2(S))\ne \gamma_1(S)$,  then $S$ is exceptional and $p$ is $\F$-essential.  This contradicts  Lemma \ref{prop:not p^6 and CSZ2 essential}. Hence the only option is $P= \gamma_1(S)$. This proves the statement.
\end{proof}

%% file: Loc2.tex
\section{Locating $\F$-essential subgroups in groups of maximal class II}\label{sec: proof MT1 2}

\blue{In this section, we continue to prepare for the proof of Theorem~\ref{MT1}. In particular,  we continue to work under Hypothesis~\ref{hp.conj}.}
We start by creating a compendium facts that we have established about saturated fusion systems which satisfy Hypothesis~\ref{hp.conj}.
\begin{lemma}\label{facts} Suppose that $E \in \E_\F$ is a witness.
\begin{enumerate}\item $E < \gamma_1(S)$, $E$ is not normal in $S$ and $|E| \ge p^4$;
\item if $S$ is exceptional, then $C_S(Z_2(S))$ is not $\F$-essential and $\gamma_1(S)$ is not extraspecial;
\item $p \ge 7$ and $p^7\le |S| < p^{2p-4}$;
\item $\Omega_1(\gamma_1(S))$ is not abelian;
\item $O_p(\F)=1$;
\item $\Out_\F(S)$ acts faithfully on $S/\gamma_1(S)$ and is cyclic of order dividing $p-1$;
\item if $\gamma_1(S)$ is $\F$-essential, then $S$ is not exceptional and $\Out_\F(\gamma_1(S))\cong \PSL_2(p)$ or $\PGL_2(p)$.
\end{enumerate}
\end{lemma}

\begin{proof} Part (i) follows from Lemmas~\ref{not normal} and \ref{no pearls}. \blue{Part (ii) is Proposition~\ref{prop:not p^6 and CSZ2 essential} and, as in this case Hypothesis~\ref{Hyp7.1} holds, Lemma~\ref{Lem:no ess in extrasp}.
Part (iii) is a combination of Lemmas~\ref{p7} and \ref{Omega.non.ab}. Part (iv) follows from Lemma~\ref{Omega.non.ab}.}

Part (v) is precisely Proposition~\ref{propOp}. Using part (iv), (vi) follows from Corollary~\ref{c12}. Finally, for part (vii), if $S$ is exceptional, then Proposition~\ref{gamma1-essD} implies $\gamma_1(S)$ is extraspecial, which is against (iv).  Hence $S$ is not exceptional and Hypothesis~\ref{Hyp8.1} holds. Thus $O^{p'}(\Out_\F(\gamma_1(S))) \cong \PSL_2(p)$ by Lemma~\ref{gamma1-essC}.  Using (vi) gives $\Out_\F(\gamma_1(S))\cong \PSL_2(p)$ or $\PGL_2(p)$.
\end{proof}

\blue{
Suppose that $E$ is a witness. Then $E < \gamma_1(S)$  by Lemma~\ref{facts} (i).  If $N_\F(E) \subseteq   N_\F(\gamma_1(S))$, then $E$ is $ N_\F(\gamma_1(S))$-essential by Lemma~\ref{F-essential G-essential} and we obtain the contradiction  $E<\gamma_1(S)\le O_p(  N_\F(\gamma_1(S)))\le E$ from Lemma~\ref{OpinE}.  Hence $N_\F(E)\not \subseteq  N_\F(\gamma_1(S))$. Hence the set of subgroups $T$ of $\gamma_1(S)$ with $N_\F(T) \not \subseteq N_\F(\gamma_1(S))$ is non-empty. We now set up the notation which shall be used in the remainder of this section.

\begin{notation}\label{notationT} Set $\mathcal G=N_\F(\gamma_1(S))$. From among all non-trivial subgroups $T\le \gamma_1(S)$, select one which satisfies the following conditions in the specified order.
 \begin{enumerate}
\item  $N_\F(T) \not \subseteq \mathcal G $;
\item $|N_S(T)|$ is maximal; and
\item $|T| $ is maximal.
\end{enumerate}
 \end{notation}
Observe  that $N_\F(S) \subseteq \mathcal G$ with equality if and only if $\gamma_1(S)$ is not $\F$-essential.}

\blue{The   discussion before Notation~\ref{notationT} shows that  a witness satisfies the first condition and so  we can conclude that a subgroup $T$ as specified in Notation~\ref{notationT} exists.}

\begin{lemma}\label{lem:Tfn} The subgroup
$T$ is fully $\F$-normalized and $N_\F(T)$ is saturated.
\end{lemma}

\begin{proof}
By \cite[Lemma I.2.6 (c)]{AKO}, there exists $\alpha \in \Hom_\F(N_S(T), S)$ such that $T\alpha$ is fully $\F$-normalized. By the Alperin-Goldschmidt Theorem,  $\alpha$ is a product of maps from $\Aut_\F(A)$, $A \in \E_\F$ with $|N_S(T)| \le |A|$ and from $\Aut_\F(S)$. Then, as $A \le \gamma_1(S)$ \blue{by Lemmas~\ref{not normal} and \ref{facts} (ii)}, the  selection method of $T$ shows that each $\Aut_\F(A)$ is contained in $\mathcal G$ as is $\Aut_\F(S)$. Let $X, Y \le N_S(T)$. Then $\Hom_{N_\F(T\alpha)}(X\alpha ,Y\alpha) \supseteq \Hom_{N_\F(T)}(X,Y)\alpha^*$.  Since $\alpha \in \mathcal G$, $\Hom_{N_\F(T\alpha)}(X\alpha ,Y\alpha)\subset \mathcal G$ if and only if $\Hom_{N_\F(T)}(X,Y)\alpha^* \subset \mathcal G$.  We conclude that $N_\F(T\alpha) \not \subset \mathcal G$.  Since $|N_S(T\alpha)|\ge |N_S(T)|$,  the maximal choice of $N_S(T)$ implies that $T$ is fully $\F$-normalized.
\end{proof}

\begin{lemma}\label{NFTnormZ} Assume that $1 \ne K \le T$. Then following statements hold.
\begin{enumerate}
\item If $K$ is $\Aut_\F(T)$-invariant, then $N_S(K)= N_S(T)$.
\item If $T$ is not normal in $S$ and  $K$ is characteristic in $N_S(T)$, then $K$ is not $\Aut_\F(T)$-invariant.
\end{enumerate}
\end{lemma}
\begin{proof} Suppose
 that $K$ is $\Aut_\F(T)$-invariant. Then $K$ is invariant under   $\Aut_S(T)$ and so is normal in $N_S(T)$.  Lemma~\ref{NFTNFR} states  that $N_\F(T)\subseteq N_\F(K)$. Therefore, if $N_S(K) > N_S(T)$,  the maximal choice of $|N_S(T)|$ implies $N_\F(T)\subseteq   N_\F(K) \subseteq \mathcal G$ which  contradicts the choice of $T$. This proves (i).

Part (ii) follows from (i).
\end{proof}
\begin{lemma}\label{lem:not normal S}  If $K$ is a normal subgroup of $S$ which is contained in $T$, then  $K$ is not $\Aut_\F(T)$-invariant. In particular, $T$ is not normal in $S$.
\end{lemma}
\begin{proof}
Suppose that $K\le T$  is $\Aut_\F(T)$-invariant and normal in $S$. Then Lemma~\ref{NFTnormZ} (i) implies that $T$ is normal in $S$.    Hence $N_\F(T)$ is a fusion system on $S$.  Since $T \le O_p(N_\F(T))$ and $O_p(\F)=1$, $N_\F(T) \ne \F$. Application of Lemma~\ref{lem:subsystem1} yields $N_\F(T) \subseteq \mathcal G$, a contradiction.
\end{proof}

\subsection{The case $T$ is $S$-centric}

\blue{
In this subsection we assume

\begin{hypothesis}\label{HypCST}  Hypothesis~\ref{hp.conj} holds and adopting Notation~\ref{notationT} we have $T$ is $S$-centric; that is $C_S(T)\le T$.
\end{hypothesis}}

As $T$ is fully $\F$-normalized by Lemma~\ref{lem:Tfn}, Hypothesis~\ref{HypCST} implies $T$ is $\F$-centric. Since $Z(S) \le T$ and $T$ is not normal in $S$ \blue{by Lemma~\ref{lem:not normal S}}, Lemma~\ref{mc-normalizer} implies $N_S(T) \le \gamma_1(S)$.
By Theorem~\ref{model} there exists a model $G$  for $N_\F(T)$. Choose $G_1$ such that $G \ge G_1> N_S(T)$ and $G_1$ has minimal order such that $\F_{N_S(T)}(G_1) \not \subseteq \mathcal G$.

\begin{lemma}\label{opG1=T} Assume that Hypothesis~\ref{HypCST} holds. We have $T= O_p(G_1)=O_p(G) $ and $C_{G_1}(T) \le T$. \end{lemma}

\begin{proof} Notice that $N_S(O_p(G_1)) \ge N_S(T)$ and $O_p(G_1) \ge T$.  As  $\F_{N_S(T)}(G_1)=N_{\F_{N_S(T)}(G_1)}(O_p(G_1)) \subseteq N_\F(O_p(G_1))$, we get $N_\F(O_p(G_1)) \not\subseteq \mathcal G$. The maximal choice of $T$ now yields $T= O_p(G_1)$. We now have $T \le O_p(G) \le O_p(G_1)=T$ and so $T= O_p(G)=O_p(G_1)$. As $G$ is a model for $N_\F(T)$, $C_{G_1}(T) \le C_G(T) \le T$.
\end{proof}

\begin{lemma}\label{omegagood} Assume that Hypothesis~\ref{HypCST} holds. Then the following hold.
\begin{enumerate}
\item  $\Omega_1(\gamma_1(S))\not \le T$; and
\item $Z(\Omega_1(\gamma_1(S))) \le \Omega_1(T)$.\end{enumerate}
\end{lemma}

\begin{proof}  If $\Omega_1(\gamma_1(S))\le T$, then $\Omega_1(\gamma_1(S))= \Omega_1(T)$ is normalized by $S$ and is $\Aut_\F(T)$-invariant contrary to Lemma~\ref{lem:not normal S}. Hence $\Omega_1(\gamma_1(S))\not \le T$. This proves (i).

Let $R= N_{Z(\Omega_1(\gamma_1(S)))}(T)$.  Then, as   $\gamma_1(S)$ is regular by Lemma~\ref{mc-facts} (ii), using Lemma~\ref{lem:regular}(iii) yields  $$[T,R] \le T \cap \Omega_1(\gamma_1(S)) \le \Omega_1(T)$$ and $[\Omega_1(T),R]\le [\Omega_1(\gamma_1(S)),R]=1$.  Hence $R$ centralizes $T/\Omega_1(T)$ and $\Omega_1(T)$. Now Lemmas~\ref{lem:series} and \ref{omegagood} imply  $\Aut_R(T) \le O_p(\Aut_\F(T))=\Inn(T)$.  Hence $R\le T$. Applying Lemma~\ref{lem:K in E} yields $Z(\Omega_1(\gamma_1(S))) \le T$, as claimed.
\end{proof}

Before reading the next lemma the following example is worthy of some consideration.  Let $X= 2^6:3^.\Sym(6)$. Then $X$ is isomorphic to a $2$-local subgroup in $\mathrm M_{24}$.  Let $R \in\syl_2(X)$.  Then $\F_R(X)$ has three $\F_R(X)$-essential subgroups $E_0=R \cap X'$, $E_1$ and $E_2$ where $E_1/O_2(X)\cong E_2/O_2(X)$ are elementary abelian of order $8$. We have $\Out_{\F_R(X)}(E_i) \cong \Sym(3)$ for $i=0,1,2$. Here is the point: $X= \langle N_X(E_1),N_X(E_2)\rangle$, however $\F_R(X) \ne \langle \F_R(N_X(E_1)),\F_R(N_X(E_2))\rangle$ by \cite[Lemma 3.13]{parkersemerarocomputing}.

\begin{lemma}\label{G1-minpara} Assume that Hypothesis~\ref{HypCST} holds. Then $N_S(T)$ is contained in a unique maximal subgroup of $G_1$.
\end{lemma}

\begin{proof}
Suppose that $M_1 $ and $M_2$ are maximal subgroups of $G_1$ with $N_S(T) \le M_1 \cap M_2 $ and $M_1 \ne M_2$.
By the minimal choice of $G_1$,
$$ \langle \F_{N_S(T)}(M_1), \F_{N_S(T)}(M_2)\rangle \subseteq \mathcal G.$$
In particular,  \begin{eqnarray*}\Aut_{\mathcal G}(T)&\supseteq& \langle \Aut_{\F_{N_S(T)}(M_1)}(T), \Aut_{\F_{N_S(T)}(M_2)}(T) \rangle \\&=&\langle \Aut_{M_1}(T), \Aut_{M_2}(T)\rangle= \Aut_{G_1}(T)\end{eqnarray*} as $G_1=\langle M_1, M_2\rangle.$ Since  $T$ is $\F$-centric, $Z(\gamma_1(S)) \le T$ and $\Aut_{\mathcal G}(T)$ leaves $Z(\gamma_1(S))$ invariant. Since  $N_\F(Z(\gamma_1(S)))$ is a fusion system on $S$ and $O_p(\F)=1$,  $N_\F(Z(\gamma_1(S))) \subseteq \mathcal G$ by Lemma~\ref{lem:subsystem1}.
Therefore $$\Aut_{\F_{N_S(T)}(G_1)}(T)=\Aut_{G_1}(T)\subseteq \Aut_{\mathcal G}(T)\subseteq \mathcal G
 = N_\F(Z(\gamma_1(S))).$$    Since $\F_{N_S(T)}(G_1) = N_{\F_{N_S(T)}(G_1)}(T)$, application of Lemma~\ref{NFTNFR} gives  $\F_{N_S(T)}(G_1) \subseteq \mathcal G$, a contradiction. Hence $N_S(T)$ is contained in a unique maximal subgroup of $G_1$.
\end{proof}

Let $U$ be the unique maximal subgroup of $G_1$ which contains $N_S(T)$.

\begin{lemma}\label{Jgood} Assume that Hypothesis~\ref{HypCST} holds. If  $K$ is a non-trivial characteristic subgroup of $N_S(T)$, then $N_{G_1}(K) \le U$.
\end{lemma}

\begin{proof}
 We have $N_{G_1}(K) \ge N_S(T)$ and so, as $U$ is the unique maximal subgroup of $G_1$ which contains $N_S(T)$, if $U$ does not contain $N_{G_1}(K)$, then $K$ is normalized by $G_1$. Suppose that $K$ is normal in $G_1$. Then, by Lemma~\ref{opG1=T},  $K \le O_p(G_1)=T \le \gamma_1(S)$ and $\F_{N_S(T)}(G_1) \subseteq N_\F(K)$ which means that $N_\F(K) \not \subseteq \mathcal G$.  It follows from the maximal choice of $|N_S(T)|$ that $|N_S(N_S(T))| \le |N_S(K)| \le |N_S(T)|$, and so we must have that $T$ is normal in $S$.  This contradicts Lemma~\ref{lem:not normal S} and  demonstrates  $N_{G_1}(K) \le U$.
\end{proof}

Recall from \cite{CGT} that for a finite group $X$ and $R \in \syl_p(X)$,  $$C(X,R)=\langle N_X(K) \mid 1\ne K \text { a characteristic subgroup of } R\rangle.$$
In \cite{CGT}, they also define
$$C^*(X,R)=\langle N_X(K) \mid 1\ne K \text { a characteristic subgroup of } B(R) \text{ or }\Omega_1(Z(R))\rangle,$$
where $B(R)$ is the Baumann subgroup (see \cite[Definition 1.1 and just before]{CGT}). Since $B(R)$ and $\Omega_1(Z(R))$ are characteristic in $R$, we have $C^*(X,R) \le C(X,R)$.

\begin{lemma}\label{cgt} Assume that Hypothesis~\ref{HypCST} holds. Then there is a natural number $b \ge 1 $ such that $O^p(G_1)/O_p(O^p(G_1)) \cong \SL_2(p^b)$ and $O_p(O^p(G_1))$ is a natural $\GF(p)O^p(G_1)/O_p(O^p(G_1))$-module. Furthermore, $G_1/T \cong \SL_2(p^b)$.
\end{lemma}

\begin{proof} By Lemma~\ref{Jgood}, we have $C^*(G_1,N_S(T))\le C(G_1,N_S(T)) \le U$. Lemma \ref{opG1=T} states $C_{G_1}(T) \le T$ and $T=O_p(G_1)$. Since $U$ is the unique maximal subgroup of $G_1$ containing $N_S(T)$, we may apply \cite[Corollary 1.9]{CGT} to obtain $B(N_S(T))$-blocks $B_1, \dots, B_r$ such that the product $B_1 \dots B_r$ is normal in $G_1$, $[B_i, B_j]=1$ for $1 \le i < j\le r$ and  $G_1=(B_1 \dots B_r)C(G_1,N_S(T))$. Furthermore, \blue{as $U$ is the unique maximal subgroup  of $G_1$ containing $N_S(T)$, we have $G_1=(B_1 \dots B_r)N_S(T)$ and} $N_S(T)$ permutes $\{B_1, \dots, B_r\}$ transitively by conjugation. In particular, $r$ is a power of $p$.
The definition of a $B(N_S(T))$-block and the fact that $p \ge 5$ yields $B_1=O^p(B_1)$, $B_1 /O_p(B_1) \cong \SL_2(p^b)$ for some $b \ge 1$, $O_p(B_1)= \Omega_1(Z(O_p(B_1)))$ and $O_p(B_1)$ is a natural $\GF(p)B_1/O_p(B_1)$-module. Assume $r \ge p$.  Then $|S| \ge |N_S(T)| \ge p^{3p}$ which contradicts Lemma~\ref{facts} (iii).  Hence $r< p$, and, as $r$ is a power of $p$, we must have $r=1$. It follows that $B_1$ is normal in $G_1$ and thus  $B_1=O^p(G_1)$.  Finally, we note that if $G_ 1/T =G_1/O_p(G_1) \not \cong \SL_2(p^b)$, then some $p$ element of $G_1$ must induce a non-trivial field automorphism on $B_1/O_p(B_1)$.  But then $p \ge b$ and $|N_S(T)\cap B_1|= p^{3p}$, which once again contradicts  Lemma~\ref{facts} (iii). Hence $G_1/O_p(G_1)  \cong \SL_2(p^b)$ and this proves the lemma.
\end{proof}

We define $$B_1= O^p(G_1)\text { and }W=O_p(B_1).$$

\begin{lemma}\label{It is L2p} Assume that Hypothesis~\ref{HypCST} holds. Then  $G_1/T \cong \SL_2(p)$ and $W=[T,B_1]$ has order $p^2$.
\end{lemma}

\begin{proof}
By Lemma~\ref{cgt}, there is a natural number $b$ such that $ B_1/O_p(B_1)\cong G_1/T \cong \SL_2(p^b)$ and $G_1=B_1T$. Select $g$ of order $p^b-1$ in $N_{G_1}(N_S(T))$ and set $\theta= c_g$. Then, by saturation, $\theta$ is the restriction of a morphism $\theta^*\in \Aut_\F(N_S(T))$.  The maximal choice of $T$ implies that $\theta^*$ is the restriction of a morphism $\hat \theta \in \Aut_{\mathcal G}(\gamma_1(S))$.
 In particular, $\hat \theta$ has order a multiple of $p^b-1$.  We shall show that $b=1$.
 Since $\Out_\F(S)$ has order dividing $p-1$ by Lemma~\ref{facts} (vi),  we may suppose that $\mathcal G \ne  N_\F(S)$. Hence $\gamma_1(S)$ is $\F$-essential and $\mathcal G= N_\F(\gamma_1(S))$. In this case   $\Out_\F(\gamma_1(S)) \cong \PSL_2(p)$ or $\PGL_2(p)$ by Lemma~\ref{facts} (vii), and, as the maximal order of a $p'$-element of $\PGL_2(p)$ is $p+1$,  we conclude that $b=1$. Therefore $G_1/T \cong \SL_2(p)$ and  $W=[T,B_1]$ has order $p^2$.
\end{proof}

\begin{lemma}\label{lem:tau} Assume that Hypothesis~\ref{HypCST} holds. Suppose that $\mathcal G= N_\F(S)$ and let $t \in G_1$ be an involution.  Then $c_t$ is the restriction of $\tau \in \Aut_{\F}(S)$, $W= [N_S(T), \tau]$, $Z(S)\le W$ and $Z(S)$ is inverted by $\tau$.
\end{lemma}

\begin{proof} Since $t$ is an involution, $tT \in Z(G_1/T)$ and so $t \in N_{G_1}(N_S(T))$. Let  $\tilde \tau $ be the element of $\Aut_\F(N_S(T))$ which restricts to $c_t$. The maximality of $T$ implies  $N_\F(N_S(T)) \subseteq \mathcal G$ and so $\tilde \tau$ is the restriction of  a morphism $\tau$ in $\Aut_\F(S)$. Since $t$ centralizes $N_S(T)/T$, $[T,t]\le [T,O^{p'}(G_1)] =W$ and $t$ inverts $W$, we have $W=[N_S(T),t]= [N_S(T),\tau]$.  If $Z(S) \not \le W$ then, as $Z(S)\le T$, $V=Z(S)W$ is normalized by $B_1N_S(T)=G_1$ and is contained in $Z(T)$.
As   $W$ is the unique non-central $G_1$-chief factor in $V$, we get $V= C_V(\tau) \times [V,\tau]$ and  $C_V(\tau) =C_V(G_1)$. Hence, as $Z(S)$ is $\langle \tau\rangle$-invariant $Z(S) \le C_V(\tau)= C_V(G_1)$.  But then $\F_{N_S(T)}(G_1) \subseteq N_\F(Z(S)) \subseteq \mathcal G$ by Lemma~\ref{lem:subsystem1}.  Since this is not the case, we conclude that $Z(S) \le W$ and $Z(S)$ is inverted by $\tau$.
\end{proof}

\begin{lemma}\label{lem:gamma1 is ess} \blue{Assume that Hypothesis~\ref{HypCST} holds. Then $\gamma_1(S)$ is $\F$-essential and $S$ is not exceptional. In particular, Hypothesis~\ref{Hyp8.1} holds.}
\end{lemma}

\begin{proof} Assume that $\gamma_1(S)$ is not $\F$-essential.  Then $\mathcal G= N_\F(S)$.   Let $t \in G_1$ and $\tau$ be as in Lemma~\ref{lem:tau}. Then $$Z(S)\le [N_S(T),\tau]= W.$$
By Lemma~\ref{mc-factsb} (iii) and (vi) the group $\gamma_1(S)/Z(S)$ is the unique $2$-step centralizer in $S/Z(S)$. Hence $$[T,Z_3(S)]\le [\gamma_1(S),Z_3(S)] \le Z(S).$$ Since $Z(S) \le T$, we deduce that $Z_3(S)\le N_S(T)$ and consequently we also have $[Z_3(S),\tau]\le Z_3(S) \cap W$.

By Lemma~\ref{action}, $\tau$ does not centralize $Z_3(S)/Z(S)$. Hence, as $\tau$ inverts $Z(S)$, we have  $[Z_3(S),\tau] \le [T, \tau] = W$.  As $|S| \ge p^7$ by Lemma~\ref{facts} (iii), $Z_3(S)$ is abelian and so $Z_3(S) \le C_{N_S(T)}(W)=T$. Furthermore,  $Z_5(S) \le \gamma_1(S)$ and it follows that $Z_5(S) \le N_S(T)$. Therefore $$[Z_5(S),\tau] \le[N_S(T),\tau] \le W \le Z_3(S)$$ and this contradicts Lemma~\ref{action}. Hence $\gamma_1(S)$ is $\F$-essential. Lemma~\ref{facts}(vii) yields $S$ is not exceptional.
\end{proof}

\begin{lemma}\label{g1 not ess1} Assume that Hypothesis~\ref{HypCST} holds. Then the following hold:
\begin{enumerate}
\item  $\Out_\F(\gamma_1(S)) \cong \PSL_2(p)$ or $\PGL_2(p)$;
\item $\gamma_1(S)= N_S(T)$; and
\item $\gamma_1(S)=\Omega_1(\gamma_1(S))$ has class $2$ and order $p^{p-1}$.
\end{enumerate}
\end{lemma}

\begin{proof} By Lemma \ref{lem:gamma1 is ess}, we have Hypothesis~\ref{Hyp8.1} is satisfied.  Part (i) is just Lemma~\ref{facts}(vii). The group $\Omega_1(\gamma_1(S))$ is not abelian by Lemma~\ref{facts}(iv).
From Lemma~\ref{gamma1-essC} we obtain   $\Omega_1(\gamma_1(S))$ has order $p^{p-1}$ and nilpotency class $2$ and $$[\Omega_1(\gamma_1(S)),\gamma_1(S)]\le Z(\Omega_1(\gamma_1(S))).$$
Hence, using Lemma~\ref{omegagood},
$$[\Omega_1(T),\gamma_1(S)]\le [\Omega_1(\gamma_1(S)),\gamma_1(S)] \le Z(\Omega_1(\gamma_1(S)))\le \Omega_1(T).$$
In particular, $\gamma_1(S)$ normalizes $\Omega_1(T)$. Therefore, as $N_\F(\Omega_1(T)) \supseteq N_\F(T)\not \subseteq \mathcal G$, the maximal choice of $N_S(T)$ implies that $N_S(T)\ge \gamma_1(S)$. As $T$ is not normal in $S$ by Lemma~\ref{lem:not normal S}, we now have $N_S(T)= \gamma_1(S)$ which is part (ii).

By Lemma~\ref{It is L2p}, $G_1/T \cong \SL_2(p)$ and so $T$ is a maximal subgroup of $N_S(T)$. As $\Omega_1(\gamma_1(S)) \not \le T$ by Lemma~\ref{omegagood}, we have $\gamma_1(S)= \Omega_1(\gamma_1(S))T$. Hence Lemmas~\ref{agemo} and \ref{mc-facts}(ii) imply $\agemO^1(T)= \agemO^1(\gamma_1(S))$. If  $\agemO^1(\gamma_1(S)) \ne 1$, then $N_\F(T) \le N_\F(\agemO^1(\gamma_1(S))) \subseteq \mathcal G$ and this contradicts the choice of $T$. Therefore  $\agemO^1(\gamma_1(S)) = 1$ and we conclude $\gamma_1(S)=\Omega_1(\gamma_1(S))$. Thus part (iii) holds.
\end{proof}

\begin{prop}\label{the centric case} If Hypothesis~\ref{hp.conj} holds, then $T$ is not $S$-centric.
\end{prop}

\begin{proof}\blue{ Suppose that $T$ is $S$-centric. Then Hypothesis~\ref{HypCST} holds.}
\blue{From Lemma~\ref{lem:gamma1 is ess} we know Hypothesis~\ref{Hyp8.1} holds.} Lemma~\ref{g1 not ess1} implies that $N_S(T)= \gamma_1(S)$ has exponent $p$, order $p^{p-1}$ and nilpotency class $2$. We also know that   $\Out_\F(\gamma_1(S)) \cong \PSL_2(p)$ or $\PGL_2(p)$. Let $L=\langle \Aut_S(\gamma_1(S))^{\Aut_{\F}(\gamma_1(S))}\rangle$ be the preimage of $O^{p'}(\Aut_\F(\gamma_1(S)))$.

As in Lemma~\ref{lem:gamma1 is ess}, take $\theta \in N_{G_1}(N_S(T))$ of order $p-1$ and let  $\rho$   the element of $\Aut_\F(N_S(T))=\Aut_\F(\gamma_1(S))$ which restricts to $\theta$.
 Set $H=\langle \rho\rangle$.  Then, as $H $ centralizes $T/W$, we have $|[N_S(T),H]|=|[\gamma_1(S),H]|=p^3$.
Since $H \cap L$ has index at most $2$ in $H$, we have $|(H\cap L)\Inn(\gamma_1(S))/\Inn(\gamma_1(S))|=(p-1)/2$ and subgroups of this order normalize a Sylow $p$-subgroup of $L$.   We calculate, using $|\gamma_1(S)|=p^{p-1}$, that  $$|C_{\gamma_1(S)}(\rho^2)|= |\gamma_1(S)/[\gamma_1(S),\rho^2]|=p^{p-4}. $$  However, Lemma~\ref{gamma1-essC} states that $|C_{\gamma_1(S)}(\rho^2)|= p^2$ and so we conclude that  $p<7$. This contradicts Lemma~\ref{facts} (iii) and proves that $T$ is not $\F$-centric.
\end{proof}
\subsection{The case $T$ is not $S$-centric}

In this subsection, we continue the notation so far established and, in addition, assume that
\blue{\begin{hypothesis}\label{HypCST2}  Hypothesis~\ref{hp.conj} holds and adopting Notation~\ref{notationT} we have $C_S(T)\not \le T$.
\end{hypothesis}}

\begin{lemma}\label{existessential}
 Assume that Hypothesis~\ref{HypCST2} holds. Then  $\E_{N_\F(T)}\ne \emptyset$. \end{lemma}

\begin{proof} Suppose that $\E_{N_\F(T)}=\emptyset$.  Then, because $N_\F(T)$ is saturated by Lemma~\ref{lem:Tfn}, $N_\F(T) \subseteq N_\F(N_S(T))$. Hence  $N_\F(N_S(T)) \not\subseteq \mathcal G$ and, as $T\ne S$, this contradicts the maximal choice of $T$. We conclude $\E_{N_\F(T)}\ne \emptyset$.
\end{proof}

Since $N_\F(T)$ is saturated and $\E_{N_\F(T)} \ne \emptyset$,  the Alperin-Goldschmidt fusion theorem implies there is an $N_\F(T)$-essential subgroup $P$ such that $\Aut_{N_\F(T)}(P) \not \subseteq \mathcal G$.

 \begin{notation}\label{not:P} The subgroup $P\le N_S(T)$ is an $N_\F(T)$-essential subgroup  of maximal order such that  $\Aut_{N_\F(T)}(P) \not \subseteq \mathcal G$.
 \end{notation}

\begin{lemma}\label{lem:lem1}
 If Hypothesis~\ref{HypCST2} holds, then
\begin{enumerate}
\item $T<P$;
\item $N_S(P) \le \gamma_1(S)$;
\item $|N_S(P)| < |N_S(T)|$;  and
\item $P < N_{N_S(T)}(P) < N_S(T)$.
\end{enumerate}\end{lemma}

\begin{proof} \blue{
By Lemma~\ref{OpinE}, $T \le O_p(\N_\F(T))\le P$. Since $T $ is not $S$-centric, $T$ is not $N_S(T)$-centric and so $T<P$. This is (i).}

 Suppose that $P \not \le \gamma_1(S)$.  Then Lemma~\ref{facts} (i) implies that $P$ is not contained in any $\F$-essential subgroups.  Therefore the elements of $\Aut_\F(P)$ are all restrictions of elements in $\Aut_\F(S)$ and  this means that $O_p(\Out_\F(P))= \Out_S(P)$. In particular, $\Aut_\F(P)$ has a unique Sylow $p$-subgroup.  As $\Aut_{N_\F(T)}(P)$ is a subgroup of $\Aut_\F(P)$, we have $\Aut_{N_\F(T)}(P)$ has a unique Sylow $p$-subgroup and this contradicts $P$ being $N_\F(T)$-essential.  Hence $P \le \gamma_1(S)$. Since $Z(S) \le N_S(T)$, we have $Z(S) \le C_{N_S(T)}(P) \le P$ and so, either, $P$ is normal in $S$ or $N_S(P) \le \gamma_1(S)$ by Lemma~\ref{mc-normalizer}.   As $T$ is not normal in $S$ by Lemma \ref{lem:not normal S}, the maximal choice of $T$ implies $N_S(P)\ne S$. Hence $N_S(P) \le \gamma_1(S)$. This proves (ii). If $|N_S(P)| \ge |N_S(T)|$ then as $P > T$, we have a contradiction to the maximal choice of $T$. Thus $|N_S(P)| < |N_S(T)|$ so (iii) holds. Part (iii) yields  $P < N_{N_S(T)}(P) < N_S(T)$ which is (iv).
\end{proof}

\blue{Recall from Section~\ref{sec:PF}, that  the subgroup $H_{N_\F(T)}(P)$   of $\Aut_{N_\F(T)}(P)$   is generated by those $N_\F(T)$-automorphisms of $P$ which extend to $N_\F(T)$-isomorphisms between strictly larger subgroups of $N_S(T)$.}

\begin{lemma}\label{lem:P fact}
\blue{ Assume that Hypothesis~\ref{HypCST2} holds.
 Let $A, B \le N_{N_S(T)}(P)$ with $P\le A \cap B$. Then
   $\Hom_{N_\F(T)}(A,B) \subseteq \mathcal G$ and $$H_{N_\F(T)}(P)= N_{\Aut_{N_\F(T)}(P)}(\Aut_{N_S(T)}(P))\subseteq \mathcal G.$$}
\end{lemma}

\begin{proof} From the maximal choice of $P$, we know that any $N_\F(T)$-essential subgroup $P_1$ containing $A$ has $\Aut_{N_\F(T)}(P_1) \subset \mathcal G$. It follows   that $\Hom_{N_\F(T)}(A,B)\subseteq  \mathcal G$.  In particular, the elements of  $\Hom_{N_\F(T)}(A,B)$  are restrictions of morphisms in $N_\F(\gamma_1(S))$.  If $\alpha \in \Hom_{N_\F(T)}(A,B)$ and $P=P\alpha$, then taking any $\hat \alpha \in \Aut_{\F}(\gamma_1(S))$ which restricts to $\alpha$, we have $P\hat\alpha= P$, $T= T\hat \alpha$ and, because of Lemma~\ref{lem:lem1},  $$N_{N_S(T)}(P)\hat \alpha= N_{N_{\gamma_1(S)}(T)}(P)\hat \alpha=  N_{N_{\gamma_1(S)}(T)}(P)=N_{N_S(T)}(P).$$ Hence  $H_{N_\F(T)}(P)= N_{\Aut_{N_\F(T)}(P)}(\Aut_{N_S(T)}(P))$, as stated.
\end{proof}

\begin{lemma}\label{lem:CST not in P}  Assume that Hypothesis~\ref{HypCST2} holds. Then $C_S(T) \not \le P$. In particular, $\Out_{C_S(T)}(P) \ne 1$.
\end{lemma}

\begin{proof} If $C_S(T) \le P$, then $C_S(T)T= C_P(T)T$ is normalized by $N_S(T)$ and is invariant under the
 action of $\Aut_{N_\F(T)}(P) \not \subseteq \mathcal G$. Thus the maximal choice of $T$ yields $C_S(T)T = T$. Hence $C_S(T) \leq T$, a contradiction. Therefore $C_S(T) \not \le P$. Since $PC_S(T)$ is a subgroup of $S$, we have $N_{C_S(T)}(P) \not \le P$  by Lemma~\ref{lem:K in E} and so $\Out_{C_S(T)}(P) \ne 1$.
\end{proof}

\begin{lemma}\label{Z(S) not in T} Assume that Hypothesis~\ref{HypCST2} holds. Then $\N_\F(Z(S)) \subseteq \mathcal G$ and  $Z(S) \not \le T$. In particular, $\Omega_1(P) \not \le T$.
\end{lemma}

\begin{proof} By Lemma~\ref{lem:subsystem1}, $\N_\F(Z(S)) \subseteq \mathcal G$.

As $T$ is $\Aut_{N_{\F}(T)}(P)$-invariant, the group $K=\langle \Aut_{C_S(T)}(P)^{\Aut_{N_{\F}(T)}(P)}\rangle$ centralizes $T$.
Aiming for a contradiction, suppose $Z(S)\le T$. Then $K$ centralizes $Z(S)$ and so $$K \Inn(P)\subseteq \N_\F(Z(S)) \subseteq \mathcal G.$$
Since $C_S(T) \not \le P$ by Lemma~\ref{lem:CST not in P}, $ K \not \le \Inn(P)$ and $\Aut_{N_{\F}(T)}(P)= KH_{N_\F(T)}(P)$ by the Frattini Argument.  Hence we get $\Aut_{N_{\F}(T)}(P) \subseteq \mathcal G$ from  Lemma~\ref{lem:P fact}, a contradiction. Thus $Z(S) \not \le T$. Since $Z(S) \le C_{N_S(T)}(P) \le P$, we have $Z(S) \le \Omega_1(P)$ and so $\Omega_1(P) \not \le T$.
\end{proof}

\begin{lemma}\label{G=NFS}  Assume that Hypothesis~\ref{HypCST2} holds. Then $\mathcal G= N_\F(S)$.
\end{lemma}

\begin{proof} Suppose that $\mathcal G\ne  N_\F(S)$. Then $\gamma_1(S)$ is $\F$-essential and  Lemma~\ref{facts} (vii) says that $S$ is not exceptional. Also Lemmas~\ref{mc-normalizer} and \ref{NFTnormZ} imply $N_S(T) \le \gamma_1(S)$.  As $\Omega_1(\gamma_1(S))$ is not abelian by Lemma~\ref{Omega.non.ab}, Lemma~\ref{gamma1-essC} implies that $\Omega_1(\gamma_1(S))$ has nilpotency class $2$ and order $p^{p-1}$ with $Z(\gamma_1(S))= Z(\Omega_1(\gamma_1(S)))=[\Omega_1(\gamma_1(S)),\gamma_1(S)]$. In particular, $Z(\Omega_1(\gamma_1(S))) \le N_S(T)$ and  $Z(\Omega_1(\gamma_1(S))) \le C_{N_S(T)}(P) \le P$. Thus $$[\Omega_1(P), \gamma_1(S)]\le [\Omega_1(\gamma_1(S)), \gamma_1(S)] = Z(\Omega_1(\gamma_1(S)))\le \Omega_1(P)$$ and so $\Omega_1(P)$ is normal in $\gamma_1(S)$.  In particular, as $N_S(T)\le \gamma_1(S)$, $N_S(T) \le N_S(T\Omega_1(P))$. By Lemma~\ref{NFTNFR} we have  $N_{N_\F(T)}(P) \subseteq  N_{N_\F(T)}(T\Omega_1(P))$ and so, as $N_{N_\F(T)}(P) \not \subseteq \mathcal G$, $N_\F(T\Omega_1(P)) \not \subseteq \mathcal G$.  The maximal choice of $T$ now implies that $T = T\Omega_1(P)$. Therefore $  \Omega_1(P) \le T$, and this contradicts Lemma~\ref{Z(S) not in T}.
\end{proof}

\begin{notation}\label{not:L} Define $$L= O^{p'}(\Aut_{N_\F(T)}(P))=\langle \Aut_{N_S(T)}(P)^{\Aut_{N_\F(T)}(P)}\rangle.$$\end{notation}
Then the Frattini Argument yields
$$\Aut_{N_\F(T)}(P)=  LN_{\Aut_{N_\F(T)}(P)}(\Aut_{N_S(T)}(P)).$$
 \begin{lemma}\label{ZOpL}  Assume that Hypothesis~\ref{HypCST2} holds. Then  $Z(S)$ is not    $L$-invariant.
 \end{lemma}

\begin{proof} Suppose that $Z(S)$ is  normalized by  $L$.  Then, as $\Aut_{N_\F(T)}(P)=  LN_{\Aut_{N_\F(T)}(P)}(\Aut_{N_S(T)}(P))$   and $ N_{\Aut_{N_\F(T)}(P)}(\Aut_{N_S(T)}(P))\subseteq \mathcal G= N_\F(S)$ by Lemmas~\ref{lem:P fact} and \ref{G=NFS}, we deduce from Lemma~\ref{Z(S) not in T} that $\Aut_{N_\F(T)}(P) \subseteq N_\F(Z(S)) \subseteq \mathcal G$, which is a contradiction.
\end{proof}

 The next three lemmas limit the structure of $L/\Inn(P)$.

\begin{lemma}\label{Op' gone}  Assume that Hypothesis~\ref{HypCST2} holds. Then $O_{p'}(\Out_{\N_\F(T)}(P))$ is centralized by $\Out_{N_S(T)}(P)$. In particular, $L/\Inn(P)$ centralizes $O_{p'}(\Out_{\N_\F(T)}(P))$.
\end{lemma}

\begin{proof}
Suppose false. Then there exists an $x\in \Aut_{N_S(T)}(P)$ such that $x\Inn(P)$ does not centralize $O_{p'}(\Out_{\N_\F(T)}(P))$. Let $K \le \Aut_{\N_\F(T)}(P)$ be the preimage of $[O_{p'}(\Out_{\N_\F(T)}(P)),x\Inn(P)]$.  Then $K\langle x \rangle$ is $p$-soluble and, since $p$ is odd, there exists a  non-central $K\langle x \rangle$-chief factor $V$   in $\Omega_1(P)$ by \cite[Theorem 5.3.10]{Gor}.  By Proposition~\ref{Hall-Hig},  $|V| \ge p^{p-1}$. Since $\Omega_1(P) < \Omega_1(\gamma_1(S))$ and $1<\Omega_1(T) \le \Omega_1(P)$ is $L$-invariant, Lemma~\ref{mc-facts} (iv) implies $\Omega_1(P)= \Omega_1(T)\le T$ and this contradicts Lemma~\ref{Z(S) not in T}.
\end{proof}

\begin{lemma}\label{Out cyc}  Assume that Hypothesis~\ref{HypCST2} holds. Then $\Out_{N_S(T)}(P)$ is cyclic.
\end{lemma}

\begin{proof}
If $\Out_{N_S(T)}(P)$ is not cyclic, then, as $p$ is odd, $\Out_{N_S(T)}(P)$ has an elementary abelian subgroup of order $p^2$.  Using Lemma~\ref{strongly p structure}, we obtain $K=\Out_{N_\F(T)}(P)/O_{p'}(\Out_{N_\F(T)}(P))$ is an almost simple group. Because $p \ge 7$ by Lemma~\ref{facts}  (iii), using Proposition~\ref{SE-p2} we obtain that  $K \cong \PSL_2(p^a)$ with $a \ge 2$, $\PSU_3(p^a)$ with $a \ge 1$, $\Alt(2p)$ or $p=11$ and $L\cong \J_4$. Since $H_{N_\F(T)}(P) \subset \mathcal G=N_\F(S)$ by Lemmas~\ref{lem:P fact} and \ref{G=NFS},  we get $N_{L}(\Aut_{N_S(T)}(P))$ is cyclic of order dividing $p-1$ from Lemma~\ref{facts} (vi).
Using \cite[Theorem 7.6.2]{GLS3}, we find  this is not compatible with any of the candidates for $O^p(L/\Inn(P))/Z(L/\Inn(P))$. Thus $\Out_{N_S(T)}(P)$  is cyclic.
\end{proof}

\begin{lemma}\label{L2ptype}  Assume that Hypothesis~\ref{HypCST2} holds. Then  $L/\Inn(P)\cong \PSL_2(p)$ or $\SL_2(p)$.
\end{lemma}

\begin{proof} Since $L/\Inn(P)$ centralizes $O^{p'}(\Out_{N_\F(T)}(P))$ by Lemma~\ref{Op' gone} and $O_p(L)=\Inn(P)$, $L/\Inn(P)$ centralizes the Fitting subgroup of $\Out_{N_\F(T)}(P)$ and so $E(\Out_{N_\F(T)}(P)) \ne 1$ and, by  Lemma~\ref{Out cyc}, there is a unique component $K/\Inn(P)$ of  $\Out_{N_\F(T)}(P)$ which has order divisible by $p$.  Since $\Aut_{N_S(T)}(P) \le L$,  $K$ is a normal subgroup of $L$ and $\Out_{N_\F(T)}(P)= H_\F(P) K$. Since $H_\F(P)$ leaves  $Z(S)$ invariant, and $L \le  H_\F(P) K $, $K$ does not leave $Z(S)$ invariant by Lemma~\ref{ZOpL}.

Assume that $K/\Inn(P)$ does not have a quotient isomorphic to $\PSL_2(p)$. Then $K$ is not of $\mathrm L_2(p)$-type and $K/\Inn(P)$ is quasisimple. Hence, as $Z(S)$ is not $K$-invariant,  setting $U=\langle Z(S)^{K}\rangle \le \Omega_1(Z(P))$, $U$ has a non-central $K$-chief factor. Since $U$ is $\Aut_{N_S(T)}(P)K$-invariant,  Theorem~\ref{feit} yields $|U| \ge p^{2(p-1)/3}$, $|\Out_{N_S(T)}(P)|=p$ and $U$ is indecomposable as a $\GF(p)\Out_{N_S(T)}(P)$-module. In  particular, $[U,\Out_{N_S(T)}(P);\lceil 2(p-1)/3\rceil-1] \ne 1$.  This with Lemma~\ref{nilp.class.gamma1} gives $\lceil 2(p-1)/3\rceil-1 < (p+1)/2$ and yields $p=7$. But then, as $p=7$, Lemma~\ref{nilp.class.gamma1} additionally
 tells us that $S$ is exceptional. As $|S| \ge 7^7$ by Lemma~\ref{facts} (iii), Lemma~\ref{mc-factsb} (v) implies $|S|=7^8$. Since $T \le P$ and $U \not \le T$, we also have $|P|\ge 7^5$.
 By Lemma~\ref{lem:lem1}, $P< N_{N_S(T)}(P) < N_S(T)$. Hence $N_S(T)$ is a maximal subgroup of $S$.  Since $S$ is exceptional, either $Z(N_S(T))= Z(S)$ or $N_S(T)= C_S(Z_2(S))$. In the former case, as $T$ is normal in $N_S(T)$, we have $Z(S) \le T$ which contradicts Lemma~\ref{Z(S) not in T}. Hence $N_S(T)= C_S(Z_2(S))$ and consequently $Z_2(S) \le P$. Since $P \le \gamma_1(S)$ by Lemma~\ref{lem:lem1}, $ P \le \gamma_1(S)  \cap C_S(Z_2(S))=\gamma_2(S)$.  Now $[\gamma_{2}(S),\gamma_2(S)]\le \gamma_5(S)= Z_2(S)\le P$ and so $\gamma_2(S)= N_{N_S(T)}(P)$ and $\gamma_2(S)$ acts quadratically on $P$ contrary to $[U,\Out_{N_S(T)}(P);\lceil 2(p-1)/3\rceil-1] \ne 1$.  We conclude that $K/\Inn(P)$ is of $\mathrm L_2(p)$-type and it follows that $K/\Inn(P) \cong \PSL_2(p)$ or $\SL_2(p)$. Since $O_p(L)=\Inn(P)$ and $\Out_{N_S(T)}(P)$ is cyclic, we further deduce that $L= K$.
\end{proof}

By Lemma~\ref{L2ptype}, $\Aut_{N_S(T)}(P)$ has order $p$.  Hence Lemma~\ref{lem:CST not in P} implies $$\Aut_{C_S(T)}(P) \Inn(P) = \Aut_{N_S(T)}(P).$$

We now establish some notation which will play an important role in the remaining lemmas of this section. Let $ \langle \theta\rangle$ be a complement in $L$  to $\Aut_{N_S(T)}(P)=\Aut_{C_S(T)}(P)$  chosen so that $  \langle \theta\rangle \le  \langle \Aut_{C_S(T)}(P)^{\Aut_{N_\F(T)}(P)}\rangle \le C_\F(T) $. We know $\theta$ has order $(p-1)/2$ when  $L/\Inn(P) \cong \PSL_2(p)$  and order $p-1$  when $L/\Inn(P)\cong\SL_2(p)$. Since $\theta\in  \langle \Aut_{C_S(T)}(P)^{\Aut_{N_\F(T)}(P)}\rangle$, we also know $\theta$ centralizes $T$.

Because $\theta$ normalizes $\Aut_{N_S(T)}(P)$,  $\theta$ is the restriction of a morphism $$\tilde \theta\in \Aut_{N_\F(T)}(N_{N_S(T)}(P))\subseteq H_{N_\F(T)}(P) \subseteq \mathcal G=N_\F(S)$$ by Lemmas~\ref{lem:P fact} and \ref{G=NFS}. Hence $\theta$ is in fact  the restriction of an element $\hat \theta$ of $\Aut_\F(S)$ and we may assume  that $\hat \theta$  has $p'$-order.

If $\theta$ has order $p-1$, then Lemma~\ref{facts} (vi) shows that $\hat \theta$ has order $p-1$ and so  $\hat \tau= \hat \theta^{(p-1)/2}$ is an involution which restricts to $\tau= \theta^{(p-1)/2}$. Recall  that the  action for $\hat \theta$ on $\gamma_1(S)$ is the subject of Lemmas~\ref{action} and \ref{centralizer auto}.

\begin{lemma}\label{Z(S) and theta} Assume that Hypothesis~\ref{HypCST2} holds. Then $Z(S) \not \le C_{\Omega_1(P)}(\theta)$.
\end{lemma}

\begin{proof} Suppose that $Z(S)$ is centralized by $\theta$ and define  $V=\langle Z(S)^{L}\rangle \Omega_1(T)$. Since $Z(S) \not \le C_V(O^{p}(L))$ by Lemma~\ref{ZOpL}, we can select $U$ maximal such that $$1\ne \Omega_1(T) \le C_V(O^{p}(L)) \le U<V$$ is $L$-invariant. Then $V/U$ is irreducible as a $\GF(p)L$-module and, by the definition of $V$,  $Z(S) \not \le U$.  Set $\ov V = V/C_{V}(O^p(L))$. Then  $|\ov V| < p^p$ by Lemma~\ref{mc-facts} (iv). Combining Lemma~\ref{L2ptype} and Corollary~\ref{L2p} yields  $V= Z(S)U$ and $\ov V = \ov{Z(S)}\ov U > \ov U$. Lemma~\ref{lem:mss} implies $\ov V = C_{\ov V}(L) \ov U$. By coprime action,  $C_{\ov V}(L)\le C_{\ov V}(O^p(L))=1$. Hence $\ov V = \ov U$, a contradiction. Therefore  $Z(S)$ is not centralized by $\theta$.
\end{proof}

\begin{lemma}\label{Z2 in NST}  \blue{Assume that Hypothesis~\ref{HypCST2} holds. If $\Omega_1(P)/\Omega_1(T)$ is abelian, then}  $Z_2(S)\le C_{N_S(T)}( \Omega_1(N_S(T)))$. In particular, $Z_2(S) \le N_S(T)$.
\end{lemma}

\begin{proof} If $S$ is not exceptional, then, as $T$ is not normal in $S$ by Lemma~\ref{lem:not normal S}, Lemma~\ref{mc-normalizer} implies  $N_S(T) \le \gamma_1(S)$. The claim then follows as $$Z_2(S) \le Z(\gamma_1(S)) \le Z(N_{\gamma_1(S)}(T)) =  Z(N_S(T)).$$

Suppose that $S$ is exceptional.  Then $Z(\gamma_1(S))= Z(S)$.  If $\Omega_1(N_S(T))\le C_S(Z_2(S))$, then   $$Z_2(S) \le C_S(\Omega_1(T)) \le N_S(\Omega_1(T))=N_S(T)$$ by Lemma~\ref{NFTnormZ}. Thus $Z_2(S)\le C_{N_S(T)}( \Omega_1(N_S(T)))$ in this case.

So assume that $\Omega_1(N_S(T))\not \le C_S(Z_2(S))$. Then $N_S(T) \not \le C_S(Z_2(S))$ and, as $T$ is not normal in $S$ by Lemma~\ref{lem:not normal S}, Lemma~\ref{mc-normalizer} implies $N_S(T) \le \gamma_1(S)$. Since  $\Omega_1(N_S(T))\not \le C_S(Z_2(S))$,  $\Omega_1(\gamma_1(S)) \not \le \gamma_2(S)$ and, as $S$ has maximal class, we deduce  $\gamma_1(S) =\Omega_1(\gamma_1(S))$ which has exponent $p$ by Lemmas~\ref{mc-facts} (ii) and \ref{lem:regular} (ii).  In particular, $P= \Omega_1(P)$ and $T=\Omega_1(T)$.
By Lemmas~\ref{lem:lem1} and \ref{L2ptype}, $N_{N_S(T)}(P) < N_S(T)$ and $N_{N_S(T)}(P)/P$ has order $p$. Therefore   we can pick $x \in N_{N_S(T)}(N_{N_S(T)}(P))\setminus N_{N_S(T)}(P)$ such that $N_{N_S(T)}(P)=PP^x$ and $P \cap P^x$ has index $p$ in $P$.  Using Lemma~\ref{L2ptype} we can find $\ell \in L$ such that $ L = \langle \Inn(P),\Aut_{P^x}(P), \Aut_{P^x}(P)^\ell\rangle$. Then, as $\Omega_1(P)/\Omega_1(T)=P/T$ is abelian, $O^p(L)$ centralizes $(P \cap P^x\cap P^{x\ell})/T$  and centralizes $T$ and so coprime action implies  $|P/C_P(O^{p}(L))|\le p^2$.  It follows that  $L/\Inn(P) \cong \SL_2(p)$,  and $\theta$ has order $p-1$. By Lemmas~\ref{mc-factsb}(vi) and ~\ref{centralizer auto}, $|C_{\gamma_1(S)/Z(S)}(\hat \theta)|\le p$. As $Z(S)$ is not $L$-invariant by Lemma \ref{ZOpL}, it follows that $|C_P(O^p(L))|\le p$ and so $T=C_P(O^p(L))$ has order $p$ and $|P|=p^3$. In particular, we have $|N_{N_S(T)}(P)|=p^4$ and $P/T$ is an $(N_\F(T)/T)$-pearl.

 Since $ p^8\le|S|\le p^{p+1}$ by Lemmas~\ref{mc-factsb}(v)  and  \ref{facts} (iii) and $\gamma_1(S)$ has exponent $p$, we have $Z_4(S)$ is elementary abelian. We establish some notation for the action of $\hat \theta$ which we know has order $p-1$.  So let $x\in C_S(Z_2(S))\setminus \gamma_1(S)$ and $s_1 \in \gamma_1(S)\setminus \gamma_2(S)$.  Then we may suppose that \begin{eqnarray*}x\hat\theta &\equiv& x^a \mod \gamma_2(S)\\
 s_1 \hat \theta &\equiv &s_1^b \mod \gamma_2(S)\end{eqnarray*} where $a, b \in \GF(p)^\times$. We also know that $\hat \theta $ centralizes $T$. By Lemma~\ref{action}, $$s_{n-1}\hat \theta = s_{n-1}^c$$ where $c={a^{n-3}b^2}$ and, as $P/T$ is a natural $\GF(p)\SL_2(p)$-module for $y\in P/Z(S)T$ we have $$y\hat \theta \equiv y^{c^{-1}} \mod Z(S)T.$$

 Set $V=Z_4(S)/Z(S)$.  Then, by Lemma~\ref{action}, we may select elements  $v_{n-4}, v_{n-3}$ and $v_{n-2} $ of $V$ such that $v_j\hat\theta  = v_j^{a^{j-1}b}$. It particular, the action of $\hat \theta$ on each of these elements is different and they correspond to eigenvectors of $\hat \theta|_V$ on $V$ with different eigenvalues.  Pick $t \in T^\#$, then
 $\hat \theta|_V$ and $c_t|_V$ commute, and so $T$ normalizes the eigenspaces of $\theta$ on $V$.  Since the $\theta$-eigenspaces on  $V$     have order $p$, we conclude that $V= C_V(T)$. This shows that $[Z_4(S),T]\le Z(S)$ and so  $|C_{Z_4(S)}(T)|\ge p^3$.

 Suppose that  $[P, Z_4(S)]  \le Z(S)$.  Then $C_{Z_4(S)}(T)\le N_{N_S(T)}(P)$ as $Z(S)\le P$. Since $|N_{N_S(T)}(P)|=p^4$ and $ N_{N_S(T)}(P)$ is non-abelian, $TC_{Z_4(S)}(T)$ has order $p^3$.  Hence $T \le Z_4(S)$. But then $Z_4(S)= N_{N_S(T)}(P)$, a contradiction as $Z_4(S)$ is abelian.   We have proved that $$[P, Z_4(S)] \not \le Z(S).$$
  In particular, as $S/Z(S)$ has positive degree of commutativity, $P \not \le \gamma_2(S)$.

 Suppose that $T \not \le \gamma_2(S)$.  Then  $\langle T^S\rangle= \gamma_1(S)$ and so, as $[Z_4(S),T]\le Z(S)$, we have \begin{eqnarray*}
 [Z_4(S),P]&\le& [Z_4(S), \gamma_1(S)]=  [Z_4(S),\langle T^S\rangle]\\&=&\langle [Z_4(S),T]^{S}\rangle\le Z(S),\end{eqnarray*}
which is a contradiction.  Hence $T \le \gamma_2(S)$.  Since $Z(S) \le \gamma_2(S)$ and  $P \not \le \gamma_2(S)$, for $y \in P \setminus Z(S)T$, we have $y \not \in \gamma_2(S)$.  Hence $$y\hat \theta \equiv y^b \mod \gamma_2(S)$$ and $$y \hat \theta\equiv  y^{c^{-1}} \mod \gamma_2(S).$$ Therefore $c=b^{-1}$.
Suppose that $T \not \le \gamma_3(S)$, then as $\hat \theta$ centralizes $T$, we have $ab =1$ by Lemma~\ref{action}. Now applying Lemma~\ref{action} again yields $$s_3\hat \theta\equiv  s_3^{b^{-1}} \mod \gamma_3(S)= s_3^c\mod \gamma_3(S)$$ and also $$s_{n-1}\hat \theta = s_{n-1}^{b^{-1}}=s_{n-1}^c.$$ Lemma~\ref{centralizer auto} yields $p-1$ divides $n-1-3$. Hence $n \ge p+3$ and this contradicts $n \le p+1$.
 Therefore $T \le \gamma_3(S)$.

   Since $[Z_3(S),\gamma_3(S)]=1$, we now have $Z_3(S) \le N_{N_S(T)}(P)$.
As $P \not \le \gamma_2(S)$, $P$ does not centralize $Z_3(S)$.  Hence $PZ_3(S)= N_{N_S(T)}(P)$ and $Z_3(S)\cap P = Z(N_{N_S(T)}(P))= Z(S)T.$  In particular, $T \le Z_3(S)$ and $Z_5(S) \le \gamma_3(S) \le N_S(T)$. Now $|Z_5(S)/Z_3(S)|=p^2$ and is centralized by $N_S(T)\le \gamma_1(S)$.  Therefore $N_S(T)/T$ does not have maximal class.  However, $P/T$ is an $(N_\F(T)/T)$-pearl and consequently Lemma~\ref{pearls1} (i) implies that $N_S(T)/T$ does have maximal class. We have derived a contradiction and this proves the lemma.
\end{proof}

\begin{lemma}\label{it's SL2p again}  Assume that Hypothesis~\ref{HypCST2} holds. Then we have $L/\Inn(P)\not \cong \PSL_2(p)$.
\end{lemma}

\begin{proof}
\blue{Suppose that $L/\Inn(P) \cong \PSL_2(p)$.}
We start by considering the case $|C_{\Omega_1(P)}(\theta)| \le p^2$.  In this case, as $\theta$ centralizes $\Omega_1(T) \le \Omega_1(P)$, $|C_{\Omega_1(P)/\Omega_1(T)}(\theta) | \le p$.  Since $Z(S)\le \Omega_1(P)$ and $O^p(L)$ does not centralize $Z(S)$ by Lemma~\ref{ZOpL},  $\Omega_1(P)/\Omega_1(T)$ contains at least one non-central $L$-chief factor  contained in $\langle Z(S)^{L}\rangle$. \blue{As  $|\Omega_1(P)/\Omega_1(T)| < p^p$, Lemma~\ref{L2p} implies} each $L$-chief factor in $\Omega_1(P)/\Omega_1(T)$ contributes $p$ to $|C_{\Omega_1(P)/\Omega_1(T)}(\theta )|$ and  therefore $\Omega_1(P)/\Omega_1(T)$ is a non-central $L$-chief factor. In particular, Lemma~\ref{L2p}  implies $\Omega_1(P)/\Omega_1(T)$ has order $p^a$ with $a$ odd in the range $3\le a \le p-2$,  $\Out_{N_{N_S(T)}(P)}(P)$ acts indecomposably on $\Omega_1(P)/\Omega_1(T)$ and, as $|C_{\Omega_1(P)}(\theta)| \le p^2$, $|\Omega_1(T)|=p$.

We will frequently use  $$ [\Omega_1(P)/\Omega_1(T),N_{N_S(T)}(P),N_{N_S(T)}(P)]\ne 1$$
which is a consequence of $a \ge 3$ and $\Out_{N_{N_S(T)}(P)}(P)$ acts indecomposably on $\Omega_1(P)/\Omega_1(T)$.
We also remark that $$ \Omega_1(T) \langle Z(S)^{L}\rangle\le Z(\Omega_1(P)) \le \Omega_1(P)$$ and, as  $ \langle Z(S)^{L}\rangle$ has a non-central $L$-chief factor, we obtain $\Omega_1(P) = Z(\Omega_1(P))$ is abelian.

By Lemma~\ref{Z2 in NST},  $Z_2(S)\le C_{N_S(T)}( \Omega_1(N_S(T)))\le N_S(T)$.
Since $[P, Z_2(S)]\le Z(S) \le P$, $Z_2(S) \le N_{N_S(T)}(P)$.  If $Z_2(S)\not  \le P$, then, as $|N_{N_S(T)}(P)/P|=p$,  $  N_{N_S(T)}(P)= PZ_2(S)$ and we obtain $$[P, N_{N_S(T)}(P), N_{N_S(T)}(P)]= [P, Z_2(S),Z_2(S)] \le[Z(S),Z_2(S)]=1,$$
which is a contradiction.  Thus $Z_2(S) \le P$  and so $Z_2(S) \le \Omega_1(P)$.

 Since $Z_2(S) \le C_{N_S(T)}( \Omega_1(N_S(T)))$, $Z_2(S)$ is centralized by $N_{\Omega_1(N_S(T))}(P)$. Assume that $N_{\Omega_1(N_S(T))}(P)\le P$. Then, by Lemma~\ref{lem:K in E}, $\Omega_1(N_S(T))\le P$ and so $\Omega_1(P)=\Omega_1(N_S(T))$.  But then $\Omega_1(P)T=\Omega_1(N_S(T))T$, $N_\F(\Omega_1(P)T)\supseteq N_\F(P) \not \subseteq \mathcal G$ and $N_S(\Omega_1(P)T)\ge  N_S(T)$. The maximal choice of $T$ implies $T \ge \Omega_1(P) \ge Z(S)$, and this contradicts Lemma~\ref{Z(S) not in T}.
  Hence $N_{\Omega_1(N_S(T))}(P)\not \le P$ and $N_{N_S(T)}(P)=PN_{\Omega_1(N_S(T))}(P)$.  As $N_{\Omega_1(N_S(T))}(P)$ centralizes $Z_2(S)$ and $\Out_{N_S(T)}(P)$ acts indecomposably on $\Omega_1(P)/\Omega_1(T)$, we deduce that $$ Z_2(S)\Omega_1(T)/\Omega_1(T) \le C_{\Omega_1(P)/\Omega_1(T)}(N_{N_S(T)}(P))= Z(S)\Omega_1(T)/\Omega_1(T).$$ Hence, as $\Omega_1(T)$ has order $p$, we have $\Omega_1(T) \le Z_2(S)$.  In particular,  $\gamma_2(S)$ normalizes  $T$ and $Z_2(S)=Z(S)\Omega_1(T) \le P$.

Since $|S| \ge p^7$ by Lemma~\ref{facts}(iii), $Z_4(S) \le \gamma_2(S) \le N_S(T)$.  Hence $[P,Z_4(S)] \le Z_2(S)\le P$ implies $Z_4(S) \le N_{N_S(T)}(P)$. In addition $$[P,Z_4(S),Z_4(S)]\le [\gamma_1(S), Z_4(S),Z_4(S)]\le [Z_2(S), Z_4(S)]= 1.$$ Since $N_{N_S(T)}(P)$ does not act quadratically on $\Omega_1(P)/\Omega_1(T)$, we have $Z_4(S) \le P$. Since $Z_4(S)$ has exponent $p$ by Lemma~\ref{mc-facts} (iii) and (vi), $Z_4(S) \le \Omega_1(P)$. Now \begin{eqnarray*}[Z_4(S)/\Omega_1(T),N_{N_S(T)}(P)]&\le& [Z_4(S)/\Omega_1(T),\gamma_1(S)]\le Z_2(S)/\Omega_1(T)\\& =& Z(S)\Omega_1(T)/\Omega_1(T)\end{eqnarray*} and, as $|Z_4(S)/Z(S)\Omega_1(T)|=p^2$, we have a contradiction to the indecomposable action of $ \Out_{N_S(T)}(P)$  on  $\Omega_1(P)/\Omega_1(T)$. This contradiction shows that $|C_{\Omega_1(P)}(\theta)| > p^2$.

By Lemma~\ref{Z(S) and theta} the morphism $\theta$ does not centralize $\Z(S)$. Hence
\[ |C_{\Omega_1(P/Z(S))}(\theta)| > p^2 .\]
Consider the group $S/Z(S)$ and note that $S$ is  not exceptional. Since $\hat \theta$ has order  divisible by $(p-1)/2$, and $|C_{\Omega_1(\gamma_1(S/Z(S)))}(\hat \theta)|\ge  |C_{\Omega_1(P/Z(S))}(\theta)| > p^2$,  Lemma~\ref{centralizer auto} implies $|\Omega_1(\gamma_1(S/Z(S)))|=p^p$. By Lemma~\ref{mc-facts}(iv) we get $|S/Z(S)|=p^{p+1}$, and so $|S|=p^{p+2}$ and $\agemO^1(\gamma_1(S))\leq \Z(S)=\gamma_{p+1}(S)$. However this contradicts Lemma~\ref{mc-facts}(iii) which states that $\agemO^1(\gamma_1(S))=\gamma_{p}(S)>Z(S)$.
\end{proof}

\begin{lemma}\label{lem:we should cite this} Assume that Hypothesis~\ref{HypCST2} holds. Then $L/\Inn(P) \not \cong \SL_2(p)$.
\end{lemma}

\begin{proof} Suppose  that  $L/\Inn(P)  \cong \SL_2(p)$.    Then,   $\hat \theta$ has order $p-1$ and acts faithfully on $S/\gamma_1(S)$.
If  $|C_{\Omega_1(\gamma_1(S))}(\theta)|\ge p^2$, then Lemma~\ref{centralizer auto} applied to $S/Z(S)$ implies that $Z(S)$ is centralized by $\theta$. This is impossible by  Lemma~\ref{Z(S) and theta}. Hence  $|C_{\Omega_1(\gamma_1(S))}(\theta)|\le p$.
Applying Lemma~\ref{L2p} (i) and (ii), delivers all the non-central $L$-chief factors in $\Omega_1(P)/\Omega_1(T)$ are faithful $\GF(p)L/\Inn(P)$-modules.  In particular, $\tau$ inverts $\Omega_1(P)/\Omega_1(T)$ and so $\Omega_1(P)/\Omega_1(T)$ is abelian and $\Omega_1(T)$ has order $p$.

Suppose that $Z_3(S) \le N_S(T)$.  Then, as $[Z_3(S),P]\le Z(S)$, we obtain $Z_3(S)\le N_{N_S(T)}(P)$.

Assume that $Z_3(S) \not \le P$. Then $[P,Z_3(S)]\le [Z_3(S),\gamma_1(S)]\le  Z(S)$. Hence $\Omega_1(P)/\Omega_1(T)$ has exactly one non-central $L$-chief factor and  $|P/\Omega_1(T)|=p^2$ by Lemma~\ref{uniserial}.  Hence $|Z_3(S) \cap P|= p^2$. If $\Omega_1(T) \not \le Z_3(S)$, then $P= \Omega_1(T)(P \cap Z_3(S))$ and this means that $Z_3(S)$ centralizes $P$. Since $P$ is $N_\F(T)$-centric, this is impossible. Thus  $Z_3(S)P/P\cong Z_3(S)/(Z_3(S) \cap P)$ and $\Omega_1(T)$ are both centralized by $\tau$. As $Z(S) \le P$ and $Z(S) \cap \Omega_1(T)=1$, we have $Z_3(S)/Z(S)$ is centralized by $\tau$.  This is impossible by  Lemma~\ref{centralizer auto}.
Hence $Z_3(S) \le P$ and so $Z_3(S) \le \Omega_1(P)$ by Lemma~\ref{mc-facts}(iii) and (vi). Now Lemma~\ref{centralizer auto} implies $1 \ne C_{Z_3(S)}(\tau) \le C_{\Omega_1(P)}(\tau)= \Omega_1(T)$.  As $\Omega_1(T)$ has order $p$, this means that $\Omega_1(T) \le Z_3(S)$.

We conclude that, if $Z_3(S) \le N_S(T)$, then $\Omega_1(T)\le Z_3(S)\le \Omega_1(P)$ and, in particular, $N_{S}(T) \ge \gamma_3(S)$.

Continue to assume that $Z_3(S) \le N_S(T)$.  Then $Z_4(S) \le  \gamma_3(S) \le N_S(T)$ by Lemma~\ref{facts} (iii).   Since $Z_3(S) \le P$, $Z_4(S) \le N_{N_S(T)}(P)$.
Suppose that $Z_4(S) \not \le P$. Then $\tau$ centralizes $Z_4(S)P/P$ and so $\tau$ centralizes $Z_4(S)/Z_3(S)$. Lemma~\ref{action} implies $\tau$ inverts $Z_3(S)/Z_2(S)$ and so $\Omega_1(T) \le Z_2(S)$. Since $[P,Z_4(S)]\le Z_2(S)$, $\Omega_1(P)/\Omega_1(T)$ has order $p^2$ by Lemma~\ref{uniserial}.  Hence $\Omega_1(P)=Z_3(S)$ and we then have $[\Omega_1(P),Z_4(S)]\le [Z_3(S),\gamma_3(S)] =1$, a contradiction.  Therefore  $Z_4(S) \le P$. Since $Z_4(S)$ has exponent $p$  by Lemma~\ref{mc-facts}(iii) and (vi), $Z_4(S) \le \Omega_1(P)$ and $Z_4(S)/\Omega_1(T)$ is inverted by $\tau$.  Lemma~\ref{action} yields $\Omega_1(T)^\# \subseteq Z_3(S) \setminus Z_2(S)$ and $Z_2(S)$ is inverted by $\tau$.  In particular, $S$ is exceptional and so $|S| \ge p^8$, $\gamma_2(S)$ has exponent $p$ by Lemma~\ref{mc-facts} (vi) and \ref{mc-factsb}(v).
Therefore $Z_5(S) \le \gamma_3(S) \le N_S(T)$ and $[P, Z_5(S)]\le Z_3(S)\le P$.
If $Z_5(S) \le P$, then   $Z_5(S) \le \Omega_1(P)$ and $C_{Z_5(S)}(\tau)=C_{\Omega_1(P)}(\tau)= \Omega_1(T)$ has order $p$. This contradicts Lemma~\ref{action}. Hence $Z_5(S) \not \le P$, $Z_4(S)\le \Omega_1(P)$ and, in particular, $|\Omega_1(P)| \ge p^4$. As  $$[P, Z_5(S),Z_5(S)]\le [Z_3(S), Z_5(S)]=1$$ and $Z_3(S)/Z(S)$ has order $p^2$, we have that $\Omega_1(P)/\Omega_1(T)$ has two $L$-chief factors and they are both natural $\GF(p)L/P$-modules (of order $p^2$) by Lemma~\ref{uniserial}. Hence $|P|=p^5$ and $|P:C_{P}(Z_5(S))| \ge p^2$.  If $Z_5(S)$ is abelian, then $Z_4(S) \le C_{P}(Z_5(S))$ and $|P:C_{P}(Z_5(S))| \le p$, a contradiction.  Hence $Z_5(S)$ is non-abelian and consequently $|S|=p^8$.  Since $[Z_4(S), \Omega_1(P)]\le\Omega_1(P)' \cap Z_2(S) \le \Omega_1(T)\cap Z_2(S)=1$, $Z_4(S) \le Z(\Omega_1(P))$ and we deduce that $\Omega_1(P)$ is abelian. On the other hand, $C_{\gamma_1(S)}(Z_4(S)) = Z_k(S)$ for some $k \ge 4$. As $Z_5(S)$ is non-abelian, we have that $C_{\gamma_1(S)}(Z_4(S)) = Z_4(S)$ which means that $Z_4(S) < P \le Z_4(S)$, a contradiction.   We have demonstrated that $Z_3(S)$ does not normalize $T$.

Assume that $S$ is not exceptional. Then, as $Z_2(S)\le Z(\gamma_1(S))$, $Z_2(S)\le N_S(T) $ and  so $Z_2(S) \le P$.  Since $\tau$ inverts $Z_2(S)\Omega_1(T)/\Omega_1(T)$, Lemma~\ref{centralizer auto} implies that $Z_2(S)\ge \Omega_1(T)$. Hence $\Omega_1(T)$ is centralized by $\gamma_1(S)$ and Lemma~\ref{NFTnormZ} (i) implies $\gamma_1(S)$ normalizes $T$.  Since $Z_3(S) \le \gamma_1(S)$, this is impossible.  Thus $S$ is exceptional.

   By Lemma~\ref{Z2 in NST}, $Z_2(S) \le C_{N_S(T)}(\Omega_1(N_S(T)))$ and so, in particular, $Z_2(S) \le C_{N_S(T)}(\Omega_1(P))$.  Because $L/P$ acts faithfully on $\Omega_1(P)/\Omega_1(T)$, we deduce that $Z_2(S) \le \Omega_1(P)$. If $\Omega_1(T) \le Z_2(S)$, then $N_S(T) \ge \gamma_2(S)$ and then $Z_3(S)$ normalizes $T$, a contradiction. Hence $Z_2(S)\cong Z_2(S)\Omega_1(T)/\Omega_1(T)$ is inverted by $\tau$.

 It follows that $\hat \tau$ centralizes $S/C_S(Z_2(S))$ and so $$[\gamma_1(S),\hat \tau] \le\gamma_1(S) \cap C_S(Z_2(S))= \gamma_2(S).$$  As $T$ is not normalized by $Z_3(S)$, $\Omega_1(T)$ is not centralized by $Z_3(S)$ by Lemma~\ref{NFTnormZ} (i).  Hence $\Omega_1(T) \not \le \gamma_3(S)$, but $\Omega_1(T)$ does centralize $Z_2(S)$ and so $\Omega_1(T) \le \gamma_2(S)$. Thus $\gamma_2(S)/\gamma_3(S)= \Omega_1(T)\gamma_3(S)/\gamma_3(S)$ is centralized by $\hat \tau$.  Therefore $\hat \tau$ centralizes $\gamma_1(S)/\gamma_3(S)$ and this finally contradicts Lemma~\ref{action}. We have shown that $L/\Inn(P) \not \cong \SL_2(p)$.
  \end{proof}

\begin{prop}\label{prop:not centric} If Hypothesis~\ref{hp.conj} holds. Then $T$ is $S$-centric.
\end{prop}
\begin{proof}\blue{ If $T$ is not $S$-centric, then Hypothesis~\ref{HypCST2} holds. Recall the definitions of $P$ and $L$ from Notation~\ref{not:P} and Notation~\ref{not:L}. Then Lemma~\ref{L2ptype} yields $L/\Inn(P) \cong \PSL_2(p)$ or $\SL_2(p)$, whereas Lemmas~\ref{it's SL2p again} asserts that $L/\Inn(P) \not \cong \PSL_2(p)$ and Lemma~\ref{lem:we should cite this} states that $L/\Inn(P) \not\cong \SL_2(p)$. This is impossible. Thus $T$ is $S$-centric.}
\end{proof}

\begin{proof}[Proof of Theorem~\ref{MT1}] As we remarked at the beginning of Section~\ref{sec: proof MT1 1},
 Lemmas~\ref{p=2} and \ref{p=3} show that Theorem~\ref{MT1} holds if $p\le 3$ as in this case $S$ is not exceptional by Lemma~\ref{mc-factsb}(v).   Hence we may assume that  Hypothesis~\ref{hyp1} holds.\blue{ Assume Theorem~\ref{MT1} is false.  Then Hypothesis~\ref{hp.conj} and Notation~\ref{notationT} hold. Now we obtain a contradiction as Proposition~\ref{the centric case} yields $T$ is not $S$-centric while Proposition~\ref{prop:not centric} asserts that $T$ is   $S$-centric.  We conclude that the main statement in Theorem~\ref{MT1} is true.}
   If $S$ is exceptional, then 
   we obtain $\mathcal P(\F)=\mathcal P_a(\F)$ from Proposition~\ref{Prop:ess.in.except}. This concludes the proof.
\end{proof}

%% file: Exceptional.tex
 \section[Exceptional maximal class groups: the proof of Theorem~B]{The saturated fusion systems on exceptional maximal class groups: the proof of Theorem~B}\label{sec: proof Raul app}

 The objective of this section is to prove Theorem~\ref{Raul app}, which for convenience we repeat below.
Recall that in the introduction we defined   $S(p)$ as the unique split extension of an extraspecial group of exponent $p$ and order $p^{p-2}$ by a cyclic group of order $p$ which has maximal class \cite[Proposition 8.1]{Raul}.
\begin{customthm}{B}\label{Raul app}
Suppose that $p \ge 5$,   $S$ is an exceptional maximal class $p$-group of order at least $p^4$ and $\F$ is a saturated fusion system on $S$. Assume that $\F \ne  N_\F(S)$.  Then one of the following holds.
\begin{enumerate}
\item $\gamma_1(S)$ is extraspecial, and, if  $\F\ne N_\F(\gamma_1(S))$, then  one of the following holds:
\begin{enumerate}
\item $S $ is isomorphic to a Sylow $p$-subgroup of $\G_2(p)$  and either
\begin{enumerate}[label=(\greek*)]
\item $\F= N_\F(C_S(Z_2(S)))$, $O^{p'}(\Out_\F(C_S(Z_2(S))))\cong \SL_2(p)$;
\item $p=5$, $1\ne O_p(\F) \le \gamma_2(S)$, $\F \cong \F_S(5^3{}^. \SL_3(5))$;
\item  $p \ge 5$ and $\F=\F_S(\G_2(p))$;
\item $p=5$ and $\F=\F_S(G)$ where $G= \Ly, \HN, \Aut(\HN)$ or $\B$; or
\item $p=7$ and either $\F$ is exotic (27 examples) or $\F= \F_S(\Mo)$.
 \end{enumerate}
 \item $p \ge 11$, $S\cong S(p)$, $\mathcal P(\F)= \mathcal P_a(\F)\not=\emptyset$ and, if $\gamma_1(S)$ is $\F$-essential, then $\Out_\F(S) \cong \GF(p)^\times \times \GF(p)^\times$, $O^{p'}(\Out_\F(\gamma_2(S))) \cong \SL_2(p)$ and $\gamma_1(S)/Z(\gamma_1(S))$ is the unique $(p-3)$-dimensional irreducible $\GF(p)\SL_2(p)$-module.

     \end{enumerate}
      \item $p=5$,  $S= \mathrm{SmallGroup(5^6,661)}$, $O_5(\F)=C_S(Z_2(S))$ is the unique $\F$-essential subgroup, $\Out_\F(S)$ is cyclic of order $4$, $\Out_\F(C_S(Z_2(S))) \cong \SL_2(5)$ and $\F$ is unique.
     \end{enumerate}
In particular,  if $\F \ne N_\F(\gamma_1(S))$, then $\F=O^p(\F)$ and, in addition,   $O_p(\F) =1 $ in all cases other than parts (i)(a)($\alpha$), (i)(a)($\beta$) and part (ii).
     \end{customthm}

\begin{proof}
 Let $\F$ be a saturated fusion system on  $S$ and assume that $\F \ne N_\F(S)$. Then $\E_\F$ is non-empty and  $$ \E_\F \subseteq \mathcal P_a(\F) \cup \{C_S(Z_2(S)),\gamma_1(S)\}$$ by Theorem~\ref{MT1}.

Suppose $\E_\F = \{ \gamma_1(S)\}$. Then $\F= N_\F(\gamma_1(S))$ and $\gamma_1(S)$ is extraspecial by Proposition~\ref{gamma1-essD}. From now on, suppose $\E_\F \neq \{ \gamma_1(S)\}$.  Then by Proposition~\ref{Prop:ess.in.except} either  $\gamma_1(S)$ is extraspecial or conclusion (ii) holds. So suppose $\gamma_1(S)$ is extraspecial and let $R=O_p(\F)$.

Assume $R \neq 1$. Then $R$ is normal in $S$ and, for $E \in \E_\F$, $R \le E$ and $R$ is $\Aut_\F(E)$-invariant by Lemma \ref{OpinE}.  In particular, $\F$ has no $\F$-pearls as all the $\F$-pearls are abelian.  Since $\E_\F \neq \{ \gamma_1(S)\}$ by assumption, we deduce that $C_S(Z_2(S)) \in \E_\F$. Thus Proposition~\ref{Prop:ess.in.except} implies that $|S| =p^6$ and $O^{p'}(\Out_\F(C_S(Z_2(S))))\cong \SL_2(p)$ acts transitively on the maximal subgroups of $C_S(Z_2(S))$. As $\gamma_2(S)$ has exponent $p$ by Lemma~\ref{mc-facts} (vi) and $\Phi(C_S(Z_2(S)))=\gamma_4(S)$, $C_S(Z_2(S))$ has exponent $p$.  In particular, there exists an element $x$ of order $p$ in $C_S(Z_2(S))$ such that $S= \gamma_1(S)\langle x\rangle$. Hence \cite[Proposition 8.1]{Raul} implies that $S$ is isomorphic to a Sylow $p$-subgroup of $\G_2(p)$.   If $|\E_\F| =1$ then $\E_\F = {C_S(Z_2(S))}$ and we are in case (i)(a)($\alpha$). If  $|\E_\F|=2$, then $\E_\F=\{\gamma_1(S), C_S(Z_2(S))\}$. We have  $R \le C_S(Z_2(S)) \cap \gamma_1(S)= \gamma_2(S)$ and, as $R$ is $\Aut_\F( C_S(Z_2(S))) $-invariant and $\Aut_\F( C_S(Z_2(S))) $ acts irreducibly on $C_S(Z_2(S))/\Phi( C_S(Z_2(S))) $, we obtain $R \le \Phi( C_S(Z_2(S))) =\gamma_3(S)$.
Since $\Aut_\F(C_S(Z_2(S)))$ does not normalize $Z(S)$, we have $R= Z_2(S)$ or $R= Z_3(S) = \gamma_3(S)$.
Assume that $R= Z_2(S)$. Then $\Aut_\F(\gamma_1(S))$ normalizes $C_{\gamma_1(S)}(Z_2(S))=\gamma_2(S)$ and $\Aut_\F(\gamma_1(S))$ acts on $\gamma_2(S)/Z_2(S)$ which has order $p^2$. Since $\gamma_2(S)/\gamma_3(S)$, $Z_2(S)/Z(S)$ and $Z(S)$ are centralized by $O^{p'}(\Aut_\F(S))$, we have $O^{p'}(\Out_\F(\gamma_1(S))) \cong \SL_2(p)$. Now using an  element $\tau \in \Aut_\F(\gamma_1(S))$ which maps to the involution in $O^{p'}(\Out_\F(\gamma_1(S)))$ we have $[\gamma_1(S),\tau]$ is extraspecial of order $p^3$ and is normalized by $S$.  Since $[\gamma_1(S),\tau]\ne \gamma_3(S)$, this is impossible.  Hence $R= \gamma_3(S)$ and $R$ is $\F$-centric.  It follows that there is a model $H$ for $N_\F(R)$.  From the structure of $\Aut(R) \cong \GL_3(p)$, we deduce that $O^{p'}(H)/R \cong \SL_3(p)$ as $|\mathcal E_\F|\ge 2$. Let $P= C_{O^{p'}(H)}(Z(S))$, then $P/\gamma_1(S) \cong \SL_2(p)$ and $O^{p'}(H)$ acts on the natural $\GF(p)\SL_2(p)$-modules $\gamma_1(S)/R$ and $R/Z(S)$. Since $S$ has maximal class, $\gamma_1(S)/Z(S)$ is a non-split extension for $P/\gamma_1(S)$. Now we may apply \cite[Proposition 4.2]{p.index2} and obtain $p=5$.  This, together with the model theorem gives (i)(a)($\beta$).

Hence we may now suppose that $R=1$.  \cite[Main Theorem]{Raul} implies that either $S$  is isomorphic to a Sylow $p$-subgroup of $\G_2(p)$ or $S \cong S(p)$. In the first case, \cite[Theorem 1.1]{G2p} shows that $\F$ is known and it is one of the fusion systems described in (i)(a)($\gamma$), ($\delta$) and ($\epsilon$).  If $S \cong S(p)$, then $\F$ is known by  \cite[Theorem 2]{Raul} and corresponds to case (i)(b).
\end{proof}

%% file: NotExceptional.tex
\section[non-exceptional maximal class groups]{The saturated fusion systems on non-exceptional maximal class groups }\label{sec: non excep g1}

\blue{In this section we will lay the ground work for the proof of Theorem~\ref{non exc} which will be presented in the next section.} So suppose that $p$ is an odd prime and $S$ is a maximal class $p$-group of order at least $p^4$.
Following \cite{p.index2}, we define $$\Delta = \GF(p)^\times \times \GF(p)^\times$$ which has order $(p-1)^2$ and, for $i \in \mathbb Z$, we define the diagonal subgroups  $$\Delta_i= \{(r,r^{i})\mid r \in \GF(p)^\times\}.$$  Most important here are $\Delta_0$ and $\Delta_{-1}$. Indeed, these diagonals  originate from the existence of $\F$-pearls, as we shall see in Lemma~\ref{lem:mu image}.
Suppose that $\alpha \in \Aut(S)$, $x \in S\setminus \gamma_1(S)$ and $z \in Z(S)^\#$. Then there exist $r, s \in \GF(p)^\times$ which are independent of $x$ and $z$  such that $$x\alpha \gamma_1(S)  = x^r\gamma_1(S)$$ and $$z\alpha= z^s.$$
With this notation established,   define the homomorphism
 \begin{eqnarray*}\mu: \Aut(S)&\rightarrow&   \Delta\\\alpha &\mapsto&(r,s).\end{eqnarray*}We also use the following definitions from  \cite{p.index2}.
\begin{definition} Suppose that  $p$ is a prime, $G$ is a group, and $U \in \Syl_p(G)$. Then
   $G$ is in   class $\mathcal G_p^\wedge$ if and only if $O_p(G)=1$, $|U|=p$, and $|\Aut_G(U)|=p-1$.
   \end{definition}

\begin{definition}
Suppose that  $G\in \mathcal G_p^\wedge$ and let $V$ be a faithful $\GF(p)G$-module. Then $G$ is \emph{minimally active} on $V$ if and only if the matrix representing $u\in U^\#$ in its action on $V$ has one non-trivial Jordan block.
\end{definition}

We will need the following   consequence of the results of Craven, Oliver and Semeraro  \cite{p.index2} as listed in Appendix~\ref{AppA}.

\begin{prop}\label{app COS} Suppose that \blue{$p$ is an odd prime,} $\mathcal G$ is a reduced fusion system on a $p$-group $S$ of maximal class of order at least $p^4$ with $\gamma_1(S)$ elementary abelian and  $\mathcal G$-essential. Assume that $\Out_{\mathcal G}(S) \cong \Aut_\mathcal G(S)\mu =\Delta_{0}$ or $\Delta_{-1}$ and $|\gamma_1(S)|\le p^{p-1}$. Then $p \ge 5$, $\Aut_\mathcal G(\gamma_1(S)) \cong \Sym(p)$ or $\PGL_2(p)$, $\gamma_1(S)$ is   irreducible as a $\GF(p)\Aut_\mathcal G(\gamma_1(S))$-module \blue{and $|\gamma_1(S)|=p^{p-2}$.} Furthermore, $\Aut_\mathcal G(S)\mu =\Delta_{-1}$  and $\mathcal P (\mathcal G)=\mathcal P_a (\mathcal G)$ is non-empty.
\end{prop}

\begin{proof} Set $V= \gamma_1(S)$, $G=\Aut_\mathcal G(V)$ and $G_0=F^*(G)$. By Theorem~\ref{MT1} we have
 \[\E_\mathcal G = \mathcal P(\mathcal G) \cup \{ V \}\]
and since $\mathcal G$ is reduced, $\mathcal P(\mathcal G) \ne \emptyset$.  Because $\Out_\mathcal G(S)\cong \Aut_\mathcal G(S)\mu = \Delta_{0}$ or $\Delta_{-1}$, we know $Z(G)=1$ from the definition of $\Delta_0$ and $\Delta_{-1}$.  Now, as $Z(G)= 1$ and $\Aut_\mathcal G(S)\mu = \Delta_{0}$ or $\Delta_{-1}$,   applying Theorem~\ref{TheoremAbelian}  we have that  $\Aut_\mathcal G(\gamma_1(S)) \cong \Sym(p)$ or $\PGL_2(p)$ and that $\mathcal G$ is described in lines 3, 4, 29, 30, 33 or 34 of Table~\ref{Tab1}. If $\Aut_\mathcal G(S)\mu = \Delta_{0}$, then inspection of the consequence of column six of Table~\ref{Tab1} \blue{as read from Table~\ref{Tab2.1} yields  a  contradiction to the structure of $\Aut_\mathcal G(S)\mu$.} If   $\Aut_\mathcal G(S)\mu = \Delta_{-1}$, \blue{ then, again using Tables~\ref{Tab2.1} and \ref{Tab1},} we obtain cases 3 and  4 and that III from Table~\ref{Tab2.1} holds. Thus $|V|=p^{p-2}$ and, additionally, $\mathcal P(\mathcal G) =\mathcal P_a(\mathcal G)$.
\end{proof}

From here on we assume that $\F$ is a saturated fusion system on $S$ and that $S$ is not exceptional. By Theorem~\ref{MT1}, the $\F$-essential subgroups are contained in $\mathcal P(\F) \cup \{\gamma_1(S)\}$. If $\mathcal P(\F)$ is empty, then $\F= N_\F(\gamma_1(S))$ and, if $\mathcal P(\F)$ is non-empty, then $O_p(\F)\le Z(S)$.  Hence    $\F \ne N_\F(\gamma_1(S))$ if and only if $ \mathcal P(\F)$ is non-empty.

\begin{lemma}\label{lem:mu image} Assume that $\gamma_1(S)$ is not abelian and   $P \in  \mathcal P(\F)$.
 \begin{enumerate}
 \item  $\Out_\F(S)$ is a Hall $p'$-subgroup of $\Out(S)$, is cyclic of order  $p-1$ and acts faithfully on $S/\gamma_1(S)$.
 \item $\Aut_\F(S)= N_{\Aut_\F(S)}(P) \Inn(S)$.
 \item $\Out_\F(P) \cong \SL_2(p)$.
 \item If $P$ is abelian, then $\Aut_\F(S) \mu = \Delta_{-1}$.
   \item If $P$ is extraspecial, then $\Aut_\F(S) \mu = \Delta_{0}$.
   \item Either $\mathcal P(\F)= \mathcal P_a(\F)$ or $\mathcal P(\F)=\mathcal P_e(\F)$.
 \end{enumerate}
In particular, if $\gamma_1(S)$ is $\F$-essential, then   $ \Out_\F(\gamma_1(S)) $ is a member of $\mathcal G^\wedge_p$ and $Z(\Out_\F(\gamma_1(S)))=1$.
\end{lemma}

\begin{proof} Parts (i), (ii) and (iii) are  Lemma~\ref{type.pearl}.

Suppose that $P$ is abelian.  Then Lemma~\ref{pearls2} (iii) implies that $P \cap \gamma_1(S)= Z(S)$ and $S= P\gamma_1(S)$.  Let $\tau \in \Aut_\F(S)$  be such that $\tau$ restricts to a generator of $N_{\Aut_\F(P)}(\Aut_S(P))$. As $\Aut_\F(P) \cong \SL_2(p)$ by (iii), the element $\tau$ raises elements of $S/\gamma_1(S) \cong P/Z(S)$ to the power $r$ and elements of $Z(S)$ to the power $r^{-1}$ for some $r \in \GF(p)^\times$. Hence $\Aut_\F(S)\mu= \Delta_{-1}$ in this case. If $P$ is extraspecial, then,  as $\Aut_\F(P) \cong \SL_2(p)$, $Z(S) $ is centralized by $\Aut_\F(P)$.  Thus $\Aut_\F(S)\mu= \Delta_{0}$. As $|\Out_\F(S)|=p-1$ by (i), we have shown parts (iv) and (v) hold.  Finally (vi) follows from (iv) and (v).
\end{proof}

For the rest of this section, we assume that
\begin{hypothesis}\label{hyp last}~
\begin{enumerate}
\item $p \ge 5$, $S$ has maximal class, is not exceptional and has order at least $p^4$;
\item $\F$ is a saturated fusion system on $S$;
\item  $\gamma_1(S)$ is $\F$-essential and is not abelian; and
\item $P\in \mathcal P(\F)$.
\end{enumerate}
\end{hypothesis}
To ease notation we set $$Q=\gamma_1(S),\; G= \Out_\F(Q) \text{ and } G_0 = F^*(G).$$
  We also put $$V= \Omega_1(Z(Q)) \text{ and  }S_1= VP.$$

 As $S$ is not exceptional, $V\ge Z_2(S)$ and so $|V| \ge p^2$.
Set  $$H= \Inn(S_1)\langle \phi|_{S_1} \mid \phi \in N_{\Aut_\F(S)}(P)\rangle \le \Aut_\F(S_1)$$ and
$$B= \{\phi|_V\mid \phi \in \Aut_\F(Q)\} \le \Aut_\F(V).$$
 Proposition~\ref{prop:sub sat} implies that  $$\F_0=\langle \Aut_\F(P), B,H\rangle $$ is a saturated fusion system on $S_1$ and, by construction, $V$ has index $p$ in $S_1$ which has maximal class. Furthermore $P$ is an $\F_0$-pearl. Set $$V_0= V \cap \hyp(\F_0).$$

\begin{lemma}\label{tau action}\blue{Assume that Hypothesis~\ref{hyp last} (i),(ii) and (iv) hold with $Q$ non-abelian but not necessarily $\F$-essential.}
Suppose that $\tau\in \Aut_\F(S)$ has order $p-1$ and let $s_i \in \gamma_i(S)\setminus \gamma_{i+1}(S)$ for $1 \le i \le n-1$.
 Then either
 \begin{enumerate} \item $P$ is abelian, $\tau$ acts fixed-point-freely on $V_0$ and centralizes $V/V_0$ and for all $j$, $$s_{n-j}\gamma_{n-j+1}(S)\tau =  s_{n-j}^{r^{-j}}\gamma_{n-j+1}(S)$$ where $\tau \mu = (r,r^{-1}) \in \Delta_{-1}$; or
 \item  $P$ is extraspecial,  $\tau$ acts fixed-point-freely on $V/Z(S)$ and centralizes $Z(S)$ and, for all $j$,
 $$s_{n-(j+1)}\gamma_{n-j}(S)\tau  = s_{n-(j+1)}^{r^{-j}}\gamma_{n-j}(S)  $$ where $\tau \mu = (r,1) \in \Delta_{0}$.
  \end{enumerate}
\end{lemma}

\begin{proof} Recall that $|S|=p^n$. Assume that $P$ is abelian. Then Lemma~\ref{lem:mu image} (iv) gives $\Aut_\F(S)\mu=\Delta_{-1}$. Hence we take $\tau = (r,r^{-1})$. Then, for $x \in P \setminus Q$ we have $xQ\tau = x^rQ$, and  $s_{n-1}\tau = s_{n-1}^{r^{-1}}$. We may write $s_1\tau \gamma_2(S)= s_1^b\gamma_2(S)$. Then, Lemma~\ref{action} shows that  $r^{-1}= r^{n-2}b$, $$s_{n-j}\gamma_{n-j+1}(S)\tau = s_{n-j}^{r^{n-j-1}b}\gamma_{n-j+1}(S)\tau = s_{n-j}^{r^{-j}}\gamma_{n-j+1}(S).$$
Since $r^{-1}$ has order $p-1$, we now have $\tau$ acts fixed-point-freely on $V_0$ which has order $p^{p-2}$ and $\tau$ centralizes $V/V_0= \gamma_{n-(p-1)}(S)/\gamma_{n-(p-2)}(S)$.  Since $O^{p'}(G)$ centralizes $V/V_0$, $G$ centralizes $V/V_0$. Thus (i) holds.

The proof of (ii) follows similarly.
\end{proof}

\begin{lemma}\label{F on V} \blue{Assume that Hypothesis~\ref{hyp last} holds.} Then
\begin{enumerate}
\item   $G$ is faithful and  minimally active on $V$.
\item $\Aut_{\F_0}(V) \cong G$.
\item Either $O_p(\F_0)=1$ or $P$ is extraspecial and $O_p(\F_0)= Z(S)$.
\item  If $P$ is abelian, then $\Out_{O^{p}(\F_0)}(\hyp(\F_0))\cong \Aut_{O^{p}(\F_0)}(\hyp(\F_0))\mu =\Delta_{-1}$  and, if $P$ is extraspecial   $\Out_{O^{p}(\F_0)}(\hyp(\F_0))\cong \Aut_{O^{p}(\F_0)}(\hyp(\F_0))\mu =\Delta_{0}$.
\item   If $P$ is extraspecial and $O_p(\F_0)= Z(S)$, then $$\Aut_{O^{p}(\F_0/Z(S))}(\hyp(\F_0)/Z(S))\mu=\Delta_{-1}.$$
\item If $P$ is abelian, then $O_p(\F) = O_p(\F_0) =1$ and $O^{p}(\F_0)$ is reduced.
\item If $P$ is extraspecial, then either $O_p(\F) = O_p(\F_0) =1$ and $O^{p}(\F_0)$ is reduced or $O_p(\F) = O_p(\F_0)= Z(S)$ and $O^p(\F_0/Z(S))$ is reduced.

\end{enumerate}
\end{lemma}

\begin{proof} Since $V$ is centralized by $Q$, $V$ is a $\GF(p)G$-module.  As $C_V(S)= Z(S)$ has order $p$,  $S/Q$ acts with  a unique Jordan block on $V$ and as $|V| \ge p^2$ this block is not trivial. Therefore $G/C_G(V)$ is minimally active on $V$ and $C_G(V)$ is a $p'$-group. Assume that $C_G(V)\ne 1$. Let $K \le \Aut_\F(Q)$ be of $p'$-order satisfy $K\Inn(Q)/\Inn(Q)= C_G(V)$.  By Lemma~\ref{lem:mu image} (i), $C_{K \Inn(Q)/\Inn(Q)}(S/Q)=1$.  In particular, $[K, N_{\Aut_S(Q)}(K)]=K$ by coprime action. Using  \cite[Theorem 5.3.10]{Gor},  we obtain that $K$ acts faithfully on $\Omega_1(Q)$.  Since $K$ centralizes $V$, $KN_{\Aut_S(Q)}(K)$ acts faithfully on $\Omega_1(Q)/V$ which has order at most $p^{p-2}$ by Lemma~\ref{mc-facts} (vi). However, Proposition~\ref{Hall-Hig} implies that $|\Omega_1(Q)/V| \ge p^{p-1}$, which is a contradiction. Hence $C_G(V)=1$ and  $G$ acts faithfully on $V$. This proves (i).

To see (ii), just note that the restriction map $\Aut_\F(Q) \rightarrow \Aut_{\F}(V)$ has kernel $Q$  as $G$ acts faithfully on $V$. Thus (ii) follows  immediately from (i).

Since $O_p(\F_0) \le P$, is normal in $S$  and   $\Aut_{\F_0}(P)$-invariant by Lemma~\ref{OpinE}, if $O_p(\F_0) \ne 1$, then $O_p(\F_0)= Z(P)=Z(S)$ and $P$ is extraspecial. This proves (iii).

Because $\Aut_{\F_0}(PV)= H$, and $\Aut_{\F_0}(P) = O^{p}(\Aut_{\F_0}(P))\le O^{p}(\F_0)$ by Lemma~\ref{l:focf}, we have  $\Aut_{O^{p}(\F_0)}(PV)\mu= N_{\Aut_\F(\hyp(\F_0))}(P) \mu$ and so (iv) follows from Lemma~\ref{lem:mu image} (iv) and (v).

Assume that $O_p(\F_0)= Z(S)$.  Then $P/Z(S)$ is an abelian $\F_0/Z(S)$-pearl.  The argument which proves Lemma~\ref{lem:mu image} (iv), also establishes part (v).

As $\Aut_{\F_0}^{(P)}(\hyp(\F_0)) = \Aut_{\F_0}(\hyp(\F_0))$,  Lemma~\ref{pearls2.5} and the fact that $V$ is abelian provide us the hypothesis of Lemma~\ref{oli 1.4} and this yields  $O^{p'}(O^p(\F_0))=O^{p}(\F_0)$.  Hence $O^p(\F_0)$ is reduced whenever $O_p(\F_0)=1$ and otherwise we have $O^{p}(\F_0)/Z(S)$ is reduced.
\end{proof}

\begin{lemma}\label{app COS2} \blue{Assume that Hypothesis~\ref{hyp last} holds.}  If $O_p(\F)=1$, then $P$ is abelian, $V_0$ is irreducible as a $\GF(p)G$-module, $|V_0|= p^{p-2}$ and $G\cong \Sym(p)$ or $\PGL_2(p)$.
\end{lemma}
\begin{proof} Because of Lemma~\ref{F on V} (iv), (vi) and (vii), we have $O^{p}(\F_0)$ is reduced and, as $\gamma_1(S_1)=V$, we may apply Proposition~\ref{app COS} to $O^p(\F_0)$ to obtain the result.
\end{proof}

\begin{lemma}\label{extraspecialcase} \blue{Assume that Hypothesis~\ref{hyp last} holds.} If  $O_p(\F)\ne 1$, then  $P$ is extraspecial, $G \cong \Sym(p)$ or $\PGL_2(p)$,  $O_p(\F)= Z(S)=Z(P)$ and $V=V_0$ is abelian with \blue{$G$-composition factors $O_p(\F)=Z(S)$ which is} centralized by $G$ and $V/Z(S)$ which is an irreducible $\GF(p)G$-module of order $p^{p-2}$.
\end{lemma}

\begin{proof} By  Lemma~\ref{F on V} (iii) and (vii), $P$ is extraspecial and  $O^p(\F_0/Z(S))$ is reduced. In addition, Lemma~\ref{F on V} (v) gives $\Aut_{O^{p}(\F_0/Z(S))}(\hyp(\F_0)/Z(S))\mu=\Delta_{-1}$. Now Proposition~\ref{app COS} yields $G \cong \Sym(p)$, or $\PGL_2(p)$ and $V_{ 0}/Z(S)$ has order $p^{p-2}$ is irreducible as a $\GF(p)G$-module. Thus $|V_{ {0}}|= p^{p-1}$ and $O^{p'}(G)$ centralizes $Z(S)$. In particular, $V=V_0$. Since $\Aut_\F(P)$ also centralizes $Z(S)$, we have $G$ centralizes $Z(S)=O_p(\F)$ and this completes the proof of the lemma.
\end{proof}

\begin{lemma}\label{V =Omega} \blue{Assume that Hypothesis~\ref{hyp last} holds.} Then $V=\Omega_1(Q)$ has order $p^{p-1}$ and $|S| > {p^{p+1}}$.  In particular, $\Omega_1(Q) < Q$.
\end{lemma}

\begin{proof} We know  $|V_0| \ge p^{p-2}$ by Lemmas~\ref{app COS2} and \ref{extraspecialcase}. Suppose that $|S|\le p^{p+1}$. If $|S|= p^p$, then $V_0$ has index $p$ in $Q$ and as $V_0 \le Z(Q)$, we deduce that $Q$ is abelian, a contradiction.  Suppose that $|S|= p^{p+1}$. If $P$ is extraspecial, then Lemma~\ref{extraspecialcase} implies that $V_0= V$ has order $p^{p-1}$. Thus, in this case, $V$ has index $p$ in $Q$. Since $V \le Z(Q)$, we have $Q$ is abelian, a contradiction.  So suppose that $P$ is abelian.
 Then $Q/V_0$ has order $p^2$ and $Q/V_0$ is abelian. Since $G$ does not embed in $\GL_2(p)$, we must have $[Q/V_0,O^{p'}(G)]=1$.  But then $S/V_0$ is abelian, a contradiction.  Thus $|S| > p^{p+1}$ and so $\Omega_1(Q)$ has order $p^{p-1}$ by Lemma~\ref{mc-facts} (iv).
In particular, $V_0$ has index at most $p$ in $\Omega_1(Q)$ with equality if $P$ is extraspecial.  So assume $P$ is abelian. Then, as $Q$ is the $2$-step centralizer in $S$ and $V_0$ is an irreducible $\GF(p)G$-module, we deduce that $V_0 > [\Omega_1(Q),Q] =1$ and  so $\Omega_1(Q) \le Z(Q)$.  Hence $V= \Omega_1(Q)$ in both cases.  This proves the lemma.
\end{proof}

There are two possibilities for the structure of $V$ as a $\GF(p)G$-module dependent upon the type of $P$.  If $P$ is abelian, then $V$ is a non-split extension of a submodule $V_0$ of order $p^{p-2}$ (that is $Z_{p-2}(S)$)  by a quotient of order $p$, 
 whereas, if $P$ is extraspecial, then there is a submodule of order $p$   (that is $Z(S)$)  and a quotient of order $p^{p-2}$ which is isomorphic to $V_0$.
In each case there is a unique proper submodule and a unique non-trivial quotient.

In the next lemma, we set $\Omega_0(Q)=1$.
Also, its useful to remember that  if $|Q / \Omega_{j-1}(Q)| > p^p$ for some $j\ge 0$,  then $|\Omega_{j}(Q)/\Omega_{j-1}(Q)| = p^{p-1}$ by Lemma \ref{mc-facts}(iii) applied to $S/\Omega_{j-1}(S)$.

\begin{lemma}\label{G action Q} \blue{Assume that Hypothesis~\ref{hyp last} holds.}  Suppose that $j\ge 1$ is maximal such that $\Omega_j(Q) \le Q$ and $|\Omega_j(Q)/\Omega_{j-1}(Q)|= p^{p-1}$.  Then, for $1\le k \le j$,  as $\GF(p) \Aut_\F(Q)$-modules,  $\Omega_{k}(Q)/\Omega_{k-1}(Q) \cong V$ and, either $Q=\Omega_j(Q)$  or, as $\GF(p)\Aut_\F(Q)$-modules
$$Q/\Omega_j(Q) \cong \begin{cases} Z(S)&P \text{ extraspecial } \\ V_0 & P\text{ abelian.}\end{cases}$$ In particular, $\Omega_m(Q)/\Omega_{m-1}(Q)$ is centralized by $Q$ for all $m$ and the (ascending) $\Aut_\F(Q)$-chief factors in the unique  $\Aut_\F(Q)$-chief series in $Q$  alternate between having  order $p$ and order $p^{p-2}$  until it reaches $Q$.
\end{lemma}

\begin{proof} From Lemma~\ref{V =Omega} we know that $Q> \Omega_1(Q)=V$ and  $V \le Z(Q)$.  In particular, $\Omega_2(Q)> V$. Suppose that $\ell\ge 0$ is minimal such that $\Omega_{\ell+1}(Q)/\Omega_{\ell}(Q) $ is not isomorphic to $V$. By Lemma~\ref{V =Omega}, $\ell \ge 1$.

Let $\Omega_\ell(Q)<W \le  \Omega_{\ell+1}(Q)$ be defined as  $$W/\Omega_{\ell}(Q)= C_{\Omega_{\ell+1}(Q)/\Omega_\ell(Q)}(Q).$$
\blue{Notice that $W' \Omega_{\ell-1}(Q)/\Omega_{\ell-1}(Q)  \le \Omega_{\ell}(Q)/\Omega_{\ell-1}(Q)$ and so $W' \Omega_{\ell-1}(Q)/\Omega_{\ell-1}(Q)$   has exponent $p$. Also, as $ \Omega_{\ell}(Q)/\Omega_{\ell-1}(Q)\cong V$ as a  $\GF(p)\Aut_\F(Q)$-module by the definition of $\ell$, $\Omega_{\ell}(Q)/\Omega_{\ell-1}(Q) $ is centralized by $Q$ and so is $W' \Omega_{\ell-1}(Q)/\Omega_{\ell-1}(Q)$.}

Consider the map  $\theta: W \rightarrow \Omega_{\ell}(Q)/\Omega_{\ell-1}(Q)$ defined by $w\mapsto w^p\Omega_{\ell-1}(Q)$. Then, for $x, y \in W$, using \cite[Lemma 2.2 (ii)]{Gor} together with $[x,y]^p\in \Omega_{\ell-1}(Q)$, we have  $$(xy)\theta=(xy)^p\Omega_{\ell-1}(Q)= x^py^p[x,y]^{\frac 1 2 p(p-1)}\Omega_{\ell-1}(Q)= x^py^p\Omega_{\ell-1}(Q)=x\theta y \theta$$ as $p$ is odd.
 In addition,  for $\alpha \in \Aut_\F(Q)$, $$(xy)\alpha\theta= (x\alpha)^p(y\alpha)^p\Omega_{\ell-1}(Q)= x^p\alpha y^p\alpha\Omega_{\ell-1}(Q) = ((x^py^p)\Omega_{\ell-1}(Q))\alpha = (xy)\theta\alpha.$$  Hence $\theta$ is $\Aut_\F(Q)$-invariant.  As $Q$ centralizes $ \Omega_{\ell}(Q)/\Omega_{\ell-1}(Q)$ and $W/\Omega_\ell(Q)$, the map $\theta$  is a $\GF(p)G$-module homomorphism and $\ker \theta =\Omega_{\ell}(Q)$ as $Q$ is regular.

 Suppose that $\ell+1 \le j$. Then $\Omega_{\ell+1}(Q)/\Omega_{\ell}(Q)$ has order $p^{p-1}$ and, as $Q$ is the $2$-step centralizer, $|[\Omega_{\ell+1}(Q), Q]\Omega_{\ell}(Q)/\Omega_{\ell}(Q)|\le p^{p-3}$ and $|W/\Omega_{\ell}(Q)|\ge p^2$.  Since $\ker \theta =\Omega_{\ell}(Q)$, $W\theta$ has order at least $p^2$. If $P$ is extraspecial, we know that the unique proper submodule of $V$ has order $p$ and so, in this case, $\theta$ is onto and $W= \Omega_{\ell+1}(Q)$, which is a contradiction. If $P$ is abelian, then $W\theta$ has order at least $p^{p-2}$ and so $|W/\Omega_{\ell}(Q)| \ge p^{p-2}$.  Since $W \ne  \Omega_{\ell+1}(Q)$,  we know that $W/\Omega_{\ell}(Q) $ is isomorphic to the unique proper submodule $V_0$  of $V$.  In particular, $W/\Omega_{\ell}(Q)$ is an irreducible $\GF(p)G$-module.  As $|[\Omega_{\ell+1}(Q), Q]\Omega_{\ell}(Q)/\Omega_{\ell}(Q)|\le p^{p-3}$, this is a contradiction.  We have proved that, for $1\le k \le j$,  as $\GF(p) \Aut_\F(Q)$-modules,  $\Omega_{k}(Q)/\Omega_{k-1}(Q) \cong V$.  In particular, the result is proved if $Q=\Omega_j(Q)$.
 Hence we assume that $Q > \Omega_j(Q)$ which means  $\ell+1> j$.

As  $\ell+1> j$  and, as $j \geq \ell$ by the maximal choice of $j$, we get  $\ell=j$. By the maximal choice of $j$, $\Omega_{j+1}= Q$ and $|Q/\Omega_{j}(Q)| < p^{p-1}$. If $P$ is abelian, then $W/\Omega_j(Q)\cong W\theta \cong V_0 $ which has order $p ^{p-2}$ and so $W = Q$.  If $P$ is extraspecial, then, as $|W\theta |< p^{p-1}$, we conclude $W\theta$ has order $p$ and $W= Q$ as $Q$ is a $2$-step centralizer. This concludes the proof.
\end{proof}

\begin{lemma}\label{sec-abelian} \blue{Assume that Hypothesis~\ref{hyp last} holds.}
Suppose that $R_2< R_1 \le Q$ with $R_1$ and $R_2$ both $\Aut_\F(Q)$-invariant. If $|R_1/R_2|\le p^{p-1}$, then $R_1/R_2$ is elementary abelian and centralized by $Q$.
\end{lemma}

\begin{proof} We may assume that $R_1/R_2$ is not a $\Aut_\F(Q)$-chief factor as this case is clear. Hence Lemma~\ref{G action Q} implies that $|R_1/R_2|=p^{p-1}$ and that there are exactly two $\Aut_\F(Q)$-chief factors in $R_1/R_2$ one of order $p$ and one of order $p^{p-2}$. Since $Q$ is a $2$-step centralizer and $S$ has maximal class we know that $|R_1:[R_1,Q]|\ge p^2$ and $|C_{R_1/R_2}(Q)|\ge p^2$. If follows that $Q$ centralizes $R_1/R_2$ and, in particular, $R_1/R_2$ is abelian. Since $\agemO^1(R_1/R_2)$ and $\Omega_1(R_1/R_2)$ are $\Aut_\F(Q)$-invariant and  $\agemO^1(R_1/R_2)\le \Omega_1(R_1/R_2)$ , it  follows that $R_1/R_2$ is elementary abelian.
\end{proof}

For $\Alt(p)$ there is a unique irreducible $\GF(p)$-module of dimension $p-2$ and this is the irreducible heart of the degree $p$ permutation module.  For $\PSL_2(p)$, again there is a unique irreducible module of dimension $p-2$ and this time it is the module that we denoted by $\VV_{p-3}$ in Section~\ref{sec:Reps} which consists of homogeneous polynomials in $\GF(p)[x,y]$ of degree $p-3$. Lemma~\ref{tau action} now gives us a unique action of $G= G_0\langle \tau \Inn(Q)\rangle$ on $V$. In the case that $G \cong \Sym(p)$, this module is in fact the quotient of the $p$-dimensional permutation module by the sum of all vectors in the natural basis. In the case of $\PGL_2(p)$, it is less natural to define.

\begin{lemma}\label{O^p}  \blue{Assume that Hypothesis~\ref{hyp last} holds.} Suppose that $Q= Q_1 > Q_2> \dots >Q_\ell=1$ is an $\Aut_\F(Q)$-chief series in $Q$. Then  $O^p(\F) \ne \F$ if and only if $|Q_1/Q_2|=p$.
Furthermore, if $O^p(\F) \ne \F$, then $\hyp(\F)= P\gamma_2(S)$ and $\gamma_1(\hyp(\F))=\gamma_2(S)$.

\end{lemma}
\begin{proof}
By Lemma~\ref{l:focf}(i) $$\foc(\F) = \langle [g,\alpha] \mid g \in R \in \E_\F \cup\{S\} \text { and } \alpha \in  \Aut_\F(R)\rangle.$$
Lemma~\ref{lem:mu image} implies that $P \le \foc(\F)$ and so $\foc(\F)=P[Q,\Aut_\F(Q)]$.

Lemma~\ref{G action Q} says that, as $\GF(p)\Aut_\F(Q)$-modules,  either $Q/Q_2\cong V_0$ or $Q/Q_2\cong Z(S)$   is centralized by $\Aut_\F(Q)$.
If $Q/Q_2 \cong V_0$, then $Q=[Q,\Aut_\F(Q))]Q_2$ and so, as $S$ has maximal class, $Q= [Q,\Aut_\F(Q)]$ and $S=\foc(\F)$. Hence, using Lemma~\ref{l:focf}(ii) yields $\F=O^p(\F)$. Thus, if $|Q_1/Q_2|>p$, then $\F=O^p(\F)$.

Conversely, if $|Q/Q_2|$ has order $p$, then $Q_2=\gamma_2(S)$ and  $Q_2/Q_3\cong V_0$ by Lemma \ref{G action Q}. Hence $[Q, \Aut_\F(Q)]=\gamma_2(S)$ in this case. This means that $\foc(\F)= P\gamma_2(S)=\hyp(\F)$ and $\gamma_1(\hyp(\F))=\gamma_2(S)$. In particular, Lemma~\ref{l:focf}(ii) implies $\F \ne O^p(\F)$.
 \end{proof}

\begin{lemma}\label{F not OpF} \blue{Assume that Hypothesis~\ref{hyp last} holds.}
Then $\F \ne O^p(\F)$.
\end{lemma}

\begin{proof} Assume that $\F= O^p(\F)$ and let    $Q= Q_1 > Q_2> \dots >Q_\ell=1$ be an $\Aut_\F(Q)$-chief series in $Q$. Then Lemmas~\ref{G action Q} and \ref{O^p} imply that $Q/Q_2 \cong V_0$ as a $\GF(p)G$-module.

 By hypothesis $Q$ is non-abelian. We intend to show that this cannot be the case. Choose $N \le Q$ maximal in $Q$ such that $N$ is $\Aut_\F(Q)$-invariant and $Q/N$ is non-abelian. Since $\Aut_\F(Q)$ acts on $Q/N$, it is sufficient to show that $Q/N$ cannot admit $\Aut_\F(Q)$.  Thus we work with the quotient $Q/N$ and to make the notation lighter we assume that $N=1$.  Set $M= Q'$. Then, the maximal choice of   $N$ implies that $M$  is the smallest non-trivial  $\Aut_\F(Q)$-invariant subgroup of $Q$. In particular $M$ is elementary abelian,  $M\le Z(Q)$ and by Lemma~\ref{G action Q} either $M$ has order $p$ and is centralized by $G$ or $M \cong V_0$ as a $\GF(p)G$-module.
Since $Q$ is now a quotient of $\gamma_1(S)$, we need to argue that $|Q| \ge p^{p+1}$.
  Note that $[Q \colon Z(Q)] \geq [Q \colon Q_2] = p^{p-2}$ since $Q$ is non-abelian. Also, $|Z(Q)|>p$ because $Q$ is the $2$-step centralizer and so Lemma~\ref{G action Q} implies $|Z(Q)| \geq p^{p-2}$. Hence, as $p \ge 5$, $|Q|\geq p^{2(p-2)}\geq p^{p+1}$.

Set $\ov Q= Q/\agemO^1(Q)$. As $|Q| \ge p^{p+1}$,  Lemmas~\ref{lem:regular} (iv) and \ref{mc-facts} (iii) yield $|\ov Q|=|Q/\agemO^1(Q)|= |\Omega_1(Q)|=p^{p-1}$. Lemma~\ref{sec-abelian} implies that $\ov Q$ and $\Omega_1(Q)$ are elementary abelian.

Define  $$\kappa: \ov Q \times \ov Q \rightarrow M$$ by $(x\agemO^1(Q), y\agemO^1(Q))\kappa=[x,y]
.$
Since $M$ is elementary abelian and central in $Q$, we have $[x^p,y]=[x,y]^p=1$ for all $x,y \in Q$. In particular, $\agemO^1(Q) \le Z(Q)$ and $\kappa$ is a well-defined,  surjective, $G$-invariant, alternating, bilinear map. Therefore there  is a unique $\GF(p)G$-module homomorphism from $\Lambda^2(\ov Q)$ onto $M$.
Notice that $M$ either has order $p$ or $p^{p-2}$.
Suppose that $G \cong \Sym(p)$.  In this case, $\ov Q/C_{\ov Q}(G')$ is the unique $\GF(p)\Alt(p)$-module of dimension $p-2$ and $\mathrm H^1(\Alt(p), \ov Q/C_{\ov Q}(G'))$ has dimension $1$ by Lemma~\ref{H1dim}. Hence  $\ov Q/C_{\ov Q}(G')$ is uniquely determined as the submodule of dimension $p-1$ of the natural permutation module for $\Alt(p)$ on $p$ points. However, this means that $\Lambda^2(\ov Q)$ is the module described in Lemma~\ref{Symmetric(p)Module} and this   has no quotients isomorphic to $  M$.

Suppose that $G = \PGL_2(p)$.  Then $\ov Q$ is an indecomposable $\GF(p)\PSL_2(p)$-module with socle the trivial $1$-dimensional module and quotient of dimension $p-2$.    By Lemma~\ref{tensor2}, there are unique irreducible quotients of $\Lambda^2(\ov Q)$ of dimensions $1$ and of dimension $p-2$ only if the same is true for $\Lambda^2(\ov Q/C_{\ov Q}(G))$. Since $|\ov Q/C_{\ov Q}(G)|=p^{p-2}$ and $p-2$ is odd, $\ov Q/C_{\ov Q}(G)$ cannot support an alternating bilinear form and thus $|M |=p^{p-2}$.

By Lemma~\ref{tau action}, $\tau$ acts fixed-point-freely on $M$ and on $\ov Q/C_{\ov Q}(G)$. Since $\agemO^1(Q) \le Z(Q)$ and $\agemO^1(Q)$ has index $p$ in the preimage of $C_{\ov Q}(G)$, we get that  the preimage of $C_{\ov Q}(G)$ is abelian. Let $Q^*$ represent the quotient of $\ov Q$ by $C_{\ov Q}(G)$ and we consider the  map from $\kappa^*:\Lambda^2(Q^*)\rightarrow M$.  Let $t_1, \dots, t_{p-2}$ be eigenvectors for the action of $\tau$ on $Q^*$ where $t_j \in \gamma_j(S)\setminus \gamma_{j+1}(S)$.  Since the action of $G$ on $M$ and on $Q^*$ is isomorphic to the action of $G$ on $V_0$, Lemma~\ref{tau action} implies $t_i\tau = t_i^{r^i}$ for $1\le i\le p-2$.  It follows that, as $r^{p-1}=1$, $$[t_1,t_{p-2}]\tau = [t_1^r,t_{p-2}^{r^{(p-2)}}]=[t_1,t_{p-2}]^{(rr^{p-2})}= [t_1,t_{p-2}].$$ Since $C_{M}(\tau)=1$, we deduce that $[t_1,t_{p-2}]=1$.
Thus $\kappa^*(t_1\wedge t_{p-2})= [t_1,t_{p-2}]=1$ and this contradicts Lemma~\ref{non-zero}.  This contradiction proves that $Q$ is abelian. However, $Q$ is not abelian by hypothesis and so we deduce that $\F\ne O^p(\F)$ as claimed.
\end{proof}

\begin{lemma}\label{Orders} \blue{Assume that Hypothesis~\ref{hyp last} holds.} Then
\begin{enumerate}
\item  $\gamma_2(S)$ is abelian and   is  $\Aut_\F(Q)$-invariant;
\item $Z(Q) = \agemO^1(Q)$, $|Q:Z(Q)| =p^{p-1}$ and $Q$ has nilpotency class $2$;
\item $P$ is abelian and  $[Q,Q] = V_0$;
\item $|S| = p^{j(p-1)+1}$ for some $j \ge 2$ and \blue{$S$ has sectional rank $p-1$}; and
\item $\mathcal P(\F)= P^\F$.
\end{enumerate}
\end{lemma}
\begin{proof}  By Lemmas~\ref{O^p} and \ref{F not OpF} we have $\F \ne O^p(\F)$, $\hyp(\F) = P\gamma_2(S)$ and  $\gamma_1(\hyp(\F))  =\gamma_2(S)$. As $O^p(O^p(\F))=O^p(\F)$ and $O^p(\F)$ satisfies Hypothesis \ref{hyp last} except $\gamma_1(\hyp(\F))$ being abelian, Lemma~\ref{F not OpF} implies that $\gamma_2(S)$ is abelian. This is (i).

 Set $X=\langle r^p \mid r\in Q \setminus \gamma_2(S)\rangle$. Let $t\in Q\setminus \gamma_2(S)$. Then $t^p$ is centralized by $t$ and $\gamma_2(S)$. Hence $t^p \in Z(Q)$ and so $X \le Z(Q)$ is a normal subgroup of $S$ and $t \in \Omega_1(Q/X)$.  Since the preimage of $ \Omega_1(Q/X)$ is normal in $S$ and $S$ has maximal class we deduce that $Q/X$ has exponent $p$. Hence $\agemO^1(Q) \le X \le  \agemO^1(Q) $ which means that $X= \agemO^1(Q)$. It follows that $Z(Q) \ge \agemO^1(Q)$. Since $|Q/\agemO^1(Q)|=|\Omega_1(Q)| = p^{p-1}$ by Lemmas~\ref{lem:regular} (iii) and \ref{V =Omega} we can use Lemma~\ref{G action Q} to conclude $Z(Q)= \agemO^1(Q)$ as surely $Z(Q)\ne \gamma_2(S)$.
Also,  $Q/\agemO^1(Q)$ is elementary abelian by Lemma~\ref{sec-abelian}. Hence $Q' \le Z(Q)$ and $Q$ has nilpotency class $2$. This proves (ii).

Let $x,y \in Q$, then, by (ii),  $x^p \in Z(Q)$ and so $1=[x^p,y]=[x,y]^p$. Hence $Q' \le \Omega_1(Q)=V$.
Now the commutator map $\kappa: Q/X \times Q/X \rightarrow V$ again determines a $\GF(p)G$-module homomorphisms $\kappa^*:\Lambda^2(Q/X) \rightarrow V$. Now $Q/X$ has $\gamma_2(S)/X \cong V_0$ as its unique submodule and we know that $\Lambda^2(\gamma_2(S)/X)$ is in the kernel of $\kappa^*$. We have $\Lambda^2(Q/X)/\Lambda^2(\gamma_2(S)/X) \cong \gamma_2(X)\cong V_0$ as $\GF(p)G$-modules. Hence the image of $\kappa^*$ in $V$ is isomorphic to $V_0$. If $P$ is extraspecial, then $Z(S)$ is the unique proper subgroup of $V$ which is $G$-invariant. Hence $P$ is abelian and $Q'=V_0$. Thus (iii) holds.

 As for part (iv), if  $Q= Q_1 > Q_2> \dots > Q_{\ell -1} >Q_\ell=1$ is an $\Aut_\F(Q)$-chief series in $Q$, then, by part (i), $Q_2=\gamma_2(S)$ and, by part (iii), $Q_{\ell -1} = Q'$ has order $p^{p-2}$.  Hence Lemma \ref{G action Q}  implies that $|Q|=p^{j(p-1)}$ for some $j\geq 1$. Thus Lemma \ref{V =Omega} yields  $|S| =p^{j(p-1) +1}$ for some $j\geq 2$. \blue{In particular $|S| \geq p^{2p}$ and \cite[Theorem A]{pearls} implies that $S$ has sectional rank $p-1$.}

Note that $P\gamma_2(S)$ is $\Aut_\F(S)$-invariant and $\Aut_\F(S)$ centralizes $\gamma_1(S)/\gamma_2(S)$. Hence $P\gamma_2(S)$ and $\gamma_1(S)$ are the only $\Aut_\F(S)$-invariant maximal subgroups of $S$. Therefore every $\F$-pearl is contained in $P\gamma_2(S)$ and from this it follows that $\mathcal P(\F)= \{P^x\mid x \in S\}= P^\F$. Hence (v) holds.
\end{proof}

%% file: Proofs.tex
\section{The proofs of Theorems~A and ~C}\label{sec:proofs}

In this short section we prove Theorems~\ref{compound} and ~\ref{non exc} as well as Corollary~\ref{cor: no pearls}. We begin with Theorem~\ref{non exc}.

\begin{customthm}{C}\label{non exc}
Suppose that $p$ is an odd prime, $S$ is a maximal   class $p$-group of order at least $p^4 $ and $\F$ is a saturated fusion system on $S$. Assume that $S$ is not exceptional,   $\gamma_1(S)$ is not abelian and $\F \ne N_\F(\gamma_1(S))$. Then one of the following holds:
\begin{enumerate}
\item $\E_\F=\mathcal P_a(\F)$,  $|S: \hyp(\F)|\le p$ with $|S: \hyp(\F)|= p$  if and only if $|S|=p^{j(p-1)+1}$  for some $j \ge 2$.   Furthermore,  \blue{either $O^p(\F)$ is simple and exotic or $p=3$ and   $O^3(\F)$ is realized by $\PSL_3(q)$ for suitable prime powers $q$.}
\item $p \ge 5$, $\E_\F=    \mathcal P_e(\F)$,     $O_p(\F)=Z(S)$, $|S: \hyp(\F)|\le p$ with $|S: \hyp(\F)|= p$  if and only if $|S|=p^{j(p-1)+2}$  for some $j \ge 2$.  Furthermore,    $O^p(\F/Z(S))$ is simple  and exotic.
\item $p \ge 5$, $\E_\F= \mathcal P_a(\F) \cup \{\gamma_1(S)\}$,  $O_p(\F)=1$, $\F \ne O^p(\F)$ and
\begin{enumerate}
\item $ \mathcal P_a(\F)$ is a single $\F$-class, $|S|=p^{j(p-1)+1}$ for some $j \ge 2$ and $S$ has sectional rank $p-1$;
\item $\Out_\F(\gamma_1(S)) \cong \Sym(p)$ or $\PGL_2(p)$;
\item$Z(\gamma_1(S))=\agemO^1(\gamma_1(S))$ has index $p^{p-1}$ in $\gamma_1(S)$, $\gamma_1(S)'< \Omega_1(\gamma_1(S))$ has order $p^{p-2}$ and  $\gamma_2(S)$ is abelian  but not elementary abelian;
\item  every composition factor of $\Aut_\F(\gamma_1(S))$ on $\gamma_1(S)$ has order $p$ or $p^{p-2}$ and the composition factors of order $p$ are centralized by the automorphism group $\Aut_\F(\gamma_1(S))$;
\item for $P \in \mathcal P_a(\F)$, $\hyp(\F)= P\gamma_2(S)$,   $O^p(\F)$ is a saturated fusion system on $P\gamma_2(S)$,  and $\Aut_{O^p(\F)}(\gamma_2(S)) \cong \Sym(p)$ or $\PGL_2(p)$.
 \end{enumerate}
\end{enumerate}
Furthermore, in all cases  $\Out_\F(S)$ is a Hall $p'$-subgroup of $\Out(S)$ and is cyclic of order $p-1$ and, if  $| S |  = p^n$, and $P\in \mathcal P(\F)$, then either $\mathcal P(\F)=P^S$ or  $\mathcal E_\F=\mathcal P(\F)$ and $n  \equiv \eps \pmod { p-1 }$
where $\eps = 0$ if $P\in \mathcal P_a(\F)$ and  $\eps = 1$ if $P\in \mathcal P_e(\F)$.\end{customthm}

\begin{proof} Suppose that $\F$ is a saturated fusion system on $S$ with $|S|\ge p^4$. In addition, suppose $S$  has maximal nilpotency class,  is not exceptional and has $Q=\gamma_1(S)$ non-abelian. Assume that $\F \ne N_\F(Q)$. By  Theorem~\ref{MT1},  $\E_\F \subseteq \mathcal P(\F)\cup \{Q\}$ and so $\Pp(\F)$ is  non-empty.

\blue{Assume that $p=3$. Then the saturated fusion systems are presented in Appendix~\ref{AppB}. Since we require $\gamma_1(S)$ to be non-abelian, the discussion after Theorem~\ref{thm:black} says we only need to inspect Table~\ref{this one} for $S= \mathrm B(2\ell;1,0,2), \mathrm B(2\ell, 1,0,0)$, $\ell \ge 3$ and $\B(2k+1,1,0,0)$ with $k \ge 2$. This shows that $\mathcal E_\F=\mathcal P_a(\F)$, $O_p(\F)=1$, and $|S:\hyp(\F)|=3$ if and only if $|S|=3^{2k+1}$. Furthermore, either $S$ is one of  $\mathrm B(2\ell;1,0,2)$ or $\mathrm B(2\ell, 1,0,0)$ with $\ell \ge 3$ and $\F$ is simple and exotic or $S= \B(2k+1,1,0,0)$ and $O^3(\F)$ is realised by $\PSL_3(q)$ for suitable $q$. Hence (i) holds when there are no saturated fusion systems $\F$ with extraspecial $\F$-pearls or with $\gamma_1(S)\in \E_\F$.

Assume from now on that $p\ge 5$.} By Lemma~\ref{lem:mu image} (vi), either $\mathcal P(\F)= \mathcal P_a(\F)$ or $\mathcal P(\F)=\mathcal P_e(\F)$. Furthermore, \blue{if $\E_\F =\mathcal P_e(\F) \cup \{Q\}$, then Hypothesis~\ref{hyp last} holds and this contradicts Lemma~\ref{Orders} (iii).} Thus $\E_\F \ne\mathcal P_e(\F) \cup \{Q\}$.  By Lemma~\ref{lem:mu image}, $\Out_\F(S)$ is a Hall $p'$-subgroup of $\Out(S)$ and   is cyclic of order $p-1$. This, together with \cite[Theorem 3.15]{pearls} establishes the chaser to the theorem.

Suppose that $\E_\F= \mathcal P_a(\F)$. Then Lemma~\ref{pearls5} gives (i).

If $\E_\F= \mathcal P_e(\F)$, then $O_p(\F)=Z(S)$. Since $P\le \hyp(\F)$ and $\hyp(\F)$ is normal in $S$, $|S:\hyp(\F)|\le p$. Furthermore, $|S: \hyp(\F)|= p$  if and only if $s_1\gamma_2(S)$ is centralized by $\Aut_\F(S)$ which is if and only if $n-2\equiv 0 \pmod {p-1}$ by  Lemma~\ref{tau action} (ii). Hence $O^p(\F) \subset \F$ if and only if $|S|=p^{j(p-1)+2}$ with $j \ge 1$.  
 In addition, $P/Z(S)$ is an abelian $\F/Z(S)$-pearl. In the case that $Q/Z(S)$ is non-abelian,   $\F/Z(S)$ satisfies (i) with $|S/Z(S)|=p^{n-1}\ge p^4$. So, in particular, $O^p(\F/Z(S))$ is simple and exotic and $j \ge 2$ when $O^p(\F)\ne \F$.
   \blue{On the other hand, if $Q/Z(S)$ is abelian, then applying \cite[Theorem 2.8 (a)(i) and a(iv)]{p.index}   delivers $O^p(\F/Z(S))$ is simple and exotic. Suppose that $O^p(\F) \subset \F$ and $j=1$. Then $|S|= p^{p+1}$. We have $Q'= Z(S)$. In addition, $O^p(\F)$ has extraspecial pearls and $|\hyp(\F)|=p^{p}$ and so \cite[Theorem A]{pearls} applied $O^p(\F)$ yields $\gamma_1(\hyp(\F))=\gamma_2(S)$ is elementary abelian.
   Let $\tau\in \Out_\F(S)$ have order $p-1$. Then $\tau$ centralizes $s_1\gamma_2(S)$ and  $Z(S)$.   
   Choose $k$ maximal so that $[s_1,s_k]\ne 1$. Then $[s_1,s_k]\in Z(S)$ and
   $$[s_1,s_k]=[s_1,s_k]\tau=[s_1\tau,s_k\tau].$$
Since $1<k< n-1=p$ and $\tau$ has order $p-1$, $s_k\tau= s_k^bg_{k+1}$ for some $b \in \GF(p)^\times \setminus\{1\}$ and $g_{k+1}\in \gamma_{k+1}(S)$ by Lemma~\ref{centralizer auto}. We also have $s_1\tau= s_1 g_2$ for some $g_2\in \gamma_2(S)$.
Since $\gamma_2(S)$ is abelian, the maximal choice of $k$ gives $$[s_1,s_k]=[s_1\tau,s_k\tau]=[s_1g_2,s_k^bg_{k+1}]=[s_1,s_k]^b,$$
a contradiction.}

Suppose that $\E_\F= \mathcal P_a(\F)\cup\{Q\}$.   Then Hypothesis~\ref{hyp last} holds.  That $\F \ne O^p(\F)$ is just Lemma~\ref{F not OpF}. Now part (iii)(a) is Lemma~\ref{Orders}(iv) and (v), part (iii)(b) is Lemma~\ref{app COS2}, part (iii)(c) is Lemma~\ref{Orders}(i), (ii), (iii) and (iv), part (iii)(d) is Lemma~\ref{G action Q}, finally part (iii)(e) follows from Lemmas~\ref{O^p} and \ref{Orders}(i) as clearly  $\Aut_{O^p(\F)}(\gamma_2(S)) \cong \Sym(p)$ or $\PGL_2(p)$ by (iii)(b).
\end{proof}

\begin{customthm}{A}  Suppose that $\F$ is a reduced saturated fusion system on a $p$-group $S$ of maximal class \blue{of order at least $p^4$.}  Then  one of the following statements  holds.
\begin{enumerate}
\item $\gamma_1(S)$ is non-abelian,   $S$ is not exceptional,  $\mathcal \E_\F=\mathcal P_a(\F)$, and $\F$ is simple and exotic.
\item $\gamma_1(S)$ is non-abelian, $S$ is exceptional and either  \begin{enumerate}\item
$p \ge 5$ and $\F=\F_S(\G_2(p))$;
\item $p=5$, $S$ is isomorphic to a Sylow $5$-subgroup of $\G_2(5)$ and $\F=\F_S(G)$ where $G$ is one of the sporadic simple groups $\Ly, \HN$ or $\B$;
\item $p=7$,  $S$ is isomorphic to a Sylow $7$-subgroup of $\G_2(7)$ and either $\F$ is exotic (20 examples) or $\F= \F_S(\Mo)$ where $\Mo$ denotes the monster; or
\item $p \ge 11$, $S$ is uniquely determined of order $p^{p-1}$, $\mathcal P(\F)= \mathcal P_a(\F)\not=\emptyset$ and, if $\gamma_1(S)$ is $\F$-essential, then $\Out_\F(S) \cong \GF(p)^\times \times \GF(p)^\times$, $O^{p'}(\Out_\F(\gamma_2(S))) \cong \SL_2(p)$ and $\gamma_1(S)/Z(\gamma_1(S))$ is the  $(p-3)$-dimensional irreducible $\GF(p)\SL_2(p)$-module.
\end{enumerate}
\item $\gamma_1(S)$ is abelian and $\F$ is described by Theorem~\ref{TheoremAbelian}. \end{enumerate}
\end{customthm}

\begin{proof} Suppose that $\F$ is reduced ($O_p(\F)=1$ and $\F=O^p(\F)=O^{p'}(\F)$). In addition we may assume that $\gamma_1(S)$ is non-abelian.   If $S$ is not exceptional, then, as $\F$ is reduced,  Theorem~\ref{non exc} (ii) and (iii) cannot hold as in the first case $O_p(\F) \ne 1$ and in the second case $\F \ne O^p(\F)$. Hence Theorem~\ref{non exc} (ii) holds and in particular $\E_\F$ consists of abelian $\F$-pearls.   This is point (i) of the theorem.

Suppose that $S$ is exceptional.  Then Theorem~\ref{Raul app} applies. In particular,  this immediately gives parts (ii)(a) and  (ii)(d) of the theorem, completing the proof for $p\geq 11$. For $p=5$ and $\F\neq \F_S(\mathrm G_2(5))$, Theorem~\ref{Raul app} (i)(a)($\alpha$), (i)(a)($\beta$) and (ii) cannot hold, as $O_5(\F)\ne 1$ in these cases.  Hence Theorem~\ref{Raul app} (i)(a)($\delta$) holds. We deduce that $\F$ is not realized by $\Aut(\HN)$ and this gives point (ii)(b) of the theorem.

Finally, when $p=7$ and $S $ is isomorphic to a Sylow $7$-subgroup of $\mathrm G_2(7)$, then we use \cite[Table 5.1]{G2p} to determine how many of the examples are reduced.  There are 20 exotic examples all appearing as subsystems of $\F_S(\mathrm M)$. This is point (ii)(c) of the theorem and completes the proof.
\end{proof}

\begin{customcor}{\ref{cor: no pearls}}Let $p$ be a prime,  $S$ be a $p$-group of maximal   class and  let $\F$ be a  saturated fusion system on $S$ with $O_p(\F)=1$. If $\mathcal P(\F)$ is empty, then $S$   is isomorphic to a Sylow $p$-subgroup  of $\G_2(p)$ and either
 \begin{enumerate}
 \item $\F= \F_{S}(\G_2(p))$;
 \item  $p=5$ and  $\F=\F_S(G)$ where $G= \Ly, \HN, \Aut(\HN)$ or $\B$;
 \item  $p=7$,  $\F$ is exotic and the $\F$-essential subgroups are $C_S(Z_2(S))$ and $\gamma_1(S)$, with  $\Out_\F(C_S(Z_2(S)))\cong \GL_2(7)$, $\Out_\F(\gamma_1(S)) \cong 3\times 2^. \Sym(7)$,   and $\Out_\F(S) \cong \GF(7)^\times \times  \GF(7)^\times$.
 \end{enumerate} \end{customcor}

\begin{proof}[Proof of Corollary~\ref{cor: no pearls}.] Suppose  $\F$ has no  $\F$-pearls and  $O_p(\F)=1$. Theorem~\ref{non exc} implies that \blue{either $\gamma_1(S)$ is abelian} or $S$ is exceptional. \blue{If $\gamma_1(S)$ is abelian, then $\mathcal P(\F)$ is non-empty by \cite[Lemma 2.3]{p.index}. So assume that $S$ is exceptional.}
 Then examining Theorem~\ref{Raul app} we see that part (i)(a) must hold. In particular, $S$ is isomorphic to a Sylow $p$-subgroup of $\mathrm G_2(p)$. Using  \cite[Theorem 1.1 and Table 5.1]{G2p} yields the result.
\end{proof}

%% file: Examples.tex
\section{A series of examples with non-abelian $2$-step centralizer}\label{sec:examples}
Let $p$ be an odd prime, then by Dirichlet's Theorem~\cite{Dirichlet} there exists a prime $r$ such that $r \equiv 1 \pmod{p^k}$.   Let $T $ be a Sylow $p$-subgroup of $\GF(r^p)^\times$ and   $ M$ be the monomial subgroup of $\GL_p(r^p)$ which has all matrix entries in $T$.   Then $T$ is a cyclic group  and  $|M|= |T|^pp!$. Notice that $Z(M)\cong T$ and consists of scalar matrices. We denote by $D$  the subgroup of diagonal matrices of $M$. Let $R\in \syl_p(M)$.  We claim $R/Z(M)$ has maximal class. Let $\pi$ be the permutation matrix corresponding to the permutation $(1,2, \dots, p)$: $$\pi =\left( \begin{smallmatrix}
0&1&0&\dots &0\\
0&0&1&\dots&0\\
\vdots&\vdots&\vdots&\vdots&\vdots\\
0&0&0&\dots&1\\
1&0&0&\dots &0\end{smallmatrix}\right). $$ We may assume that $\pi \in R$. Then a typical element of $R$ has the form $$\mathrm{diag}(d_1, \dots,d_p)\pi^j\in D\langle \pi\rangle$$ where $d_i \in T$ and $1\le j \le p$.  Now a   calculation shows that the set of matrices in $R$ which centralize  $\pi$ mod $Z(M)$ is $$C=\{\mathrm{diag}(d, de,de^2,\dots ,de^{p-1})\pi^j\mid d,e\in T, e^p=1, j \in \mathbb Z\} \le \GL_p(r). $$
Because $(C\cap D)/Z(M)=\{\mathrm{diag}(1,e,\dots,e^{p-1})Z(M)\mid e \in T, e^p=1\}$ has order $p$,   Lemma~\ref{p2centralizer} yields $R/Z(M)$ has maximal class. \blue{ Since $\pi \in \SL_p(r)$,  and $$\det(\mathrm{diag}(1,e,\dots,e^{p-1}))= e^{p(p-1)/2}=1$$ we see the image of $C$ in $\PGL_p(r)$ is contained in $\PSL_p(r)$.} Using \cite[Section 4]{AlperinFong}  we record the well-known fact:

\begin{lemma}
If $p$ divides $r-1$, then the Sylow $p$-subgroups of $G=\PGL_p(r)$ and $G=\PSL_p(r)$ have maximal class. Furthermore, in both cases,\blue{we have $\Aut_G(\wt C)\cong \SL_2(p)$ where $\wt C$ is the image of $C$ in $\PGL_2(r)$.}\qed
\end{lemma}

Since $p^{k}$ is the highest power of $p$ which  divides $r-1$,  $p^{k+1}$ exactly divides $r^p-1$  (see \cite[Lemmas 4.1(iv) and 4.2(ii)]{MPR}). Let $\sigma:\GF(r^p)\rightarrow \GF(r^p)$ be the field automorphism given by $x \mapsto x^r$ and extend $\sigma$ to the standard Frobenius automorphism of $\GL_p(r^p)$ which acts as $\sigma$ on each matrix entry. We denote this automorphism by $\sigma$ as well. Define  $M^*$ to be the semidirect product of $M$ and $\langle \sigma\rangle$,  identify $M$ and $\langle \sigma\rangle$ with their images in $M^*$ and let $R^*= R\langle \sigma\rangle$. Consider the subgroup $C^*$ of $R^*$ which centralizes $\pi$ mod $Z(M)$.  Since $\pi \in \GL_p(r)$, we see $$C^*=C\langle \sigma\rangle=\{\mathrm{diag}(d, de,de^2,\dots ,de^{p-1})\pi^j\sigma^k\mid d,e\in T, e^p=1, j,k\in \mathbb Z\} $$ and so $|C^*/Z(M)|=p^3$.
Let $R>R_0\ge Z(M) $ be such that $R_0/Z(M)$ is a Sylow $p$-subgroup of $\PSL_p(r^p)$ and put $D_0= D \cap R_0$. Then $R^*>R>R_0$ and $R^*/R_0=\langle R_0\sigma, R_0\mathrm{diag}(d,1,\dots,1)\rangle$ which has order $p|T|$.  Now let $R_1= R_0\langle \mathrm{diag}(d,1,\dots, 1)\sigma\rangle$.  Then $\pi\in R_1$,  $C^* \cap R_1=C$ and so $R_1/Z(M)$ has maximal class by Lemma~\ref{p2centralizer} and $$\gamma_1(R_1/Z(M))= \langle D_0/Z(M), Z(M)\mathrm{diag}(d,1,\dots, 1)\sigma\rangle$$ is not abelian as $\sigma$ does not centralize $D_0/Z(M)$.

\begin{example}\label{ExThmNonExc} \blue{Assume that $p>3$ is a prime.} Put
$X_1= \SL_p(r^p)$, $X = X_1R_1$ and $\ov X= X/Z(X)$. Set $\F=\F_{\ov R_1}(\ov X)$. Then \begin{enumerate}
 \item $\ov{R_1}$ has maximal class and $\gamma_1(\ov {R_1})$ is non-abelian;
  \item $ \ov C$ and $\gamma_1(\ov{R_1})$ are $\F$-essential with $\ov C$ an abelian $\F$-pearl;
   \item $\Aut_\F(\ov C) \cong \SL_2(p)$, $\Out_\F(\gamma_1(\ov{R_1}))\cong\Sym(p)$; and \item   $O_p(\F)=1$. \end{enumerate}
       In particular,   there exist   realizable saturated fusion systems which satisfy Theorem~\ref{non exc} (iv) with $\gamma_1(S)$ non-abelian and $\Out_\F(\gamma_1(S))\cong \Sym(p)$.
        By \cite[Theorem 6.2]{parkersemerarocomputing}, the subfusion systems generated by $N_{\Aut_\F(\ov{R_1})}(\ov C)$ and $\Aut_\F(\ov C)$ is also saturated and this gives an  example of Theorem~\ref{non exc} (i).
        \end{example} 

%% file: AppendixA.tex
\section[maximal class $p$-groups with $\gamma_1(S)$ abelian]{Saturated fusion systems on maximal class $p$-groups with $\gamma_1(S)$ abelian}\label{AppA}

In this appendix   we curate a list of the reduced fusion systems on maximal class $p$-groups which have $\gamma_1(S)$ abelian. Thus we present a synopsis of the main results from \cite{p.index,p.index2, p.index3}. For this, we first establish some further notation from \cite{p.index,p.index2, p.index3} which is honed to our special situation.  First there are the sets of $\F$-essential subgroups in the sets  $\mathcal H$ and $\mathcal B$ which are introduced in \cite{p.index}.
 We continue to use the notation introduced in Subsection~\ref{SS3.1} so that
 $S/\gamma_2(S)=\langle x\gamma_2(S), s_1\gamma_2(S)\rangle.$ \blue{We also recall the definition of the set $\mathcal P(\F)=\mathcal P_a(\F)\cup\mathcal P_e(\F)$ of $\F$-pearls from Definition~\ref{pearl:def}.}
 If $\mathcal P(\F)$ is non-empty, then we may additionally assume that $x$ has order $p$. For  $0 \le i \le p-1$, define $$H_i=  \langle xs_1^i\rangle Z(S)\text{ and }B_i= \langle xs_1^i\rangle Z_2(S).$$
Then $\mathcal H= \bigcup_{i=0}^{p-1}H_i^S$ and $\mathcal B=\bigcup_{i=0}^{p-1} B_i^S.$
 In our notation we have \begin{eqnarray*}\mathcal  P_a(\F)&=&\{H\in \mathcal H\mid H \text { is } \F\text{-essential}\}\\\mathcal P_e(\F)&=&\{B\in \mathcal B\mid B \text { is } \F\text{-essential}\}.\end{eqnarray*}
We also let $\mathcal P^i_a(\F) $ consist of  those $\F$-pearls which are $S$-conjugate to $H_i$ and $\mathcal P^i_e(\F)$ contain the $\F$-pearls $S$-conjugate to $B_i$. Finally, for $I \subseteq \mathbb Z/p\mathbb Z$ and $b \in \{a,e\}$ define $$\mathcal P^I_b(\F)= \bigcup _{i\in I}\mathcal P_b^i(\F) \text{ and }\mathcal P_b^*(\F)= \bigcup _{i=1}^{p-1}\mathcal P_b^i(\F).$$

Set $G= \Aut_\F(\gamma_1(S))$.  Then, we define $$\mu_1: N_G(\Aut_S(\gamma_1(S))) \rightarrow \Delta$$ by $$\alpha \mu_1= \beta \mu$$ where $\beta$ in $\Aut_\F(S)$, $\beta|_{\gamma_1(S)}=\alpha$ \blue{ and $\mu$ is as defined at the beginning of Section~\ref{sec: non excep g1}}.  In \cite[page 218]{p.index2} it is explained why $\mu_1$ (denoted $\mu_A$) is well-defined. If $X$ is a finite cyclic group and $n$ divides $|X|$, then $\frac 1 n X$ denotes the unique subgroup of $X$ which has index $n$. Almost all the other notation that we require can be found in the introduction to Section~\ref{sec: non excep g1} one exception being the cyclic subgroups of $\Delta$ defined as  $$\Delta_{k/\ell}=\{(u^\ell,u^k)\mid u \in \mathbb Z/p\mathbb Z^\times\}$$ whenever $k$ and $\ell$ are coprime.

Suppose that $\F$ is a reduced saturated fusion system on $S$ where $S$ has maximal class, $\gamma_1(S)$ is abelian and $|S|=p^n > p^3$.  Since we are primarily focussed on the cases when $\gamma_1(S)$ is $\F$-essential, we can sift through the results in  \cite{p.index2, p.index3}.

By Lemma~\ref{mc-facts} (iv) we know that either $|\Omega_1(\gamma_1(S))| \le p^{p-1}$ or $|S|= p^{p+1}$ and $|\Omega_1(\gamma_1(S))|=p^p$. Furthermore, Lemma~\ref{p2centralizer} implies that $\Aut_S(\gamma_1(S))$ acts on $\Omega_1(\gamma_1(S))$ with a single Jordan block.
By
\cite[Proposition 3.7]{p.index2},  if $\Omega_1(\gamma_1(S))$ has order at most $p^p$, then $\Aut_S(\gamma_1(S))$ operates on $\Omega_1(\gamma_1(S))$ with a single Jordan block and so from the results in \cite{p.index2, p.index3} we just need to collate the ones with $|\Omega_1(\gamma_1(S))| \le p^p$.  If  $\gamma_1(S)> \Omega_1(\gamma_1(S))$, then   $|\Omega_1(\gamma_1(S))|\ne p^p$ and so \cite[Theorem A]{p.index3}  implies
$|\Omega_1(\gamma_1(S))|=p^{p-1}$. In particular, if $\Omega_1(\gamma_1(S))$ is irreducible, then $\gamma_1(S)$ is homocyclic. In Table~\ref{Tab2.1} we present conditions on the structure of $\Out_\F(S)$ which determine the various possibilities for constellations of $\F$-pearls in a reduced fusion system.  This table is transcribed from \cite[Theorem 2.8, Table 2.1]{p.index2}.

\begin{table}[h] \resizebox{\textwidth}{!}{%
\begin{tabular}{|c|cccc|}
\hline
&$(N_G(\Aut_S(\gamma_1(S)))\mu_1$&$G= O^{p'}(G)X$&$|\gamma_1(S)|=p^m $&Pearls\\
\hline
I&$\Delta$&$X=N_G(\Aut_S(\gamma_1(S)))$&$m\equiv 0 \pmod {p-1}$&$\mathcal P_a^0(\F)\cup \mathcal P_e^{*}(\F)$\\
\hline
II&$\Delta$&$X=N_G(\Aut_S(\gamma_1(S)))$&$m\equiv p-2\pmod {p-1}$&$\mathcal P_a^*(\F)\cup \mathcal P_e^{0}(\F)$\\
\hline
III&$\ge \Delta_{-1}$&$X=(\Delta_{-1})\mu_1^{-1}$&$m\equiv p-2\pmod {p-1}$&$\mathcal P_a^I(\F) $, $I \subseteq \mathbb Z/p\mathbb Z$\\
&$\ge \Delta_{-1}$&$X=(\Delta_{-1})\mu_1^{-1}$& \noindent\rule{1.5cm}{0.4pt} &$\mathcal P_a^0(\F) $\\
\hline
IV&$\ge \Delta_0$&$X=(\Delta_{0})\mu_1^{-1}$&$m\equiv 0 \pmod {p-1}$&$\mathcal P_e^I (\F) $, $I \subseteq \mathbb Z/p\mathbb Z$\\
 &$\ge \Delta_0$&$X=(\Delta_{0})\mu_1^{-1}$&\noindent\rule{1.5cm}{0.4pt}&$\mathcal P_e^0(\F) $ \\
 \hline
\end{tabular}}
\caption{Configurations of pearls determined by $(N_G(\Aut_S(\gamma_1(S)))\mu_1$ where $G= \Aut_\F(\gamma_1(S))$ and $X \le N_G(\Aut_S(\gamma_1(S)))$.}\label{Tab2.1}
\end{table}

\begin{theorem}\label{TheoremAbelian} Suppose that $p$ is an odd prime, $S$ has maximal class of order at least $p^4$ and $\gamma_1(S)$ is abelian. If $\F$ is a reduced saturated fusion system on $S$, then $\E_\F \subseteq \{\gamma_1(S)\}\cup \mathcal P(\F)$ and $\mathcal P(\F)$ is non-empty. Furthermore, one of the following holds:
\begin{enumerate}
\item $\E_\F= \mathcal P(\F)$.
\item $|\Omega_1(\gamma_1(S))| < p^{p-1}$, $\gamma_1(S)= \Omega_1(\gamma_1(S))$, the candidates  for $\Aut_\F(\gamma_1(S))$ and the configurations of $\F$-pearls  are listed in the first section of Table~\ref{Tab1}.
\item  $|\Omega_1(\gamma_1(S))|=p^{p}$, $\gamma_1(S)= \Omega_1(\gamma_1(S))$, $|S|=p^{p+1}$ and the possibilities for $\Aut_\F(\gamma_1(S))$   and the configurations of $\F$-pearls  are listed in the second section of Table~\ref{Tab1}.
\item $|\Omega_1(\gamma_1(S))| = p^{p-1}$, $\Omega_1(\gamma_1(S))$ is irreducible as a $\GF(p)\Aut_\F(\gamma_1(S))$-module, $\gamma_1(S)$ is homocyclic of order $p^{a(p-1)}$ for some $a \ge 1$ and  the possibilities for $\Aut_\F(\gamma_1(S))$ and    the configurations of $\F$-pearls  are listed in the third section of Table~\ref{Tab1}.
\item $|\Omega_1(\gamma_1(S))|=p^{p-1}$,  $\Omega_1(\gamma_1(S))$ is indecomposable but not irreducible as a $\GF(p)\Aut_\F(\gamma_1(S))$-module, $\gamma_1(S)$ not necessarily homocyclic  and  the possibilities for $\Aut_\F(\gamma_1(S))$ and    the configurations of $\F$-pearls are listed in the fourth  section of Table~\ref{Tab1}.
\end{enumerate}
\end{theorem}
\begin{proof} This mostly follows from our previous discussion and by combining the results of \cite[Corollary 2.10 and Theorem 4.1]{p.index2} with  \cite[Theorem A]{p.index3}. However,  lines 29, 30 and 33 and 34  in Table~\ref{Tab1} need   further explanation to confirm column 6, where the possibilities for $\F$-pearls are described. So suppose that $\F$ is reduced.  For lines 29 and 33, we note that if IV from Table~\ref{Tab2.1} holds, then $\mathcal P(\F)= \mathcal P_e(\F)$ and $Z(S)= O_p(\F)$ which is impossible. The possibility that III holds when considering lines 30 and 34 leads to $\Aut_\F(S)\mu = \Delta_{-1}$.   Using Lemma~\ref{action} we obtain  $[\gamma_1(S), \Aut_\F(S)] \le \gamma_2(S)$ and so, as $[\gamma_1(S), O^{p'}(\Aut_\F(\gamma_1(S))]= \gamma_2(S)$, we have $[\gamma_1(S), \Aut_\F(\gamma_1(S))] =\gamma_2(S)$. Now applying \cite[Lemma 2.7 (b)]{p.index2}  shows that $\F \ne O^p(\F)$, which contradicts $\F$ being reduced.
    \end{proof}

\renewcommand{\arraystretch}{1.1}
\begin{table}[H]\label{Tab1} \resizebox{\textwidth}{!}{%
\begin{tabular}{|c|lclcc|}
\hline
Row&$p$&$Y=O^{p'}(\Aut_\F(\gamma_1(S)))$&$\Omega_1(\gamma_1(S)) $&$(N_Y(\Aut_\F(S)))\mu$&Pearls\\
\hline
1&$p$&$\SL_2(p)$ &$p^{e+1}\cong \VV_e$&$\{(u^2,u^{e}) \mid u\in \mathbb Z/p\mathbb Z^\times \}$& III or IV\\
&&& $e \le p-4$,  odd&&\\
2&$p$&$\PSL_2(p)$ &$p^{e+1}\cong \VV_e$&$\{(u^2,u^{e}) \mid u\in \mathbb Z/p\mathbb Z^\times\}$& III or IV\\
&&& $e \le p-5$,   even&&\\
3&$p$&$\PSL_2(p)$ &$p^{p-2}\cong \VV_{p-3}$&$\frac 1 2 \Delta _{-1}$& II,  III or IV\\
4&$p$&$\Alt(p)$&$p^{p-2}$&$\frac 1 2 \Delta_{-1}$& II,  III or IV\\
5&$7$&$2^{.}\Alt(7)$&$7^4$&$\Delta_{3/2}$&III or IV\\
6&$11$&$\J_1$& $11^{7}$&$\Delta_3$&III or IV\\
\hline\hline
7&$p$&$\PSL_2(p)$&$p^{p}=\VV_{p-1}$&$\frac 1 2 \Delta_1$&III or IV\\
8&$p$&$\PSL_2(p)$&$p^{p}=\substack{{\:}\\{\;}\\\VV_0\\\VV_{p-3}\\\VV_0\\{\;}}$&$\frac 1 2 \Delta_0$&III\\
9&$p$&$\Alt(p)$&$p^{p}= \substack{{\;}\\1\\{p-2}\\{1}\\{\;}}$&$\frac 1 2 \Delta_0$&III\\
10&$p$&$\Alt(p+1)$&$p^p$&$\frac 1 2 \Delta_0$&III or IV\\
11&$p$&$ Y\le  O^{p'}((p-1)\wr \Sym(p))$& $p^p$&\noindent\rule{1.5cm}{0.4pt} &III or IV \\
&& $Y/O_{p'}(Y)\cong \Alt(p) $&&&\\
12&$p$&$ Y\le  O^{p'}((p-1)\wr \Sym(p))$&$p^p$& \noindent\rule{1.5cm}{0.4pt} &III or IV\\
&&$|Y/O_{p'}(Y)|=p$&&&\\
13&$7$&$\PSU_3(3)$& $7^7$&$\frac 1 2 \Delta_0$&III or IV\\
14&$7$&$\SL_2(8)$& $7^7$&$\frac 1 3 \Delta_1$&III or IV\\
15&$7$&$\Sp_6(2)$& $7^7$&$\Delta_3$&III or IV\\

\hline
\hline
16&$p$&$\SL_2(p)$&$p^{p-1}=\VV_{p-2}$&$\{(u^2, u^{-1}) \mid u\in \mathbb Z/p\mathbb Z^\times \}$&I, III or IV\\
17&$5$&$2^{.}\Alt(6)$& $5^4$&$\Delta_{1/2}$&I, III or IV\\
18&$5$&$4\circ 2^{1+4}.\Alt(6)$& $5^4$& \noindent\rule{1.5cm}{0.4pt}&I, III or IV\\
19&$5$&$ 2^{1+4}_-.\Alt(5)$& $5^4$& \noindent\rule{1.5cm}{0.4pt}&I, III or IV \\
20&$5$&$4\circ 2^{1+4}_-.\Alt(5)$& $5^{4}$& \noindent\rule{1.5cm}{0.4pt}&I, III or IV\\
21&$5$&$ 2^{1+4}_-.5$& $5^{4}$& \noindent\rule{1.5cm}{0.4pt}&I, III or IV\\
22&$7$&$6^{.}\PSL_3(4)$&$7^6$&$\{(u^2, w)\mid u, w \in \mathbb Z/p\mathbb Z\}$&I, III or IV\\
23&$7$&$6^{.}_1\PSU_4(3)$&$7^6$&$\{(u^2, w)\mid u, w \in \mathbb Z/p\mathbb Z\}$&I, III or IV \\
24&$7$&$\PSU_3(3)$& $7^6$&$\frac 1 2 \Delta_1$&I, III or IV\\
25&$11$&$\PSU_5(2)$&$11^{10}$&$\frac 1 2 \Delta_2$&I, III or IV\\
26&$11$&$2^{.}\mathrm M_{12}$& $11^{10}$,  $11^{10}$&$\Delta_{1/2}$,  $\Delta_{7/2}$&I, III or IV\\
27&$11$&$2^{.}\mathrm M_{22}$& $11^{10}$,  $11^{10}$&$\Delta_{1/2}$,  $\Delta_{7/2}$&I, III or IV\\
28&$13$&$\PSU_{3}(4)$&$13^{12}$&$\frac 1 3 \Delta _1$&I, III or IV\\
\hline
\hline
29&$p$&$\Alt(p)$&$p^{p-1}= \substack{p-2 \\1\\{\:}}$&$\frac 1 2 \Delta_{0}$&I or III\\
30&$p$&$\Alt(p)$&$p^{p-1}= \substack{1\\p-2\\{\;}}$&$\frac 1 2 \Delta_{-1}$&I  or IV\\
31&$p$&$\SL_2(p)$&$p^{p-1}=\substack{{\;}\\\VV_f\\\VV_e\\{\;}}$&$\{(u^2, u^{e}) \mid u\in \mathbb Z/p\mathbb Z^\times \}$&I, III or IV\\
&&& $e+f=p-3$,   $e$ odd&&\\
32&$p$&$\PSL_2(p)$&$p^{p-1}=\substack{ {\;}\\\VV_f\\ \VV_e\\ {\;}}$ &$\{(u^2, u^{e}) \mid u\in \mathbb Z/p\mathbb Z^\times \}$&I, III or IV\\
 &&&$e+f=p-3$, $ef\ne 0$, $e$ even&&\\
 33&$p$&$\PSL_2(p)$&$p^{p-1}=\substack{ {\;}\\\VV_{p-3}\\ \VV_0\\ {\;}}= \substack{p-2 \\1\\{\:}}$ &$\frac 1 2 \Delta_{0}$ &I or III\\
  34&$p$&$\PSL_2(p)$&$p^{p-1}=\substack{ {\;}\\\VV_{0}\\ \VV_{p-3}\\ {\;}}= \substack{1 \\p-2\\{\:}}$ &$\frac 1 2 \Delta_{-1}$&I  or IV\\
\hline
\end{tabular} }
\caption{The reduced saturated fusion systems $\F$ on maximal class $p$-groups of order at least $ p^4$, $p$ odd, with an abelian $\F$-essential subgroup of index $p$.  Where expressions are of the form $p-j$ for some natural number $j$, we require $p$ to be large enough to ensure that $p-j \ge 2$.}\label{Tab1}
\end{table}

Finally, we trim \cite[Table 2.2]{p.index2} and provide the list of realizable reduced fusion systems on maximal class $p$-groups $S$ with $\gamma_1(S)$ abelian and $|S| \ge p^4$. In this table $\nu_p(m)$ denotes that exponent of the highest power of $p$ which divides $m$. All the fusion systems not listed in Table~\ref{Tab3} but listed in Table~\ref{Tab1} are exotic.

\begin{table}[H] \resizebox{\textwidth}{!}{%
\begin{tabular}{|llllclccc|}
\hline Line&$G$&$p$&Conditions&Rank$(\gamma_1(S))$&$e$&$|\gamma_1(S)|$&$\Aut_G(\gamma_1(S))$ &Pearls\\
\hline$2$, IV&$\PSp_4(p)$&$p$&\noindent\rule{1.2cm}{0.4pt}&$3$&$1$&$p^3$&$\GL_2(p)/\{\pm I\}$&$\mathcal P_e^0(\F)$\\
$11$, III&$\Alt(p^2)$&$p$&\noindent\rule{1.2cm}{0.4pt}&$p$&$1$&$p^p$&$\frac 1 2 (p-1)\wr \Sym(p)$&$\mathcal P_a^0(\F)$\\$4$, III&
$\PSL_p(q)$&$p$&$\nu_p(q-1)=1$, $p >3$&$p-2$&$1$&$p^{p-2}$&$\Sym(p)$&$\mathcal P_a^0(\F)\cup \mathcal P_a^*(\F)$\\$34$, III&
$\PSL_p(q)$&$p$&$\nu_p(q-1)\ge 2$, $p >3$&$p-1$&$\nu_p(q-1)$&$p^{e(p-1)-1}$&$\Sym(p)$&$\mathcal P_a^0(\F)\cup \mathcal P_a^*(\F)$\\$10$, IV&
$\PSL_{p+1}(q)$&$p$&$\nu_p(q-1)=1$&$p$&$1$&$p^{p}$&$\Sym(p+1)$&$\mathcal P_e^0(\F)$\\$11$, IV&
$\mathrm P\Omega_{2p}^+(q)$&$p$&$\nu_p(q-1)= 1$&$p$&$1$&$p^p$&$ 2^{p-1}:\Sym(p)$&$\mathcal P_e^0(\F)$\\$16$, IV&
${}^2\mathrm F_4(2^{2n+1})$&$3$&$2n+1 \ge 3$&$2$&$\nu_3(q+1)$&$3^{2e}$&$\GL_2(3)$&$\mathcal P_e^0(\F)\cup \mathcal P_e^*(\F)$\\$15$, IV&
 $\mathrm E_7(q)$&$7$&$\nu_7(q-1)=1$&$7$&$1$&$7^7$&$\mathrm W(\mathrm E_7)= 2 \times \Sp_6(2)$&$\mathcal P_e^0(\F)$\\$18$, I&
 $\mathrm E_8(q)$&$5$&$\nu_5(q^2+1)\ge 1$&$5$&$\nu_5(q^2+1)$&$5^{4e}$&$  (4\circ 2^{1+4}).\Sym(6)$&$\mathcal P_a^0\cup\mathcal P_e^*(\F)$\\$3$, II&
$\mathrm{Co}_1$&$5$&\noindent\rule{1.2cm}{0.4pt}&3&$1$&$3$&$4 \times \Sym(5)$&$\mathcal P_e^0\cup\mathcal P_a^*(\F)$\\
\hline
\end{tabular}}
\caption{Realizable, reduced fusion systems $\F=\F_S(G)$ on maximal class $p$-groups $S$, $p \ge 3$, $|S|\ge p^4$ and $\gamma_1(S)$ abelian of exponent $e$. }\label{Tab3}
\end{table}

%% file: AppendixB.tex
\section{The saturated fusion systems on maximal class $3$-groups}\label{AppB}

In this appendix, we bring together the outcome of  various investigations into saturated fusion systems on maximal class $3$-groups. These results are mainly extracted from \cite{DRV,PSrank2,RV}.

There are two maximal class $3$-groups of order $3^3$, they are both extraspecial one of  exponent $3$ and one of exponent $9$.

\begin{theorem}\label{max31} Suppose that $S$ is extraspecial of order $3^3$ and $\F$ is a saturated fusion system on $S$. Then either $\F=N_\F(S)$ or $S$ has  exponent $3$, $\E_\F=\mathcal P_a(\F)$ and one of the following holds.
\begin{enumerate}
\item $\F= \F_S(3^2{:}\SL_2(3))$ or $\F_S(3^2{:}\GL_2(3))$;
\item $\F= \F_S(\PSL_3(3))$ or $\F=\F_S(\PSL_3(3){:}2)$;
\item $\F=\F_S({}^2\mathrm F_4(2)')$; or
\item $\F=\F_S(\J_4)$.
\end{enumerate}
\end{theorem}
\begin{proof} The extraspecial group of order $3^3$ and exponent $9$ is meatacyclic. Hence,  in this case, $\F= N_\F(S)$ by \cite[Proposition 5.4]{Stancu1}
(see also \cite[Theorem 7.5]{Craven}). If $S$ is extraspecial of exponent 3, we read the result from \cite[Theorem 1.1]{RV}.
\end{proof}

The classification of maximal class $3$-groups is due to Blackburn  \cite{black}.   We take their presentations  from \cite[Theorem A.2]{DRV}. For $n \ge 4$, and $\beta, \gamma, \delta \in \{0,1,2\}$, define $$\mathrm B(n; \beta, \gamma, \delta)=\langle x, s_1,\dots,s_{n-1}\mid \textbf{R1},\textbf{R2}, \textbf{R3}, \textbf{R4}, \textbf{R5}, \textbf{R6}\rangle$$ where, understanding that $s_{n}=s_{n+1}=1$, the relations are as follows:
\begin{enumerate}
\item[\textbf{R1}:] $s_i=[s_{i-1},x]$ for $i\in\{2, \dots, n-1\}$;
\item [\textbf{R2}:] $[s_1,s_i]=1$ for $i\in\{3, \dots, n-1\}$;
\item [\textbf{R3}:] $s_i^3s_{i+1}^3s_{i+2} = 1$ for $i \in \{2, \dots, n-1\}$;
\item [\textbf{R4}:] $[s_1,s_2]=s_{n-1}^\beta$;
\item [\textbf{R5}:] $s_1^3s_2^3s_3^{}=s_{n-1}^\gamma$; and
\item [\textbf{R6}:] $x^3=s_{n-1}^\delta$.
\end{enumerate}

\begin{theorem}[Blackburn]\label{thm:black} Suppose that $S$ is a maximal class $3$-group of order  $3^n$ with $n \ge 4$. Then $S\cong \mathrm{B}(n;\beta, \gamma, \delta)$ for  some $\beta, \gamma, \delta\in\{0,1,2\}$. Furthermore, $S$ is metabelian and, unless $S \cong \B(4;0,1,0)$, $S$ has rank $2$.
\end{theorem}

\begin{proof}   This comes from the text before and after \cite[Theorems 4.2 and 4.3]{black}.   \end{proof}

There are isomorphisms between some of the Blackburn group. Using the discussion after \cite[Theorem 4.3]{black},
the full list of groups of order $3^4$ is given as $\mathrm{B}(4;0,\gamma,\delta)$  where   $$(0,\gamma, \delta) \in\Sigma_4=\{(0,1,0),(0,2,0),(0,0,0), (0,0,1)\}.$$ Since these groups have order $3^4$, they each have  an abelian subgroup of index $3$.  The group  $\mathrm{B}(4;0,1,0)$ is the unique maximal class $3$-group of rank $3$. Thus $\mathrm{B}(4;0,1,0)$ is isomorphic to a Sylow $3$-subgroup of $\Sym(9)$. For $n \ge 5$, the groups with no abelian maximal subgroups are given by $$(\beta,\gamma, \delta)\in \Theta =\{(1,0,0), (1,0,1), (1,0,2)\}.$$
With this we can write down the a full irredundant list of maximal class $3$-groups of order at least $3^5$:
\begin{enumerate}
\item for $n$ odd, $$(\beta,\gamma, \delta) \in \Theta \cup\{(0,1,0), (0,0,1), (0,0,0)\}.$$
\item for $n$ even, $$(\beta,\gamma, \delta) \in \Theta \cup\Sigma_4.$$
\end{enumerate}
So, for $n\ge 5$, there are  six maximal class $3$-groups when $n$ is odd, and  seven when $n$ is even.

\begin{lemma}\label{max32}
Assume that $S =\mathrm{B}(4;0,1,0) $ is a Sylow $3$-subgroup of $\Sym(9)$ and $\F$ is a saturated fusion system  on $S$ with $\F\ne N_\F(S)$. Set $A= \langle x,s_3\rangle$ and $E=\langle s_2,A\rangle$. Then $\gamma_1(S)=\langle s_1,s_2,s_3\rangle$ and $\F$ is described as follows:
\begin{enumerate}
\item $\E_\F=\{\gamma_1(S)\}$, $\F= N_\F(\gamma_1(S))$ and $\Aut_\F(\gamma_1(S))\cong \Frob(39)$, $ \Frob(39) \times 2$, $\Alt(4)$,  $2\times \Alt(4)$, $\Sym(4)$ two different actions,  or $2\times \Sym(4)$;
    \item $\E_\F= \mathcal P_e(\mathcal F)=\{E\}$, $\F=N_\F(E)$ and either $\Aut_\F(E)\cong {3^{2}}{:}\SL_3(3)$ or $3^{2}{:}\GL_2(3)$;
     \item $\E_\F=  \{E,\gamma_1(S)\}= \mathcal P_e(\F) \cup \{\gamma_1(S)\}$ and $\F= \F_{S}(G)$ with $G \cong \PSp_4(3)$ or $\PSp_4(3){:}2$;
     \item $\E_\F= A^\F\cup\{\gamma_1(S)\}=\mathcal P_a(\F) \cup \{\gamma_1(S)\}$ and $\F= \F_{S}(G)$ with $G \cong \Alt(9)$ or $\Sym(9)$;
        or
\item $\E_\F= \mathcal P_a(\F)= A^\F$,   $\Aut_\F(A)\cong \SL_3(3)$ or $\GL_2(3)$,   and $O^{3'}(\F)$ is simple and exotic.
\end{enumerate}
In particular, $\E_\F\subseteq \mathcal P(\F)\cup\{ \gamma_1(S)\}$.
\end{lemma}

 \begin{proof} The examples have been enumerated  by computer using the procedures from \cite{parkersemerarocomputing}. The code is in Subsection \ref{CB3}.  This confirms that there are 15 examples. \end{proof}

\begin{theorem}\label{max33} Suppose that $S= \mathrm{B}(n;\beta, \gamma, \delta)$ is a maximal class $3$-group of rank $2$ and order $3^n$ with $n \ge 4$. Assume $\F$ is a saturated fusion system on $S$ with $\F \ne N_\F(S)$. Then  $$S \cong \begin{cases}\mathrm{B}(2k;0,0,0),   \mathrm{B}(2k;0,2,0)&k \ge 2\cr
  \mathrm{B}(2\ell;0,1 , 0),\mathrm{B}(2\ell ;1,0,0),\mathrm{B}(2\ell;1,0,2)& \ell \ge 3\cr
\mathrm{B}(2k+1;0,0 , 0) , \mathrm{B}(2k+1;1,0,0)&k\ge 2 \end{cases}$$ and $\F$ is as described in Table~\ref{this one}.
Furthermore, for each of the fusion systems tabulated,  $\mathcal E_\F\subseteq \mathcal P(\F)\cup\{\gamma_1(S)\}$ and, if $\gamma_1(S)\in \E_\F$, then $S \cong \mathrm{B}(2k+1;0,0,0)$ with $k \ge 2$ and $\gamma_1(S)$ is homocyclic of rank $2$.
\end{theorem}

\begin{proof} This is a compilation of \cite[Theorem 5.10]{DRV} and \cite[Theorem 1.1]{PSrank2}. The final statement is obtained by inspection of  Table \ref{this one}. \end{proof}

\begin{table}[H]\label{Tab5}
\resizebox{\textwidth}{!}{
\begin{tabular}{|c|c|ccc|ccc|c|c|}
\hline
 Group&$|\Out_\F(S)|$&$\Out_\F(A_0)$&$\Out_\F(A_1)$&$\Out_\F(A_{-1})$&$\Out_\F(E_0)$& $\Out_\F(E_1)$&$\Out_\F(E_{-1})$&$\Out_\F(\gamma_1)$&Example\\

 \hline
$\mathrm{B}(2k;0,0,0)$&$2 $&$\SL_2(3)$& & &&&&&$\F_{\mathrm{DRV}}(3^{2k},1)$\\\hdashline
$\mathrm{B}(2k;0,0,0)$&$2$&&$\SL_2(3)$&$\SL_2(3)$&&&&&$\F_{\mathrm{DRV}}(3^{2k},2)$\\
\hdashline
$\mathrm{B}(2k;0,0,0)$&$2$&$\SL_2(3)$&$\SL_2(3)$&$\SL_2(3)$&&&&&$ \PSL_3(q_1)$\\\hdashline
$\mathrm{B}(2k;0,0,0)$&$2$&&&&$\SL_2(3)$&&&&$3^{\cdot}\PGL_3(q_2)$, $k>2$\\
 & &&&&  &&&&$N_\F(E_0)$, $k=2$\\\hdashline
$\mathrm{B}(2k;0,0,0)$&$2^2$&$\GL_2(3)$&& &&&&&$\F_{\mathrm{DRV}}(3^{2k},1).2$\\\hdashline
$\mathrm{B}(2k;0,0,0)$&$2^2$&&$\SL_2(3)$&$A_1\sim_\F A_{-1}$&&&&&$\F_{\mathrm{DRV}}(3^{2k},2).2$\\\hdashline

$\mathrm{B}(2k;0,0,0)$&$2^2$&$\GL_2(3)$&$\SL_2(3)$&$A_1\sim_\F A_{-1}$&&&&&$\PSL_3(q_1).2$\\\hdashline
$\mathrm{B}(2k;0,0,0)$&$2^2$& & & &$\GL_2(3)$&&&&$3^{\cdot}\PGL_3(q_2).2$, $k >2$\\
 & & & & & &&&&$N_\F(E_0)$, $k=2$\\\hdashline
$\mathrm{B}(2k;0,0,0)$&$2^2$& &$\SL_2(3)$ &$A_1\sim_\F A_{-1}$ &$\GL_2(3)$&&&&${}^3\mathrm D_4(q_3)$\\
\hline

$\mathrm{B}(2k;0,2,0)$&$2$&  $\SL_2(3)$ &$*$&$*$ &&$*$&$*$&&$\F_{\mathrm{DRV}}(3^{2k},4)$, $k>2$\\

 & &  & &  && & &&$\F_{\mathrm{DRV}}(3^{4},3)$, $k=2$\\\hdashline

$\mathrm{B}(2k;0,2,0)$&$2$& &  $*$&$*$ &$\SL_2(3)$&$*$&$*$&&$3^{\cdot}\PGL_3(q_2)$, $k>2$\\
&& &   &  & & & &&$N_\F(E_0)$, $k=2$\\\hdashline

$\mathrm{B}(2k;0,2,0)$&$2^2$&  $\GL_2(3)$ &$*$&$*$ &&$*$&$*$&&$\F_{\mathrm{DRV}}(3^{2k},4).2$\\

 & &  & &  && & &&$\F_{\mathrm{DRV}}(3^{4},3).2$, $k=2$\\\hdashline

$\mathrm{B}(2k;0,2,0)$&$2^2$& & $*$&$*$ &$\GL_2(3)$&$*$&$*$&&$3^{\cdot}\PGL_3(q_2).2$, $k>2$\\
&& &   &  & & & &&$N_\F(E_0)$, $k=2$\\

\hline

$\mathrm{B}(2\ell;0,1,0)$&$2$&  $\SL_2(3)$ &$*$&$*$ &&$*$&$*$&&$\F_{\mathrm{DRV}}(3^{2\ell},3)$\\\hdashline

$\mathrm{B}(2\ell;0,1,0)$&$2$& & $*$&$*$&$\SL_2(3)$&$*$&$*$&&$3^{\cdot}\PGL_3(q_2)$\\\hdashline

$\mathrm{B}(2\ell;0,1,0)$&$2^2$&  $\GL_2(3)$ &$*$&$*$ &&$*$&$*$&&$\F_{\mathrm{DRV}}(3^{2\ell},3).2$\\\hdashline

$\mathrm{B}(2\ell;0,1,0)$&$2^2$& &  $*$&$*$ &$\GL_2(3)$&$*$&$*$&&$3^{\cdot}\PGL_3(q_2).2$\\
\hline

$\mathrm{B}(2\ell;1,0,2)$&$2$&$*$&$\SL_2(3)$ &&$*$&&&&$\F (3^{2\ell},7)$\\\hdashline
$\mathrm{B}(2\ell;1,0,2)$&$2$&$* $&&$\SL_2(3)$ &$*$&&&&$\F (3^{2\ell},8)$\\\hdashline
$\mathrm{B}(2\ell;1,0,2)$&$2$&$*$&$\SL_2(3)$&$\SL_2(3)$& $*$ &&&&$\F (3^{2\ell},9)$\\\hline
$\mathrm{B}(2\ell;1,0,0)$&$2$&$\SL_2(3)$& $*$&$*$ &&$*$&$*$&&  $\F (3^{2\ell},6)$\\
\hline
$\mathrm{B}(2k+1;0,0,0)$&$2$& &&&$\SL_2(3)$&&&&$3^{\cdot}\F_{\mathrm{DRV}}(3^{2k},1)$\\\hdashline
$\mathrm{B}(2k+1;0,0,0)$&$2$&&&&&$\SL_2(3)$&$\SL_2(3)$&&
$3^{\cdot}\F_{\mathrm{DRV}}(3^{2k},2)$\\\hdashline
$\mathrm{B}(2k+1;0,0,0)$&$2$&&&  &$\SL_2(3)$&$\SL_2(3)$&$\SL_2(3)$&&$\SL_3(q_1)$\\\hdashline
 $\mathrm{B}(2k+1;0,0,0)$&$2$&&&&&&&$\SL_2(3)$&$N_\F(\gamma_1(S))$\\\hdashline
$\mathrm{B}(2k+1;0,0,0)$&$2$&$\SL_2(3)$&& &&&&&$\PGL_3(q_1)$\\\hdashline

$\mathrm{B}(2k+1;0,0,0)$&$2^2$&  &  &&&&&$\GL_2(3)$&$N_\F(\gamma_1(S))$\\\hdashline

 $\mathrm{B}(2k+1;0,0,0)$&$2^2$&&&&&$\SL_2(3)$&$E_1\sim_\F E_{-1}$&&$3.\F_{\mathrm{DRV}}(3^{2k},2).2$\\\hdashline

 $\mathrm{B}(2k+1;0,0,0)$&$2^2$&&&&&$\SL_2(3)$&$E_1\sim_\F E_{-1}$&$\GL_2(3)$&$\F_{\mathrm{DRV}}(3^{2k+1},1)$\\\hdashline

$\mathrm{B}(2k+1;0,0,0)$&$2^2$& & & &$\GL_2(3)$&&&&$3.\F_{\mathrm{DRV}}(3^{2k},1).2$\\\hdashline
$\mathrm{B}(2k+1;0,0,0)$&$2^2$& & & &$\GL_2(3)$&&&$\GL_2(3)$&$\F_{\mathrm{DRV}}(3^{2k+1},2)$\\\hdashline

$\mathrm{B}(2k+1;0,0,0)$&$2^2$& & & &$\GL_2(3)$&$\SL_2(3)$&$E_1\sim_\F E_{-1}$&&$\SL_3(q_1).2$\\\hdashline
$\mathrm{B}(2k+1;0,0,0)$&$2^2$& & & &$\GL_2(3)$&$\SL_2(3)$&$E_1\sim_\F E_{-1}$&$\GL_2(3)$&${}^2\mathrm F_4(q_4)$\\\hdashline

$\mathrm{B}(2k+1;0,0,0)$&$2^2$&$\GL_2(3)$& & &&&&&$ \PGL_3(q_1).2$\\\hdashline
$\mathrm{B}(2k+1;0,0,0)$&$2^2$&$\GL_2(3)$& & &&&&$\GL_2(3)$&$\F_{\mathrm{DRV}}(3^{2k+1},3)$\\\hdashline
 $\mathrm{B}(2k+1;0,0,0)$&$2^2$&$\GL_2(3)$&&&&$\SL_2(3)$&$E_1\sim_\F E_{-1}$ &&
$\F _{\mathrm{DRV}}(3^{2k+1},5)$\\\hdashline
 $\mathrm{B}(2k+1;0,0,0)$&$2^2$&$\GL_2(3)$&&&&$\SL_2(3)$&$E_1\sim_\F E_{-1}$ &$\GL_2(3)$&
$\F_{\mathrm{DRV}}(3^{2k+1},4)$\\
\hline
$\mathrm{B}(2k+1;1,0,0)$&$2$&$\SL_2(3)$& $*$&$*$ &&$*$&$*$&& $ \PSL_3(q^3_1).3$\\
\hline

\end{tabular} }
\caption{The saturated fusion systems on maximal class $3$-groups of rank $2$ and order at least $3^4$. }\label{this one}
\end{table}

In Table~\ref{this one}, $q_1$, $q_2$, $q_3$ and $q_4$ are prime powers with $\nu_3(q_1-1)=\nu_3(q_4^2-1)=k$ and $\nu_3(q_2-1)=\nu_3(q_3^2-1)=k-1$ where  $\nu_3(m)$ denotes that exponent of the highest power of $3$ which divides $m$. Of course, $q_4$ is   an odd power of $2$.

To understand the data presented in  Table \ref{this one},  let $S=\mathrm{B}(n;\beta,\gamma,\delta)$ with $n \ge 4$ and assume that $S$ has rank $2$. Let $\F$ be a saturated fusion system on $S$.
It is easy to see that $Z(S)= \langle s_{n-1}\rangle$, $Z_2(S)=\langle s_{n-2},s_{n-1}\rangle$ and $\gamma_1(S)=\langle s_i\mid 1\le i \le n-1\rangle$. For $i\in \{0,1,-1\}$ define $$A_i=\langle xs_1^i,s_{n-1}\rangle,$$ and $$E_i= \langle A_i,s_{n-2}\rangle.$$
When the subgroups $A_i$, $E_i$ have exponent $3$, then up to $S$-conjugacy they are the candidates to be $\F$-pearls. If there is a $*$ in  Table \ref{this one}, this indicates that for the given $S$, the corresponding subgroups $A_i$ and $E_i$ do not have exponent $3$ (see \cite[Table 1]{PSrank2}).  We do not make this indication for $\gamma_1(S)$ as it does not have exponent $3$. The notation $A_{1}\sim_\F A_{-1}$, $E_{1}\sim_\F E_{-1}$ means that these subgroups are $\F$-conjugate. We only show this in the case when $A_{1}$ or $E_{1}$ is an $\F$-pearl. Reading across a row of  Table \ref{this one}, first comes the Blackburn name of $S$, then  $|\Out_\F(S)|$. As we move a long the row, an entry displays the structure of  $\Out_\F(A_i)$, $\Out_\F(E_i)$ or $\Out_\F(\gamma_1(S))$  in the case that the corresponding subgroup  is $\F$-essential. A group or $N_\F(E_0)$ in the final column indicates that the fusion system is realizable, all the other tabulated fusion systems are exotic.  Where the  exotic fusion system is in \cite[Theorem 5.10]{DRV}, this is indicated by a subscript DRV and we have adhered to their names.

 The fusion systems $\F_{\mathrm{DRV}}(3^{2k+1},5)$ are exotic and incorrectly labeled in \cite[Table 6]{DRV} (hence the strange numbering of the system).  The fusion systems on $S=\mathrm{B}(2k+1;1,0,0)$ included as the  final row of Table~\ref{this one} are claimed to be exotic in \cite{PSrank2}.  In fact they can be constructed by the methods in Example~\ref{ExThmNonExc}.

\begin{theorem}\label{outcome}Suppose that $\F$ is a saturated fusion system on a non-abelian maximal class $3$-group. Then $\F$ is known, $\E_\F\subseteq \mathcal P(\F)\cup\{ \gamma_1(S)\}$ and, if  $\gamma_1(S)\in \E_\F$, then $\gamma_1(S)$ is abelian.\end{theorem}

\begin{proof} Combine Theorems~\ref{max31} and \ref{max33} and  Lemma~\ref{max32}. \end{proof}

%% file: AppendixC.tex
\section{Computer code}\label{AppC}

In this appendix we provide the {\sc Magma} code that we have used   for the work. We have used the functionality developed in \cite{parkersemerarocomputing}. To use these subroutines, the intrinsics provide in \cite{PSGIT} must be available. For this you will be required to attach the package file following the instructions provided here
\begin{center} \begin{verbatim} https://magma.maths.usyd.edu.au/magma/handbook/text/24#168.\end{verbatim}
\end{center}

\subsection{The {\sc Magma} code for Example~\ref{5^7thm}}\label{C1.2}

This code requires the fusion system package written by Parker and Semeraro \cite{parkersemerarocomputing, PSGIT}.

\begin{verbatim}

X:= ASL(2,5); S5:= Sylow(X,5);
BA:= Normalizer(X,S5);

Y:= PSp(4,5); S5:= Sylow(Y,5);
B:= Normalizer(Y,Centre(S5));
B1:= DerivedSubgroup(B);
BE:=Normalizer(B1,S5);

// BA is isomorphic to the normalizer of
//a Sylow 5-subgroup of 5^2:\SL_2(5) and BE
//is isomorphic to the  normalizer of a
//Sylow 5-subgroup of 5^{1+2}:\SL_2(5)

AP:=[];
EP:=[];

SS:= SmallGroups(5^7,IsMaximalClass);
#SS;

SSA:=[];
for x in SS do
	L:= LowerCentralSeries(x);
	G1:=Centralizer(x,L[2],L[4]);
	if not IsAbelian(G1) then Append(~SSA,x);
	end if;
end for;
#SSA;

for x in SSA do MakeAutos(x); end for;
SSAA:= [x : x in SSA |#Sylow(x`autoperm,2) eq 4];

#SSAA;

for X in SSAA do
	L:= LowerCentralSeries(X);
	G1:=Centralizer(X,L[2],L[4]);
	AP[Index(SSAA,X)] :=[];
	EP[Index(SSAA,X)] :=[];
 	SAP:= Subgroups(X:OrderEqual:= 25);
 	SAP:= {x`subgroup: x in SAP| not x`subgroup subset G1};
 	SEP:=  Subgroups(X:OrderEqual:= 125);
  	SEP:= {x`subgroup: x in SEP|not IsAbelian(x`subgroup)
  		and not x`subgroup subset G1 and Exponent(x`subgroup) eq 5};
	H:= Sylow(X`autoperm,2);
	H:= SubInvMap(X`autopermmap,X`autogrp,H);
	AFS:= sub<X`autogrp|Inn(X),H>;
	SAP:={x: x in SAP|IsOdd(#AutOrbit(X,x,AFS))};
	SEP:={x: x in SEP|IsOdd(#AutOrbit(X,x,AFS))};
	
	for x in SAP do
		NSx:= Normalizer(X,x);
		A, B, C:= AutOrbit(X,x,AFS);
		MakeAutos(NSx);	
		W:= Sylow(SubMap(X`autopermmap,X`autoperm,B),2);
		K:= SubInvMap(X`autopermmap,X`autogrp,W);
		K:= sub<NSx`autogrp|K>;
		B := Holomorph(NSx,K);
		a:= IsIsomorphic(B,BA);
		if a then Append(~AP[Index(SSAA,X)] ,x); end if;
	end for;

	for x in SEP do
		NSx:= Normalizer(X,x);
		A, B, C:= AutOrbit(X,x,AFS);
		MakeAutos(NSx);	
		W:= Sylow(SubMap(X`autopermmap,X`autoperm,B),2);
		K:= SubInvMap(X`autopermmap,X`autogrp,W);
		K:= sub<NSx`autogrp|K>;
		B := Holomorph(NSx,K);
		a:= IsIsomorphic(B,BE);
		if a then Append(~EP[Index(SSAA,X)] ,x); end if;
	end for;
end for;

IA:={i:i in [1..#SSAA]|#AP[i] ne 0};
IE:={i:i in [1..#SSAA]|#EP[i] ne 0};
print "there are", #{x:x in AP|#x ne 0},
"maximal class groups of order 5^7
with abelian pearls.";

print "there are", #{x:x in EP|#x ne 0},
"maximal class groups of order 5^7
with extraspecial pearls.";

JJ:={1297,1308,1321,1360,1363,1374,1384};
for i in IA do
	for j in JJ do
		if IsIsomorphic(SmallGroup(5^7,j),SSAA[i]) then
			i,j;end if;
	end for;
end for;

for i in IE do
	for j in JJ do
		if IsIsomorphic(SmallGroup(5^7,j),SSAA[i]) then
			i,j; end if;
	end for;
end for;
JJ:={1297,1308,1321,1360,1363,1374,1384};
A:=[];
for j in JJ do
A[j]:= AllFusionSystems(SmallGroup(5^7,j):
                OpTriv:=false, pPerfect:=false);
end for;
for j in JJ do
	if IsDefined(A,j) then A[j]; end if;
end for;

for j in JJ do
	S:= SmallGroup(5^7,j);
	L:= LowerCentralSeries(S);
	G1:=Centralizer(S,L[2],L[4]);
	j, NilpotencyClass(G1);
end for;

\end{verbatim}
\subsection{The {\sc Magma} code required to verify Lemma~\ref{Symmetric(p)Module}}\label{CSym(p)}

This code verifies that when $p\in \{5,7\}$, the exterior square of the module $S^{p-1,1}$ for $\GF(p)\Sym(p)$ has two composition factors and no quotient factor which is $1$-dimensional or isomorphic to $D^{p-1,1}$.
\begin{verbatim}
for p in {5,7} do
G:= Sym(p);
V:= PermutationModule(G,GF(p));
W:= sub<V|V.1-V.1*G.1>;
DPminus11:= W/Socle(W);
U:= ExteriorSquare(W);
C:=CompositionFactors(U);
C;
IsIrreducible(U/Socle(U));
IsIsomorphic(U/Socle(U),DPminus11);
IsIsomorphic(Socle(U),DPminus11);
end for;
\end{verbatim}

\subsection{The {\sc Magma}  code for Lemma~\ref{prop:CSZ2 essential}}\label{subsec:CSZ2}
This code requires the fusion system package written by Parker and Semeraro \cite{parkersemerarocomputing,PSGIT}.

\begin{verbatim}
SS:= SmallGroups(5^6,IsMaximalClass);
#SS;
TT:=[];
for S in SS do
	L:= LowerCentralSeries(S);
	gamma1:= Centralizer(S,L[2],L[4]);
	CSZ2:= Centralizer(S,L[4]);
	if gamma1 ne CSZ2 then   Append(~TT,S);
	end if;
end for;
#TT;

TT2:=[];
for S in TT do
	L:= LowerCentralSeries(S);
	gamma1:= Centralizer(S,L[2],L[4]);
	CSZ2:= Centralizer(S,L[4]);
	MakeAutos(CSZ2);
	if not IsExtraSpecial(gamma1) and
        not IsSoluble(CSZ2`autoperm) then
		Append(~TT2,S);
	end if;
end for;
#TT2;

TT3:=[];
for S in TT2 do MakeAutos(S);
	if not IsNilpotent(S`autoperm) then  Append(~TT3,S);
	end if;
end for;

#TT3;
S:= TT3[1];
A:=AllFusionSystems(S:OpTriv:= false,pPerfect:= false);
A;
\end{verbatim}

\subsection{The group providing an example for Lemma~\ref{gamma1-essC}}\label{CEx}
This code constructs a group $G$ of shape $7^{3+3}{:}\PGL_2(7)$ which has Sylow $7$-subgroup $S$ of maximal class and $\gamma_1(S)$ special with centre of order $7^3$. It further shows that $S$ cannot support a saturated fusion system $\mathcal G$ containing $\mathcal G$-pearls.
\begin{verbatim}
G:= FreeGroup(5);
R<r1,r2,r3,s,t>:= FreeGroup(5);
R:= quo<R|r1^7,r2^7,r3^7,((r1,r2),r1),((r1,r2),r2),
((r1,r2),r3),((r1,r3),r1),((r1,r3),r2),((r1,r3),r3),
((r2,r3),r1),((r2,r3),r2),((r2,r3),r3),
s^3, t^2, (s,t)^4,(t*s)^8,
r1^s = r1^5*r2^2*(r1,r2)^2*(r1,r3)*(r2,r3)^3,
r2^s =r1*r2*r3*(r1,r2)^5*(r1,r3)^4*(r2,r3)^2,
r3^s = r1^4*r2^2*r3*(r1,r3)^3*(r2,r3)^6,
r1^t = r2*r3*(r1,r2)^3*(r1,r3)^2*(r2,r3),
r2^t =r1^5*r2^3*r3^2*(r1,r2)^4*(r1,r3)^3*(r2,r3)^5,
r3^t =r1^3*r2^4*r3^5*(r1,r2)*(r1,r3)^5*(r2,r3)^4>;
G:= CosetImage(R,sub<R|s,t>);
ChiefFactors(G);
S:=Sylow(G,7);
IsMaximalClass(S);

//We now check if G can be decorated with pearls

p:= FactoredOrder(S)[1][1];
L:= LowerCentralSeries(S);
gamma1:= Centralizer(S,L[2],L[4]);
Z2:= L[#L-2];
R:= Centralizer(S,Z2);

M:= MaximalSubgroups(S);
PotAbelianPearls:=[];
PotExtraspecialPearls:=[];
for x in M do  y:= x`subgroup;
	if y ne gamma1 and y ne R then
		if exists(xx){z:z in y|
            not z in gamma1 and Order(z) eq p and not z in R}
		then
        Append(~PotExtraspecialPearls,sub<S|Z2,xx>);
        Append(~PotAbelianPearls,sub<S|L[#L-1],xx>);
		end if;
	end if;
end for;

J1:= Centralizer(PSp(4,p),Centre(Sylow(PSp(4,p),7)));
BE:= Normalizer(J1,Sylow(J1,p));
BA:= Normalizer(ASL(2,p),Sylow(ASL(2,p),p));

 for x in PotExtraspecialPearls do
 	 if IsIsomorphic(Normalizer(G,x),BE) then
 	 	"S can be decorated with extraspecial pearls";
 	 	else "this potential extraspecial pearl,
                is not a pearl."; end if;
 end for;

 for x in PotAbelianPearls do
 	 if IsIsomorphic(Normalizer(G,x),BA) then
 	 	"S can be decorated with abelian pearls";
 	 	else "this potential abelian pearl,
                is not a pearl."; end if;
 end for;
\end{verbatim}

\subsection{The {\sc Magma}  code for Lemma~\ref{max32}}\label{CB3}

This code requires the fusion system package written by Parker and Semeraro \cite{parkersemerarocomputing, PSGIT}.

\begin{verbatim}
S:= Sylow(Sym(9),3);
A:= AllFusionSystems(S:OpTriv:= false,pPerfect:= false);
#A;
\end{verbatim}